\documentclass[11pt,a4paper]{amsart}
\usepackage{relsize}
\usepackage{amssymb}
\usepackage{amsfonts}
\usepackage{amsthm}
\usepackage{amsmath}
\usepackage{mathtools}
\usepackage{amscd}
\usepackage[utf8]{inputenc}
\usepackage{t1enc}
\usepackage[mathscr]{eucal}
\usepackage{indentfirst}
\usepackage{graphicx}
\usepackage{graphics}
\usepackage{wasysym}
\usepackage{pict2e}
\usepackage{lmodern}    
\usepackage{ifthen}  
\usepackage{epic}
\numberwithin{equation}{section}
\usepackage[margin=3.6cm]{geometry}
\usepackage{epstopdf} 
\usepackage{tikz-cd}
\tikzcdset{scale cd/.style={every label/.append style={scale=#1},
    cells={nodes={scale=#1}}}}
\usetikzlibrary{matrix,arrows,decorations.pathmorphing}
\usepackage{yhmath}
\author{Bence Hevesi}
\email{bencehevesi@math.cnrs.fr}
\address{Université Paris-Saclay\\ CNRS, Laboratoire de mathématiques d’Orsay\\ 91405, Orsay, France}

\usepackage{mathrsfs}
\usepackage{lipsum}

\usepackage[
backend=biber,
style=alphabetic,
maxnames=8
]{biblatex}

\usepackage{color}   
\usepackage{hyperref}
\hypersetup{
    colorlinks=false, 
    linktoc=all,     
    linkcolor=blue,  
}
 \usepackage{url}
 \usepackage{textcomp}

\theoremstyle{plain}
\newtheorem{Th}{Theorem}[section]
\newtheorem{Lemma}[Th]{Lemma}
\newtheorem{Cor}[Th]{Corollary}
\newtheorem{Prop}[Th]{Proposition}
\newtheorem{hyp}[Th]{Hypothesis}

\theoremstyle{definition}

\newtheorem{Def}[Th]{Definition}
\newtheorem{Conj}[Th]{Conjecture}
\newtheorem{Rem}[Th]{Remark}
\newtheorem{?}[Th]{Problem}
\newtheorem{Ex}[Th]{Example}
\newtheorem{assumption}[Th]{Assumption}

\newcommand{\Hom}{{\rm{Hom}}}

\newenvironment{claim}[1]{\par\noindent\underline{Claim:}\space#1}{}

\begin{document}
\title{Ordinary parts and local-global compatibility at $\ell=p$}
\makeatletter
\@namedef{subjclassname@2020}{%
 \textup{2020} Mathematics Subject Classification}
\makeatother

\begin{abstract}
We prove local-global compatibility results at $\ell=p$ for the torsion automorphic Galois representations constructed by Scholze, generalising the work of Caraiani--Newton. In particular, we verify, up to a nilpotent ideal, the local-global compatibility conjecture at $\ell=p$ of Gee--Newton in the case of imaginary CM fields under some technical assumptions.

The key new ingredient is a local-global compatibility result for $Q$-ordinary self-dual automorphic representations for arbitrary parabolic subgroups.
\end{abstract}
\maketitle

\tableofcontents

\section{Introduction}
In this article, we are concerned with the question of local-global compatibility of the global Langlands correspondence at $\ell=p$ for $\textnormal{GL}_n$ over imaginary CM number fields. Verifying local-global compatibility at $\ell=p$, among other things, proved to be essential in establishing modularity lifting results. In the case of imaginary CM fields, this problem has seen a breakthrough when the authors of \cite{ACC23} laid out and executed a strategy to attack it. Another breakthrough came very recently with the novel work of Caraiani--Newton. In \cite{CN23}, they introduced many new ideas to the method, yielding a robust argument, and significantly stronger results. In this article, we follow and seemingly push the ideas of \cite{CN23} close to their limit. As a result, in characteristic $0$ in the residually irreducible case, under some genericity assumption, we prove that the automorphic Galois representations constructed in \cite{HLTT16}, \cite{Sch15} satisfy the expected local-global compatibility at $\ell=p$, up to only bounding the nilpotence of the monodromy (see Theorem~\ref{Th1.1.1} below). Moreover, our results go beyond characteristic $0$ automorphic Galois representations and yield local-global compatibility for the torsion automorphic Galois representations constructed in \cite{Sch15} as well, vastly generalising the results of \cite{ACC23}, and \cite{CN23} (see Theorem~\ref{Th1.1.4}).

We now describe our results in more detail. Let $F$ be a number field and fix a prime number $p$. Consider a regular algebraic cuspidal automorphic representation (RACAR) $\pi$ of $\textnormal{GL}_n(\mathbf{A}_F)$ and an isomorphism $t:\overline{\mathbf{Q}}_p\xrightarrow{\sim} \mathbf{C}$. Langlands reciprocity predicts the existence of a continuous semisimple Galois representation
\begin{equation*}
    r_t(\pi):\textnormal{Gal}(\overline{F}/F)\to \textnormal{GL}_n(\overline{\mathbf{Q}}_p)
\end{equation*}
determined by the property that, for each finite place $v\notin S:=\{v\mid \pi_v \textnormal{ is ramified}\}\cup \{v|p\}$ in $F$, $r_t(\pi)|_{\textnormal{Gal}(\overline{F}_v/F_v)}$ is unramified, admitting an isomorphism
\begin{equation*}
    \textnormal{WD}(r_t(\pi)|_{\textnormal{Gal}(\overline{F}_v/F_v)})^{F-ss}\cong t^{-1}\textnormal{rec}(\pi_v\otimes |\det|_v^{\frac{1-n}{2}}).
\end{equation*}
Moreover, $r_t(\pi)$ is expected to satisfy local-global compatibility also at places in $S$ governed by the shape of $\pi$ at $S$ and at $\infty$. In particular, for a place $v$ of $F$ above $p$, it is predicted that $r_t(\pi)$ is de Rham at $v$ with labelled Hodge--Tate weights that can be read off from the infinitesimal character of $\pi_{\infty}$, and it further admits an isomorphism
\begin{equation*}
    \textnormal{WD}(r_t(\pi)|_{\textnormal{Gal}(\overline{F}_v/F_v)})^{F-ss}\cong t^{-1}\textnormal{rec}(\pi_v\otimes |\det|_v^{\frac{1-n}{2}}).
\end{equation*}
Assuming that $F$ is an imaginary CM field, the representations $r_t(\pi)$ have been constructed in \cite{HLTT16}, and \cite{Sch15}. We prove the following, making progress on verifying local-global compatibility at $p$ for the automorphic Galois representations of \cite{HLTT16}.
\begin{Th}\label{Th1.1.1}
    Let $F$ be an imaginary CM field, $t:\overline{\mathbf{Q}}_p\xrightarrow{\sim} \mathbf{C}$ be any isomorphism, and $\pi$ be a regular algebraic cuspidal automorphic representation of $\textnormal{GL}_n(\mathbf{A}_F)$ of weight $\lambda \in (\mathbf{Z}_+^n)^{\Hom(F,\mathbf{C})}$. Assume that
    \begin{itemize}
        \item $\overline{r_t(\pi)}$ is irreducible and decomposed generic in the sense of \cite{CS19} (see Definition~\ref{Def2.25}).
    \end{itemize}
    Then, for any $p$-adic place $v| p$ of $F$, and $\iota:F_v\hookrightarrow \overline{\mathbf{Q}}_p$, $r_t(\pi)$ is de Rham at $v$, with labelled $\iota$-Hodge--Tate weights $\lambda_{t\circ \iota,n}<...<\lambda_{t\circ \iota,1}+n-1$ and we have
    \begin{equation}\label{eq1.1.1}
        \textnormal{WD}(r_t(\pi)|_{\textnormal{Gal}(\overline{F}_v/F_v)})^{F-ss}\preceq t^{-1}\textnormal{rec}(\pi_v\otimes |\det|_v^{\frac{1-n}{2}}).
    \end{equation}
\end{Th}
\begin{Rem}
    For the notion of "$\preceq$" in \ref{eq1.1.1}, see \S\ref{sec2.6}. It in particular means that we have an isomorphism
    \begin{equation*}
        \textnormal{WD}(r_t(\pi)|_{\textnormal{Gal}(\overline{F}_v/F_v)})^{ss}\cong (t^{-1}\textnormal{rec}(\pi_v\otimes |\det|_v^{\frac{1-n}{2}}))^{ss}
    \end{equation*}
    and additionally says that, in some precise sense, the monodromy of the LHS is at least as nilpotent as the monodromy of the RHS. As a consequence, assuming that the monodromy of $\textnormal{rec}(\pi_v\otimes |\det|_v^{\frac{1-n}{2}})$ is in fact $0$, we obtain an isomorphism
    \begin{equation*}
        \textnormal{WD}(r_t(\pi)|_{\textnormal{Gal}(\overline{F}_v/F_v)})^{F-ss}\cong t^{-1}\textnormal{rec}(\pi_v\otimes |\det|_v^{\frac{1-n}{2}}).
    \end{equation*}
\end{Rem}
\begin{Rem}
With the additional assumption that $\pi$ is conjugate self-dual, local-global compatibility of $r_t(\pi)$ is known in full generality at each place by work of many people (see \cite{Sh20} for a survey on this).

Beyond the self-dual case, local-global compatibility analogous to Theorem~\ref{Th1.1.1} at places prime-to-$p$ (without the assumptions on the residual representation) was proved in \cite{Var14}. However, having a better understanding of the monodromy has proved to be a rather subtle task. Nevertheless, some results in this direction can be found in \cite{AN21}, \cite{Yan21} and \cite{Mat24}.
    
At $\ell=p$, local-global compatibility results beyond the self-dual case were first established in \cite{ACC23} where they treated the Fontaine--Laffaille and ordinary cases. Their so-called "degree shifting argument" is essentially the only known way to attack local-global compatibility at $p$ without self-duality conditions and relies heavily on the deep vanishing result for unitary Shimura varieties of \cite{CS19}.\footnote{This is where the decomposed generic assumption of our results originates.} As we will discuss, it is an important feature of their work that they also prove local-global compatibility for Galois representations attached to genuine torsion Hecke eigenclasses. Presence of such torsion classes is one of the main subtleties of our setup.

Very recently, there is also work of A'Campo, \cite{AC22}, that, under the assumptions of Theorem~\ref{Th1.1.1}, proves that $r_t(\pi)$ is de Rham with the expected Hodge--Tate weights and further verifies that the representation is crystalline at every $p$-adic place where $\pi$ is spherical. Although his argument is a generalisation of the (Fontaine--Laffaille) degree shifting argument of \cite{ACC23}, it only works in characteristic $0$.

Finally, Caraiani--Newton \cite{CN23} go further and prove that automorphic Galois representations coming from Hecke eigenclasses with hyperspecial level at $p$ must be crystalline with the correct Hodge--Tate weights (cf. \textit{loc. cit.} Theorem 4.3.1). The novelty of their work is that it also deals with torsion classes. It relies on a clever combination of a generalisation of the Fontaine--Laffaille and ordinary degree shifting arguments. This adds up to a very robust method that keeps track of torsion as well. Our work generalises theirs to deal with arbitrary level and to also keep track of local Hecke operators at $p$.
\end{Rem}

We deduce Theorem~\ref{Th1.1.1} from a more general result that also deals with the torsion automorphic Galois representations constructed in \cite{Sch15}. In this generality, formulating a precise conjecture already requires some work that we summarise now assuming that $F$ is an imaginary CM field and for details we point the reader to \S\ref{sec5}.

Let $K\subset \textnormal{GL}_n(\mathbf{A}_F^{\infty})$ be a good compact open subgroup (in the sense of \cite{ACC23}, \S2.1.1) with $K_p=\prod_v \textnormal{GL}_n(\mathcal{O}_{F_v})$ and let $X_K$ be the corresponding locally symmetric space for $\textnormal{GL}_{n}/F$. Let $E\subset \overline{\mathbf{Q}}_p$ be a large enough finite field extension so that the images of all embeddings $F\hookrightarrow \overline{\mathbf{Q}}_p$ lie in $E$, and set $\mathcal{O}$ to be its ring of integers, and $\varpi\in \mathcal{O}$ be a choice of uniformiser. Given a highest weight vector $\lambda=(\lambda_v)_{v|p}$ for $(\textnormal{Res}_{F/\mathbf{Q}}\textnormal{GL}_n)_E$, and a so-called Weil--Deligne inertial type\footnote{This is simply, for each $p$-adic place $v$ of $F$, a choice of an isomorphism class of a pair of an inertial type and a monodromy operator.} $\tau=(\tau_v)_{v|p}$ at $p$ (see \S\ref{sec2.6} for a definition), we obtain an $\mathcal{O}$-local system $\mathcal{V}_{(\lambda,\tau)}$ on $X_K$. The corresponding Betti cohomology groups $H^{\ast}(X_K,\mathcal{V}_{(\lambda,\tau)})$ admit an action of an abstract Hecke algebra $\mathbf{T}^{\lambda,\tau}=\mathbf{T}\otimes_{\mathcal{O}}\mathfrak{z}_{\lambda,\tau}^{\circ}$. Here $\mathbf{T}$ is the usual abstract spherical Hecke algebra (over $\mathcal{O}$) acting at unramified places and $\mathfrak{z}_{\lambda,\tau}^{\circ}=\otimes_{v|p}\mathfrak{z}_{\lambda_v,\tau_v}^{\circ}$ is an $\mathcal{O}$-flat algebra consisting of Hecke operators at $p$ admitting a natural identification $\mathfrak{z}_{\lambda_v,\tau_v}^{\circ}[1/p]\cong \mathfrak{z}_{\tau_v}$ where the latter denotes the Bernstein centre corresponding to $\tau_v$. Denote by $\mathbf{T}^{\lambda,\tau}(K)$ the quotient of $\mathbf{T}^{\lambda,\tau}$ acting faithfully on $H^{\ast}(X_K,\mathcal{V}_{(\lambda,\tau)})$. Given a maximal ideal $\mathfrak{m}\subset \mathbf{T}^{\lambda,\tau}(K)$, Scholze constructed a continuous Galois representation
\begin{equation*}
    \overline{\rho}_{\mathfrak{m}}:\textnormal{Gal}(\overline{F}/F)\to \textnormal{GL}_n(\mathbf{T}^{\lambda,\tau}(K)/\mathfrak{m})
\end{equation*}
and, assuming that $\overline{\rho}_{\mathfrak{m}}$ is absolutely irreducible (in other words, $\mathfrak{m}$ is non-Eisenstein), a lift to a continuous representation
\begin{equation*}
    \rho_{\mathfrak{m}}:\textnormal{Gal}(\overline{F}/F)\to \textnormal{GL}_n(\mathbf{T}^{\lambda,\tau}(K)_{\mathfrak{m}}/I),
\end{equation*}
where $I\subset \mathbf{T}^{\lambda,\tau}(K)_{\mathfrak{m}}$ is a nilpotent ideal with $I^4=0$ (for the bound on the nilpotence degree, see \cite{NT16}). Both of these representations are determined by insisting that they satisfy local-global compatibility at prime-to-$p$ unramified places.

As explained by the work of Calegari--Geraghty (\cite{CG18}), to prove automorphy lifting theorems in the positive defect\footnote{For the definition of the defect of a group, we point the reader for instance to \cite{GN20}, \S2.1.} case (such as the case of $\textnormal{GL}_n$ over an imaginary CM field), it is essential to verify local-global compatibility of $\rho_{\mathfrak{m}}$. The reason why this does not follow from characteristic $0$ local-global compatibility is that, as we move away from the defect $0$ case (such as the case of $\textnormal{GL}_2$ over a totally real field or, more generally, the case of Shimura varieties), the cohomology group $H^{\ast}(X_K,\mathcal{V}_{(\lambda,\tau)})_{\mathfrak{m}}$ typically contains genuine torsion.\footnote{To support this sentence mathematically, we invite the reader to look at for instance \cite{BV13} where, in the defect $1$ case (such as the case of $\textnormal{GL}_2$ over an imaginary quadratic field), they prove exponential growth of torsion in the cohomology of locally symmetric spaces as we deepen the level while the space of characteristic $0$ automorphic forms contributing to the cohomology stays small.} In particular, one seeks an integral version of local-global compatibility. Such a conjecture was formulated in \cite{CG18} in the Fontaine--Laffaille case (cf. \textit{loc. cit.} Conjecture B) and indeed, under some assumptions, this was proved in \cite{ACC23}. The special feature of the Fontaine--Laffaille case is that there the integral $p$-adic Hodge theory is as well-behaved as it could be. As a consequence, one can make a conjecture that is reminiscent of the characteristic $0$ one. Similarly, the ordinary case of \cite{ACC23} is again special enough to make a more naive yet valid prediction. However, as one moves away from these two cases, it is less clear what one should ask for. The first signs (known to the author) of a general conjecture (in the potentially crystalline case) appeared in \cite{CEGGPS16} and a precise formulation (in the crystalline case) can be found in \cite{GN20}, Conjecture 5.1.12. Caraiani and Newton recently proved a large part of the conjecture of Gee--Newton (cf. \cite{CN23}, Theorem 1.3).

The two key ideas in \cite{GN20}, Conjecture 5.1.12 are to use Kisin's potentially semistable ($p$-torsion free) local deformation rings and that semisimple local Langlands can be interpolated over
the generic fibre of these deformation rings. In
particular, we obtain a workable notion of being "torsion potentially semistable of a given weight and inertial type". Moreover, by $p$-adically rescaling, we can test compatibility of the action of Hecke operators at $p$ using the interpolation map. 

Namely, for $v|p$ setting $\rho_v:=\rho_{\mathfrak{m}}|_{\textnormal{Gal}(\overline{F}_v/F_v)}$, an $\mathcal{O}$-flat quotient $R^{\lambda_v,\preceq \tau_v}_{\overline{\rho}_v}$ of the unrestricted local framed deformation ring $R^{\Box}_{\overline{\rho}_v}$ is constructed in \cite{Kis07}. It is characterised by the property that its $\overline{E}$-points $\rho:\textnormal{Gal}(\overline{F}_v/F_v)\to \textnormal{GL}_n(\overline{E})$ are exactly those lifts of $\overline{\rho}_v$ that are potentially semistable with Hodge--Tate weights determined by $\lambda_v$ (by the usual $\rho$-shift) and such that the Weil--Deligne inertial type of $\textnormal{WD}(\rho)$ is bounded by $\tau_v$. Moreover, on the generic fiber one can construct a map
\begin{equation*}
    \eta: \mathfrak{z}_{\tau_v}\to R^{\lambda_v,\preceq\tau_v}_{\overline{\rho}_v}[1/p]
\end{equation*}
interpolating the (Tate-normalised) semisimple local Langlands correspondence. Although $\eta$ does not necessarily send $\mathfrak{z}_{\lambda_v,\tau_v}^{\circ}$ into $R^{\lambda_v,\preceq \tau_v}_{\overline{\rho}_v}$, we can introduce the subring 
\begin{equation*}
\mathfrak{z}_{\lambda_v,\tau_v}^{\circ,\textnormal{int}}:=\eta^{-1}(R_{\overline{\rho}_v}^{\lambda_v,\preceq \tau_v})\cap \mathfrak{z}_{\lambda_v,\tau_v}^{\circ}\subset \mathfrak{z}_{\lambda_v,\tau_v}^{\circ}.
\end{equation*} 
This ring still has the property $\mathfrak{z}_{\lambda_v,\tau_v}^{\circ,\textnormal{int}}[1/p]=\mathfrak{z}_{\tau_v}$. Local-global compatibility can then be phrased by asking that there exists a (necessarily unique) dotted arrow making the diagram
\begin{equation}\label{eq1.1.2}
\begin{tikzcd}
R^{\Box}_{\overline{\rho}_v} \arrow{r}{\rho_v} \arrow[two heads]{d}{}
&\mathbf{T}^{\lambda,\tau}(K)_{\mathfrak{m}}/I \\
R^{\lambda_v,\preceq \tau_v}_{\overline{\rho}_v} \arrow[hook]{d}{} \arrow[dotted]{ur}{} &\mathfrak{z}_{\lambda_v,\tau_v}^{\circ,\textnormal{int}}\arrow{u}[swap]{\textnormal{nat}}\arrow[hook]{d}{}\arrow{l}{\eta\mid_{\mathfrak{z}_{\lambda_v,\tau_v}^{\circ,\textnormal{int}}}}\\
R^{\lambda_v,\preceq\tau_v}_{\overline{\rho}_v}[1/p]&\arrow{l}{\eta} \mathfrak{z}_{\tau_v}
\end{tikzcd}
\end{equation}
commutative. Here the map $\textnormal{nat}$ denotes the natural map towards the faithful Hecke algebra (along the inclusion $\mathfrak{z}_{\lambda_v,\tau_v}^{\circ,\textnormal{int}}\subset \mathfrak{z}_{\lambda_v,\tau_v}^{\circ}$) and the inclusions at the bottom are the ones induced by inverting $p$. Our main result is then (roughly) as follows (cf. Theorem~\ref{Th5.7}).
\begin{Th}\label{Th1.1.4}
    Let $F$ be an imaginary CM field that contains an imaginary quadratic field $F_0$ in which $p$ splits with totally real subfield $F^+$. Assume that the finite set of places $T$ is as in Theorem~\ref{Th5.7}. Let $\Bar{v}$ be a $p$-adic place in $F^+$. Assume the following:
    \begin{enumerate}
    \item There is a $p$-adic place $\Bar{v}'$ of $F^+$ such that $\Bar{v}\neq \Bar{v}'$ and
    \begin{equation*}
        \sum_{\Bar{v}''\neq \Bar{v},\Bar{v}'}[F^+_{\Bar{v}''}:\mathbf{Q}_p]\geq \frac{1}{2}[F^+:\mathbf{Q}]
    \end{equation*}
    where the sum runs over $p$-adic places of $F^+$;
    \item The representation $\overline{\rho}_{\mathfrak{m}}$ is decomposed generic.
\end{enumerate}
Then, up to possibly enlarging $I$, there is a (necessarily unique) dotted arrow making (\ref{eq1.1.2}) commutative.
\end{Th}
\begin{Rem}
    The further constraint on $T$ in the statement of Theorem~\ref{Th5.7} is a mild one that can be always achieved by enlarging $T$. This condition is so that we can appeal to the unconditional base change results of \cite{Shi14} and it is already present in \cite{Sch15}.

    Assumption i) is essential to our methods and already appears in \cite{ACC23}, and \cite{CN23}. In particular, this rules out the case of $F$ being an imaginary quadratic field. Using Theorem~\ref{Th1.1.1}, it seems plausible to weaken this assumption, for instance allowing one to prove the theorem for $n=2$, and $[F:\mathbf{Q}]=4$. 

    Finally, assumption ii) on $\overline{\rho}_{\mathfrak{m}}$ is crucial to be able to use the vanishing results of \cite{CS19}.
\end{Rem}
\begin{Rem}
    We point out that Theorem~\ref{Th1.1.4} yields a new result even in the crystalline case since it also matches the crystalline Frobenius with the Frobenius at $v$ under local Langlands. Namely, when $\tau_v$ is assumed to be trivial (i.e. when the inertial type is the trivial representation of the inertia subgroup and the monodromy is the $0$ matrix), \cite{CN23} proves the existence of a dotted arrow making the top triangle of (\ref{eq1.1.2}) commute. Note that in the crystalline case $\mathfrak{z}_{\tau_v}\cong E[T_1,...,T_n^{\pm 1}]$ is the usual spherical Hecke algebra, and $\mathfrak{z}_{\lambda_v,\tau_v}^{\circ}\subset \mathfrak{z}_{\tau_v}$ is some suitable subalgebra acting on $H^{\ast}(X_K,\mathcal{V}_{(\lambda,\tau)})$. Therefore, the commutativity of the lower triangle of (\ref{eq1.1.2}) roughly says that, beyond $\rho_{\mathfrak{m}}|_{\textnormal{Gal}(\overline{F}_v/F_v)}$ being crystalline with the right Hodge--Tate weights, its crystalline Frobenius too is compatible with the local Langlands correspondence (i.e. with the map $\eta$). In particular, in characteristic $0$, for a given automorphic representation $\pi$, we prove that the crystalline Frobenius on the filtered $\varphi$-module associated with $r_t(\pi)|_{\textnormal{Gal}(\overline{F}_v/F_v)}$ coincides with the Frobenius on $t^{-1}\textnormal{rec}(\pi_v\otimes |\det|^{\frac{1-n}{2}}_v)$.

    Although in general for proving small $R=\mathbf{T}$ results it often suffices to prove factorisation through the right potentially semistable deformation rings at $p$, we remark that understanding the crystalline Frobenius is for instance essential to make the big $R=\mathbf{T}$ results of \cite{GN20}, \S5 unconditional. 
\end{Rem} 
In the rest of the introduction, we briefly explain our approach to Theorem~\ref{Th1.1.4} focusing on what it adds to the proof of \cite{CN23}. The first observation in the proof of \textit{loc. cit.} is that to ensure the existence of a dotted arrow making the top triangle of \ref{eq1.1.2} commute, it suffices to construct a characteristic $0$ lift of $\rho_{v}$ mod $\varpi^m$ with the right $p$-adic Hodge theoretic properties for every $m\in \mathbf{Z}_{\geq 1}$. More precisely, assume for a second that $\tau$ is trivial and fix integers $m\geq 1$, $q\geq 0$. Setting $A:= \mathbf{T}(H^q(X_K,\mathcal{V}_{\lambda}/\varpi^m)_{\mathfrak{m}})$, Caraiani and Newton construct a finite flat $\mathcal{O}$-algebra $\widetilde{A}$,\footnote{$\widetilde{A}$ is a faithful Hecke algebra for some $P$-ordinary middle degree (Betti) cohomology of some $U(n,n)$ Shimura variety where $P$ is the Siegel parabolic subgroup in $U(n,n)$.} a so-called degree shifting map
\begin{equation*}
    \overline{\mathcal{S}}:\widetilde{A}\to A,
\end{equation*}
and a lift $\widetilde{\rho}_v:\textnormal{Gal}(\overline{F}_v/F_v)\to \textnormal{GL}_n(\widetilde{A})$ of $\rho_v:\textnormal{Gal}(\overline{F}_v/F_v)\to \textnormal{GL}_n(A)$ under $\overline{\mathcal{S}}$. Moreover, after inverting $p$, the lift $\widetilde{\rho}_v$ becomes crystalline with the expected Hodge--Tate weights. Then a simple diagram chase gives the factorisation of the top triangle in \ref{eq1.1.2} for $\rho_v:\textnormal{Gal}(\overline{F}_v/F_v)\to \textnormal{GL}_n(A)$. Taking limit over $m$ then finishes the proof. In particular, their proof consists of two main steps. The first of which is constructing the transfer map $\overline{\mathcal{S}}$ (that purely takes place on the automorphic side) descending the Satake transfer for the parabolic subgroup $P$ in $U(n,n)$. The other step for them is proving a $P$-ordinary local-global compatibility result in characteristic $0$ for conjugate self-dual automorphic Galois representations (cf. \cite{CN23}, Theorem 3.1.2).

For us to conclude, we need an enrichment of their output. Namely, now set $\tau$ to be arbitrary, and consider $A^{\lambda,\tau}:=\mathbf{T}^{\lambda,\tau}(H^q(X_K,\mathcal{V}_{(\lambda,\tau)})_{\mathfrak{m}})$ that is now naturally a $\mathfrak{z}_{\lambda_v,\tau_v}^{\circ}$-algebra. We too construct a finite flat $\mathcal{O}$-algebra $\widetilde{A}^{\lambda_{v},\tau_v}$ by similarly taking the faithful Hecke algebra of some $P$-ordinary middle degree cohomology of a $U(n,n)$-Shimura variety. However, in our case we prove that it is equipped with the structure of a $\mathfrak{z}_{\lambda_v,\tau_v}^{\circ}$-algebra. We then perform the degree shifting but we further carefully make sure that the obtained map $\overline{\mathcal{S}}^v$ is a map of $\mathfrak{z}_{\lambda_v,\tau_v}^{\circ}$-algebras. This requires some extension of Hida theory for torsion Betti cohomology, and a study of how Hecke operators at $p$ behave under Poincar\'e duality. The other ingredient in the proof of Theorem~\ref{Th1.1.4} is a strong enough $P$-ordinary local-global compatibility result in characteristic $0$ for conjugate self-dual automorphic representations. Here our main result is as follows (cf. Theorem~\ref{Th4.27} and Theorem~\ref{Th4.30}).
\begin{Th}\label{Th1.1.7}
    Let $\pi$ be a regular algebraic conjugate self-dual cuspidal automorphic representation of $\textnormal{GL}_n(\mathbf{A}_F)$ of weight $\lambda\in (\mathbf{Z}^n_+)^{\Hom(F,\mathbf{C})}$, $t:\overline{\mathbf{Q}}_p\xrightarrow{\sim}\mathbf{C}$ be a fixed field isomorphism and $v|p$ be a $p$-adic place of $F$. Let $Q\subset \textnormal{GL}_n$ be the standard parabolic subgroup corresponding to a partition $n=n_1+...+n_t$, and denote by $M$ its Levi quotient. Assume that $t^{-1}\pi_v$ is $Q$-ordinary of weight $\lambda_v$. 
    
    Then the $Q$-ordinary part $(t^{-1}\pi_v)^{Q\textnormal{-ord}}$ is an irreducible smooth $\overline{\mathbf{Q}}_p$-representation of $M(F_v)=\textnormal{GL}_{n_1}(F_v)\times ...\times \textnormal{GL}_{n_t}(F_v)$. Write $(t^{-1}\pi_v)^{Q\textnormal{-ord}}=\pi_1\otimes...\otimes \pi_t$. Then there is moreover an isomorphism
    \begin{equation*}
    r_t(\pi)|_{\textnormal{Gal}(\overline{F}_v/F_v)}\sim \begin{pmatrix}
\rho_1 & \ast & ... & \ast\\
0 & \rho_2 & ... & \ast\\
. & . & . & .\\
. & . & . & .\\
0 & ... & 0 & \rho_t
\end{pmatrix}
\end{equation*}
where, for $1\leq j\leq t$,
\begin{equation*}
    \rho_j:\textnormal{Gal}(\overline{F}_v/F_v)\to \textnormal{GL}_{n_j}(\overline{\mathbf{Q}}_p)
\end{equation*}
is potentially semistable such that, for every embedding $\iota:F_v\hookrightarrow \overline{\mathbf{Q}}_p$, the labelled $\iota$-Hodge--Tate weights of $\rho_j$ are given by
\begin{equation*}
    \lambda_{t\circ \iota,n+1-(n_1+...+n_{j})}+n_1+...+n_j-1>...>\lambda_{t\circ \iota,n+1-(n_1+...+n_{j-1}+1)}+n_1+...+n_{j-1}
\end{equation*}
and we have an isomorphism
\begin{equation*}
    \textnormal{WD}(\rho_j)^{F-ss}\cong \textnormal{rec}(\pi_j\otimes|\det|^{n_1+...+n_{j-1}}_v\otimes|\det|_v^{\frac{1-n_j}{2}} ).
\end{equation*}
\end{Th}
With Theorem~\ref{Th1.1.7}, and the degree shifting map $\overline{\mathcal{S}}^v$ in hand, the arguments of \cite{CN23} (most notably the determinant argument of \textit{loc. cit.} \S3.2) produce a lift $\widetilde{\rho}_v:\textnormal{Gal}(\overline{F}_v/F_v)\to \textnormal{GL}_n(\widetilde{A}^{\lambda_v,\tau_v})$ of $\rho_v$ under $\overline{\mathcal{S}}^v$. Thanks to Theorem~\ref{Th1.1.7}, there is a dotted arrow making a diagram like \ref{eq1.1.2} commute for $\widetilde{\rho}_v$. Then an easy diagram chase using that $\overline{\mathcal{S}}^v$ is a $\mathfrak{z}_{\lambda_v,\tau_v}^{\circ}$-algebra map proves the same for $\rho_v$ valued in $A^{\lambda,\tau}$. Passing to the limit over $m$ then allows one to conclude.
\begin{Rem}
    Theorem~\ref{Th1.1.7} generalises the ordinary local-global compatibility result \cite{Ger18}, Corollary 2.7.8 (see also \cite{Tho15}, Theorem 2.4) to the case of a general standard parabolic subgroup. The main extra difficulty is having a generalisation of \cite{Ger18}, Lemma 5.4 2) for $Q$-ordinary parts for a general standard parabolic subgroup. This is the first part of Theorem~\ref{Th1.1.7} and is a question purely in the realm of smooth representations of $p$-adic reductive groups. This part of the theorem we prove in a rather general setup in \S\ref{sec4.2} that might be of independent interest.
\end{Rem}
\begin{Rem}
    As it does not require any significant extra work, we in fact also produce (locally at $p$) $Q$-ordinary lifts with expected shape of Galois representations associated with $Q$-ordinary torsion Hecke eigenclasses for any standard parabolic subgroup $Q\subset \textnormal{GL}_n$ (see Proposition~\ref{Prop6.6}). Although we do not investigate it further, after formulating a suitable $Q$-ordinary local-global compatibility conjecture, these $Q$-ordinary lifts seem to have the potential to make progress on such conjectures.
\end{Rem}

The organisation of the paper is as follows. In Section 2, we collect the preliminaries on the cohomology of locally symmetric spaces and introduce the necessary local systems and the corresponding Hecke actions for the group $\textnormal{GL}_{n/F}$ and the quasi-split unitary group $\widetilde{G}$. In Section 3, we recollect and generalise Hida theory in the Betti setting for general standard parabolic subgroups. In Section 4, we revisit and further develop a theory of ordinary parts for locally algebraic representations. Moreover, we finish the section with proving Theorem~\ref{Th1.1.7}. In Section 5, we formulate a torsion local-global compatibility conjecture following \cite{GN20}. In Section 6, we execute the strategy outlined in the introduction to prove Theorem~\ref{Th1.1.4} and conclude the section with deducing Theorem~\ref{Th1.1.1}. In the Appendix, we make a brief recollection on the smooth representation theory of $\textnormal{GL}_n$ over a $p$-adic field and prove a couple of relevant technical lemmas that are only used at the end of \S\ref{sec4.2}.
\subsection*{Acknowledgements}
First and foremost, I would like to express my gratitude to Ana Caraiani for suggesting this project, generously sharing her knowledge and ideas on the subject and for spending much of her time listening to my ideas. This work would not have been completed without her support. I am also indebted to James Newton for the many illuminating conversations and for constantly showing interest in my work. I am grateful to Fred Diamond for providing plenty of feedback on this work and asking even more great questions. I thank Lambert A'Campo, Ana Caraiani, Fred Diamond, James Newton and Matteo Tamiozzo for comments on an earlier draft of this paper.

This work was supported by the Engineering and Physical Sciences Research Council [EP/S021590/1]. 
The EPSRC Centre for Doctoral Training in Geometry and Number Theory (The London School of Geometry and Number Theory), University College London, King's College London and Imperial College London. This project has received funding from the European Research Council (ERC) under the European Union’s Horizon 2020 research and innovation programme (grant agreement No. 804176).

\section*{Notation and Conventions}\label{Notations}
Given a number field $F$, we will denote by $S(F)$ its set of finite places and by $S_p(F)$ its set of $p$-adic places. We set $G_F$ to be the absolute Galois group $\textnormal{Gal}(\overline{F}/F)$ and for a finite set $S\subset S(F)$ we denote by $G_{F,S}$ the quotient of $G_F$ corresponding to the maximal Galois extension of $F$, unramified outside $S$. For $v\in S(F)$, set $F_v$ to be the $v$-adic completion of $F$, fix a choice of uniformiser $\varpi_v$ and set $k_v:=\mathcal{O}_{F_v}/\varpi_v$ to be its residue field. Set $G_{F_v}:=\textnormal{Gal}(\overline{F}_v/F_v)$. Moreover, $I_{F_v}\subset G_{F_v}$ will denote its inertia subgroup and set $\textnormal{Frob}_v\in G_{F_v}/I_{F_v}$ to be the geometric Frobenius. We denote by $\mathbf{A}_F$ the ring of adeles of $F$ and for $S$ a finite set of places of $F$ we denote by $\mathbf{A}_{F}^S$ its prime-to-$S$ part and set $\widehat{\mathcal{O}}_{F,S}:=\prod_{v\in S}\mathcal{O}_{F_v}$.

For $G$ a reductive group over a number field $F$ and a finite set $S\subset S(F)$, we will denote by $G^S:=G(\mathbf{A}_{F}^{S\cup \infty})$ and by $G_S:=G(\prod_{v\in S}F_v)$. Moreover, if $S=S_p(F)$, we only write $G^p$, respectively $G_p$.

Let $G$ be a linear algebraic group over $\mathcal{O}_L$ where $\mathcal{O}_L$ is the ring of integers of a finite extension $L/\mathbf{Q}_p$. Let $\varpi_L\in \mathcal{O}_L$ be a choice of uniformiser. For $n\in \mathbf{Z}_{\geq 0}$, denote $\ker(G(\mathcal{O}_L)\to G(\mathcal{O}_L/\varpi_L^n))$ by $G^n$.

For $G$ a reductive group over a finite extension $L/\mathbf{Q}_p$ with parabolic subgroup $Q=M\ltimes N \subset G$, we denote by $\delta_Q:M(L)\to \mathbf{Q}^{\times}$ the corresponding modulus character $x\mapsto |\det(\textnormal{ad}(x)|\textnormal{Lie}N)|_L$. For a smooth representation $\sigma$ of $M(L)$, we denote by $\textnormal{n-Ind}_{Q(L)}^{G(L)}\sigma$ the normalised parabolic induction $\textnormal{Ind}_{Q(L)}^{G(L)}\delta_Q^{1/2}\sigma$. For a smooth $\overline{\mathbf{Q}}_p$-representation $\pi$ of $G(L)$, we will denote by $J_Q(\pi)$ its \textit{unnormalised} Jacquet module associated with $Q$. In general we denote by $\textnormal{c-Ind}$ compact induction for smooth representations.

For a smooth irreducible representation $\pi$ of $\textnormal{GL}_n(L)$ with supercuspidal support $(\textnormal{GL}_{n_1}(L)\times...\times \textnormal{GL}_{n_k}(L),\pi_1\otimes...\otimes \pi_k)$, we set $\textnormal{SC}(\pi):=\{\pi_1,...,\pi_k\}$.

Set $W_L$ to be the Weil group of $L$ and write $\textnormal{Art}_L:L^{\times}\xrightarrow{\sim} W_L^{\textnormal{ab}}$ for the Artin map of local class field theory normalised by sending uniformisers to lifts of the geometric Frobenius.

We denote by $\textnormal{rec}_L$ the local Langlands correspondence for $L$. If it's clear from the context, we will just write $\textnormal{rec}$ instead. Moreover, for $\pi$ an irreducible admissible $\textnormal{GL}_n(L)$-representation, we set $\textnormal{rec}^T(\pi)=\textnormal{rec}(\pi\otimes|\cdot|_L^{\frac{1-n}{2}})$. Then $\textnormal{rec}^T$ commutes with $\textnormal{Aut}(\mathbf{C})$ and therefore $\textnormal{rec}^T$ makes sense over $\overline{\mathbf{Q}}_p$ by choosing an abstract isomorphism $t:\overline{\mathbf{Q}}_p\xrightarrow{\sim} \mathbf{C}$. In the literature, this is often called the \textit{Tate normalisation} of the local Langlands correspondence.

We fix an algebraic closure $\overline{\mathbf{Q}}_p/\mathbf{Q}_p$ and denote by $\textnormal{val}_p$ the $p$-adic valuation normalised by setting $\textnormal{val}_p(p)=1$. By convention, the $\iota$-Hodge--Tate weights of the $p$-adic cyclotomic character $\epsilon:G_L\to \mathbf{Z}_p^{\times}$ are $-1$.

For an $\ell$-adic Galois representation $\rho:G_L\to \textnormal{GL}_n(\overline{\mathbf{Q}}_{\ell})$ with $\ell\neq p$, we denote by $\textnormal{WD}(\rho)$ the associated Weil--Deligne representation. For a de Rham $p$-adic Galois representation $\rho:G_L\to \textnormal{GL}_n(\overline{\mathbf{Q}}_p)$, we denote by $\textnormal{WD}(\rho)$ the associated Weil--Deligne representation provided by the recipe of Fontaine.

For a Weil--Deligne representation $(r,N)$, we denote by $(r,N)^{F-ss}$ its Frobenius semisimplification and by $(r,N)^{ss}$ its semisimplification.

For a ring $R$, we denote by $D(R)$ the derived category of left $R$-modules and by $D^+(R)$ resp. $D^b(R)$ the bounded below resp. bounded derived category. Given moreover a locally profinite group $G$, we will denote by $\textnormal{Mod}_{\textnormal{sm}}(R[G])$ the category of smooth $R[G]$-modules and denote by $D_{\textnormal{sm}}(R[G])$ its derived category and by $D^+_{\textnormal{sm}}(R[G])$ resp. $D^b_{\textnormal{sm}}(R[G])$ its bounded below resp. bounded derived category.

Given a topological group $G$, and a topological space $X$ with a continuous right action of $G$, we denote by $\textnormal{Sh}_G(X)$ the category of $G$-equivariant sheaves on $X$ in the sense of \cite{NT16}, Definition 2.22, (2). Moreover, for a ring $R$, we denote by $\textnormal{Sh}_G(X,R)$ the category of $G$-equivariant sheaves of $R$-modules on $X$.

For $G/L$ a split reductive group with a choice of a Borel subgroup $B$ and a maximal torus $T$, denote by $w_0^G$ the longest element in the Weyl group $W_G:=W(G,T)$. For a standard parabolic subgroup $Q\subset G$ with Levi decomposition $M\ltimes N$, set $W^Q$ to be the set of minimal length representatives of $ W_G/W_Q$. We denote by $w_0^Q$ the longest element in $W^Q$ that is, in fact, given by $w_0^Gw_0^M$. Similar notations apply to ${}^QW$. Moreover, for another standard parabolic subgroup $Q'\subset G$ with Levi decomposition $M'\ltimes N'$, denote by ${}^{Q'}W^{Q}$ the set of minimal length representatives of $W_{Q'}\backslash W_G/W_Q$.

For an integer $n\geq 1$, we denote by $\mathbf{Z}_+^n\subset \mathbf{Z}^n$ the subset of $n$-tuples of integers $(k_1,...,k_n)$ satisfying $k_1\geq...\geq k_n$.

\section{Preliminaries}

We collect the necessary preliminaries regarding locally symmetric spaces and their cohomology following \cite{CN23}, \textsection2. The main differences compared to the setup of \cite{ACC23}, \textsection2 are the use of different infinite level locally symmetric spaces, which take into account the profinite topology on the adelic part.

\subsection{Locally symmetric spaces}\label{sec2.1}Consider a number field $F$ and a connected linear algebraic group $G$ over $F$, with a model over $\mathcal{O}_F$, still denoted by $G$. One can then associate with $\textnormal{Res}_{F/\mathbf{Q}}G$ its symmetric space $X^G$, a homogeneous $G(F\otimes_{\mathbf{Q}}\mathbf{R})$-space, as is defined in \cite{BS73}, \textsection2. By \cite{BS73}, Lemma 2.1, $X^G$ is unique up to isomorphism of homogeneous $G(F\otimes_{\mathbf{Q}}\mathbf{R})$-spaces.

Given a good compact open subgroup $K_G\subset G(\mathbf{A}^{\infty}_F)$ in the sense of \cite{ACC23}, \textsection2.1.1, we can form the corresponding locally symmetric space
\begin{equation*}
    X_{K_G}:=G(F)\backslash(X^G\times G(\mathbf{A}_F^{\infty})/K_G),
\end{equation*}
a smooth orientable Riemannian manifold. Borel--Serre then construct a partial compactification $\overline{X}^G$ of $X^G$ (cf. \cite{BS73}, \textsection7.1) and form the compactified locally symmetric space
\begin{equation*}
    \overline{X}_{K_G}:=G(F)\backslash(\overline{X}^G\times G(\mathbf{A}_F^{\infty})/K_G),
\end{equation*}
an orientable compact smooth manifold with corners and with interior $X_{K_G}$. We will denote the corresponding boundary by $\partial X^G:=\overline{X}^G\setminus X^G$ resp. $\partial X_{K_G}:=\overline{X}_{K_G}\setminus X_{K_G}$. 

Following \cite{CN23}, we define the sets
\begin{equation*}
    \mathfrak{X}_G:=\varprojlim_{K_G}X_{K_G}, \textnormal{ }\overline{\mathfrak{X}}_G:=\varprojlim_{K_G} \overline{X}_{K_G}, \textnormal{ }\partial\mathfrak{X}_G=\varprojlim_{K_G}\partial X_{K_G}
\end{equation*}
where the limits run over good subgroups of $G(\mathbf{A}_{F}^{\infty})$, and make them into topological spaces by endowing them with the projective limit topology. The latter two then become compact Hausdorff spaces, being projective limits of such. They all are equipped with the natural continuous right action of $G(\mathbf{A}_F^{\infty})$. Moreover, the induced action of any good subgroup $K_G\subset G(\mathbf{A}_F^{\infty})$ on $\overline{\mathfrak{X}}_G$, and $\partial \mathfrak{X}_G$ is free in the sense of \cite{NT16}, Definition 2.23. As explained in \cite{CN23}, \S2.1.1, we can also introduce these spaces as the quotient spaces
\begin{equation*}
    \mathfrak{X}_G=G(F)\setminus X^G\times G(\mathbf{A}_F^{\infty}),\textnormal{ }\overline{\mathfrak{X}}_G=G(F)\setminus \overline{X}^G\times G(\mathbf{A}_F^{\infty}),\textnormal{ and }
\end{equation*}
\begin{equation*}
    \partial \mathfrak{X}_G=G(F)\setminus \partial X^G\times G(\mathbf{A}_F^{\infty})
\end{equation*}
where the group $G(\mathbf{A}_F^{\infty})$ is equipped with its locally profinite topology. We remind the reader that, before \cite{CN23}, the more common choice was to use the discrete topology of $G(\mathbf{A}_F^{\infty})$ instead (see \cite{NT16}). They yield different spaces of course, but the produced finite level cohomology groups (with the induced Hecke actions) compare well (cf. \cite{CN23}, Lemma 2.1.5). We denote by $j:\mathfrak{X}_G\hookrightarrow \overline{\mathfrak{X}}_G$ the natural open immersion.

For later use we introduce some further notation. Set $S$ to be a finite set of finite places of $F$ and $K^S_G\subset G(\mathbf{A}^{S\cup\{\infty\}}_F)$ be a compact open subgroup that extends to some good subgroup of $G(\mathbf{A}_F^{\infty})$. We then set
\begin{equation*}
    \overline{X}_{K_G^S}:= \overline{\mathfrak{X}}_G/K_{G}^S=\varprojlim_{K_{G,S}} \overline{X}_{K_{G}^SK_{G,S}}, \textnormal{ and }\partial X_{K_G^S}:= \partial \mathfrak{X}_G/K_{G}^S=\varprojlim_{K_{G,S}} \partial X_{K_{G}^SK_{G,S}}
\end{equation*}
where the limits run over compact open subgroups $K_{G,S}\subset G(\mathbf{A}_{F}^{\infty})$ such that $K_G^SK_{G,S}$ is good.

\subsection{Coefficient systems and Hecke operators}\label{sec2.2}
For $S$ a finite set of finite places of $F$, we set $G^S:=G(\mathbf{A}_{F}^{S\cup\{\infty\}})$ and $G_S:=G(\mathbf{A}_{F,S})$, and similarly, for a good subgroup $K_G\subset G(\mathbf{A}^{\infty}_F)$, we set $K^S_G:=\prod_{v\notin S}K_{G,v}$ and $K_{G,S}:=\prod_{v\in S}K_{G,v}$.

Let $R$ be a commutative ring. Note that both $\mathfrak{X}_G$, and $\overline{\mathfrak{X}}_G$ admit a continuous right action of $G^S\times K_{G,S}$. In particular, we can consider the corresponding categories of $G^S\times K_{G,S}$-equivariant sheaves of $R$-modules $\textnormal{Sh}_{G^S\times K_{G,S}}(\mathfrak{X}_G,R)$, $\textnormal{Sh}_{G^{S}\times K_{G,S}}(\overline{\mathfrak{X}}_G,R)$ in the sense of \cite{NT16}, Definition 2.22, (2). Given a smooth $R[K_{G,S}]$-module $\mathcal{V}$, the formalism of \cite{ACC23} attaches to it a $G^S\times K_{G,S}$-equivariant sheaf $\mathcal{V}\in \textnormal{Sh}_{G^S\times K_{G,S}}(\mathfrak{X}_G,R)$ resp. $\mathcal{V}\in \textnormal{Sh}_{G^{S}\times K_{G,S}}(\overline{\mathfrak{X}}_G,R)$. Namely, one inflates $\mathcal{V}$ to get an element of $\textnormal{Mod}_{\textnormal{sm}}(R[G^S\times K_{G,S}])$ which, by \cite{NT16}, Lemma 2.25, is equivalent to the category $\textnormal{Sh}_{G^S\times K_{G,S}}(\ast,R)$ and we then pull it back along $f:\mathfrak{X}_G\to \ast$ resp. $\overline{f}:\overline{\mathfrak{X}}_G\to \ast$. As explained in \cite{Sch98}, the categories $\textnormal{Sh}_{G^S\times K_{G,S}}(\mathfrak{X}_G, R)$ and $\textnormal{Sh}_{G^S\times K_{G,S}}(\overline{\mathfrak{X}}_G,R)$ have enough injectives and both $f_{!}$ and $\overline{f}_!=\overline{f}_{\ast}$ land in $\textnormal{Sh}_{G^S\times K_{G,S}}(\ast,R)\cong \textnormal{Mod}_{\textnormal{sm}}(R[G^S\times K_{G,S}])$. Therefore, we see that both
\begin{equation*}
    R\Gamma(\overline{\mathfrak{X}}_G,j_!\mathcal{V})=\footnote{Here we use that $j_{!}$ preserves injective objects.}R(\overline{f}_{\ast}\circ j_!)\mathcal{V}=Rf_!\mathcal{V}
\end{equation*}
and
\begin{equation*}
    R\Gamma(\overline{\mathfrak{X}}_G,\mathcal{V})=R\overline{f}_{\ast}\mathcal{V}=R\overline{f}_!\mathcal{V}
\end{equation*}
lie in $D^+_{\textnormal{sm}}(R[G^S\times K_{G,S}])$. By taking derived invariants, we end up with
\begin{equation*}
    R\Gamma(K_G,R\Gamma(\overline{\mathfrak{X}}_G,\mathcal{V}))\textnormal{ and }R\Gamma(K_G,R\Gamma(\overline{\mathfrak{X}}_G,j_!\mathcal{V}))
\end{equation*}
in $D^+(\mathcal{H}(G^S,K^S_G)\otimes_{\mathbf{Z}}R)$ where $\mathcal{H}(G^S,K^S_G):=\mathbf{Z}[K_G^S\backslash G^S/K_G^S]$ is the prime-to-$S$ Hecke algebra i.e., the additive group of integer valued $K_G^S$-biinvariant functions on $G^S$ equipped with the convolution product.

On the other hand, by descent (cf. \cite{NT16}, Lemma 2.24), $\mathcal{V}$ and $j_!\mathcal{V}$ give rise to sheaves on $\overline{X}_{K_G}$, which, by abuse of notation, we denote by the same letter. Then, by combining the fact that $j$ is a homotopy equivalence and \cite{CN23}, Proposition 2.1.3, we obtain natural isomorphisms
\begin{equation*}
    R\Gamma(X_{K_G},\mathcal{V})\cong R\Gamma(\overline{X}_{K_G},\mathcal{V})\cong R\Gamma(K_G,R\Gamma(\overline{\mathfrak{X}}_G,\mathcal{V}))^{\sim}
\end{equation*}
and
\begin{equation*}
    R\Gamma_c(X_{K_G},\mathcal{V}):=R\Gamma(\overline{X}_{K_G},j_!\mathcal{V})\cong R\Gamma(K_G,R\Gamma(\overline{\mathfrak{X}}_G,j_!\mathcal{V}))^{\sim}
\end{equation*}
in $D^+(R)$, where $(-)^{\sim}$ denotes the forgetful functor. In particular, we have a ring homomorphism
\begin{equation*}
    \mathcal{H}(G^S,K_G^S)\otimes_{\mathbf{Z}}R\to \textnormal{End}_{D^+(R)}(R\Gamma_{(c)}(X_{K_G},\mathcal{V})).
\end{equation*}
The same observations also apply to $R\Gamma(\partial X_{K_G},\mathcal{V})$.

Finally, assuming that $R$ is Noetherian and $K_G'\subset K_G$ is a good normal subgroup, one sees, by writing down the complex explicitly in terms of a well-chosen simplicial complex, that $R\Gamma_{(c)}(X_{K'_G},\mathcal{V})$ is in fact a perfect object in $D^+(R[K_G/K'_G])$ (cf. \cite{CN23}, Lemma 2.1.6).

We now turn our attention to completed cohomology following the viewpoint of \cite{CN23}. For this we fix $m\in \mathbf{Z}_{\geq 1}$ and set $R=\mathcal{O}/\varpi^m$ where $\mathcal{O}$ is the ring of integers of a finite field extension $E/\mathbf{Q}_p$. Moreover, let $S\subset S_p:= S_p(F)$ be a subset of the set of $p$-adic places of $F$.
\begin{Def}\label{Def2.2.1}
Given $\mathcal{V}\in \textnormal{Mod}_{\textnormal{sm}}(\mathcal{O}/\varpi^m[K_{G,S_p}])$, we define its completed cohomology (with compact support) at $S$ of level $K^S_G$ to be
\begin{equation*}
    \pi(K^S_G,\mathcal{V}):=R\Gamma(K^S_G,R\Gamma(\overline{\mathfrak{X}}_G,\mathcal{V}))^{\sim}\cong\footnote{This identification is proved exactly the same way as in the proof of \cite{CN23}, Proposition 2.1.3.} R\Gamma(\overline{X}_{K^{S}_G},\mathcal{V})\in D^+_{\textnormal{sm}}(\mathcal{O}/\varpi^m[K_{G,S}])
\end{equation*}
resp.
\begin{equation*}
    \pi_c(K^S_G,\mathcal{V}):=R\Gamma(K^S_G,R\Gamma(\overline{\mathfrak{X}}_G,j_!\mathcal{V}))^{\sim}\cong R\Gamma(\overline{X}_{K^{S}_G},j_!\mathcal{V})\in D^+_{\textnormal{sm}}(\mathcal{O}/\varpi^m[K_{G,S}]).
\end{equation*}
Note that if $\mathcal{V}$ is inflated from an element $\mathcal{V}^S\in \textnormal{Mod}_{\textnormal{sm}}(\mathcal{O}/\varpi^m[K_{G,S_p\backslash S}])$, then the $S$-completed cohomology complexes in fact lie in $D^+_{\textnormal{sm}}(\mathcal{O}/\varpi^m[G_S])$.
\end{Def}
For any set of finite places $S_p\subset T$, $\mathcal{H}(G^T,K_G^T)\otimes_{\mathbf{Z}}\mathcal{O}/\varpi^m$ acts on the completed cohomology complexes (with compact support) as it does so on $R\Gamma(K^S_G,R\Gamma(\overline{\mathfrak{X}}_G,(j_!)\mathcal{V}))$. Moreover, one similarly defines boundary completed cohomology
\begin{equation*}
    \pi_{\partial}(K^S_G,\mathcal{V}):=R\Gamma(\partial X_{K^S_G},\mathcal{V})\in D^+_{\textnormal{sm}}(\mathcal{O}/\varpi^m[K_{G,S}]).
\end{equation*}
Note that \cite{CN23}, Lemma 2.1.7 justifies the use of the term completed cohomology i.e., it shows that, after taking cohomology groups, we get back Emerton's completed cohomology as defined in \cite{Em06}. 

Finally, we have the usual phenomenon of completed cohomology at $S$ being independent of the weight at $S$ (cf. \cite{CN23}, Lemma 2.1.8). To state this more precisely, let $\mathcal{V}\in \textnormal{Mod}_{\textnormal{sm}}(\mathcal{O}/\varpi^m[K_{G,S_p\setminus S}\times \Delta_S])$ where $\Delta_S\subset G_S$ is a submonoid containing a compact open subgroup $U_S\subset G_S$ and assume that $\mathcal{V}$ is $\mathcal{O}/\varpi^m$-flat. We can view $\mathcal{V}$ as a $G^{S_p}\times K_{G,S_p\setminus S}\times U_S$-equivariant sheaf on $\overline{\mathfrak{X}}_G$.
\begin{Lemma}\label{Lemma2.2}
We have canonical isomorphisms
\begin{equation*}
    R\Gamma(\overline{\mathfrak{X}}_{G},\mathcal{V})\cong R\Gamma(\overline{\mathfrak{X}}_{G},\mathcal{O}/\varpi^m)\otimes_{\mathcal{O}/\varpi^m}\mathcal{V},
\end{equation*}
and
\begin{equation*}
    R\Gamma(\partial\mathfrak{X}_{G},\mathcal{V})\cong R\Gamma(\partial\mathfrak{X}_{G},\mathcal{O}/\varpi^m)\otimes_{\mathcal{O}/\varpi^m}\mathcal{V}
\end{equation*}
in $D^+_{\textnormal{sm}}(\mathcal{O}/\varpi^m[G^{S_p}\times K_{G,S_p\setminus S}\times U_S])$.
\end{Lemma}
\begin{proof}
This is \cite{CN23}, Lemma 2.1.8.
\end{proof}
Following \cite{CN23}, we use the lemma above to define the object
\begin{equation*}
    R\Gamma(\overline{\mathfrak{X}}_G,\mathcal{V}):= R\Gamma(\overline{\mathfrak{X}}_G,\mathcal{O}/\varpi^m)\otimes_{\mathcal{O}/\varpi^m}\mathcal{V}\in D^+_{\textnormal{sm}}(\mathcal{O}/\varpi^m[G^{S_p}\times K_{G,S_p\setminus S}\times \Delta_S])
\end{equation*}
in a way that is independent of the choice of $U_S$.
In particular, we can endow $R\Gamma(X_{K_G}, \mathcal{V})$ with a natural $\mathcal{H}(G^{S_p},K^{S_p}_G)\otimes \mathcal{H}(\Delta_S,U_S)$-action for any choice of compact open subgroup $U_S\subset G_S$

We also have a version of Lemma~\ref{Lemma2.2} with $\mathcal{V}$ being a complex of sheaves with bounded cohomology. Namely, consider $\mathcal{V}\in D^b_{\textnormal{sm}}(\mathcal{O}/\varpi^m[K_{G,S_p\setminus S}])$ and denote also by $\mathcal{V}$ the associated object in $D^b(\textnormal{Sh}_{G^{S_p\setminus S}\times K_{G,S_p\setminus S}}(\overline{\mathfrak{X}}_G,\mathcal{O}/\varpi^m))$. One can then make sense of the derived tensor product
\begin{equation*}
    R\Gamma(\overline{\mathfrak{X}}_G,\mathcal{O}/\varpi^m)\otimes_{\mathcal{O}/\varpi^m}^{\mathbf{L}}-:D^b_{\textnormal{sm}}(\mathcal{O}/\varpi^m[K_{G,S_p\setminus S}])\to D^b_{\textnormal{sm}}(\mathcal{O}/\varpi^m[G^{S_p\setminus S}\times K_{G,S_p\setminus S}])
\end{equation*}
as explained on page 13 of \cite{CN23}.\footnote{Note that it is the consideration of the derived tensor product that forces us to switch here to the bounded derived category.} Then \textit{loc. cit.} Lemma 2.1.9 shows that there is a canonical isomorphism
\begin{equation*}
    R\Gamma(\overline{\mathfrak{X}}_G,\mathcal{O}/\varpi^m)\otimes_{\mathcal{O}/\varpi^m}^{\mathbf{L}}\mathcal{V}\xrightarrow{\sim}R\Gamma(\overline{\mathfrak{X}}_G,\mathcal{V})
\end{equation*}
in $D^b_{\textnormal{sm}}(G^{S_p\setminus S}\times K_{G,S_p\setminus S})$.
\vspace{5mm}

We finally define cohomology complexes with $\mathcal{O}$-coefficients by taking homotopy limit. For this we start with an $\mathcal{O}[K_{G,S}]$-module $\mathcal{V}$, finite free as an $\mathcal{O}$-module. We set
\begin{equation*}
    R\Gamma_{(c)}(X_{K_G},\mathcal{V}):=\varprojlim_m R\Gamma_{(c)}(X_{K_G},\mathcal{V}/\varpi^m)\in D^+(\mathcal{O}),
\end{equation*}
where the projective limit is understood as a homotopy limit. One can endow these complexes with an action of the Hecke algebra $\mathcal{H}(G^S,K_G^S)\otimes_{\mathbf{Z}}\mathcal{O}$ (see the footnote on Page 14 of \cite{CN23} for a discussion on how it relates to the Hecke action defined in \cite{NT16}).

The next lemma explains that after taking cohomology we get back classical Betti cohomology with $\mathcal{O}$-coefficients. Such an argument will be used at several places whenever the Mittag--Leffler condition holds.
\begin{Lemma}\label{Lemma2.3}
Let $K_G\subset G(\mathbf{A}^{\infty}_{F})$ be a good subgroup, $S$ be a finite set of finite places of $F$ and $\mathcal{V}$ an $\mathcal{O}[K_{G,S}]$-module, finite free as an $\mathcal{O}$-module. Then, for every $j\in \mathbf{Z}_{\geq 0}$, we have a natural identification
\begin{equation*}
    H^j(R\Gamma_{(c)}(X_{K_G},\mathcal{V}))\cong H^j_{(c)}(X_{K_G},\mathcal{V}).
\end{equation*}
\end{Lemma}
\begin{proof}
By \cite[\href{https://stacks.math.columbia.edu/tag/0CQE}{Lemma 0CQE}]{stacks-project}, it suffices to prove that the higher inverse limit
\begin{equation*}
    R^1\varprojlim_mH^j_{(c)}(X_{K_G},\mathcal{V}/\varpi^m)
\end{equation*} 
vanishes. This vanishing is ensured once we prove that the inverse system
\begin{equation*}
    \{H^j_{(c)}(X_{K_G},\mathcal{V}/\varpi^m)\}_{m\geq 0}
\end{equation*}
satisfies Mittag-Leffler. To see this, we note that, by \cite{CN23}, Lemma 2.1.6, all the appearing cohomology groups are finite. Therefore, the images of all the transition maps in the inverse system must stabilise after finitely many steps.
\end{proof}

We now recall the definition of the unnormalised Satake transform. Assume that $G$ is reductive and $P=M\rtimes N\subset G$ is a parabolic subgroup with its Levi decomposition. Given a good subgroup $K_G\subset G(\mathbf{A}_F^{\infty})$, set $K_P=K_G\cap P(\mathbf{A}^{\infty}_F)$, $K_N= K_G\cap N(\mathbf{A}^{\infty}_F)$ and $K_M=\textnormal{im}(K_P\to M(\mathbf{A}_F^{\infty}))$. We call $K_G$ \textit{decomposed} with respect to $P=M\rtimes N$, if $K_P=K_M\rtimes K_N$; equivalently, if $K_M=K_G\cap M(\mathbf{A}^{\infty}_F)$.

Assume $K_G$ is decomposed with respect to $P=M\rtimes N$ and let $S$ be a finite set of finite places such that, for $v\notin S$, $K_{G,v}$ is a hyperspecial maximal compact in $G(F_v)$. Then we have the usual maps on Hecke algebras
\begin{equation*}
    r_P:\mathcal{H}(G^S,K_G^S)\to \mathcal{H}(P^S,K_P^S) \textnormal{ and } r_M:\mathcal{H}(P^S,K_P^S)\to \mathcal{H}(M^S,K_M^S)
\end{equation*}
given by "restriction to $P$" and "integration along $N$", respectively (cf. \cite{NT16}, 2.2.3, 2.2.4). We then can define $\mathcal{S}:=r_M\circ r_P$, the unnormalised Satake transform.

\subsection{Hecke algebras of types}\label{sec2.3} We will make use of the Hecke algebra and Bernstein centre action at $p$ with respect to locally algebraic types. For this we briefly recall the content of Appendix A.1 and A.4 of \cite{Vig04}. First we consider the following general setup. Let $R$ be a commutative ring, $G$ be a locally profinite group, $K\subset G$ an open subgroup and $\sigma$ be a smooth $R[K]$-module, finitely generated over $R$. 

Given $\sigma$, we can define the corresponding Hecke algebra
\begin{equation*}
    \mathcal{H}(\sigma):=\textnormal{End}_{R[G]}(\textnormal{c-Ind}_{K}^{G}\sigma).
\end{equation*} For $\pi \in \textnormal{Mod}_{\textnormal{sm}}(R[G])$, $\mathcal{H}(\sigma)$ acts on the space of invariants
\begin{equation*}
    \Hom_{R[G]}(\textnormal{c-Ind}_{K}^{G}\sigma,\pi)=\Hom_{R[K]}(\sigma,\pi)=\sigma^{\vee}\otimes_{R[K]} \pi
\end{equation*}
on the right.
We also note that $\mathcal{H}(\sigma)$ can be identified with the convolution algebra of compactly supported functions $f:G\to \textnormal{End}_R(\sigma)$ satisfying $f(k_1gk_2)=\sigma(k_1)f(g)\sigma(k_2)$ for every $k_1,k_2\in K$ and $g\in G$. The isomorphism is realised by acting with the convolution algebra on $\textnormal{c-Ind}_{K}^{G}\sigma$ via convolution. From this description it is clear that $\mathcal{H}(\sigma)$ is spanned over $R$ by elements represented by pairs $[g,\psi]$ where $g\in G$ and $\psi\in \textnormal{End}_R(\sigma)$ such that 
\begin{equation}\label{eq2.1}
    \sigma(k)\circ \psi=\psi\circ \sigma (g^{-1}k g)
\end{equation} 
for every $k\in K\cap gKg^{-1}$. More precisely, such a pair gives rise to the function $G\to \textnormal{End}_R(\sigma)$ supported on $KgK$ that sends $g$ to $\psi$. Moreover, under the mentioned identification, $[g,\psi]$ acts on $\phi\in \Hom_{R[K]}(\sigma,\pi)$ by the formula
\begin{equation*}
    \phi\cdot [g,\psi]:v\mapsto \sum_i\pi(g_i)^{-1}\phi([g,\psi](g_i)v)
\end{equation*}
for $v\in \sigma$ and $KgK=\coprod_iKg_i$.

Note that we have an anti-involution
\begin{equation*}
    \mathcal{H}(\sigma)\xrightarrow{\sim}\mathcal{H}(\sigma^{\vee}),
\end{equation*}
\begin{equation*}
    [g,\psi]\mapsto [g^{-1},\psi^{t}].
\end{equation*}
Moreover, $[h,\chi]=[g^{-1},\psi^t]\in \mathcal{H}(\sigma^{\vee})$ acts on $\sum_jf_j\otimes p_j\in \sigma^{\vee}\otimes_{R[K]}\pi=(\sigma^{\vee}\otimes_R\pi)^K$ via the formula
\begin{equation*}
    [h,\chi]\cdot (\sum_jf_j\otimes p_j)=\sum_i\left(\sum_j[h,\chi](h_i)f_j\otimes \pi(h_i)(p_j)\right)=
\end{equation*}
\begin{equation*}
    =\sum_i\left(\sum_j[g^{-1},\psi^t](g_i^{-1})f_j\otimes \pi(g_i^{-1})(p_j)\right)
\end{equation*}
for $\coprod_i h_iK=KhK=Kg^{-1}K=\coprod_ig_i^{-1}K=(\coprod_iKg_i)^{-1}$. This gives rise to a left action of $\mathcal{H}(\sigma^{\vee})$ on $\sigma^{\vee}\otimes_{R[K]}\pi$. Moreover, as the computation shows, the anti-isomorphism $\mathcal{H}(\sigma)\xrightarrow{\sim}\mathcal{H}(\sigma^{\vee})$ intertwines the two actions under the identification $\Hom_{R[K]}(\sigma,\pi)\cong \sigma^{\vee}\otimes_{R[K]}\pi$.

We are now ready to equip the cohomology of locally symmetric spaces with a Hecke action at $p$. For this, we revisit the setup of \S\ref{sec2.1}, and $G$ will again denote a connected linear algebraic group $G$ over $F$ admitting a model over $\mathcal{O}_F$. Before stating the lemma, we remind the reader that, for a compact open subgroup $K_G\subset G(\mathbf{A}_F^{\infty})$ and a smooth $\mathcal{O}/\varpi^m[K_{G,S}]$-module $\sigma$, we introduced a $G^S\times K_{G,S}$-equivariant sheaf on $\overline{\mathfrak{X}}_G$ and, by descent, a sheaf on $\overline{X}_{K_G}$. Moreover, by abuse of notation, we denoted all of these objects by $\sigma$.
\begin{Lemma}\label{Lem2.4}
Let $S\subset S(F)$ be a finite set of finite places, $K_G\subset G(\mathbf{A}^{\infty}_F)$ a good subgroup, and $ \sigma\in \textnormal{Mod}_{\textnormal{sm}}(\mathcal{O}/\varpi^m[K_{G,S}])$, finite free as an $\mathcal{O}/\varpi^m$-module. Then the diagram of derived functors
\begin{equation*}
    \begin{tikzcd}[scale cd=0.7]
	{D^+\textnormal{Sh}_{G(\mathbf{A}_F^{\infty})}(\overline{\mathfrak{X}}_{G},\mathcal{O}/\varpi^m)} && {D^+_{\textnormal{sm}}(\mathcal{O}/\varpi^m[G(\mathbf{A}_F^{\infty})])} && {D^+(\mathcal{H}(G^S,K_G^S)\times\mathcal{H}(\sigma^{\vee}))} \\
	{D^+\textnormal{Sh}_{K_{G}}(\overline{\mathfrak{X}}_{G},\mathcal{O}/\varpi^m)} && {D^+\textnormal{Sh}(\overline{X}_{K_G},\mathcal{O}/\varpi^m)} && {D^+(\mathcal{O}/\varpi^m)}
	\arrow["{R\Gamma(\overline{\mathfrak{X}}_{G},-)}", from=1-1, to=1-3]
	\arrow["{\textnormal{forget}}"', from=1-1, to=2-1]
	\arrow["{\textnormal{descent}}"', from=2-1, to=2-3]
	\arrow["{R\textnormal{Hom}_{\mathcal{O}/\varpi^m[K_{G}]}(\sigma,-)}", from=1-3, to=1-5]
	\arrow["{R\Gamma(\overline{X}_{K_G},-\otimes_{\mathcal{O}/\varpi^m}\sigma^{\vee})}"', from=2-3, to=2-5]
	\arrow["{\textnormal{forget}}", from=1-5, to=2-5]
\end{tikzcd}
\end{equation*}
is commutative.\footnote{Here by the lower right horizontal arrow we mean the composition of first tensoring over $\mathcal{O}/\varpi^m$ with $\sigma^{\vee}\in \textnormal{Sh}(\overline{X}_{K_G},\mathcal{O}/\varpi^m)$ and then applying derived invariants to the obtained object.}
\end{Lemma}
\begin{proof}
Throughout the proof, we repeatedly use \cite{We94}, Corollary 10.8.3 without further mention.
According to \cite{NT16}, Lemma 2.28, $\Gamma(\overline{\mathfrak{X}}_{G},-)$ preserves injectives. Moreover, the forgetful functor 
\begin{equation*}
(-)^{\sim}:D^+\textnormal{Sh}_{G(\mathbf{A}_F^{\infty})}(\overline{\mathfrak{X}}_{G},\mathcal{O}/\varpi^m)\to D^+\textnormal{Sh}_{K_{G}}(\overline{\mathfrak{X}}_{G},\mathcal{O}/\varpi^m)
\end{equation*} is exact and preserves injectives by \cite{Sch98}, \S3, Corollary 3.\footnote{Note that a running assumption in \cite{Sch98} is that the ring of coefficients is $\mathbf{C}$. However, one sees that the proof of \textit{loc. cit.} goes through without a change also with $\mathcal{O}/\varpi^m$-coefficients.} In particular, we get that the composition of the functors on the top of the square is naturally isomorphic to the functor
\begin{equation*}
    R\Gamma(K_G,R\Gamma(\overline{\mathfrak{X}}_{G},(-)^{\sim})\otimes_{\mathcal{O}/\varpi^m}\sigma^{\vee}): D^+\textnormal{Sh}_{G(\mathbf{A}_F^{\infty})}(\overline{\mathfrak{X}}_{G},\mathcal{O}/\varpi^m)\to D^+(\mathcal{O}/\varpi^m).
\end{equation*}

On the other hand, the proof of Lemma~\ref{Lemma2.2} (cf. \cite{CN23}, Lemma 2.1.8) with arbitrary $\mathcal{G}\in D^+\textnormal{Sh}_{K_G}(\overline{\mathfrak{X}}_G,\mathcal{O}/\varpi^m)$ in place of $\mathcal{O}/\varpi^m$ shows that we have a natural isomorphism
\begin{equation*}
R\Gamma(\overline{\mathfrak{X}}_{G},-)\otimes_{\mathcal{O}/\varpi^m}\sigma^{\vee}\cong R\Gamma(\overline{\mathfrak{X}}_{G},-\otimes_{\mathcal{O}/\varpi^m}\sigma^{\vee})
\end{equation*}
of derived functors $D^+\textnormal{Sh}_{K_{G}}(\overline{\mathfrak{X}}_{G},\mathcal{O}/\varpi^m)\to D^+_{\textnormal{sm}}(\mathcal{O}/\varpi^m[K_{G}])$.

Finally, we note that the descent functor respects tensor products since the pullback functor (its inverse) does. We then obtain a natural isomorphism of derived functors
\begin{equation*}
    R\Gamma(K_G,R\Gamma(\overline{\mathfrak{X}}_{G},-\otimes_{\mathcal{O}/\varpi^m}\sigma^{\vee}))\cong R\Gamma(\overline{X}_{K_{G}},-\otimes_{\mathcal{O}/\varpi^m}\sigma^{\vee}).
\end{equation*}
By putting these observations together, we can conclude.
\end{proof}

We spell out a generalisation of \cite{ACC23}, Proposition 2.2.22 that will be used to make twisting arguments in the proof of local-global compatibility analogous to the ones in \textit{loc. cit.}, Corollary 4.4.8. Let $G=\textnormal{GL}_{n,F}$, and $K\subset \textnormal{GL}_n(\mathbf{A}_{F}^{\infty})$ be a good subgroup, and $\chi: G_F\to \mathcal{O}^{\times}$ be a continuous character such that $\chi\circ \textnormal{Art}_{F_v}$ is trivial on $\det(K_v)$ for each finite place $v\notin T$ of $F$. Let $\sigma\in \textnormal{Mod}_{\textnormal{sm}}(\mathcal{O}/\varpi^m[K_p])$, finite free as an $\mathcal{O}/\varpi^m$-module. Set $\sigma_{\chi}:K_p\to \mathcal{O}^{\times}$ to be the continuous character defined by
\begin{equation*}
    (k_v)_{v\in S_p}\mapsto \prod_{v\in S_p}\chi(\textnormal{Art}_{F_v}(\det(k_v))).
\end{equation*}
Define the isomorphism of $\mathcal{O}$-algebras
\begin{equation*}
    f_{\chi}:\mathcal{H}(G^T,K^T)\otimes_{\mathbf{Z}}\mathcal{H}(\sigma^{\vee})\to \mathcal{H}(G^T,K^T)\otimes_{\mathbf{Z}}\mathcal{H}(\sigma^{\vee}\otimes \sigma_{\chi^{-1}})
\end{equation*}
sending a function $f:\textnormal{GL}_n(\mathbf{A}_F)\to \textnormal{End}(\sigma^{\vee})$ lying in the source of $f_{\chi}$ to the function $f_{\chi}(f):g\mapsto \chi(\textnormal{Art}_F(\textnormal{det}(g)))^{-1}f(g)$.
\begin{Lemma}\label{Lemma2.5}
    Let $K\subset \textnormal{GL}_n(\mathbf{A}_{F}^{\infty})$ be a good subgroup, and $\chi: G_F\to \mathcal{O}^{\times}$ be a continuous character such that $\chi\circ \textnormal{Art}_{F_v}$ is trivial on $\det(K_v)$ for each finite place $v\notin T$ of $F$. Let $\sigma\in \textnormal{Mod}_{\textnormal{sm}}(\mathcal{O}/\varpi^m[K_p])$, finite free as an $\mathcal{O}/\varpi^m$-module. Then there is an isomorphism
    \begin{equation*}
        R\Gamma(X_K,\sigma^{\vee})\cong R\Gamma(X_K,\sigma^{\vee}\otimes \sigma_{\chi^{-1}})
    \end{equation*}
    in $D^+(\mathcal{O}/\varpi^m)$ that is $\mathcal{H}(G^T,K^T)\otimes_{\mathbf{Z}}\mathcal{H}(\sigma^{\vee})$-equivariant when we consider its usual action on the left and the one induced by pre-composition with $f_{\chi}$ on the right.
\end{Lemma}
\begin{proof}
    To see this, we introduce some constructions. For any $\textnormal{GL}_n(\mathbf{A}_{F}^{\infty})$-equivariant sheaf $\mathcal{F}\in\textnormal{Sh}_{\textnormal{GL}_n(\mathbf{A}_F^{\infty})}(\overline{\mathfrak{X}}_{\textnormal{GL}_n},\mathcal{O}/\varpi^m)$, consider the map
    \begin{equation}\label{eq2.2}
        \Gamma(\overline{\mathfrak{X}}_{\textnormal{GL}_n},\mathcal{F})\to \Gamma(\overline{\mathfrak{X}}_{\textnormal{GL}_n},\mathcal{F}),
    \end{equation}
    \begin{equation*}
        s\mapsto s^{\chi}
    \end{equation*}
    defined by the formula $s^{\chi}((x,g)):=\chi(\textnormal{Art}_F(\det(g)))s((x,g)).$
    An easy computation shows that \ref{eq2.2} becomes $\textnormal{GL}_n(\mathbf{A}_F^{\infty})$-equivariant when we twist the target by the character $g\mapsto \chi(\textnormal{Art}_F(\det(g)))^{-1}$. In particular, \ref{eq2.2} descends to a map
\begin{equation}\label{eq2.3}
    \Hom_{\mathcal{O}/\varpi^m[K]}(\sigma,\Gamma(\overline{\mathfrak{X}}_{\textnormal{GL}_n},\mathcal{F}))\to \Hom_{\mathcal{O}/\varpi^m[K]}(\sigma\otimes \sigma_{\chi},\Gamma(\overline{\mathfrak{X}}_{\textnormal{GL}_n},\mathcal{F})).
\end{equation} After unravelling the definition of the Hecke action, one sees that this amounts to saying that the induced map \ref{eq2.3} satisfies the desired Hecke-equivariance of the lemma. 

We can conclude by choosing an injective resolution $\mathcal{O}/\varpi^m\to \mathcal{I}^{\bullet}$ in the category $\textnormal{Sh}_{\textnormal{GL}_n(\mathbf{A}_F^{\infty})}(\overline{\mathfrak{X}}_{\textnormal{GL}_n},\mathcal{O}/\varpi^m)$ and applying the previous observation with the choice of $\mathcal{F}=\mathcal{I}^i$ for every $i\in \mathbf{Z}_{\geq 0}$. Indeed, this is because, by Lemma~\ref{Lem2.4}, applying the forgetful functor to 
\begin{equation*}
    \Hom_{\mathcal{O}/\varpi^m[K]}(\sigma,\Gamma(\overline{\mathfrak{X}}_{\textnormal{GL}_n},\mathcal{I}^{\bullet}))\textnormal{, and }\Hom_{\mathcal{O}/\varpi^m[K]}(\sigma\otimes \sigma_{\chi},\Gamma(\overline{\mathfrak{X}}_{\textnormal{GL}_n},\mathcal{I}^{\bullet}))
\end{equation*} computes $R\Gamma(X_K,\sigma^{\vee})$, and $R\Gamma(X_K,\sigma^{\vee}\otimes \sigma_{\chi^{-1}})$, respectively. Moreover, the Hecke actions come from these identifications.
\end{proof}

\subsection{Hecke equivariance of Poincar\'e duality}\label{sec2.4} In the proof of our local-global compatibility results, we appeal to the Poincar\'e duality isomorphism for the cohomology of the locally symmetric spaces attached to $\textnormal{GL}_{n}/F$. To keep track of the Hecke algebra actions during this process, it is crucial to verify that Poincar\'e duality is equivariant with respect to suitable Hecke actions on the two sides. This is already checked for instance in \cite{NT16}, Proposition 3.7 for unramified prime-to-$p$ places. We will need a version of this Hecke equivariancy for the Hecke algebra actions at $p$. This will need slightly more care and will be handled in this subsection.

We return to our setup in \S\ref{sec2.2}. We let $S\subset S_p(F)$ to be a set of $p$-adic places, $K\subset G(\mathbf{A}_F^{\infty})$ a good subgroup. Let $\sigma\in \textnormal{Mod}_{\textnormal{sm}}(\mathcal{O}/\varpi^m[K_S])$, which we also assume to be finite free over $\mathcal{O}/\varpi^m$. Then, by the previous subsection, for any $G_S$-equivariant sheaf $\mathcal{G}$ on $\overline{X}_{K^S}$ in the sense of \cite{NT16}, Section 2.4, the complex
\begin{equation*}
    R\textnormal{Hom}_{\mathcal{O}/\varpi^m[K_S]}(\sigma,R\Gamma(\overline{X}_{K^S},\mathcal{G}))
\end{equation*}
lies in $D^+(\mathcal{H}(\sigma)^{\textnormal{op}})= D^+(\mathcal{H}(\sigma^{\vee}))$. Then, if we let $\pi:\overline{X}_{K^S}\to \overline{X}_K$ to be the natural projection and $\overline{f}:\overline{X}_{K^S}\to \ast$, Lemma~\ref{Lem2.4} implies that there is an induced morphism of algebras
\begin{equation*}
    \mathcal{H}(\sigma^{\vee})\to \textnormal{End}_{D^+(\mathcal{O}/\varpi^m)}(R\Gamma(\overline{X}_K,\pi_{\ast}^{K_S}(\mathcal{G}\otimes_{\mathcal{O}/\varpi^m}\overline{f}^{\ast}\sigma^{\vee})).
\end{equation*}
To prove Hecke equivariancy of Poincar\'e duality for $R\Gamma(X_K,\sigma^{\vee})$, we give a different description of this action as in \cite{NT16}, Lemma 2.19. As we are working with places above $p$, which are not assumed to be unramified, we need to refine the constructions of \cite{NT16}.

We pick an element $[g,\psi]\in \mathcal{H}(\sigma)^{\textnormal{op}}$ as in the previous subsection. We define an action of this element on cohomology. Let $K':=K^SK'_S$ where $K'_S=K_S\cap g^{-1}K_Sg$. By abuse of notation, whenever $K''_S\subset K_S$ is a compact open subgroup, we denote by $\pi$ the projection $\overline{X}_{K^S}\to \overline{X}_{K^SK''_S}$. We consider the correspondence
$$\begin{tikzcd}
&\overline{X}_{K'}\arrow{ld}[swap]{p_1}\arrow{rd}{p_2} \\
\overline{X}_K& &\overline{X}_K
\end{tikzcd}
$$
where $p_1$ is the natural projection and $p_2$ is given by the map $\overline{X}_{K'}\xrightarrow[]{\cdot g^{-1}}\overline{X}_{K^S\cdot gK'_Sg^{-1}}$ followed by the natural projection. By the intertwining property \ref{eq2.1} of $\psi$, it descends to a map

\begin{equation*}
p_1^{\ast}\pi_{\ast}^{K_S}\overline{f}^{\ast}\sigma=\pi_{\ast}^{K'_S}\overline{f}^{\ast}\sigma
    \xrightarrow[]{\psi}(g^{-1})^{\ast}\pi_{\ast}^{gK'_Sg^{-1}}\overline{f}^{\ast}\sigma=p_2^{\ast}\pi_{\ast}^{K_S}\overline{f}^{\ast}\sigma,
\end{equation*}
giving a cohomological correspondence. For any $G_S$-equivariant sheaf of $\mathcal{O}/\varpi^m$-modules $\mathcal{G}$ on $\overline{X}_{K^S}$, we get an induced map
\begin{equation*}
    \Psi:R\Gamma(\overline{X}_{K'},p_2^{\ast}\pi_{\ast}^{K_S}(\mathcal{G}\otimes_{\mathcal{O}/\varpi^m}\overline{f}^{\ast}\sigma^{\vee}))\cong R\Hom_{\textnormal{Sh}(\overline{X}_{K'},\mathcal{O}/\varpi^m)}(p_2^{\ast}\pi_{\ast}^{K_S}\overline{f}^{\ast}\sigma,p_2^{\ast}\pi_{\ast}^{K_S}\mathcal{G})\to
    \end{equation*}
    \begin{equation*}
    R\Hom_{\textnormal{Sh}(\overline{X}_{K'},\mathcal{O}/\varpi^m)}(p_1^{\ast}\pi_{\ast}^{K_S}\overline{f}^{\ast}\sigma,p_1^{\ast}\pi_{\ast}^{K_S}\mathcal{G})\cong R\Gamma(\overline{X}_{K'},p_1^{\ast}\pi_{\ast}^{K_S}(\mathcal{G}\otimes_{\mathcal{O}/\varpi^m}\overline{f}^{\ast}\sigma^{\vee})),
\end{equation*}
where the map in the middle is induced by the map $\psi$ in the first component and by the multiplication $g^{-1}:p_2^{\ast}\pi_{\ast}^{K_S}\mathcal{G}\cong p_1^{\ast}\pi_{\ast}^{K_S}\mathcal{G}$ in the second component.

We then obtain an endomorphism
\begin{equation*}
    \theta([g,\psi]) \in \textnormal{End}_{D^+(\mathcal{O}/\varpi^m)}(R\Gamma(\overline{X}_K,\pi_{\ast}^{K_S}(\mathcal{G}\otimes_{\mathcal{O}/\varpi^m}\overline{f}^{\ast}\sigma^{\vee})))
\end{equation*} defined by
\begin{equation*}
    R\Gamma(\overline{X}_K,\pi_{\ast}^{K_S}(\mathcal{G}\otimes_{\mathcal{O}/\varpi^m}\overline{f}^{\ast}\sigma^{\vee}))\xrightarrow{p_2^{\ast}}R\Gamma(\overline{X}_{K'},p_2^{\ast}\pi_{\ast}^{K_S}(\mathcal{G}\otimes_{\mathcal{O}/\varpi^m}\overline{f}^{\ast}\sigma^{\vee}))\xrightarrow{\Psi}
\end{equation*}
\begin{equation*}
    \to R\Gamma(\overline{X}_{K'},p_1^{\ast}\pi_{\ast}^{K_S}(\mathcal{G}\otimes_{\mathcal{O}/\varpi^m}\overline{f}^{\ast}\sigma^{\vee}))=R\Gamma(\overline{X}_{K},p_{1,{\ast}}p_1^{\ast}\pi_{\ast}^{K_S}(\mathcal{G}\otimes_{\mathcal{O}/\varpi^m}\overline{f}^{\ast}\sigma^{\vee}))\xrightarrow{\textnormal{tr}}
\end{equation*}
\begin{equation*}
    \to R\Gamma(\overline{X}_{K},\pi_{\ast}^{K_S}(\mathcal{G}\otimes_{\mathcal{O}/\varpi^m}\overline{f}^{\ast}\sigma^{\vee}))
\end{equation*}
where the last map is the trace map coming from the canonical map induced by the adjoint pair $(p_{1,\ast}=p_{1,!},p_1^{\ast}=p_1^!)$.
\begin{Lemma}\label{Lemma2.6}
The endomorphism $\theta([g,\psi])$ of $R\Gamma(\overline{X}_{K},\pi_{\ast}^{K_S}(\mathcal{G}\otimes_{\mathcal{O}/\varpi^m}\overline{f}^{\ast}\sigma^{\vee}))$ coincides with the one induced by $[g,\psi]$ via the recipe of Lemma~\ref{Lem2.4}.
\end{Lemma}
\begin{proof}
    Note that we have a natural isomorphism
    \begin{equation*}
        R\Gamma(\overline{X}_{K},\pi_{\ast}^{K_S}(\mathcal{G}\otimes_{\mathcal{O}/\varpi^m}\overline{f}^{\ast}\sigma^{\vee}))\cong R\Hom_{\textnormal{Sh}(\overline{X}_K,\mathcal{O}/\varpi^m)}(\pi_{\ast}^{K_S}\overline{f}^{\ast}\sigma,\pi_{\ast}^{K_S}\mathcal{G})
    \end{equation*}
    in $D^+(\mathcal{O}/\varpi^m)$. Indeed, this follows from the fact that the functor 
    \begin{equation*}
        \pi_{\ast}^{K_S}:\textnormal{Sh}_{K_S}(\overline{X}_{K^S},\mathcal{O}/\varpi^m)\xrightarrow{}\textnormal{Sh}(\overline{X}_K,\mathcal{O}/\varpi^m)
    \end{equation*} is an equivalence of categories with inverse $\pi^{\ast}$ and pullbacks commute with tensor products. We obtain analogous descriptions of the other two complexes appearing in the definition of $\theta([g,\psi])$.

    We now pick an injective resolution $\mathcal{G}\to \mathcal{I}^{\bullet}$ in $D^+(\textnormal{Sh}_{G_S}(\overline{X}_{K^S},\mathcal{O}/\varpi^m))$. Notice that, by \cite{Sch98}, \S3, Corollary 3, it gives rise to an injective resolution in $D^+(\textnormal{Sh}_{K_S}(\overline{X}_{K^S},\mathcal{O}/\varpi^m))$, and $D^+(\textnormal{Sh}_{K_S'}(\overline{X}_{K^S},\mathcal{O}/\varpi^m))$. Moreover, since the functors $\pi_{\ast}^{K_S}$, $p_1^{\ast}\pi_{\ast}^{K_S}$, and $p_2^{\ast}\pi_{\ast}^{K_S}$ are all equivalences, $\pi_{\ast}^{K_S}\mathcal{I}^{\bullet}$, $p_1^{\ast}\pi_{\ast}^{K_S}\mathcal{I}^{\bullet}$, and $p_2^{\ast}\pi_{\ast}^{K_S}\mathcal{I}^{\bullet}$ are injective resolutions.  If we combine this with the discussion of the previous paragraph, we see that
    \begin{equation*}
        \Gamma(\overline{X}_K,\pi_{\ast}^{K_S}(\mathcal{I}^{\bullet}\otimes_{\mathcal{O}/\varpi^m}\overline{f}^{\ast}\sigma^{\vee}))=\Hom_{K_S}(\sigma,\mathcal{I}^{\bullet}(\overline{X}_{K^S}))
    \end{equation*}
    computes $R\Gamma(\overline{X}_K,\pi_{\ast}^{K_S}(\mathcal{G}\otimes_{\mathcal{O}/\varpi^m}\overline{f}^{\ast}\sigma^{\vee}))$. Analogous observations apply  to the other two complexes appearing in the definition of $\theta([g,\psi])$. We reduced the lemma to comparing the two actions in
    \begin{equation*}
        \textnormal{End}_{\mathcal{O}/\varpi^m}(\Hom_{K_S}(\sigma,\mathcal{I}^{i}(\overline{X}_{K^S})))
    \end{equation*}
    for every $i\in \mathbf{Z}$. The lemma then follows from the concrete description of the effect of the trace map on global sections and an easy unravelling of the definitions.
\end{proof}

\begin{Rem}
Note that our discussion applies also to the case of $R\Gamma_c(X_K,\sigma^{\vee})$. To see this, denote by $j:\mathfrak{X}_G\hookrightarrow \overline{\mathfrak{X}}_G$ the natural open immersion and by $\overline{f}:\overline{\mathfrak{X}}_G\to \ast$ the projection to the point. Then the claim follows from the fact that we have an isomorphism
\begin{equation*}
j_{!}j^{\ast}\overline{f}^{\ast}\sigma^{\vee}\cong j_{!}\mathcal{O}/\varpi^m\otimes_{\mathcal{O}/\varpi^m}\overline{f}^{\ast}\sigma^{\vee}
\end{equation*}
of $G(\mathbf{A}_F^{\infty})$-equivariant sheaves.
This identification follows from an equivariant version of \cite{KS94}, Proposition 2.5.13. 
\end{Rem}

\begin{Cor}\label{Cor2.8}
Given $\sigma\in \textnormal{Mod}_{\textnormal{sm}}(\mathcal{O}/\varpi^m[K_S])$, finite free as an $\mathcal{O}/\varpi^m$-module. The Verdier duality isomorphism

\begin{equation*}
    R\Hom_{\mathcal{O}/\varpi^m}(R\Gamma_c(X_K,\sigma^{\vee}),\mathcal{O}/\varpi^m)\cong R\Gamma(X_K,\sigma)[\dim_{\mathbf{R}}X_K]
\end{equation*}
is equivariant with respect to the natural \textit{left} action of $\mathcal{H}(\sigma)$ on the right and the one induced by the anti-isomorphism
\begin{equation*}
    \mathcal{H}(\sigma)\xrightarrow{\sim}\mathcal{H}(\sigma^{\vee}),
\end{equation*}
\begin{equation*}
    [g,\psi]\mapsto [g^{-1},\psi^t]
\end{equation*}
on the left.
\end{Cor}
\begin{proof}
Just as in the proof of \cite{NT16}, Proposition 3.7, this follows from the functoriality of Verdier duality and Lemma~\ref{Lemma2.6} taking into account that passing to duals interchanges pullbacks with traces in the definition of $\theta([g,\psi])$.
\end{proof}

\subsection{The quasi-split unitary group}\label{Sec2.5}
From now on we specialise to our setup of interest. The two groups we will be interested in are the quasi-split unitary group $U(n,n)$ and the general linear group appearing as its Levi subgroup. In particular, we fix an integer $n\geq 2$ and an imaginary CM field $F$ with maximal totally real subfield $F^+\subset F$. Denote by $c\in \textnormal{Gal}(F/F^+)$ its complex conjugation and set $\overline{S}_p:=S_p(F^+)$ resp. $S_p:=S_p(F)$. Consider the $2n\times 2n$ matrix
$$J_n:=\begin{pmatrix}
0 & \Psi_n \\
-\Psi_n & 0
\end{pmatrix}$$
where $\Psi_n$ denotes the $n\times n$ matrix with $1$'s on the anti-diagonal and $0$'s elsewhere. We then set $\widetilde{G}/\mathcal{O}_{F^+}$ to be the group scheme that, for an $\mathcal{O}_{F^+}$-algebra $R$, has $R$-points given by
\begin{equation*}
    \widetilde{G}(R)=\{g\in \textnormal{GL}_{2n}(R\otimes_{\mathcal{O}_{F^+}}\mathcal{O}_F)\mid {}^tgJ_ng^c=J_n\}
\end{equation*}
where $^t(-)$ denotes the transpose matrix. This is an integral model of the quasi-split unitary group $U(n,n)/F^+$, a form of $\textnormal{GL}_{2n}$, splitting after base change to $F$. In particular, it becomes reductive after base change to $\mathcal{O}_{F^+_{\Bar{v}}}$ for $\Bar{v}$ a finite place of $F^+$ which is unramified in $F$.

We let $P\subset \widetilde{G}$ to be the Siegel parabolic consisting of block upper triangular matrices with blocks of size $n\times n$. Let $P=G\ltimes U$ be a Levi decomposition such that $G$ is given by the closed subgroup of block diagonal matrices. Then $G$ can be identified with $\textnormal{Res}_{\mathcal{O}_F/\mathcal{O}_{F^+}}\textnormal{GL}_n$ as in \cite{NT16}, Lemma 5.1. Namely, if we denote by $(-)^{\ast}$ the anti-involution of $\textnormal{Res}_{\mathcal{O}_{F}/\mathcal{O}_{F^+}}\textnormal{GL}_n$ given by $A^{\ast}=\psi_n^tA^c\psi_n^{-1}$ then, by its very definition, $P\subset \widetilde{G}$ can be identified with the subgroup of $\textnormal{Res}_{\mathcal{O}_F/\mathcal{O}_{F^+}}\textnormal{GL}_{2n}$ of the form
$$\begin{pmatrix}
A & B \\
C & D
\end{pmatrix} =\begin{pmatrix}
D^{-\ast} & B \\
0 & D
\end{pmatrix}$$
where $D\in \textnormal{Res}_{\mathcal{O}_F/\mathcal{O}_{F^+}}\textnormal{GL}_n$ without any condition and $B$ is so that $B^{\ast}=B$. Under this identification, the subgroup defined by $B=0$ is the Levi subgroup $G$. Then $\begin{pmatrix}
D^{-\ast} & B \\
0 & D
\end{pmatrix}\mapsto D$ gives the identification $G\cong \textnormal{Res}_{\mathcal{O}_{F}/\mathcal{O}_{F^+}}\textnormal{GL}_n$.

We will write $\widetilde{X}$ for the symmetric space $X^{\widetilde{G}}$ and $\widetilde{X}_{\widetilde{K}}$ for the associated locally symmetric space for a good subgroup $\widetilde{K}\subset \widetilde{G}(\mathbf{A}_{F^+}^{\infty})$. Similarly, we denote by $X$ the symmetric space $X^G$ and write $X_K$ for the associated locally symmetric space for a good subgroup $K\subset G(\mathbf{A}^{\infty}_{F^+})=\textnormal{GL}_n(\mathbf{A}_F^{\infty})$. 

Write $T\subset \widetilde{B}\subset \widetilde{G}$ for the subgroup consisting, respectively, of the diagonal and upper triangular matrices of $\widetilde{G}$. These form a maximal torus and a Borel subgroup of $\widetilde{G}$. Moreover, $B=\widetilde{B}\cap G\subset G$ is the Borel subgroup of upper triangular matrices.

Recall that, for a place $\Bar{v}$ of $F^+$ splitting in $F$, a choice of place $v\mid \Bar{v}$ in $F$ gives a canonical isomorphism $\iota_v:\widetilde{G}(F^+_{\Bar{v}})\cong \textnormal{GL}_{2n}(F_v)$. Indeed, there is an isomorphism $F^+_{\Bar{v}}\otimes_{F^+}F\cong F_v\times F_{v^c}$ and $\iota_v$ is the projection to the first factor of the natural inclusion $\widetilde{G}(F^+_{\Bar{v}})\subset \textnormal{GL}_{2n}(F_v)\times \textnormal{GL}_{2n}(F_{v^c})$. Under $\iota_v$, $P(F^+_{\Bar{v}})$ is identified with the standard parabolic subgroup $P_{(n,n)}(F_v)\subset \textnormal{GL}_{2n}(F_v)$ of block upper triangular matrices of type $(n,n)$ and $G(F^+_{\Bar{v}})$ with its standard Levi subgroup of block diagonal matrices. Similarly, $\widetilde{B}(F^+_{\Bar{v}})$ is identified with the subgroup of upper triangular matrices and $T(F^+_{\Bar{v}})$ with the diagonal matrices. Moreover, for any parabolic subgroup $\widetilde{B}_{F^+_{\Bar{v}}}\subset \widetilde{Q}\subset P_{F^+_{\Bar{v}}}$, $\widetilde{Q}(F^+_{\Bar{v}})$ is identified with a standard parabolic subgroup $P_{(n_1,...,n_{t})}(F_v)\subset \textnormal{GL}_{2n}(F_v)$ where $(n_1,...,n_t)$ refines $(n,n)$. Let $\widetilde{Q}=\widetilde{M}\ltimes \widetilde{N}$ its standard Levi decomposition and set $M=G\cap \widetilde{M}$.

Note that, since the inclusion 
\begin{equation*}
    G(F^+_{\Bar{v}})=\textnormal{GL}_n(F_v)\times\textnormal{GL}_n(F_{v^c})\hookrightarrow\textnormal{GL}_{2n}(F_v)
\end{equation*}
under $\iota_v$ is given by
\begin{equation}\label{eq-unitary}
    (A,B)\mapsto \begin{pmatrix}
(\Psi_n {}^{t}B^{-1}\Psi_n)^c & 0 \\
0 & A
\end{pmatrix},
\end{equation}
we have $M(F^+_{\Bar{v}})= M_{(n_{k+1},...,n_{t})}(F_v)\times M_{(n_k,...,n_1)}(F_{v^c})\hookrightarrow G(F^+_{\Bar{v}})$
where $1\leq k\leq t$ is so that $n_1+...+n_k=n$. We set $\theta_n: \textnormal{GL}_n(F_{v^c})\cong \textnormal{GL}_n(F_{v})$ to be the map $B\mapsto (\Psi_n {}^{t}B^{-1}\Psi_n)^c$ above.
\subsection{Inertial local Langlands for $\textnormal{GL}_n$}\label{sec2.6}

Set $L=F_v$ for some $v\in S_p(F)$. According to \cite{BD84}, the category $\textnormal{Mod}_{\textnormal{sm}}(\overline{E}[G(L)])$ admits a direct sum decomposition
\begin{equation*}
     \oplus_{\Omega}\textnormal{Mod}_{\textnormal{sm}}(\overline{E}[G(L)])[\Omega]
\end{equation*}
into so-called Bernstein blocks. In terms of the local Langlands correspondence, two \textit{irreducible} representation $\pi_1,\pi_2\in \textnormal{Mod}_{\textnormal{sm}}(\overline{E}[G(L)])$ correspond to the same Bernstein block if and only if
\begin{equation*}
     \textnormal{rec}(\pi_1)^{ss}|_{I_L}\cong \textnormal{rec}(\pi_2)^{ss}|_{I_L},
\end{equation*}
 and a general representation $\pi\in  \textnormal{Mod}_{\textnormal{sm}}(\overline{E}[G(L)])$ lies in $\Omega$ if each of its Jordan--H\"older constituents does.
 Moreover, the centre $\mathfrak{z}_{\Omega}$ of the category $\textnormal{Mod}_{\textnormal{sm}}(\overline{E}[G(L)])[\Omega]$ is called the \textit{Bernstein centre} possessing the following property. Being the centre of the Bernstein block, it acts on each object lying in $\Omega$. In particular, for any \textit{irreducible} $\pi$ lying in $\Omega$, the natural action induces a character $\chi_{\pi}:\mathfrak{z}_{\Omega}\to \overline{E}$. Then, given a pair of irreducible objects $\pi_1,\pi_2\in \textnormal{Mod}_{\textnormal{sm}}(\overline{E}[G(L)])[\Omega]$, $\chi_{\pi_1}=\chi_{\pi_2}$ if and only if $\pi_1$ and $\pi_2$ have the same supercuspidal support (cf. \cite{BD84}).\footnote{For a brief overview of these results  stated with more care, see \cite{Hel16}, \S3.}

Work of Bushnell--Kutzko \cite{BK99} shows that, given any Bernstein block $\Omega$, there is always a pair $(J,\sigma)$ of a compact open $J\subset G(\mathcal{O}_L)$ and an irreducible $\overline{E}$-representation of $J$ such that $\pi\in \textnormal{Mod}_{\textnormal{sm}}(\overline{E}[G(L)])$ lies in $\Omega$ if and only if it is generated by its $\sigma$-isotypic vectors. Such a pair is then called a \textit{semisimple Bushnell--Kutzko type} for the block $\Omega$. Given a semisimple Bushnell--Kutzko type $(J,\sigma)$, \cite{BK98}, Theorem 4.3 shows that taking $\sigma$-invariants sets up an equivalence of categories
\begin{equation*}
    \textnormal{Mod}_{\textnormal{sm}}(\overline{E}[G(L)])[\Omega]\xrightarrow{\sim} \textnormal{Mod}(\mathcal{H}(\sigma)),
\end{equation*}
\begin{equation*}
    \pi\mapsto \Hom_J(\sigma,\pi)\cong\Hom_{G(L)}(\textnormal{c-Ind}_J^{G(L)}\sigma,\pi) 
\end{equation*}
where $\mathcal{H}(\sigma):=\textnormal{End}_{G(L)}(\textnormal{c-Ind}_J^{G(L)}\sigma)$. In particular, we see that the action of the Bernstein centre on $\textnormal{c-Ind}_{J}^{G(L)}\sigma$ identifies $\mathfrak{z}_{\Omega}$ with the centre $Z(\mathcal{H}(\sigma))$ of $\mathcal{H}(\sigma)$. Finally, given a Bernstein block $\Omega$, as was observed in \cite{CEGGPS16}, \S3.13,\footnote{See in particular \textit{loc. cit.} Proposition 3.23.} after possibly enlarging $E$, the Bernstein centre $\mathfrak{z}_{\Omega}$ admits a model $\mathfrak{z}_{\Omega,E}$ over $E$ acting on any $\pi\in \textnormal{Mod}_{\textnormal{sm}}(E[G(L)])$ such that $\pi\otimes_E\overline{\mathbf{Q}}_p$ lies in $\Omega$. Moreover, for any type $(J,\sigma)$ for $\Omega$ with a model $\sigma_E$ over $E$, the natural action $\mathfrak{z}_{\Omega,E}\to \textnormal{End}_{G(L)}(\textnormal{c-Ind}_J^{G(L)}\sigma_E)=\mathcal{H}(\sigma_E)$ induces an isomorphism $\mathfrak{z}_{\Omega,E}\cong Z(\mathcal{H}(\sigma_E))$ as one notes by combining \cite{CEGGPS16}, (3.15), Lemma 3.18 and Proposition 3.23. In particular, from now on for a given Bernstein block $\Omega$ we always assume that $E$ is sufficiently large so that the centre is already defined over $E$.

In \cite{SZ99}, the authors refine the Bernstein decomposition of $\textnormal{Mod}_{\textnormal{sm}}(\overline{E}[G(L)])$ to a "stratification" of the category and construct a type theory with respect to this stratification as we will discuss now. We state their results in terms of the local Langlands correspondence following \cite{BC09}, \S6.5.

\begin{Def}\label{Def2.9}
    We define a \textit{Weil--Deligne inertial type} (of $L$ over $E$) to be an isomorphism class of pairs $\tau=(\rho_{\tau},N_{\tau})$ such that $\rho_{\tau}: I_L\to \textnormal{GL}_n(E)$ is a representation of the inertia subgroup $I_L\subset W_L$ with open kernel, $N_{\tau}\in M_n(E)$ is a nilpotent matrix such that there exists a Weil--Deligne representation $(r,N)$ of $L$ and an isomorphism $(r,N)|_{I_L}\cong(\rho_{\tau},N_{\tau})$.
\end{Def}

Given a nilpotent matrix $N\in M_m(E)$, its Jordan normal form gives rise to a partition $P_{N}$ of $m$. A partition $P$ can be viewed uniquely as a decreasing function $P:\mathbf{Z}_{>0}\to \mathbf{Z}_{\geq 0}$ with finite support (where $P$ is a partition of $\sum_{i\in \mathbf{Z}_{>0}}P(i)$). Given two nilpotent matrices $N_1,N_2\in M_m(E)$, we write $N_1\preceq N_2$ if and only if $\sum_{1\leq i\leq k}P_{N_1}(i)\leq \sum_{1\leq i\leq k}P_{N_2}(i)$ for every $k\in \mathbf{Z}_{\geq 1}$.
We record the following observation (cf. \cite{BC09}, Proposition 7.8.1).
\begin{Prop}\label{Prop2.10}
Let $N_1,N_2\in M_m(E)$ be two nilpotent matrices. Then $N_1\preceq N_2$ if and only if, for all $i\in \mathbf{Z}_{\leq 1}$, we have
\begin{equation*}
    \textnormal{rank}(N^i_1)\leq \textnormal{rank}(N^i_2).
\end{equation*}
\end{Prop}
Given a Weil--Deligne inertial type $\tau=(\rho_{\tau},N_{\tau})$ of $L$ over $E$, and a finite dimensional irreducible $\overline{\mathbf{Q}}_p$-representation $\theta:I_L\to \textnormal{GL}(V_{\theta})$ with open kernel, we can consider the $\theta$-isotypic component $\rho_{\tau}[\theta]:I_L\to \textnormal{GL}(V_{\tau}[\theta])$ of $\rho_{\tau}\otimes_E\overline{\mathbf{Q}}_p$. As $N_{\tau}$ commutes with the action of $I_L$, it restricts to a nilpotent endomorphism $N_{\tau}[\theta]\in \textnormal{End}(V_{\tau}[\theta])$.

\begin{Def}\label{DefPartOrd}
    Given two Weil--Deligne inertial types $\tau_1,\tau_2$ of $L$, we write $\tau_1\preceq \tau_2$ if $\rho_{\tau_1}\cong\rho_{\tau_2}$ and $N_{\tau_1}[\theta]\preceq N_{\tau_2}[\theta]$ for every irreducible $\overline{\mathbf{Q}}_p$-representation $\theta:I_L\to \textnormal{GL}(V_{\theta})$ with open kernel. Moreover, given two Weil--Deligne representations $r_1,r_2$ of $L$, we write $r_1\preceq r_2$ if $r_1|_{I_L}\preceq r_2|_{I_L}$.
\end{Def}

We note that the partial order (on Weil--Deligne representations) appearing in Definition~\ref{DefPartOrd} is the one defined in \cite{BC09}, Definition 6.5.1,\footnote{They introduce the partial order for irreducible smooth representations of $\textnormal{GL}_n(L)$ but, under the local Langlands correspondence, it translates to our definition (cf \cite{BC09}, Proposition 7.8.1).} \cite{Var14}, Defintion 8.3 and \cite{HLLM18}, Definition 2.5.3, respectively.

\begin{Th}[\textnormal{\cite{SZ99}}]\label{Th2.12} Let $\tau$ be a Weil--Deligne inertial type. Then there is a smooth irreducible $E$-representation $\sigma(\tau)$ of $G(\mathcal{O}_L)$ such that, for any irreducible smooth representation $\pi$ of $G(L)$, the following hold.
\begin{enumerate}
    \item If $\pi|_{G(\mathcal{O}_L)}$ contains $\sigma(\tau)$, then $\textnormal{rec}(\pi)|_{I_L}\preceq \tau$;
    \item if $\textnormal{rec}(\pi)|_{I_L}\cong \tau$, then $\pi|_{G(\mathcal{O}_L)}$ contains $\sigma(\tau)$ with multiplicity one;
    \item if $\textnormal{rec}(\pi)|_{I_L}\preceq \tau$ and $\pi$ is generic, then $\pi|_{G(\mathcal{O}_L)}$ contains $\sigma(\tau)$.
\end{enumerate}
\end{Th}
\begin{proof}
    See \cite{HLLM18}, Theorem 2.5.4 and the references therein.
\end{proof}
We point out that the Theorem makes no mention about the uniqueness of $\sigma(\tau)$. Throughout this text, given a Weil--Deligne inertial type $\tau$, we work with the $\sigma(\tau)$ constructed in \cite{SZ99}.
\begin{Rem}
    Notice that when $\tau=(\rho_{\tau},N_{\tau})$ is so that $N_{\tau}=0$, we obtain the potentially crystalline inertial local Langlands of \cite{CEGGPS16}, Theorem 3.7.
\end{Rem}
Note that each Weil--Deligne inertial type $\tau$ gives rise to an inertial type $\rho_{\tau}$ in the classical sense and henceforth to a Bernstein block $\Omega$. Moreover, if $(J,\sigma)$ is a type for $\Omega$, then, by construction, $\sigma(\tau)$ is a direct summand of $\textnormal{c-Ind}_J^{G(\mathcal{O}_L)}\sigma$. In particular, $\mathfrak{z}_{\Omega}$ acts on $\textnormal{c-Ind}_{G(\mathcal{O}_L)}^{G(L)}\sigma(\tau)$. One sees that this, in fact, is a faithful action,\footnote{In fact, \cite{Pyv18b}, Theorem 7.1 shows that $\mathcal{H}(\sigma(\tau))$ is a free module over $\mathfrak{z}_{\Omega}$.} yielding an injection $\mathfrak{z}_{\Omega}\hookrightarrow\mathfrak{z}_{\tau}:= Z(\mathcal{H}(\sigma(\tau)))$.

\subsection{Local systems on $\widetilde{X}_{\widetilde{K}}$ and $X_K$}\label{sec2.7}
We now introduce the local systems we will be working with. These will be constructed from locally algebraic representations $\sigma_{\textnormal{alg}}\otimes \sigma_{\textnormal{sm}}$ where the algebraic part $\sigma_{\textnormal{alg}}$ will encode the weight of the automorphic representations considered (i.e. their shape at $\infty$) and the smooth part $\sigma_{\textnormal{sm}}$ will pin down their inertial type at $p$.

Our integral coefficient systems will depend on the choice of a prime $p$. We further set $E/\mathbf{Q}_p$ to be our coefficient field, a sufficiently large subfield of $\overline{\mathbf{Q}}_p$ finite over $\mathbf{Q}_p$ such that $\textnormal{Hom}(F,E)=\textnormal{Hom}(F,\overline{\mathbf{Q}}_p)$, and denote by $\mathcal{O}$ its ring of integers. Fix a choice of uniformiser $\varpi\in \mathcal{O}$ as well.

We start by introducing our locally algebraic representations for $G$. For each $p$-adic place $v$ of $F$, we consider a standard (possibly non-proper) parabolic subgroup $Q_v\subset \textnormal{GL}_n$ with standard Levi decomposition $Q_v= M_v\ltimes N_v$ and consider the corresponding parahoric subgroup scheme $\prescript{v}{}{\mathcal{Q}}$ of $\textnormal{GL}_{n}$. For $S\subset S_p$, we set $Q_S:=\prod_{v\in S}Q_v$ and $Q_p:=Q_{S_p}$. Set $\mathcal{Q}_v:=\prescript{v}{}{\mathcal{Q}}(\mathcal{O}_{F_v})\subset\textnormal{GL}_n(\mathcal{O}_{F_v})$ to be the corresponding parahoric subgroup. For a finite set of $p$-adic places $S\subset S_p$, we set $\mathcal{Q}_S:=\prod_{v\in S} \mathcal{Q}_v\subset\textnormal{GL}_n(\widehat{\mathcal{O}}_{F,S})$. Moreover, for any $v\in S_p$, and integers $c\geq b\geq 0$ with $c\geq 1$, we denote by $\mathcal{Q}_v(b,c)\subset \mathcal{Q}_v$ the subgroup of matrices which are block upper triangular modulo $\varpi_v^c$ and block unipotent modulo $\varpi_v^b$, where $\varpi_v\in \mathcal{O}_{F_v}$ is some choice of uniformiser. Extend the definition in the obvious way to define $\mathcal{Q}_S(b,c)$. In particular, we have $\mathcal{Q}_v(0,1)=\mathcal{Q}_v$. Note that $\mathcal{Q}_v(b,c)$ admits an Iwahori decomposition $\overline{N}_v^cM_v^bN_v^0$\footnote{Recall from \ref{Notations} that for $H$ being any of the groups $\overline{N}_v,M_v,$ or $N_v$, and $n\in \mathbf{Z}_{\leq 0}$, $H^n$ denotes the subgroup of matrices of $H(\mathcal{O}_{F_v})$ that reduce to the identity modulo $\varpi_v^n$.} and therefore the formalism of \cite{ACC23}, 2.1.9 applies. 

Such parahoric subgroups will be our level subgroups for which we introduce locally algebraic representations. These representations will then yield local systems convenient for the development of $Q_{p}$-ordinary Hida theory. After taking $Q_{p}$-ordinary parts, the cohomology of these local systems will encode integral $Q_{p}$-ordinary automorphic representations with prescribed weights at $\infty$ and inertial types at $p$. When $Q_{p}$ is taken to be $\prod_v\textnormal{GL}_n$, this will simply mean pinning down the weight of the automorphic  and the inertial type at $p$ of the whole automorphic representation. 

We first take care of the algebraic part. As usual, the character group of $(\textnormal{Res}_{F^+/\mathbf{Q}}T)_E=(\textnormal{Res}_{F/\mathbf{Q}}T_n)_E$, for $T_n\subset \textnormal{GL}_n$ the subgroup of diagonal matrices, can be identified with $(\mathbf{Z}^n)^{\Hom(F,E)}$. Denote by $\mathbf{Z}_+^n\subset \mathbf{Z}^n$ the subset of tuples $(k_1,...,k_n)$ satisfying
\begin{equation*}
    k_1\geq...\geq k_n.
\end{equation*} A character $\lambda= (\lambda_{\iota,i})\in (\mathbf{Z}^n)^{\textnormal{Hom}(F,E)}$ will then be $(\textnormal{Res}_{F^+/\mathbf{Q}}B)_E=(\textnormal{Res}_{F/\mathbf{Q}}B_n)_E$-dominant if and only if, it lies in $(\mathbf{Z}^n_+)^{\Hom(F,E)}$. In other words, for every $\iota\in \Hom(F,E)$, we have
\begin{equation*}
    \lambda_{\iota,1}\geq...\geq \lambda_{\iota,n}.
\end{equation*}
Given $\lambda \in (\mathbf{Z}^n_+)^{\Hom(F,E)}$, highest weight theory provides an integral representation of $\prod_{\iota: F\hookrightarrow E}\textnormal{GL}_n(\mathcal{O})$. This will simply be the representation constructed for instance in \cite{Ger18}, \S2.2. More precisely, if $B_n\subset \textnormal{GL}_n$ is the standard Borel of upper triangular matrices, $w_{0,n}$ denotes the longest element of the Weyl group of $\textnormal{GL}_n$ and $\iota\in \Hom(F,E)$, we consider the algebraic induction
\begin{equation*}
    \xi_{\lambda_{\iota}}:=(\textnormal{Ind}_{B_n}^{\textnormal{GL}_n}w_{0,n}\lambda_{\iota})_{/\mathcal{O}}:=\{f\in \mathcal{O}[\textnormal{GL}_n]\mid f(bg)=(w_{0,n}\lambda_{\iota})(b)f(g)
    \end{equation*}
    \begin{equation*}\textnormal{ for every }\mathcal{O}\to R, b\in B_n(R), g\in \textnormal{GL}_n(R)\},
\end{equation*}
and set $\mathcal{V}_{\lambda_{\iota}}:=\xi_{\lambda_{\iota}}(\mathcal{O})$, $V_{\lambda_{\iota}}:=\mathcal{V}_{\lambda_{\iota}}\otimes_{\mathcal{O}}E$. We then set $\mathcal{V}_{\lambda}:=\otimes_{\iota,\mathcal{O}}\mathcal{V}_{\lambda_{\iota}}$ and $V_{\lambda}:=\mathcal{V}_{\lambda}\otimes_{\mathcal{O}}E$. Note that $V_{\lambda}$ is the highest weight representation of $(\textnormal{Res}_{F^+/\mathbf{Q}}G)_E=(\textnormal{Res}_{F/\mathbf{Q}}\textnormal{GL}_n)_E\cong\prod_{\iota:F\hookrightarrow E}\textnormal{GL}_{n,E}$ of highest weight $\lambda$, a finite dimensional $E$-representation. Moreover, $\mathcal{V}_{\lambda}\subset V_{\lambda}$ is a $G(\mathcal{O})$-stable $\mathcal{O}$-lattice. In particular, for every $m\in \mathbf{Z}_{\geq 1}$, $\mathcal{V}_{\lambda}/\varpi^m$ is a smooth $\mathcal{O}/\varpi^m[\prod_{v\in S_p}\textnormal{GL}_n(\mathcal{O}_{F_v})]$-module\footnote{In particular, it becomes a smooth $\mathcal{O}/\varpi^m[\mathcal{Q}_{S_p}]$-module.} under the product of diagonal embeddings $\textnormal{GL}_n(\mathcal{O}_{F_v})\hookrightarrow \prod_{\iota:F_v\hookrightarrow E}\textnormal{GL}_n(\mathcal{O})$, finite free over $\mathcal{O}/\varpi^m$, and the formalism of \S\ref{sec2.2} applies. 

For a dominant weight $\lambda_v\in (\mathbf{Z}^n_+)^{\Hom(F_v,E)}$, set $\mathcal{V}_{w_0^{Q_v}\lambda_v}=\otimes_{\iota:F_v\hookrightarrow E}\mathcal{V}_{w_0^{Q_v}\lambda_{\iota}}$ to be the representation of $\prod_{\iota:F_v\hookrightarrow E}M_v(\mathcal{O})$ associated with $w_0^{Q_v}\lambda_v$ by the previous procedure where $w_0^{Q_v}=w_0^{M_v}w_0^{\textnormal{GL}_n}$ denotes the product of the longest Weyl group element of $M_v$, and of $\textnormal{GL}_n$. Concretely, if we assume that $Q_v=P_{(n_1,..,n_k)}$, then $w_0^{Q_v}\lambda_v=(\lambda_v^{n_1},...,\lambda_v^{n_k})$, where 
\begin{equation*}
    \lambda_v^{n_i}=(\lambda_{\iota}^{n_i})_{\iota:F_v\hookrightarrow E}=(\lambda_{\iota,n+1-(n_1+...+n_{i})},...,\lambda_{\iota,n+1-(n_1+...+n_{i-1}+1)})_{\iota}\in (\mathbf{Z}^{n_i}_+)^{\textnormal{Hom}(F_v,E)},\footnote{Here we use the convention $n_0=0$.}
\end{equation*} and
\begin{equation*}
    \mathcal{V}_{w_0^{Q_v}\lambda_v}=\mathcal{V}_{\lambda^{n_1}_v}\otimes...\otimes \mathcal{V}_{\lambda^{n_k}_v}.
\end{equation*}
We then have the analogue of \cite{CN23}, Lemma 2.1.12 with identical proof.

\begin{Lemma}\label{Lem2.14}
The natural $\prod_{\iota:F_v\hookrightarrow E}Q_v(\mathcal{O})$-equivariant map
\begin{equation*}
    \mathcal{V}_{\lambda_v}\to \mathcal{V}_{w_0^{Q_v}\lambda_v}
\end{equation*}
given by evaluation of functions at the identity is a surjection.
\end{Lemma}

We now turn to the smooth part which will be given by inflating to parahoric level the types of Schneider--Zink. Therefore, the smooth part will in fact depend on the choice of $Q_v$. For $v\in S_p$ consider a standard parabolic $Q_v$ and assume that it corresponds to a partition $(n_1,...,n_k)$ of $n$.
\begin{Def}
    We call a tuple $\underline{\tau_v}=(\tau_{v,i})_{i=1,...,k}$ consisting of ($E$-valued) $n_i$-dimensional Weil--Deligne inertial types $\tau_{v,i}$ for $F_v$ to be an inertial type of type $Q_v$.
\end{Def}
Given an inertial type $\underline{\tau_v}$ of type $Q_v$, Schneider--Zink provides a smooth irreducible $E$-representation $\sigma(\tau_{v,i})$ of $\textnormal{GL}_{n_i}(\mathcal{O}_{F_v})$ for every $i=1,...,k$ (cf. Theorem~\ref{Th2.12}). In particular, we obtain an irreducible smooth $E$-representation
\begin{equation*}
\sigma(\underline{\tau_v}):=\otimes_{i=1,...,k}\sigma(\tau_{v,i})
\end{equation*}
of $M_{v}^0=M_v(\mathcal{O}_{F_v})$.
We fix a choice of $M_{v}^0$-stable $\mathcal{O}$-lattice $\sigma(\underline{\tau_v})^{\circ}\subset \sigma(\underline{\tau_v})$. We set $c_v\geq 0$ to be the smallest integer such that $M_{v}^{c_v}=\textnormal{ker}(M_{v}^0\to \prod_{v\in S_p}M_v(\mathcal{O}_{F_v}/\varpi^{c_v}))$ acts trivially on $\sigma(\underline{\tau_v})^{\circ}$. Then $\mathcal{Q}_{v}(0,c_v)$ acts on $\sigma(\underline{\tau_v})^{\circ}$ by sending
\begin{equation*}
    \begin{pmatrix}
A_1 & . & . & . & \ast\\
. & A_2 &   &   & .\\
. &     & . &   & .\\
. &    &   & . & .\\
\ast & . & . & . & A_k
\end{pmatrix}\in \mathcal{Q}_{v}(0,c_v)
\end{equation*}
to $\sigma(\underline{\tau_v})^{\circ}(A_1,...,A_k)$. Since the corresponding map
\begin{equation*}
    \mathcal{Q}_v(0,c_v)\to M_v^0/M_{v}^{c_v}=M_v(\mathcal{O}_{F_v}/\varpi_v^{c_v})
\end{equation*} is easily checked to be a group homomorphism, this indeed defines a group action. We then denote by
$\widetilde{\sigma(\underline{\tau_v})}^{\circ}$ the representation $\sigma(\underline{\tau_v})^{\circ}$ viewed as an $\mathcal{O}[\mathcal{Q}_v(0,c_v)]$-module.

We then set our locally algebraic representation associated with the data $(Q_{S_p},\lambda,\underline{\tau}):=(Q_v,\lambda_v,\underline{\tau_v})_{v\in S_p}$ to be 
\begin{equation*}
    \mathcal{V}_{(\lambda,\underline{\tau})}^{Q_{p}}:=\mathcal{V}_{\lambda}\otimes_{\mathcal{O}}(\bigotimes_{v\in S_p,\mathcal{O}}\widetilde{\sigma(\underline{\tau_v})}^{\circ})^{\vee}.
\end{equation*} This gives rise to a $\mathcal{O}/\varpi^m[\mathcal{Q}_{p}(0,c_p)]$-module for $c_p:=\max\{c_v\}_{v\in S_p}$. By abuse of notation, we will use the same notation for the induced local systems on $X_K, \overline{X}_K$ and $\partial X_K$, respectively, for $K\subset G(\mathbf{A}_{F^+}^{\infty})$ any good subgroup with $K_p\subset \mathcal{Q}_p(0,c_p)$.

Note that when, at each place $v\in S_p$, $Q_v=\textnormal{GL}_n$, these local systems simplify. In this case we will abbreviate the notation to $\mathcal{V}_{(\lambda,\tau)}=\mathcal{V}_{\lambda}\otimes_{\mathcal{O}}\sigma(\tau)^{\circ,\vee}$. We denote the corresponding locally algebraic type by $\sigma(\lambda,\tau)^{\circ}:=(\mathcal{V}_{(\lambda,\tau)})^{\vee}$. By Lemma~\ref{Lem2.4}, the corresponding mod $\varpi^m$ cohomology groups are obtained by taking invariants of completed cohomology
\begin{equation*}
    R\Gamma_{(c)}(X_K,\mathcal{V}_{(\lambda,\tau)}/\varpi^m)\cong R\Hom_{\mathcal{O}/\varpi^m[K_{p}]}(\sigma(\lambda,\tau)^{\circ}/\varpi^m,\pi_{(c)}(K^{p},\mathcal{O}/\varpi^m))\cong
\end{equation*}
\begin{equation*}
    R\Hom_{\mathcal{O}/\varpi^m[G_{p}]}(\textnormal{c-Ind}_{K_{p}}^{G_{p}}\sigma(\lambda,\tau)^{\circ}/\varpi^m,\pi_{(c)}(K^{p},\mathcal{O}/\varpi^m)).
\end{equation*}
As a consequence, we get a natural right action of the Hecke algebra $\mathcal{H}(\sigma(\lambda,\tau)^{\circ})=\textnormal{End}_{G_{p}}(\textnormal{c-Ind}_{K_{p}}^{G_{p}}\sigma(\lambda,\tau)^{\circ})$ on $R\Gamma_{(c)}(X_K,\mathcal{V}_{(\lambda,\tau)}/\varpi^m)$. In particular,  as we will explain in \ref{sec5.3}, we obtain a natural action of an "integral Bernstein centre" $\mathfrak{z}_{\lambda,\tau}^{\circ}:=\mathcal{H}(\sigma(\lambda,\tau)^{\circ})\cap \mathfrak{z}_{\Omega}$.
 In fact, we will see that we have a similar description for general $Q$ after taking $Q$-ordinary parts due to independence of level and weight, yielding natural Hecke actions at $p$.

We now turn to treating the case of $\widetilde{G}$. From now on we introduce the following running assumption on $F$.
\begin{assumption}\label{blanket}
    Assume that our imaginary CM field $F/F^+$ is so that every $\overline{v}\in \overline{S}_p=S_p(F^+)$ splits in $F$.
\end{assumption}
Therefore we can write $\bar{v}=v\cdot v^c$ in $F$ for each $\bar{v}\in \overline{S}_p$. In particular, we fix a choice of a preferred place $v\mid \overline{v}$ in $F$. This fixes a lift $\iota: F\hookrightarrow E$ for every embedding $\overline{\iota}:F^+\hookrightarrow E$.
Therefore, it induces an identification
\begin{equation*}
(\textnormal{Res}_{F^+/\mathbf{Q}}\widetilde{G})_E=\prod_{\textnormal{Hom}(F^+,E)}\textnormal{GL}_{2n,E}
\end{equation*}
and an identification of the character group of $(\textnormal{Res}_{F^+/\mathbf{Q}}T)_E$ with $(\mathbf{Z}^{2n})^{\Hom(F^+,E)}$. This identifies a weight $\lambda=(\lambda_{\iota,i})\in (\mathbf{Z}^n)^{\Hom(F,E)}$ with $\Tilde{\lambda}=(\Tilde{\lambda}_{\overline{\iota},i})$ where
\begin{equation*}
    \Tilde{\lambda}_{\overline{\iota}}=(-w_0^{\textnormal{GL}_n}\lambda_{\iota c},\lambda_{\iota})=(-\lambda_{\iota c,n},...,-\lambda_{\iota c,1},\lambda_{\iota,1},...,\lambda_{\iota,n}).
\end{equation*}
Note that the $(\textnormal{Res}_{F^+/\mathbf{Q}}\widetilde{B})_E$-dominant weights are precisely given by $(\mathbf{Z}^{2n}_+)^{\Hom(F^+,E)}$. For such weights we can therefore define $\mathcal{V}_{\Tilde{\lambda}}\subset V_{\Tilde{\lambda}}$. For every $m\in \mathbf{Z}_{\geq 1}$, we then obtain a smooth $\mathcal{O}/\varpi^m[\prod_{\Bar{v}\in \overline{S}_p}\widetilde{G}(\mathcal{O}_{F^+_{\Bar{v}}})]$-module $\mathcal{V}_{\Tilde{\lambda}}/\varpi^m$, finite free as an $\mathcal{O}/\varpi^m$-module. These cover the algebraic parts of our locally algebraic representations.

For $\widetilde{G}$ we will only take ordinary parts at a certain subset $\overline{S}\subset \overline{S}_p$ where we wish to prove local-global compatibility using the degree shifting argument and will vary the level at the rest of the places. Therefore, we fix such a set $\overline{S}\subset \overline{S}_p$ of $p$-adic places and will introduce locally algebraic representations that are only non-algebraic at $\overline{S}$. Namely, if for $\Bar{v}\in \overline{S}$ given a tuple $(Q_{v'},\lambda_{v'},\underline{\tau_{v'}})_{v'\mid \Bar{v}}$ as before such that the weight $\Tilde{\lambda}_{\Bar{v}}$ associated with $(\lambda_v,\lambda_{v^c})$ is dominant (i.e. lies in $(\mathbf{Z}_+^{2n})^{\Hom(F^+,E)}$), we introduce a tuple $(\widetilde{Q}_{\Bar{v}},\Tilde{\lambda}_{\Bar{v}},\underline{\tau_{\Bar{v}}}=(\underline{\tau_{v}},\underline{\tau_{v^c}}))$ where $\widetilde{Q}_{\Bar{v}}=\widetilde{M}_{\Bar{v}}\ltimes\widetilde{N}_{\Bar{v}}\subset \widetilde{G}_{F^{+}_{\Bar{v}}}$ is the standard parabolic subgroup with $\widetilde{Q}_{\Bar{v}}\cap G_{F^{+}_{\Bar{v}}}=Q_v\times Q_{v^c}\subset \textnormal{GL}_{n,F_v}\times \textnormal{GL}_{n,F_{v^c}}$. As before, we denote by $\prescript{\Bar{v}}{}{\widetilde{\mathcal{Q}}}$ the corresponding parahoric subgroup scheme of $\widetilde{G}_{F^+_{\Bar{v}}}\cong^{\iota_v} \textnormal{GL}_{2n,F_v}$. We set $\widetilde{\mathcal{Q}}_{\Bar{v}}=\prescript{\Bar{v}}{}{\widetilde{\mathcal{Q}}}(\mathcal{O}_{F^+_{\Bar{v}}})$ with Iwahori decomposition $\overline{\widetilde{N}}_{\Bar{v}}^1\widetilde{M}_{\Bar{v}}^0\widetilde{N}_{\Bar{v}}^0$. We define the identification
\begin{equation*}
    \iota_v^{w_0}:\textnormal{GL}_{n}(F_v)\times \textnormal{GL}_{n}(F_{v^c})\xrightarrow{\sim}\textnormal{GL}_{n}(F_v)\times \textnormal{GL}_{n}(F_v),
\end{equation*}
\begin{equation*}
    (A,B)\mapsto (A,\theta_nB).
\end{equation*}
We then set $\widetilde{Q}_{\Bar{v}}^{w_0}=\widetilde{M}_{\Bar{v}}^{w_0}\ltimes \widetilde{N}_{\Bar{v}}^{w_0}\subset \widetilde{G}_{F^+_{\Bar{v}}}$ to be the standard parabolic subgroup associated with the Levi $\iota_v^{w_0}\widetilde{M}_{\Bar{v}}\subset \textnormal{GL}_{2n,F_v}$ under the identification $\iota_v$.
Set 
\begin{equation*}
    \Tilde{\sigma}(\underline{\tau_{\Bar{v}}})^{\circ}:=\sigma(\underline{\tau_v})^{\circ}\otimes_{\mathcal{O}}(\theta_n^{-1})^{\ast}\sigma(\underline{\tau_{v^c}})^{\circ}\in \textnormal{Mod}_{\textnormal{sm}}(\mathcal{O}[\widetilde{M}_{\Bar{v}}^{w_0,0}]).\footnote{We note that $\Tilde{\sigma}(\underline{\tau_{\Bar{v}}})^{\circ}$ is a representation of $\iota_v\widetilde{M}_{\Bar{v}}^{w_0,0}\cong \iota_{v}^{w_0}\widetilde{M}_{\Bar{v}}^0\subset \textnormal{GL}_{2n}(\mathcal{O}_{F_v})$ and we view it as a representation of $\widetilde{M}_{\Bar{v}}^{w_0,0}$ via $\iota_v$.}
\end{equation*} 
Note that here $\theta_n:\textnormal{GL}_n(F_{v^c})\cong\textnormal{GL}_n(F_v)$, as previously defined, is given by the map $B\mapsto (\Psi_n{}^tB^{-1}\Psi_n)^c$ and we are applying the pullback along $\theta^{-1}_n$.
For $\Bar{v}\in \overline{S}$, define
\begin{equation*}
\mathcal{V}^{\widetilde{Q}_{\Bar{v}}^{w_0}}_{(\Tilde{\lambda}_{\Bar{v}}, \underline{\tau_{\Bar{v}}})}:=\mathcal{V}_{\Tilde{\lambda}_{\Bar{v}}}\otimes_{\mathcal{O}}\Tilde{\sigma} \widetilde{(\underline{\tau_{\overline{v}}}) ^{\circ}}^{\vee}
\end{equation*}
where $\Tilde{\lambda}_{\Bar{v}}$ have the obvious meaning of considering the weights corresponding to all the embeddings inducing the given place.

Finally, given a tuple $(\widetilde{Q}_{\overline{S}},\Tilde{\lambda}_{\overline{S}},\underline{\tau_{\overline{S}}})=(\widetilde{Q}_{\Bar{v}},\Tilde{\lambda}_{\Bar{v}},\underline{\tau}_{\Bar{v}})_{\Bar{v}\in \overline{S}}$ coming from a tuple $(Q_p,\lambda,\underline{\tau})$ as above, and a dominant weight $\Tilde{\lambda}\in (\mathbf{Z}^{2n}_+)^{\Hom(F^+,E)}$ for $\widetilde{G}$ extending $\Tilde{\lambda}_{\overline{S}}$, we set
\begin{equation*}
    \mathcal{V}_{(\Tilde{\lambda},\underline{\tau})}^{\widetilde{Q}_{\overline{S}}^{w_0}}:=(\bigotimes_{\Bar{v}\notin \overline{S},\mathcal{O}}\mathcal{V}_{\Tilde{\lambda}_{\Bar{v}}})\otimes_{\mathcal{O}}(\bigotimes_{\Bar{v}\in\overline{S},\mathcal{O}}\mathcal{V}^{\widetilde{Q}_{\Bar{v}}^{w_0}}_{(\Tilde{\lambda}_{\Bar{v}},\underline{\tau_{\Bar{v}}})}),
\end{equation*}
a locally algebraic $\mathcal{O}$-representation of $\left(\prod_{\Bar{v}\in \overline{S}_p\setminus \overline{S}}\widetilde{G}(\mathcal{O}_{F^+_{\Bar{v}}})\right)\times \widetilde{\mathcal{Q}}^{w_0}_{\overline{S}}(0,c_p)$ for an appropriate integer $c_p\geq 1$ as before. In particular, $\mathcal{V}_{(\Tilde{\lambda},\underline{\tau})}^{\widetilde{Q}_{\overline{S}}^{w_0}}/\varpi^m$ becomes a smooth $\left(\prod_{\Bar{v}\in \overline{S}_p\setminus \overline{S}}\widetilde{G}(\mathcal{O}_{F^+_{\Bar{v}}})\right)\times \widetilde{\mathcal{Q}}^{w_0}_{\overline{S}}(0,c_p)$-module. When, for every $\Bar{v}\in\overline{S}$, the parabolic at $\Bar{v}$ is the Siegel one, we will abbreviate the notation to $\mathcal{V}^{\overline{S}}_{(\Tilde{\lambda},\tau)}$. Again, by abuse of notation we will denote identically the local systems they induce on locally symmetric spaces.

\begin{Ex}
    For the convenience of the reader, we spell out an example of $\widetilde{Q}_{\bar{v}}$, $\widetilde{Q}_{\Bar{v}}^{w_0}$, and $\widetilde{\sigma}(\underline{\tau_{\Bar{v}}})^{\circ}$.

    Let $n=3$, fix $\Bar{v}\in \overline{S}$ and write $\Bar{v}=v\cdot v^c$. We set $Q_v=\textnormal{GL}_{3}$ and $Q_{v^c}=P_{(1,2)}\subset \textnormal{GL}_{3,F_{v^c}}$, the standard parabolic subgroup with standard Levi subgroup $M_{(1,2)}=\textnormal{GL}_{1}\times \textnormal{GL}_2\subset \textnormal{GL}_{3,F_{v^c}}$. Then $\widetilde{Q}_{\Bar{v}}\subset \widetilde{G}_{F^+_{\Bar{v}}}$ is the standard parabolic subgroup with standard Levi subgroup
    \begin{equation*}
        \widetilde{M}_{\Bar{v}}=\textnormal{GL}_3\times M_{(1,2)}\subset \textnormal{GL}_{3,F_v}\times\textnormal{GL}_{3,F_{v^c}}.
    \end{equation*}

    In particular, $\iota_v\widetilde{Q}_{\Bar{v}}=P_{(2,1,3)}\subset \textnormal{GL}_{6,F_v}$. Moreover, $\iota_{v}^{w_0}\widetilde{M}_{\Bar{v}}=M_{(3,2,1)}\subset \textnormal{GL}_{6,F_v}$. Comsequently, $\widetilde{Q}_{\Bar{v}}^{w_0}\subset \widetilde{G}_{F^+_{\Bar{v}}}$ is the standard parabolic subgroup satisfying $\iota_v\widetilde{Q}_{\Bar{v}}^{w_0}=P_{(3,2,1)}\subset \textnormal{GL}_{6,F_v}$. Therefore, its standard Levi subgroup is
    \begin{equation*}
        \widetilde{M}_{\bar{v}}^{w_0}=M_{(2,1)}\times \textnormal{GL}_3\subset \textnormal{GL}_{3,F_v}\times \textnormal{GL}_{3,F_{v^c}}.
    \end{equation*}

    Further consider Weil--Deligne inertial types $\underline{\tau}_v=(\tau_v)$, and $\underline{\tau_{v^c}}=(\tau_{v^c,1},\tau_{v^c,2})$ and (a choice of) associated smooth $\mathcal{O}$-representations $\sigma(\underline{\tau_v})^{\circ}=\sigma(\tau_v)^{\circ}$, and $\sigma(\underline{\tau_{v^c}})^{\circ}=\sigma(\tau_{v^c,1})^{\circ}\otimes \sigma(\tau_{v^c,2})^{\circ}$ of $\textnormal{GL}_{3}(\mathcal{O}_{F_v})$, and $M_{(1,2)}(\mathcal{O}_{F_{v^c}})=\textnormal{GL}_1(\mathcal{O}_{F_{v^c}})\times \textnormal{GL}_2(\mathcal{O}_{F_{v^c}})$, respectively. Then $\widetilde{\sigma}(\underline{\tau_{\Bar{v}}})^{\circ}$, as a representation of $\iota_v\widetilde{M}_{\bar{v}}^{w_0,0}=\textnormal{GL}_3(\mathcal{O}_{F_v})\times \textnormal{GL}_2(\mathcal{O}_{F_v})\times\textnormal{GL}_1(\mathcal{O}_{F_v})$, is given by
    \begin{equation*}
        \sigma(\tau_v)^{\circ}\otimes\left((\theta_2^{-1})^{\ast}\sigma(\tau_{v^c,2})^{\circ}\otimes (\theta_1^{-1})^{\ast}\sigma(\tau_{v^c,1})^{\circ}\right).
    \end{equation*}
In particular, as a representation of $\widetilde{M}_{\Bar{v}}^{w_0,0}=\textnormal{GL}_{2}(\mathcal{O}_{F_v})\times\textnormal{GL}_1(\mathcal{O}_{F_v})\times \textnormal{GL}_{3}(\mathcal{O}_{F_{v^c}})$ via the identification $\iota_v$, it is given by
\begin{equation*}
    \left((\theta_2^{-1})^{\ast}\sigma(\tau_{v^c,2})^{\circ}\otimes (\theta_1^{-1})^{\ast}\sigma(\tau_{v^c,1})^{\circ}\right)\otimes(\theta_3)^{\ast}\sigma(\tau_{v,3})^{\circ}.
\end{equation*}
\end{Ex}

\vspace{5mm}

We also introduce a local system on $\widetilde{X}_{\widetilde{K}}$ suitable for a \textit{dual} version of the degree shifting argument in \S\ref{sec6.1}. For this we assume now that, for each $\Bar{v}\in \overline{S}$, $\Tilde{\lambda}_{\Bar{v}}:=(\lambda_v,-w_0^{\textnormal{GL}_n}\lambda_{v^c})$ is dominant for $\widetilde{G}$.\footnote{Note that this identification $(\mathbf{Z}^n)^{\Hom(F,E)}\cong (\mathbf{Z}^{2n})^{\Hom(F^+,E)}$ is \textit{not} the one used in the previous paragraph.} Then, for $\Bar{v}\in \overline{S}$, we set
\begin{equation*}
    \Tilde{\sigma}(\underline{\tau_{\Bar{v}}})^{\circ,w_0}:= (\theta_n^{-1})^{\ast}\sigma(\underline{\tau_{v^c}})^{\circ} \otimes_{\mathcal{O}} \sigma(\underline{\tau_v})^{\circ}\in \textnormal{Mod}_{\textnormal{sm}}(\mathcal{O}[\widetilde{M}_{\Bar{v}}^{0}]),\footnote{Again, the representation present is a representation of $\iota_v\widetilde{M}_{\Bar{v}}^{0}$ and we view it as a representation of $\widetilde{M}_{\Bar{v}}^{0}$ via $\iota_v$. Note that in particular, it is simply the representation $\sigma(\underline{\tau_v})^{\circ}\otimes_{\mathcal{O}}\sigma(\underline{\tau_{v^c}})^{\circ}$ of $\widetilde{M}_{\Bar{v}}^{0}$.}
\end{equation*}
and
\begin{equation*}
\mathcal{V}_{(\Tilde{\lambda}_{\Bar{v}},\underline{\tau_{\Bar{v}}})}^{\widetilde{Q}_{\Bar{v}},w_0^P}:=\mathcal{V}_{w_0^P\Tilde{\lambda}_{\Bar{v}}}\otimes_{\mathcal{O}} (\Tilde{\sigma}(\underline{\tau_{\Bar{v}}})^{\circ,w_0})^{\sim,\vee}\in \textnormal{Mod}(\mathcal{O}[\widetilde{\mathcal{Q}}(0,c_p)]).
\end{equation*}
For any dominant weight $\widetilde{\lambda}\in (\mathbf{Z}^{2n}_+)^{\Hom(F^+,E)}$ extending $\widetilde{\lambda}_{\overline{S}}$, we can then set
\begin{equation*}
    \mathcal{V}_{(\Tilde{\lambda},\underline{\tau})}^{\widetilde{Q}_{\overline{S}},w_0^P}:= (\bigotimes_{\Bar{v}\notin \overline{S},\mathcal{O}}\mathcal{V}_{\Tilde{\lambda}_{\Bar{v}}})\otimes_{\mathcal{O}}(\bigotimes_{\Bar{v}\in\overline{S},\mathcal{O}}\mathcal{V}^{\widetilde{Q}_{\Bar{v}},w_0^P}_{(\Tilde{\lambda}_{\Bar{v}},\underline{\tau_{\Bar{v}},})})
\end{equation*}
an $\mathcal{O}$-representation of $\left(\prod_{\Bar{v}\in \overline{S}_p\setminus \overline{S}}\widetilde{G}(\mathcal{O}_{F^+_{\Bar{v}}})\right)\times \widetilde{\mathcal{Q}}_{\overline{S}}(0,c_p)$. We denote identically the corresponding local systems induced on locally symmetric spaces for $\widetilde{G}$.

\subsection{Explicit Hecke operators}\label{sec2.8} Next we spell out the explicit formula for the usual unramified Hecke operators and $U_p$-operators as for instance in \cite{NT16}, \cite{ACC23} and \cite{CN23}. Fix for any $v$ finite place of $F$ a uniformiser $\varpi_v\in \mathcal{O}_{F_v}$. We start by introducing the usual explicit Hecke operators at unramified places. Let $v$ be a finite place of $F$, and $1\leq i\leq n$ an integer. Write $T_{v,i}\in \mathcal{H}(\textnormal{GL}_n(F_v),\textnormal{GL}_n(\mathcal{O}_{F_v}))$ for the double coset operator
\begin{equation*}
    T_{v,i}=[\textnormal{GL}_n(\mathcal{O}_{F_v})\textnormal{diag}(\varpi_v,...,\varpi_v,1,...,1)\textnormal{GL}_n(\mathcal{O}_{F_v})]
\end{equation*}
where $\varpi_v$ appears exactly $i$ times in the diagonal. We define the polynomial
\begin{equation*}
    P_v(X)=X^n-T_{v,1}X^{n-1}+...+(-1)^iq_v^{i(i-1)/2}T_{v,i}X^{n-i}+...
\end{equation*}
\begin{equation*}
    +q_v^{n(n-1)/2}T_{v,n}\in \mathcal{H}(\textnormal{GL}_n(F_v),\textnormal{GL}_n(\mathcal{O}_{F_v}))[X]
\end{equation*}
where recall that $q_v=| \mathcal{O}_{F_v}/\varpi_v|$. Note that $P_v(X)$ corresponds to the characteristic polynomial of the Frobenius element acting on $\textnormal{rec}^T_{F_v}(\pi_v)$ for $\pi_v$ any unramified representation of $\textnormal{GL}_n(F_v)$.

If $\Bar{v}$ is a finite place of $F^+$, unramified in $F$, $v$ is a choice of place of $F$ above it, and $1\leq j \leq 2n$ is an integer, then we denote by $\widetilde{T}_{v,j}\in \mathcal{H}(\widetilde{G}(F^+_{\Bar{v}}),\widetilde{G}(\mathcal{O}_{F^+_{\Bar{v}}}))\otimes_{\mathbf{Z}}\mathbf{Z}[q_{\Bar{v}}^{-1}]$ the Hecke operator denoted by $T_{G,v,j}$ in \cite{NT16}, Proposition-Definition 5.2. In particular, $q_v^{-j(2n-j)/2}\widetilde{T}_{v,j}$ is the operator corresponding to the $i$th symmetric polynomial in $2n$ variables under the dual map on Hecke algebras corresponding to the unramified endoscopic transfer from $\widetilde{G}(F^+_{\Bar{v}})$ to $\textnormal{GL}_{2n}(F_v)$ (where we also apply the (normalised) Satake isomorphism). We then define the polynomial
\begin{equation*}
    \widetilde{P}_v(X)=X^{2n}-\widetilde{T}_{v,1}X^{2n-1}+...+(-1)^jq_v^{j(j-1)/2}\widetilde{T}_{v,j}X^{2n-j}+...
\end{equation*}
\begin{equation*}
    +q_v^{2n(2n-1)/2}\widetilde{T}_{v,2n}\in \mathcal{H}(\widetilde{G}(F^+_{\Bar{v}}),\widetilde{G}(\mathcal{O}_{F^+_{\Bar{v}}}))\otimes_{\mathbf{Z}}\mathbf{Z}[q_{\Bar{v}}^{-1}][X].
\end{equation*}
This then corresponds to the characteristic polynomial of the Frobenius element acting on $\textnormal{rec}^T_{F_v}(\pi_v)$, where $\pi_v$ is the base change with respect to $F_v/F^{+}_{\Bar{v}}$ of any unramified representation $\sigma_{\Bar{v}}$ of $\widetilde{G}(F^+_{\Bar{v}})$.

We finally describe the effect of the unnormalised Satake transform $\mathcal{S}=r_G\circ r_{P}$ (for the notation, see the end of subsection \ref{sec2.2}) at unramified places. We use the following notation: for $f(X)$ a polynomial of degree $d$, with constant term a unit $a_0$, set $f^{\vee}(X):=a_0^{-1}X^df(X^{-1})$. Therefore, $f^{\vee}(X)$ is the monic polynomial with zeroes given by the inverse of the zeroes of $f(X)$. We then have the following.
\begin{Prop}\label{Prop2.16}
Let $v$ be a finite place of $F$, unramified above the place $\Bar{v}$ of $F^+$. Then the unnormalised Satake transform
\begin{equation*}
    \mathcal{S}:\mathcal{H}(\widetilde{G}(F^+_{\Bar{v}}),\widetilde{G}(\mathcal{O}_{F^+_{\Bar{v}}}))\to \mathcal{H}(G(F^+_{\Bar{v}}),G(\mathcal{O}_{F^+_{\Bar{v}}}))
\end{equation*}
sends $\widetilde{P}_v(X)$ to $P_v(X)q_v^{n(2n-1)}P_{v^c}^{\vee}(q_v^{1-2n}X)$.
\end{Prop}
\begin{proof}
This follows from the explicit formula for $\mathcal{S}$ given in \cite{NT16}, Proposition-Definition 5.3.
\end{proof}

We now turn to discussing the $U_p$-operators we will consider. As before, we assume that every $p$-adic place $\Bar{v}\in \overline{S}_p$ splits in $F$ and we fix a choice of $v|\Bar{v}$ in $S_p$.  We start with the case of $G$. Consider a tuple $(Q_{S_p},\lambda, \underline{\tau})=(Q_v,\lambda_v,\underline{\tau_v})_{v\in S_p}$ with Levi decomposition $Q_v=M_v\ltimes N_v$. Set $X_{Q_v}$ to be the set of $B_n$-dominant cocharacters in $X_{\ast}(Z(M_v))$. Concretely, if $Q_v=P_{(n_1,...,n_k)}\subset \textnormal{GL}_n$, then it identifies in $X_{\ast}(T_n)^+=\mathbf{Z}^n_+$ with the elements with jumps only at the indices $n_1+...+n_j$ for $1\leq j\leq k$. For $c\geq 1$, we then define the subset $\Delta_{\mathcal{Q}_v}(c)\subset \textnormal{GL}_n(F_v)$ given by
\begin{equation*}
    \Delta_{\mathcal{Q}_v}(c):=\coprod_{\nu\in X_{Q_v}}\mathcal{Q}_v(0,c)\nu(\varpi_v)\mathcal{Q}_v(0,c).
\end{equation*}
By \cite{CN23}, Lemma 2.1.15, the elements $\nu(\varpi_v)$ are $\mathcal{Q}_v(0,c)$-positive in the sense of \textit{loc. cit.}. In particular, $\Delta_{\mathcal{Q}_v}(c)$ forms a monoid and $\Delta_{M_v}^+:=\Delta_{\mathcal{Q}_v}(c)\cap M_v(F_v)\subset M_v(F_v)^+=\{m\in M_v(F_v)\mid mN^0_vm^{-1}\subset N^0_v\textnormal{ and }m^{-1}\overline{N}^1_vm\subset \overline{N}^{1}\}$ so the formalism of \cite{ACC23}, 2.1.9 applies. Set $\Delta_{M_v}$ to be the group generated by $\Delta_{M_v}^+$. One proves that the map
\begin{equation*}
    \mathcal{H}(\Delta^{+}_{M_v},M^0_v)\to \mathcal{H}(\Delta_{\mathcal{Q}_v}(c),\mathcal{Q}_v(0,c)),
\end{equation*}
\begin{equation*}
    [M^0_v\nu(\varpi_v)M^0_v]\mapsto [\mathcal{Q}_v(0,c)\nu(\varpi_v)\mathcal{Q}_v(0,c)]
\end{equation*}
is an isomorphism of Hecke algebras (cf. \cite{BK98}, Corollary 6.12). In particular, the latter is commutative. 

We introduce our distinguished element, the "$U_p$-operator at $v$" in our Hecke algebras. Namely, if $Q_v=P_{(n_1,...,n_k)}$, set
\begin{equation*}
    u_v^{Q_v}:=\textnormal{diag}(\varpi^{k-1},...,\varpi^{k-1},\varpi^{k-2},...,\varpi,1,...,1)\in \Delta_{\mathcal{Q}_v}
\end{equation*}
where, for $1\leq j\leq k$, $\varpi^{j-1}$ appears exactly $n_j$ times. Then, for $c\geq 1$, we denote by $U_v^{Q_v}\in\mathcal{H}(\Delta_{\mathcal{Q}_v}(c),\mathcal{Q}_v(0,c))$ the double coset operator $[\mathcal{Q}_v(0,c)u_v^{Q}\mathcal{Q}_v(0,c)]$. Note that $\Delta_{M_v}=\Delta_{M_v}^+[u_v^{Q_v,\pm 1}]$.

We introduce the usual $\lambda$-twisted action of $\Delta_{\mathcal{Q}_v}(c_p)$\footnote{For the integer $c_p\geq 1$ introduced in the previous subsection.} on $\mathcal{V}_{(\lambda_v,\underline{\tau_v})}^{Q_v}$ which will then yield our action of $U_p$-operators on the corresponding cohomology groups by the formalism of \cite{ACC23}, 2.1.9. Define the character $\alpha_{\lambda}^{Q_v}:\Delta_{\mathcal{Q}_v}\to E^{\times}$ by setting it to be trivial on $\mathcal{Q}_v$ and, for $\nu\in X_{Q_v}$, sending $\nu(\varpi_v)$ to
\begin{equation*}
    \prod_{\iota:F_v\hookrightarrow E}\iota(\varpi_v)^{\langle \nu,w_0^G\lambda_{\iota}\rangle}.
\end{equation*}
We then view $\mathcal{V}^{Q_v}_{(\lambda_v,\underline{\tau_v})}$ as an $\mathcal{O}[\Delta_{\mathcal{Q}_v}(c_p)]$-module by inflating the $\mathcal{Q}_v(0,c_p)$-action on the factor $(\widetilde{\sigma(\underline{\tau_v})}^{\circ})^{\vee}$ and acting on the factor $\mathcal{V}_{\lambda_v}$ by the recipe
\begin{equation*}
    g\cdot^{\lambda,Q_v}x:=\alpha_{\lambda}^{Q_v}(g)^{-1}g\cdot x
\end{equation*}
where $-\cdot -$ is the usual action of $g\in \Delta_{\mathcal{Q}_v}(c_p)$ on $x\in \mathcal{V}_{\lambda_v}$.
\begin{Lemma}\label{Lem2.17}
The $\mathcal{O}[\Delta_{\mathcal{Q}_v}(c_p)]$-module structure on $\mathcal{V}^{Q_v}_{(\lambda_v,\underline{\tau_v})}$ makes sense i.e., the twisted action of $\Delta_{\mathcal{Q}_v}(c_p)$ on $V_{\lambda_v}$ preserves the lattice $\mathcal{V}_{\lambda_v}$.
\end{Lemma}
\begin{proof}
This follows from the fact that $V_{\lambda_v}$ has lowest weight $w_0^G\lambda_v$. More precisely, one uses \cite{Ger18}, Lemma 2.2 to conclude.
\end{proof}
As mentioned, the formalism of \cite{ACC23}, 2.1.9 applies here. In particular, for $S_p\subset T$ any finite set of finite places, and a choice of good subgroup $K\subset \textnormal{GL}_n(\mathbf{A}^{\infty}_{F})$ with $K_{v}=\mathcal{Q}_{v}(b,c)$ for some $0\leq b\leq c$ with $c_p\leq c$, for every $v\in S_p$, we have a canonical homomorphism
\begin{equation*}
    \mathcal{H}(G^T,K^T)\otimes_{\mathbf{Z}}\mathcal{H}(\Delta_{\mathcal{Q}_{p}(c_p)},K_{p})\to \textnormal{End}_{D^{+}(\mathcal{O})}(R\Gamma(X_K,\mathcal{V}^{Q_{p}}_{(\lambda,\underline{\tau})})).
\end{equation*}
In particular, the above constructed $U_p$-operators act on the cohomology complex $R\Gamma(X_K,\mathcal{V}_{(\lambda,\underline{\tau})}^{Q_{p}})$.

For later use, set $\mathbf{T}^T(K,\lambda,\underline{\tau})$ to be the image of $\mathcal{H}(G^T,K^T)\otimes_{\mathbf{Z}}\mathcal{O}$ inside the ring $\textnormal{End}_{D^{+}(\mathcal{O})}(R\Gamma(X_K,\mathcal{V}^{Q_{p}}_{(\lambda,\underline{\tau})}))$, a finite $\mathcal{O}$-algebra.

Next we consider the case of $\widetilde{G}$. Fix a subset of $p$-adic places $\overline{S}\subset \overline{S}_p$ and, for each $\Bar{v}\in \overline{S}_p$, fix a choice $v\mid \Bar{v}$ in $S_p$. Consider a tuple $(\widetilde{Q}_{\overline{S}},\Tilde{\lambda}_{\overline{S}},\underline{\tau_{\overline{S}}}):=(\widetilde{Q}_{\Bar{v}},\Tilde{\lambda}_{\Bar{v}},\underline{\tau_{\Bar{v}}})_{\Bar{v}\in \overline{S}}$ as before.
We can then similarly define the set of $\widetilde{B}$-dominant cocharacters $X_{\widetilde{Q}_{\Bar{v}}^{w_0}}\subset X_{\ast}(Z(\widetilde{M}_{\Bar{v}}^{w_0}))$ and the corresponding open submonoids $\widetilde{\Delta}_{\widetilde{\mathcal{Q}}^{w_0}_{\Bar{v}}}(c)\subset \widetilde{G}(F^+_{\Bar{v}})$. The rest of the conclusions of \cite{CN23}, Lemma 2.1.15 will again apply. We further set $\widetilde{\Delta}_{\widetilde{\mathcal{Q}}^{w_0}_{\overline{S}}}(c_p):=\prod_{\Bar{v}\in\overline{S}}\widetilde{\Delta}_{\widetilde{\mathcal{Q}}^{w_0}_{\Bar{v}}}(c_p)$. Again, we introduce a notation for our $U_p$-operator in $\mathcal{H}(\widetilde{\Delta}_{\widetilde{\mathcal{Q}}^{w_0}_{\Bar{v}}}(c_p),\widetilde{\mathcal{Q}}^{w_0}_{\Bar{v}}(0,c_p))$. For $v|\Bar{v}$, we set $u_{\Bar{v}}^{\widetilde{Q}_{v}^{w_0}}:=\iota_v^{-1}u_v^{\iota_{v}\widetilde{Q}^{w_0}_{\Bar{v}}}$ where $\iota_{\Bar{v}}\widetilde{Q}^{w_0}_{v}\subset \textnormal{GL}_{2n}/F_v$ is the standard parabolic subgroup corresponding to $\widetilde{Q}^{w_0}_{\Bar{v}}$ under $\iota_{v}$. Then set $U_{\Bar{v}}^{\widetilde{Q}_{\Bar{v}}^{w_0}}\in \mathcal{H}(\widetilde{\Delta}_{\widetilde{\mathcal{Q}}^{w_0}_{\Bar{v}}}(c_p),\widetilde{\mathcal{Q}}^{w_0}_{\Bar{v}}(0,c_p))$ to be the double coset operator $[\widetilde{\mathcal{Q}}^{w_0}_{\Bar{v}}(0,c)u_{v}^{\widetilde{Q}^{w_0}_{\Bar{v}}}\widetilde{\mathcal{Q}}^{w_0}_{\Bar{v}}(0,c_p)]$. We note that as the notation suggests, $U_{\Bar{v}}^{\widetilde{Q}^{w_0}_{\Bar{v}}}$ is independent of the choice of $v|\Bar{v}$.

Moreover, by the exact same recipe as before, we equip $\mathcal{V}_{(\widetilde{\lambda},\underline{\tau})}^{\widetilde{Q}_{\overline{S}}}$ with an $\widetilde{\Delta}_{\widetilde{\mathcal{Q}}^{w_0}_{\overline{S}}}(c_p)$-module structure extending the natural action of $\widetilde{\mathcal{Q}}^{w_0}_{\overline{S}}(0,c_p)$.
This yields a canonical homomorphism of algebras
\begin{equation*}
    \mathcal{H}(\widetilde{G}^T,\widetilde{K}^T)\otimes_{\mathbf{Z}}\mathcal{H}(\widetilde{\Delta}_{\widetilde{\mathcal{Q}}^{w_0}_{\overline{S}}}(c_p),\widetilde{K}_{\overline{S}})\to \textnormal{End}_{D^+(\mathcal{O})}(R\Gamma(\widetilde{X}_{\widetilde{K}},\mathcal{V}_{(\Tilde{\lambda},\underline{\tau})}^{\widetilde{Q}^{w_0}_{\overline{S}}}))
\end{equation*}
for any finite set of finite places $S_p\subset T$ with $T=T^c$ and good subgroup $\widetilde{K}\subset \widetilde{G}(\mathbf{A}_{F^+}^{\infty})$ with $\widetilde{K}_{\Bar{v}}=\widetilde{\mathcal{Q}}_{\Bar{v}}^{w_0}(b,c')$ for some $0\leq b\leq c'$ with $c_p\leq c'$ for every $\Bar{v}\in \overline{S}$.
Therefore, the constructed $U_p$-operators act on the complex $R\Gamma(\widetilde{X}_{\widetilde{K}},\mathcal{V}_{(\Tilde{\lambda},\underline{\tau})}^{\widetilde{Q}^{w_0}_{\overline{S}}})$. 

We finally set $\mathbf{T}^T(\widetilde{K},\Tilde{\lambda},\underline{\tau})$ to be the corresponding faithful quotient of $\mathcal{H}(\widetilde{G}^T,\widetilde{K}^T)\otimes_{\mathbf{Z}}\mathcal{O}$.

\subsection{Automorphic Galois representations}
In this subsection, we recall the well-known results which we will need later about Galois representations attached to (mod $p$) automorphic forms. We also state the consequence of the vanishing results of Caraiani--Scholze that serves as the key ingredient for proving local-global compatibility in the style of \cite{ACC23}.

Recall that $\pi$ is called a regular algebraic conjugate self-dual cuspidal automorphic representation (RACSDCAR) of $\textnormal{GL}_n(\mathbf{A}_F)$
if it is a regular algebraic cuspidal automorphic representation (RACAR) of $\textnormal{GL}_n(\mathbf{A}_F)$ satisfying $\pi^{c}\cong \pi^{\vee}$ where $(.)^{\vee}$ denotes the contragradient of the representation.

By work of many people, such automorphic representations admit associated Galois representations satisfying local-global compatibility at every place.

\begin{Th}{\textnormal{\cite{Cl91,HT01,TY07,Sh11,CH13,BLGGT12,BLGGT11,Car12,Car14}}}\label{Th2.17}
Let $\pi$ be a RACSDCAR of $\textnormal{GL}_n(\mathbf{A}_F)$ of weight $\lambda\in(\mathbf{Z}_+^{n})^{\Hom(F,\mathbf{C})}$. Then  for any isomorphism $t:\overline{\mathbf{Q}}_p\xrightarrow{\sim}\mathbf{C}$ there is a continuous semisimple Galois representation
\begin{equation*}
    r_t(\pi):G_F\to \textnormal{GL}_n(\overline{\mathbf{Q}}_p)
\end{equation*}
satisfying the following conditions:
\begin{enumerate}
    \item We have an isomorphism $r_t(\pi)^c\cong r_t(\pi)^{\vee}(1-n)$.
    \item For each $p$-adic place $v$ of $F$, $r_t(\pi)|_{G_{F_v}}$ is potentially semistable and for each embedding $\iota:F_{v}\hookrightarrow \overline{\mathbf{Q}}_p$ we have
    \begin{equation*}
        \textnormal{HT}_{\iota}(r_t(\pi)|_{G_{F_v}})=\{\lambda_{t\circ\iota,n},\lambda_{t\circ\iota,n-1}+1,...,\lambda_{t\circ\iota,1}+n-1\}.
    \end{equation*}
    \item For each finite place $v$ of $F$, we have
    \begin{equation*}
        \textnormal{WD}(r_t(\pi)|_{G_{F_v}})^{F-ss}\cong \textnormal{rec}^T(t^{-1}\pi_v).
    \end{equation*}
\end{enumerate}
\end{Th}
Once combined with Shin's base change result \cite{Shi14} for automorphic representations of $\widetilde{G}$, we obtain the following.
\begin{Th}\label{Th2.19}
Suppose that $F$ contains an imaginary quadratic field. Let $\Tilde{\pi}$ be a $\xi$-cohomological cuspidal automorphic representation of $\widetilde{G}(\mathbf{A}_{F^+})$ for some irreducible algebraic representation $\xi$ of $\widetilde{G}_{\mathbf{C}}\cong \prod_{\Hom(F^+,\mathbf{C})}\textnormal{GL}_{2n,\mathbf{C}}$. For any field isomorphism $t:\overline{\mathbf{Q}}_p\xrightarrow{\sim}\mathbf{C}$, there exists a continuous, semisimple Galois representation
\begin{equation*}
    r_t(\Tilde{\pi}):G_F\to \textnormal{GL}_{2n}(\overline{\mathbf{Q}}_p)
\end{equation*}
satisfying the following conditions:
\begin{enumerate}
    \item For each prime $\ell\neq p$, unramified in $F$, above which $\Tilde{\pi}$ is unramified, and for each place $v$ of $F$ dividing $\ell$, $r_t(\Tilde{\pi})|_{G_{F_v}}$ is unramified and the characteristic polynomial of $r_t(\textnormal{Frob}_v)$ coincides with image of $\widetilde{P}_v(X)$ in $\overline{\mathbf{Q}}_p[X]$ corresponding to the base change of $t^{-1}(\Tilde{\pi}_{\Bar{v}})$.
    \item For each place $v$ of $F$ dividing $p$, $r_t(\Tilde{\pi})$ is potentially semistable, and for each embedding $\iota: F\hookrightarrow \overline{\mathbf{Q}}_p$, the $\iota$-labelled Hodge--Tate weights are given by
    \begin{equation*}
        \Tilde{\lambda}_{\iota,1}+2n-1>\Tilde{\lambda}_{\iota,2}+2n-2>...>\Tilde{\lambda}_{\iota, 2n},
    \end{equation*}
    where $\Tilde{\lambda}\in (\mathbf{Z}^{2n}_+)^{\Hom(F,\overline{\mathbf{Q}}_p)}$ is the highest weight of the representation $t^{-1}(\xi\otimes \xi)^{\vee}$ of $\textnormal{GL}_{2n}$ over $\overline{\mathbf{Q}}_p$.
    \item If $F_0\subset F$ is an imaginary quadratic field and $\ell$ is a prime (possibly $\ell=p)$ that splits in $F_0$, then for each place $v\mid \ell$ of $F$ lying above a place $\Bar{v}$ of $F^{+}$, there is an isomorphism
    \begin{equation*}
        \textnormal{WD}(r_t(\Tilde{\pi})|_{G_{F_v}})^{F-ss}\cong \textnormal{rec}^T(\Tilde{\pi}_{\Bar{v}}\circ \iota_v).
    \end{equation*}
\end{enumerate}
\end{Th}
\begin{proof}
For the proof see \cite{ACC23}, Theorem 2.3.3 and the references therein.
\end{proof}
\begin{Th}\label{Th2.20}
Assume that $F$ contains an imaginary quadratic field. Let $\mathfrak{m}\subset \mathbf{T}^T(K,\lambda,\underline{\tau})$ be a maximal ideal. Suppose that the finite set of places $T$ is so that $T=T^c$ and further satisfies the following condition.
\begin{itemize}
    \item Given a finite place $v$ not lying in $T$, denote its residual characteristic by $\ell$. Then either $\ell$ is unramified in $F$ and $T$ contains no $\ell$-adic places, or $\ell$ splits in some imaginary quadratic subfield $F_0\subset F$.
\end{itemize}
Then there exists a continuous semisimple Galois representation
\begin{equation*}
    \Bar{\rho}_{\mathfrak{m}}:G_{F,T}\to \textnormal{GL}_n(\mathbf{T}^T(K,\lambda,\underline{\tau})/\mathfrak{m})
\end{equation*}
such that for each finite place $v$ of $F$ not lying in $T$, the characteristic polynomial of $\Bar{\rho}_{\mathfrak{m}}(\textnormal{Frob}_v)$ coincides with the image of $P_v(X)$ in $(\mathbf{T}^T(K,\lambda,\underline{\tau})/\mathfrak{m})[X]$.
\end{Th}
\begin{proof}
This is proved the same way as \cite{ACC23}, Theorem 2.3.5.
\end{proof}
\begin{Def}
We say that a maximal ideal $\mathfrak{m}\subset \mathbf{T}^T(K,\lambda,\underline{\tau})$ is non-Eisenstein if $\Bar{\rho}_{\mathfrak{m}}$ is absolutely irreducible.
\end{Def}
\begin{Th}\label{Th2.22}
Assume that $F$ and $T$ satisfies the conditions of Theorem~\ref{Th2.20}. Let $\mathfrak{m}\subset \mathbf{T}^T(K,\lambda,\underline{\tau})$ be a non-Eisenstein maximal ideal. There exists an integer $N\geq 1$, depending only on $n$ and $[F:\mathbf{Q}]$, an ideal $I\subset \mathbf{T}^T(K,\lambda,\underline{\tau})$ satisfying $I^N=0$, and a continuous group homomorphism
\begin{equation*}
    \rho_{\mathfrak{m}}:G_{F,T}\to \textnormal{GL}_n(\mathbf{T}^T(K,\lambda,\underline{\tau})/I)
\end{equation*}
such that, for each finite place $v$ of $F$ not lying in $T$, the characteristic polynomial of $\rho_{\mathfrak{m}}(\textnormal{Frob}_v)$ coincides with the image of $P_{v}(X)$ in $(\mathbf{T}^T(K,\lambda,\underline{\tau})/I)[X]$.
\end{Th}
\begin{proof}
This is proved the same way as \cite{Sch15}, Corollary 5.4.4.
\end{proof}

\begin{Th}\label{Th2.24}
Assume that $F$ and $T$ satisfies the conditions of Theorem~\ref{Th2.20}. Let $\widetilde{\mathfrak{m}}\subset \widetilde{\mathbf{T}}^T(\widetilde{K},\Tilde{\lambda},\underline{\tau})$ be a maximal ideal. Then there is a continuous, semisimple Galois representation
\begin{equation*}
    \Bar{\rho}_{\widetilde{\mathfrak{m}}}:G_{F,T}\to \textnormal{GL}_{2n}(\widetilde{\mathbf{T}}^T(\widetilde{K},\Tilde{\lambda},\underline{\tau})/\widetilde{\mathfrak{m}})
\end{equation*}
such that for each finite place $v\notin T$ of $F$, the characteristic polynomial of $\Bar{\rho}_{\widetilde{\mathfrak{m}}}(\textnormal{Frob}_v)$ is given by the image of $\widetilde{P}_v(X)$ in $(\widetilde{\mathbf{T}}^T(\widetilde{K},\Tilde{\lambda},\underline{\tau})/\widetilde{\mathfrak{m}})[X]$.
\end{Th}
\begin{proof}
The same proof applies as in \cite{CN23}, Theorem 2.1.26 noting that the first step is to pass to deep enough level where $\mathcal{V}_{(\Tilde{\lambda},\underline{\tau})}^{\widetilde{Q}^{w_0}_{\overline{S}}}/\varpi$ is trivialised. 
\end{proof}
Finally, we discuss the key technical condition we need to have access to the vanishing result of Caraiani--Scholze.
\begin{Def}\label{Def2.25}
A continuous representation $\Bar{\rho}:G_F\to \textnormal{GL}_m(\mathcal{O}/\varpi)$ is called \textit{decomposed generic} if there exists a prime $\ell$ different from $p$ such that:
\begin{enumerate}
    \item $\ell$ splits completely in $F$;
    \item for every place $v$ of $F$ dividing $\ell$, $\Bar{\rho}|_{G_{F_v}}$ is unramified and the eigenvalues $\alpha_1,...,\alpha_m$ of $\Bar{\rho}(\textnormal{Frob}_v)$ satisfy $\alpha_i/\alpha_j\neq \ell$ for $i\neq j$.
\end{enumerate}
\end{Def}
\begin{Rem}
As explained in \cite{ACC23}, Lemma 4.3.2, once we know that $\Bar{\rho}$ is decomposed generic, an argument using Chebotarev's density theorem shows that there are infinitely many choices of $\ell$ as in Definition~\ref{Def2.25}.
\end{Rem}
\begin{Th}\textnormal{\cite{CS19,Kos21}}\label{Th2.27} Let $\widetilde{\mathfrak{m}}\subset \widetilde{\mathbf{T}}^T(\widetilde{K},\Tilde{\lambda},\underline{\tau})$ be a maximal ideal such that the associated Galois representation $\Bar{\rho}_{\widetilde{\mathfrak{m}}}$ is decomposed generic. If we set $d=\textnormal{dim}_{\mathbf{C}}\widetilde{X}_{\widetilde{K}}$, we have a $\widetilde{\mathbf{T}}^T$-equivariant diagram
\begin{equation*}
    H^d(\widetilde{X}_{\widetilde{K}},\mathcal{V}_{(\Tilde{\lambda},\underline{\tau})}^{\widetilde{Q}^{w_0}_{\overline{S}}}[1/p])_{\widetilde{\mathfrak{m}}}\hookleftarrow H^d(\widetilde{X}_{\widetilde{K}},\mathcal{V}_{(\Tilde{\lambda},\underline{\tau})}^{\widetilde{Q}^{w_0}_{\overline{S}}})_{\widetilde{\mathfrak{m}}}\twoheadrightarrow H^d(\partial\widetilde{X}_{\widetilde{K}},\mathcal{V}_{(\Tilde{\lambda},\underline{\tau})}^{\widetilde{Q}_{\overline{S}}^{w_0}})_{\widetilde{\mathfrak{m}}}.
\end{equation*}

\end{Th}
\begin{proof}
This follows from the main result of \cite{CS19} as explained in \cite{ACC23}, Theorem 4.3.3 except the extra conditions appearing in \cite{CS19} that $[F^+:\mathbf{Q}]\geq 2$ and that the length of $\Bar{\rho}_{\widetilde{\mathfrak{m}}}$ is at most $2$. These conditions were removed in \cite{Kos21}.
\end{proof}

\section{$Q$-ordinary Hida theory}
In the first part of the section, based on \cite{ACC23}, \S5.2 and \cite{CN23}, \S2.2, we spell out $Q$-ordinary Hida theory for the Betti cohomology of the locally symmetric space $X_{\widetilde{K}}$ and $X_K$, where $Q$ will be an arbitrary standard parabolic subgroup. Finally, we close the section with a computation of $Q$-ordinary parts of certain Bruhat strata of parabolic induction. This is completely analogous to \cite{ACC23}, \S5.3 and \cite{CN23}, \S2.3. More precisely, the contents of \S\ref{sec3.1}, \S\ref{sec3.2}, and \S\ref{Sec3.4} are carried out for the Borel subgroup in \cite{ACC23}, \S5.2, and for the Siegel parabolic subgroup in \cite{CN23}, \S2.2. Moreover, the content of \S\ref{sec3.5} has been worked out for the Borel subgroup in \cite{ACC23}, \S5.3, and for the Siegel parabolic subgroup in \cite{CN23}, \S2.3. Finally, Hida theory with dual coefficients for the opposite parabolic is considered in \cite{CN23} that we generalise in \S\ref{sec3.3}. A seemingly new result in \S\ref{sec3.3} is a $Q$-ordinary and Hecke equivariant Verdier duality with $\mathcal{O}/\varpi^m$-coefficients (see Proposition~\ref{Prop3.20}).

\subsection{Ordinary parts of smooth representations}\label{sec3.1}
In this text we use several incarnations of ordinary parts. To track the representation theory throughout our arguments, it is often crucial to take the point of view of Emerton \cite{Eme10}, \cite{Eme10b} on taking ordinary parts. We recollect here his approach. In fact, following \cite{ACC23}, we consider a modified version that is more convenient to do homological algebra with and coincides with the original definition on admissible representations which exhaust all the objects which we will consider here.

The setup for the subsection is as follows. We let $L/\mathbf{Q}_p$ be a finite field extension, $G/L$ a connected reductive group. Set $Q\subset G$ to be a parabolic subgroup with a Levi decomposition $Q=M\ltimes N$ and denote by $\overline{Q}=M\ltimes \overline{N}$ the opposite parabolic subgroup. Denote by $Z_M\subset M$ the centre of the Levi factor. Assume that $\mathcal{Q}\subset G(L)$ is a compact open subgroup which admits an Iwahori decomposition
\begin{equation*}
     \overline{N}^1\times M^0\times N^0 \xrightarrow[]{\sim}\mathcal{Q}\xleftarrow[]{\sim} N^0\times M^0\times \overline{N}^1
\end{equation*}
with respect to $Q$ in the sense of \cite{ACC23}, \S2.1.9. Let
\begin{equation*}
    ...\subset \mathcal{Q}(b,b)\subset ...\subset \mathcal{Q}(1,1)\subset \mathcal{Q}(0,1)=\mathcal{Q}\subset G(L)
\end{equation*}
be a cofinal family of compact open subgroups of $G(L)$ such that each $\mathcal{Q}(b,b)$ is normal in $\mathcal{Q}$ and $\mathcal{Q}(b,b)$ admits an Iwahori decomposition
\begin{equation*}
    \overline{N}^bM^bN^0\xrightarrow[]{\sim}\mathcal{Q}(b,b)
\end{equation*}
with respect to $Q$.
Then, by setting $\mathcal{Q}(b,c)=\overline{N}^c M^bN^0$ for $1\leq b\leq c$, one checks that we get a compact open subgroup of $\mathcal{Q}$ (admitting an Iwahori decomposition).
\begin{Ex}\label{ex3.1}
    For $G=\textnormal{GL}_{n,L}$, and $Q=M\ltimes N\subset G$ a parabolic subgroup standard with respect to the Borel subgroup of upper triangular matrices, we can consider the corresponding parahoric group scheme $\mathcal{Q}^{\textnormal{sch}}\subset \textnormal{GL}_n$. Then $\mathcal{Q}:=\mathcal{Q}^{\textnormal{sch}}(\mathcal{O}_L)$ admits an Iwahori decomposition $\overline{N}^1M^0N^0$, where $M^0=M(\mathcal{O}_L)$, $N^0=N(\mathcal{O}_L)$, and $\overline{N}^1$ is $\textnormal{ker}(\overline{N}(\mathcal{O}_L)\to \overline{N}(\mathcal{O}_L/\varpi_L)$. If $Q$ is the parabolic subgroup corresponding to the partition $(n_1,...,n_t)$ of $n$, $\mathcal{Q}$ will be the subgroup of matrices in $G(\mathcal{O}_L)$ that are upper block triangular of type $(n_1,...,n_t)$ modulo $\varpi_L$. For integers $0\leq b\leq c$ with $c\geq 1$, we can then set $\mathcal{Q}(b,c)$ to be the subgroup of matrices in $G(\mathcal{O}_L)$ that are block upper triangular of type $(n_1,...,n_t)$ modulo $\varpi_L^c$, and block unipotent of type $(n_1,...,n_t)$ modulo $\varpi_L^b$.

However, we consider this more general setup since it allows us to for instance set $M^1$ to be $\textnormal{ker}(M(\mathcal{O}_L)\to M(\mathcal{O}_L/\varpi_L^d)$ for some arbitrary integer $d\geq 0$. Moreover, we have the freedom to choose $N^0$ to be other compact open subgroups in $N(\mathcal{O}_L)$ that are preserved under conjugation by $M^0$. All of this allows the formalism to be more flexible and saves some space in introducing notations and applying it in later sections.

\end{Ex}
\begin{Rem}\label{Rem3.1}
    Note that an easy argument using the Iwahori decomposition shows that $M^0$ normalises $N^0=N(L)\cap \mathcal{Q}$ and each $\overline{N}^c=\overline{N}(L)\cap \mathcal{Q}(0,c)$ for $c\geq 1$.
\end{Rem}
We set
\begin{equation*}
    M^+:= \{m\in M(L)\mid mN^0m^{-1}\subset N^0\textnormal{ and }m^{-1}\overline{N}^1m\subset \overline{N}^1\}\footnote{We note that this is not the definition that appears in \cite{Eme10}, \cite{Eme10b} as there the second condition on the elements $m\in M^+$ is not present. However, this is needed for us in order to compare the two natural Hecke actions of our monoids on finite level (see Propostion~\ref{Prop3.20}). We also note that this extra condition is already present for instance in \S2.1.9 of \cite{ACC23}.}
\end{equation*}
and define $Z^+_M:=Z_M(L)\cap M^+$. By Remark~\ref{Rem3.1}, we see that in fact both $M^+\subset M(L)$ and $Z_M^+\subset Z_M(L)$ are open submonoids. We make the following assumption that will be satisfied in our context.
\begin{hyp}\label{hyp3.2}
    For any $m\in M^+$, and $c\geq 1$, we have $m^{-1}\overline{N}^cm\subset \overline{N}^c$.
\end{hyp}

Denote by $M^+\times_{Z_M^+}Z_M(L)$ the quotient of the monoid $M^+\times Z_M(L)$ by the equivalence relation generated by $(mz^+,z)\sim (m,z^+z)$ for $m\in M^+$, $z^+\in Z_M^+$ and $z\in Z_M(L)$.
\begin{Lemma}\label{Lem3.3}
    The morphism of monoids $M^+\times Z_M(L)\to M(L)$ given by multiplication factors through $M^+\times_{Z_M^+}Z_M(L)$ and descends to an isomorphism.
\end{Lemma}
\begin{proof}
    This is essentially Proposition 3.3.6 of \cite{Eme06c}. Since our definition of $M^+$ is slightly different, we provide a sketch of proof here.

    We start by picking $z_p\in Z_M^+$ such that $\{z_p^kN^0z_p^{-k}\}_{k\geq 0}$ forms a basis of neighborhoods of $0$. Such a $z_p\in Z_M^+$ exists by (the discussion above \cite{Eme10b}, Lemma 3.1.3, and) \cite{Eme10b} Lemma 3.1.3 (2),(3) and Lemma 3.1.4, (3). Note that $z_p^{-1}$ and $Z_M^+$ generate $Z_M(L)$ as a monoid. Indeed, given $z\in Z_M(L)$, then for $k\geq 0$ large enough, $z_p^kz$ will satisfy the assumption of \textit{loc. cit.} Lemma 3.1.3, (3) and Lemma 3.1.4, (3), and, in particular, will lie in $Z_M^+$. 
    
    We further see that $M^+$ and $z_p^{-1}$ generate $M(L)$ as a monoid. To see this, pick $m\in M(L)$ and assume that we have
    \begin{equation*}
        mN^0m^{-1}\subset z_p^{-k}N^0 z_p^k\textnormal{ and }m^{-1}\overline{N}^1m\subset z_p^k\overline{N}^1z_p^{-k}
    \end{equation*}
    for a large enough integer $k\geq 0$. Such $k$ always exists by \textit{loc. cit.} Lemma 3.1.3, (2) and Lemma 3.1.4, (2). Then $z_p^km=mz_p^k$ clearly lies in $M^+$.

    One can now conclude just as in the proof of \cite{Eme06c}, Proposition 3.3.6.
    \end{proof}
    From now on, we fix a $z_p\in Z_M^+$ as in the proof of Lemma~\ref{Lem3.3}. We set $\Delta_M^+\subset M^+$ to be any open submonoid containing $M^0$ and $z_p$. Moreover, for $c\geq 1$, we set $\mathcal{Q}^+(0,c)=\overline{N}^c M^+ N^0$ and $\Delta_{\mathcal{Q}}(c)=\overline{N}^c\Delta_M^+N^0$. Then \cite{ACC23}, Lemma 2.1.10 says that $\mathcal{Q}^+(0,c)$ is also a monoid, open in $G(L)$. Moreover, $\mathcal{Q}\mathcal{Q}^+(0,c) \mathcal{Q}=\mathcal{Q}^+(0,c)$, and $\mathcal{Q}^+\cap M(L)=M^+$. The same conclusion holds for $\Delta_{\mathcal{Q}}(c)$ too, since the assumption $M^0\subset \Delta_M^+$ ensures that the proof of \textit{loc. cit.} Lemma 2.1.10 applies. Moreover, set $Q^+=\mathcal{Q}^+(0,c)\cap Q(L)=M^+\ltimes N^0$ and $\Delta_Q=\Delta_{\mathcal{Q}}(c)\cap Q(L)= \Delta_M^+\ltimes N^0$. We finally set $\Delta_M$ to be the monoid generated by $\Delta_M^+$ and $z_p^{-1}$.

    Given $\pi \in \textnormal{Mod}_{\textnormal{sm}}(\mathcal{O}/\varpi^m[\Delta_Q])$, we can consider the Hecke action of $\Delta_M^+$ on the space of $N^0$-invariants $\Gamma(N^0,\pi)$. Namely, for $m\in \Delta_M^+$ and $v\in \Gamma(N^0,\pi)$, the action is given by
    \begin{equation*}
        m\cdot v:=\sum_{n\in N^0/mN^0m^{-1}}nmv.
    \end{equation*}
    We then obtain a left exact functor
    \begin{equation*}
        \Gamma(N^0,-):\textnormal{Mod}_{\textnormal{sm}}(\mathcal{O}/\varpi^m[\Delta_Q])\to \textnormal{Mod}_{\textnormal{sm}}(\mathcal{O}/\varpi^m[\Delta_M^+])
    \end{equation*}
    and, in particular, a derived functor\footnote{Here we use \cite{ACC23}, Lemma 5.2.4 to see that both categories considered are abelian with enough injectives.}
    \begin{equation*}
        R\Gamma(N^0,-):D^+_{\textnormal{sm}}(\mathcal{O}/\varpi^m[\Delta_Q])\to D^+_{\textnormal{sm}}(\mathcal{O}/\varpi^m[\Delta_M^+]).
    \end{equation*}
    We also have the exact localisation functor
    \begin{equation*}
        (-)^{Q\textnormal{-ord}}:\textnormal{Mod}_{\textnormal{sm}}(\mathcal{O}/\varpi^m[\Delta_M^+])\to \textnormal{Mod}_{\textnormal{sm}}(\mathcal{O}/\varpi^m[\Delta_M])
    \end{equation*}
    induced by the inclusion $\Delta_M^+\subset \Delta_M$. We then obtain the functor of "taking $Q$-ordinary parts at infinite level"
    \begin{equation*}
        D^+_{\textnormal{sm}}(\mathcal{O}/\varpi^m[\Delta_Q])\to D^+_{\textnormal{sm}}(\mathcal{O}/\varpi^m[\Delta_M]),
    \end{equation*}
    \begin{equation*}
        \pi\mapsto R\Gamma(N^0,\pi)^{Q\textnormal{-ord}}.
    \end{equation*}

    We now consider several versions of taking $Q$-ordinary parts at finite level and will compare them. To work in the generality we need, we will take invariants with respect to general "types" and not only the trivial representation. In particular, set $\sigma$ to be a smooth $\mathcal{O}/\varpi^m[M^0]$-modules, finite free over $\mathcal{O}/\varpi^m$. We will often abuse the notation and confuse $\sigma$ with $\textnormal{Inf}_{M^0}^{M^0\ltimes N^0}\sigma$. We define the Hecke algebras $\mathcal{H}(\sigma)^{\Delta_M^+}\subset \mathcal{H}(\sigma)^+\subset \mathcal{H}(\sigma)$ as the subalgebra generated by functions supported on $\Delta_M^+$, respectively on $M^+$.
    Finally, note that $[z_p,\textnormal{id}]$ lies $\mathcal{H}(\sigma)^{\Delta_M^+}$ and is a central element. Combined with Lemma~\ref{Lem3.3}, it easily implies that $\mathcal{H}(\sigma)^+[[z_p,\textnormal{id}]^{-1}]=\mathcal{H}(\sigma)$. We set $\mathcal{H}(\sigma)^{\Delta_M}=\mathcal{H}(\sigma)^{\Delta_M^+}[[z_p,\textnormal{id}]^{-1}]$.

    The first candidate for ordinary parts at level $\sigma$ is as follows. For $\pi\in D^+_{\textnormal{sm}}(\mathcal{O}/\varpi^m[\Delta_Q])$ we apply the functor
    \begin{equation*}
        R\Hom_{\mathcal{O}/\varpi^m[M^0]}(\sigma^{\vee},-):D^+_{\textnormal{sm}}(\mathcal{O}/\varpi^m[\Delta_M])\to D^+_{\textnormal{sm}}(\mathcal{H}(\sigma)^{\Delta_M})
    \end{equation*}
    to $R\Gamma(N^0,\pi)^{Q\textnormal{-ord}}$. Here the functor $R\Hom_{\mathcal{O}/\varpi^m[M^0]}(\sigma^{\vee},-)$ is constructed by taking the usual left Hecke action of $\mathcal{H}(\sigma)^{\Delta_M}$ on the space of $\sigma^{\vee}$-invariants.

    In order to define the other candidate, we introduce the functor
    \begin{equation*}
        \Hom_{\mathcal{O}/\varpi^m[M^0\ltimes N^0]}(\sigma^{\vee},-):\textnormal{Mod}_{\textnormal{sm}}(\mathcal{O}/\varpi^m[\Delta_Q])\to \textnormal{Mod}(\mathcal{H}(\sigma)^{\Delta_M^+})
    \end{equation*}
    where the target is the category of left $\mathcal{H}(\sigma)^{\Delta_M^+}$-modules. We spell out the definition of $\Hom_{\mathcal{O}/\varpi^m[M^0\ltimes N^0]}(\sigma^{\vee},-)$. Let $\pi\in \textnormal{Mod}_{\textnormal{sm}}(\Delta_Q)$, pick $[m,\psi]\in \mathcal{H}(\sigma)^{\Delta_M^+}$ and $\phi\in \Hom_{\mathcal{O}/\varpi^m[M^0\ltimes N^0]}(\sigma^{\vee},\pi)$. The action is defined by setting
    \begin{equation*}
        [m,\psi]\cdot \phi:v\mapsto \sum_{n\Tilde{m}\in N^0\rtimes M^0/m(N^0\rtimes M^0)m^{-1}\cap (N^0\rtimes M^0)}\pi(n\Tilde{m}m)\phi(\psi^t\circ\sigma^{\vee}((n\Tilde{m})^{-1})v).
    \end{equation*}
    Note that we have identifications of sets
    \begin{equation}\label{eq3.1}
        N^0\rtimes M^0/\left(m(N^0\rtimes M^0)m^{-1}\cap (N^0\rtimes M^0)\right)\cong 
    \end{equation}
    \begin{equation*}
        \left\{(n,\Tilde{m})\mid n\in N^0/\Tilde{m}mN^0(\Tilde{m}m)^{-1},\Tilde{m}\in M^0/mM^0m^{-1}\cap M^0\right\}\cong 
    \end{equation*}
    \begin{equation*}
        (M^0/mM^0m^{-1}\cap M^0)\times N^0/mN^0m^{-1}.
    \end{equation*}
    In particular, we obtain the following lemma.
    \begin{Lemma}\label{Lem3.4}
        We have a natural equivalence of functors
        \begin{equation*}
            \Hom_{\mathcal{O}/\varpi^m[M^0]}(\sigma^{\vee},-)\circ \Gamma(N^0,-)\cong \Hom_{\mathcal{O}/\varpi^m[M^0\ltimes N^0]}(\sigma^{\vee},-):\textnormal{Mod}_{\textnormal{sm}}(\mathcal{O}/\varpi^m[\Delta_Q])\to
        \end{equation*}
        \begin{equation*}
            \to \textnormal{Mod}(\mathcal{H}(\sigma)^{\Delta_M^+}).
        \end{equation*}
        In particular, for $\pi\in D^+_{\textnormal{sm}}(\mathcal{O}/\varpi^m[\Delta_Q])$, we have a natural isomorphism
        \begin{equation*}
            R\Hom_{\mathcal{O}/\varpi^m[M^0]}(\sigma^{\vee}, R\Gamma(N^0,\pi))\cong R\Hom_{\mathcal{O}/\varpi^m[M^0\ltimes N^0]}(\sigma^{\vee},\pi)
        \end{equation*}
        in $D^+(\mathcal{H}(\sigma)^{\Delta_M^+})$.
    \end{Lemma}
    \begin{proof}
        As the underlying $\mathcal{O}/\varpi^m$-modules of the two functors evaluated on some $\pi\in \textnormal{Mod}_{\textnormal{sm}}(\mathcal{O}/\varpi^m[\Delta_Q])$ clearly coincide, to see the first part, we only have to check that the Hecke actions match up. This follows from the definitions and the identifications \ref{eq3.1}.

        The second part follows from \cite{We94}, Corollary 10.8.3 as soon as we verify the fact that $\Gamma(N^0,-)$ carries injectives to $\Hom_{\mathcal{O}/\varpi^m[M^0]}(\sigma^{\vee},-)$-acyclics. To see this we argue just as in the proof of \cite{ACC23}, Lemma 5.2.7, (1). Namely, we have the exact forgetful functors
        \begin{equation*}
            \textnormal{Mod}_{\textnormal{sm}}(\mathcal{O}/\varpi^m[\Delta_Q])\xrightarrow[]{\alpha}\textnormal{Mod}_{\textnormal{sm}}(\mathcal{O}/\varpi^m[M^0\ltimes N^0]),
        \end{equation*}
        \begin{equation*}
            \textnormal{Mod}_{\textnormal{sm}}(\mathcal{O}/\varpi^m[\Delta_M^+])\xrightarrow[]{\beta}\textnormal{Mod}_{\textnormal{sm}}(\mathcal{O}/\varpi^m[M^0])\textnormal{ and}
        \end{equation*}
        \begin{equation*}
            \textnormal{Mod}(\mathcal{H}(\sigma)^{\Delta_M^+})\xrightarrow[]{\gamma}\textnormal{Mod}(\mathcal{O}/\varpi^m).
        \end{equation*}
        Moreover, \cite{ACC23}, Lemma 5.2.4 says that $\alpha$ and $\beta$ preserve injectives. Now pick an injective $\mathcal{I}\in \textnormal{Mod}_{\textnormal{sm}}(\mathcal{O}/\varpi^m[\Delta_Q])$. We would like to show that, for $i\geq 1$,
        \begin{equation*}
            R^i\Hom_{\mathcal{O}/\varpi^m[M^0]}(\sigma^{\vee},\Gamma(N^0,\mathcal{I}))=0.
        \end{equation*}
        As the previous discussion shows, we can equivalently verify that
        \begin{equation*}
            \gamma R^i\Hom_{\mathcal{O}/\varpi^m[M^0]}(\sigma^{\vee},\Gamma(N^0,\mathcal{I}))=R^i\Hom_{\mathcal{O}/\varpi^m[M^0]}(\sigma^{\vee},\Gamma(N^0,\alpha\mathcal{I}))=0.
        \end{equation*}
        However, $\Gamma(N^0,-)$, as a functor $\textnormal{Mod}_{\textnormal{sm}}(\mathcal{O}/\varpi^m[M^0\ltimes N^0])\to \textnormal{Mod}_{\textnormal{sm}}(\mathcal{O}/\varpi^m[M^0])$, preserves injectives as it possesses an exact left adjoint given by inflation. Therefore, $\alpha \mathcal{I}$ being injective, $R^i\Hom_{\mathcal{O}/\varpi^m[M^0]}(\sigma^{\vee},\Gamma(N^0,\alpha\mathcal{I}))$ vanishes.
    \end{proof}
    We further consider the functor
    \begin{equation*}
        (-)^{Q\textnormal{-ord}}:\textnormal{Mod}(\mathcal{H}(\sigma)^{\Delta_M^+})\to \textnormal{Mod}(\mathcal{H}(\sigma)^{\Delta_M})
    \end{equation*}
     defined by localising along $\mathcal{H}(\sigma)^{\Delta_M^+}\subset \mathcal{H}(\sigma)^{\Delta_M}$.
\begin{Lemma}\label{Lem3.5}
    The functor
    \begin{equation*}
        (-)^{Q\textnormal{-ord}}:\textnormal{Mod}_{\textnormal{sm}}(\mathcal{O}/\varpi^m[\Delta_M^+])\to \textnormal{Mod}_{\textnormal{sm}}(\mathcal{O}/\varpi^m[\Delta_M])
    \end{equation*}
    sends injectives to $\Hom_{\mathcal{O}/\varpi^m[M^0]}(\sigma^{\vee},-)$-acyclics.
\end{Lemma}
\begin{proof}
    Set $\Delta_M'^+$ to be the monoid generated by $M^0$ and $z_p$. Then, for any integer $b\geq 0$, $\mathcal{H}(\Delta_M'^+,M^b)=\mathcal{O}/\varpi^m[\Delta_M'^+/M^b]\cong \mathcal{O}/\varpi^m[M^0/M^b][z_p]$ is a polynomial ring over the Noetherian ring $\mathcal{O}/\varpi^m[M^0/M^b]$. In particular, localisation along $\mathcal{H}(\Delta_M'^+,M^b)\hookrightarrow \mathcal{H}(\Delta_M',M^b)\cong \mathcal{O}/\varpi^m[M^0/M^b][z_p^{\pm 1}]$ preserves injectives for every integer $b\geq 0$.\footnote{Indeed, this is true for arbitrary localisation of the form $R[x]\hookrightarrow R[x^{\pm1}]$ for $R$ a (not necessarily commutative) left Noetherian ring.} Therefore, the proof of \cite{ACC23}, Lemma 5.2.7, (2) applies and we get that $(-)^{Q\textnormal{-ord}}$, as a functor
    \begin{equation*}
        \textnormal{Mod}_{\textnormal{sm}}(\mathcal{O}/\varpi^m[\Delta_M'^+])\to \textnormal{Mod}_{\textnormal{sm}}(\mathcal{O}/\varpi^m[\Delta_M']),
    \end{equation*}
    preserves injectives. 
    
    We claim that the forgetful functors 
    \begin{equation*}
        \alpha:\textnormal{Mod}_{\textnormal{sm}}(\mathcal{O}/\varpi^m[\Delta_M^+])\to \textnormal{Mod}_{\textnormal{sm}}(\mathcal{O}/\varpi^m[\Delta_M'^+]),
    \end{equation*}
    \begin{equation*}
        \beta:\textnormal{Mod}_{\textnormal{sm}}(\mathcal{O}/\varpi^m[\Delta_M])\to \textnormal{Mod}_{\textnormal{sm}}(\mathcal{O}/\varpi^m[\Delta_M'])
    \end{equation*} preserve injectives. Note that once this is verified, running the argument of the second part of the proof of Lemma~\ref{Lem3.4} with $\gamma$ being the forgetful functor $\textnormal{Mod}(\mathcal{H}(\sigma)^{\Delta_M})\to \textnormal{Mod}(\mathcal{H}(\sigma)^{\Delta_M'})$ would allow us to conclude. 
    
    To see that $\alpha$, respectively $\beta$ preserves injectives, we note that it admits the functor $\mathcal{O}/\varpi^m[\Delta_M^+]\otimes_{\mathcal{O}/\varpi^m[\Delta_M'^+]}-$, respectively $\mathcal{O}/\varpi^m[\Delta_M]\otimes_{\mathcal{O}/\varpi^m[\Delta_M']}-$ as a left adjoint. Therefore, it's enough to see that the latter functors are exact. This follows from the fact that $\mathcal{O}/\varpi^m[\Delta_M^+]$, respectively $\mathcal{O}/\varpi^m[\Delta_M]$ is free as a right $\mathcal{O}/\varpi^m[\Delta_M'^+]$- , respectively $\mathcal{O}/\varpi^m[\Delta_M']$-module with set of generators given by a set of representatives of $\Delta_M/\Delta_M'$.
\end{proof}
Our other candidate for ordinary parts of some $\pi \in D^+_{\textnormal{sm}}(\mathcal{O}/\varpi^m[\Delta_Q])$ is the composition $R\Hom_{\mathcal{O}/\varpi^m[M^0\ltimes N^0]}(\sigma^{\vee},\pi)^{Q\textnormal{-ord}}$. Using Lemma~\ref{Lem3.4} and Lemma~\ref{Lem3.5} we see that the two candidates in fact coincide.
\begin{Cor}
    Given $\pi\in D^+_{\textnormal{sm}}(\mathcal{O}/\varpi^m[\Delta_Q])$, we have a natural isomorphism
    \begin{equation*}
        R\Hom_{\mathcal{O}/\varpi^m[M^0]}(\sigma^{\vee},R\Gamma(N^0,\pi)^{Q\textnormal{-ord}})\cong R\Hom_{\mathcal{O}/\varpi^m[M^0\ltimes N^0]}(\sigma^{\vee},\pi)^{Q\textnormal{-ord}}
    \end{equation*}
    in $D^+(\mathcal{H}(\sigma)^{\Delta_M})$.
\end{Cor}
\begin{proof}
    Exactness of $(-)^{Q\textnormal{-ord}}$, Lemma~\ref{Lem3.5}, and an argument just as in the proof of \cite{ACC23}, Lemma 5.2.6 implies that we have
    \begin{equation*}
        R\Hom_{\mathcal{O}/\varpi^m[M^0]}(\sigma^{\vee},-)\circ (-)^{Q\textnormal{-ord}}\cong
    \end{equation*}
    \begin{equation*}
         R(\Hom_{\mathcal{O}/\varpi^m[M^0]}(\sigma^{\vee},-)\circ (-)^{Q\textnormal{-ord}})\cong
    \end{equation*}
    \begin{equation*}
         R((-)^{Q\textnormal{-ord}}\circ\Hom_{\mathcal{O}/\varpi^m[M^0]}(\sigma^{\vee},-))\cong
    \end{equation*}
    \begin{equation*}
         (-)^{Q\textnormal{-ord}}\circ R\Hom_{\mathcal{O}/\varpi^m[M^0]}(\sigma^{\vee},-).
    \end{equation*}
    In particular, for $\pi\in D^+_{\textnormal{sm}}(\mathcal{O}/\varpi^m[\Delta_Q])$, we get a natural isomorphism
    \begin{equation*}
        R\Hom_{\mathcal{O}/\varpi^m[M^0]}(\sigma^{\vee},R\Gamma(N^0,\pi)^{Q\textnormal{-ord}})\cong R\Hom_{\mathcal{O}/\varpi^m[M^0]}(\sigma^{\vee},R\Gamma(N^0,\pi))^{Q\textnormal{-ord}}.
    \end{equation*}
   As a consequence, Lemma~\ref{Lem3.4} implies that there is a natural isomorphism
   \begin{equation*}
       R\Hom_{\mathcal{O}/\varpi^m[M^0]}(\sigma^{\vee},R\Gamma(N^0,\pi)^{Q\textnormal{-ord}})\cong R\Hom_{\mathcal{O}/\varpi^m[M^0\ltimes N^0]}(\sigma^{\vee},\pi)^{Q\textnormal{-ord}}.
   \end{equation*}
   \end{proof}
We now further assume that the action of $M^1$ on $\sigma$ is trivial. Note that it is not a serious restriction since $\sigma$ is assumed to be finite and free over $\mathcal{O}/\varpi^m$ and, by the second paragraph of Example~\ref{ex3.1}, we are free to change $M^1$ to be a smaller compact open subgroup so that it acts trivially on $\sigma$. For $c\geq 1$, set $\widetilde{\sigma}$ to be the smooth $\mathcal{O}/\varpi^m[\mathcal{Q}(0,c)]$-module defined by the map $\mathcal{Q}(0,c)\to \mathcal{Q}(0,c)/\mathcal{Q}(c,c)\cong M^0/M^c$ i.e., for $\Bar{n}mn\in \mathcal{Q}(0,c)$, we set $\widetilde{\sigma}(\Bar{n}mn)=\sigma(m)$. Consider the subalgebras $\mathcal{H}(\widetilde{\sigma})^{\Delta_{\mathcal{Q}}(c)}\subset \mathcal{H}(\widetilde{\sigma})^+\subset \mathcal{H}(\widetilde{\sigma}) $ generated by functions supported on $\Delta_{\mathcal{Q}}(c)$ and $\mathcal{Q}^+(0,c)$, respectively. We then have the following observation.
\begin{Lemma}\label{Lem3.7}
    For any $c\geq 1$, there is an isomorphism of algebras
    \begin{equation*}
        t_c:\mathcal{H}(\sigma)^{\Delta_M^+}\to \mathcal{H}(\widetilde{\sigma})^{\Delta_{\mathcal{Q}}(c)}
    \end{equation*}
    such that, for any $\pi \in \textnormal{Mod}_{\textnormal{sm}}(\mathcal{O}/\varpi^m[\Delta_{\mathcal{Q}}(c)])$,
    the inclusion
    \begin{equation*}
        \Hom_{\mathcal{O}/\varpi^m[\mathcal{Q}(0,c)]}(\widetilde{\sigma}^{\vee}, \pi)\hookrightarrow \Hom_{\mathcal{O}/\varpi^m[M^0]}(\sigma^{\vee},\Gamma(N^0,\pi))
    \end{equation*}
    intertwines the action of $\mathcal{H}(\sigma)^{\Delta_M^+}$ on the source via $t_c$ with the action of $\mathcal{H}(\sigma)^{\Delta_M^+}$ on the target.
\end{Lemma}
\begin{proof}
    Note that $\mathcal{H}(\sigma)^{\Delta_M^+}$ is freely generated as an $\mathcal{O}/\varpi^m$-module by elements of the form $[m,\psi]$ with $m\in \Delta_M^+$ running through any choice of set of representatives for $M^0\setminus \Delta_M^+/M^0$ and $\psi$ is an element of $\textnormal{End}_{\mathcal{O}/\varpi^m}(\sigma)$ intertwining $m$. Similarly,  $\mathcal{H}(\widetilde{\sigma})^{\Delta_{\mathcal{Q}}(c)}$ is freely generated as an $\mathcal{O}/\varpi^m$-module by elements of the form $[q,\widetilde{\psi}]$ with $q\in \Delta_{\mathcal{Q}}(c)$ running through any choice of set of representatives for $\mathcal{Q}(0,c)\setminus \Delta_{\mathcal{Q}}(c) /\mathcal{Q}(0,c)$ and $\widetilde{\psi}$ is an element of $\textnormal{End}_{\mathcal{O}/\varpi^m}(\widetilde{\sigma})$ intertwining $q$.
    We definte $t_c$ by sending a generator $[m,\psi]\in \mathcal{H}(\sigma)^{\Delta_M^+}$ to the function that is supported on $\mathcal{Q}(0,c)m\mathcal{Q}(0,c)$ and sends $m$ to $\psi$ regarded as an element of $\textnormal{End}_{\mathcal{O}/\varpi^m}(\widetilde{\sigma})$. In other words, $t_c([m,\psi])=[m,\psi]\in \mathcal{H}(\widetilde{\sigma})$. 
    
    We check that $t_c([m,\psi])$ indeed lies in $\mathcal{H}(\widetilde{\sigma})^{\Delta_{\mathcal{Q}}(c)}$. To do so, we pick $k\in \mathcal{Q}(0,c)\cap m\mathcal{Q}(0,c) m^{-1}$ and verify that $\widetilde{\sigma}(k)\circ \psi = \psi \circ \widetilde{\sigma}(m^{-1}km)$. Write $k=nh\Bar{n}=mn'h'\Bar{n}'m^{-1}=(mn'm^{-1})(mh'm^{-1})(m\Bar{n}'m^{-1})$ for $n,n'\in N^0$, $h,h'\in M^0$, and $\Bar{n},\Bar{n}'\in \overline{N}^c$. Using the Iwahori decomposition for $\mathcal{Q}(0,c)$, it is easy to see that we must have $n=mn'm^{-1}$, $h=mh'm^{-1}$, and $\Bar{n}=m\Bar{n}'m^{-1}$. In particular, we have
    \begin{equation*}
        \psi\circ\widetilde{\sigma}(m^{-1}km)=\psi\circ \widetilde{\sigma}(m^{-1}nmm^{-1}hmm^{-1}\Bar{n}m)=
    \end{equation*}
    \begin{equation*}
        \psi\circ \widetilde{\sigma}(n'h'\Bar{n}')=\psi\circ\sigma(h')=
    \end{equation*}
    \begin{equation*}
        \sigma(h)\circ \psi=\widetilde{\sigma}(nh\Bar{n})\circ\psi=
    \end{equation*}
    \begin{equation*}
        \widetilde{\sigma}(k)\circ\psi.
    \end{equation*}
    
    We further claim that $t_c$ yields an isomorphism $t_c:\mathcal{H}(\sigma)^{\Delta_M^+}\xrightarrow{\sim}\mathcal{H}(\widetilde{\sigma})^{\Delta_{\mathcal{Q}}(c)}$ of $\mathcal{O}/\varpi^m$-modules. To see this, we have to prove that
    \begin{enumerate}
        \item the map $M^0\setminus \Delta_M^+/M^0\to \mathcal{Q}(0,c)\setminus \Delta_{\mathcal{Q}}(c)/\mathcal{Q}(0,c)$ is bijective, and
        \item for $m\in \Delta_M^+$, $[m,\psi]\in \mathcal{H}(\sigma)$ if and only if $[m,\psi]\in \mathcal{H}(\widetilde{\sigma})$ (where in the latter case we treat $\psi$ as an element of $\textnormal{End}_{\mathcal{O}/\varpi^m}(\widetilde{\sigma})$).
    \end{enumerate}
    The first claim follows from the observation that every $q_1mq_2\in\Delta_{\mathcal{Q}}(c)$ can be written uniquely as $nm'\overline{n}'$ for $n\in N^0$, $m'\in M^0mM^0$, and $\overline{n}\in \overline{N}^c$. This further follows from the existence of an Iwahori decomposition for $\mathcal{Q}(0,c)$ and the fact that $\Delta_M^+\subset M^+$. The "only if" direction of the second claim is exactly the well-definedness of $t_c$ that we already have checked. The other direction follows from the inclusion $M^0\cap mM^0m^{-1}\subset \mathcal{Q}(0,c)\cap m\mathcal{Q}(0,c) m^{-1}$.

    Before proving that $t_c$ also respects the algebra structure, we check the last claim. To do this, we pick an element $\phi\in \Hom_{\mathcal{O}/\varpi^m[\mathcal{Q}(0,c)]}(\widetilde{\sigma}^{\vee},\pi)\subset \Hom_{\mathcal{O}/\varpi^m[M^0]}(\sigma^{\vee},\Gamma(N^0,\pi))$, and compute, first the action of $[m,\psi]$,
    \begin{equation*}
        [m,\psi]\cdot \phi :v\mapsto \sum_{\Tilde{m}\in M^0/mM^0m^{-1}\cap M^0}\pi^{N^0}(\Tilde{m}m)\phi(\psi^t\circ \sigma^{\vee}(\Tilde{m}^{-1})v)=
    \end{equation*}
    \begin{equation*}
        \sum_{\Tilde{m}\in M^0/mM^0m^{-1}\cap M^0}\pi(\Tilde{m})\sum_{n\in N^0/mN^0m^{-1}}\pi(nm)\phi(\psi^t\circ\widetilde{\sigma}^{\vee}(n^{-1}\Tilde{m}^{-1})v)=
    \end{equation*}
    \begin{equation*}
        \sum_{(\Tilde{m},n)\in (M^0/mM^0m^{-1}\cap M^0)\times (N^0/mN^0m^{-1})}\pi(\Tilde{m}nm)\phi(\psi^t\circ \widetilde{\sigma}^{\vee}(n^{-1}\Tilde{m}^{-1})v).
    \end{equation*}
    On the other hand, we have
    \begin{equation*}
        t_c([m,\psi])\cdot \phi: v\mapsto \sum_{q\in \mathcal{Q}(0,c)/\mathcal{Q}(0,c)\cap m\mathcal{Q}(0,c)m^{-1}}\pi(qm)\phi( \psi^{t}\circ \widetilde{\sigma}^{\vee}(q^{-1})v).
    \end{equation*}
Therefore, it suffices to prove that the inclusion $M^0\ltimes N^0\hookrightarrow \mathcal{Q}(0,c)$ descends to a bijection
\begin{equation*}
    (M^0/mM^0m^{-1}\cap M^0)\times (N^0/mN^0m^{-1})\cong \mathcal{Q}(0,c)/\mathcal{Q}(0,c)\cap m\mathcal{Q}(0,c)m^{-1}.
\end{equation*}
This follows from the Iwahori decomposition, the fact that $m\in \Delta_M^+\subset M^+$ and \ref{eq3.1}.

We now easily deduce that $t_c$ respects the algebra structure. Indeed, note that if we set $\pi:=\textnormal{c-Ind}_{\mathcal{Q}(0,c)}^{G(L)}\widetilde{\sigma}^{\vee}$, the action of $\mathcal{H}(\widetilde{\sigma})^{\Delta_{\mathcal{Q}(0,c)}}$ on $\pi$ gives an embedding \begin{equation*}
    \mathcal{H}(\widetilde{\sigma})^{\Delta_{\mathcal{Q}(0,c)}}\hookrightarrow \Hom_{\mathcal{O}/\varpi^m[\mathcal{Q}(0,c)]}(\widetilde{\sigma}^{\vee},\pi).
\end{equation*} In particular, $t_c$ must be an algebra homomorphism as, by the previous paragraph, $t_c$ yields an algebra action of $\mathcal{H}(\sigma)^{\Delta_M^+}$ on $\Hom_{\mathcal{O}/\varpi^m[\mathcal{Q}(0,c)]}(\widetilde{\sigma}^{\vee},\pi)$, that factors through the faithful algebra action of the target of $t_c$.
\end{proof}
As a consequence of Lemma~\ref{Lem3.7}, for every $c\geq 1$, we get a left exact functor
\begin{equation*}
    \Hom_{\mathcal{O}/\varpi^m[\mathcal{Q}(0,c)]}(\widetilde{\sigma}^{\vee},-):\textnormal{Mod}_{\textnormal{sm}}(\mathcal{O}/\varpi^m[\Delta_{\mathcal{Q}}])\to \textnormal{Mod}(\mathcal{H}(\sigma)^{\Delta_M^+}).
\end{equation*}

\begin{Lemma}
    For any $c\geq 1$, there is a natural isomorphism
    \begin{equation*}
        \Hom_{\mathcal{O}/\varpi^m[\mathcal{Q}(0,c)]}(\widetilde{\sigma}^{\vee},-)^{Q\textnormal{-ord}}\cong \Hom_{\mathcal{O}/\varpi^m[M^0\ltimes N^0]}(\sigma^{\vee},-)^{Q\textnormal{-ord}}
    \end{equation*}
    of functors
    \begin{equation*}
        \textnormal{Mod}_{\textnormal{sm}}(\mathcal{O}/\varpi^m[\Delta_{\mathcal{Q}}])\to \textnormal{Mod}(\mathcal{H}(\sigma)^{\Delta_M}).
    \end{equation*}
\end{Lemma}
\begin{proof}
    By the exactness of $(-)^{Q\textnormal{-ord}}$, for $\pi \in \textnormal{Mod}_{\textnormal{sm}}(\mathcal{O}/\varpi^m[\Delta_{\mathcal{Q}}])$, we have an inclusion
    \begin{equation}\label{eq3.2}
        \Hom_{\mathcal{O}/\varpi^m[\mathcal{Q}(0,c)]}(\widetilde{\sigma}^{\vee},\pi)^{Q\textnormal{-ord}}\hookrightarrow \Hom_{\mathcal{O}/\varpi^m[M^0\ltimes N^0]}(\sigma^{\vee},\pi)^{Q\textnormal{-ord}}
    \end{equation}
    and, by Lemma~\ref{Lem3.7}, it is $\mathcal{H}(\sigma)^{\Delta_M}$-equivariant. In particular, we are left to check that \ref{eq3.2} is a bijection. Given $\phi\in \Hom_{\mathcal{O}/\varpi^m[M^0\ltimes N^0]}(\sigma^{\vee},\pi)$, by smoothness of $\pi$ and the fact that $\sigma^{\vee}$ is finite free as an $\mathcal{O}/\varpi^m$-module, $\phi$ lies in $\Hom_{\mathcal{O}/\varpi^m[\mathcal{Q}(0,c')]}(\widetilde{\sigma}^{\vee},\pi)$ for some $c'>c$. By induction, it suffices to prove that, for some integer $k\geq 1$,
    \begin{equation*}
        [z_p,\textnormal{id}]^k\cdot \phi\in \Hom_{\mathcal{O}/\varpi^m[\mathcal{Q}(0,c'-1)]}(\widetilde{\sigma}^{\vee},\pi).
    \end{equation*}
    This, for instance, is proved in \cite{Eme10}, Lemma 3.3.1.
\end{proof}
The following Corollary then summarises the observations of the subsection.
\begin{Cor}\label{Cor3.9}
    Given $\pi \in D^+_{\textnormal{sm}}(\mathcal{O}/\varpi^m[\Delta_{\mathcal{Q}}])$, we have natural isomorphisms
    \begin{equation*}
        R\Hom_{\mathcal{O}/\varpi^m[M^0]}(\sigma^{\vee},R\Gamma(N^0,\pi)^{Q\textnormal{-ord}})\cong R\Hom_{\mathcal{O}/\varpi^m[M^0\ltimes N^0]}(\sigma^{\vee},\pi)^{Q\textnormal{-ord}}\cong
    \end{equation*}
    \begin{equation*}
        \cong R\Hom_{\mathcal{O}/\varpi^m[\mathcal{Q}(0,c)]}(\widetilde{\sigma}^{\vee},\pi)^{Q\textnormal{-ord}}
    \end{equation*}
    in $D^+(\mathcal{H}(\sigma)^{\Delta_M})$.
\end{Cor}

\subsection{$Q$-ordinary Hida theory for $G$}\label{sec3.2} We revisit the setup from \S\ref{sec2.7}. In particular, $F$ will be a CM field and $G=\textnormal{Res}_{\mathcal{O}_F/\mathcal{O}_{F^+}}\textnormal{GL}_n$. Moreover, we fix a tuple $(Q_p,\lambda,\underline{\tau})=(Q_v,\lambda_v,\underline{\tau_v})_{v\in S_p}$ as in \S\ref{sec2.7}. We will work with good subgroups $K\subset G(\mathbf{A}_{F^+}^{\infty})$ with $K^p$ being fixed and $K_p$ of the form 
\begin{equation*}
    \mathcal{Q}_p(b,c)=\prod_{v\in S_p}\mathcal{Q}_v(b,c)\subset \prod_{v\in S_p}\textnormal{GL}_n(\mathcal{O}_{F_v})
\end{equation*} for $c\geq b\geq 0$ with $c\geq c_p$ where $\mathcal{Q}_v(b,c)$ is the parahoric level subgroup corresponding to $Q_v$, just as before. We denote such a good subgroup by $K(b,c)\subset G(\mathbf{A}_{F^{+}}^{\infty})$. Recall that, given a local system $\mathcal{V}_{(\lambda,\underline{\tau})}^{Q_p}$ as in \S\ref{sec2.7}, $\mathcal{H}(\Delta_{\mathcal{Q}_p}(c),\mathcal{Q}_p(b,c))$ acts on $R\Gamma(X_{K(b,c)},\mathcal{V}_{(\lambda,\underline{\tau})}^{Q_p})$ via endomorphisms in $D^+(\mathcal{O})$ (cf. \S\ref{sec2.8}).\footnote{We emphasise that here $\Delta_{\mathcal{Q}_p}(c)$ denotes the monoid introduced in \S2.8.} In particular, for $v\in S_p$, the corresponding $U_p$-operator $U_v^{Q_v}$ acts on the mentioned complex. Following \cite{KT17}, \S2.4, we set
\begin{equation*}
    R\Gamma(X_{K(b,c)},\mathcal{V}_{(\lambda,\underline{\tau})}^{Q_p})^{Q_p\textnormal{-ord}}
\end{equation*}
to be the maximal direct summand of $R\Gamma(X_{K(b,c)},\mathcal{V}_{(\lambda,\underline{\tau})}^{Q_p})$ on which $U_v^{Q_v}$ acts invertibly for each $v\in S_p$. This will then be an object of $D^+(\mathcal{O}[\mathcal{Q}_{p}(0,c)/\mathcal{Q}_{p}(b,c)])$ with an action of the spherical Hecke algebra $\mathbf{T}^T$.

On the other hand, we can apply the formalism of \S\ref{sec3.1} with the parabolic subgroup $Q_p\subset G_p$, compact opens given by $\mathcal{Q}_p(b,c)$, open submonoid of $G_p$ given by $\Delta_{M_p}^+$ (from \S\ref{sec2.8}) and $\sigma$ being the trivial $\Delta_{M_p}^+$-module. Note that in this case $\mathcal{H}(\sigma)^{\Delta_{M_p}^+}=\mathcal{H}(\Delta_{M_p^+},M_p^0)\otimes_{\mathbf{Z}}\mathcal{O}/\varpi^m= \mathcal{O}/\varpi^m[\Delta_{M_p}^+/M_p^0]$, $\mathcal{H}(\sigma)^{\Delta_{M_p}}=\mathcal{O}/\varpi^m[\Delta_{M_p}/M_p^0]$  and $\mathcal{H}(\widetilde{\sigma})^{\Delta_{\mathcal{Q}(c)}}\cong\mathcal{H}(\Delta_{\mathcal{Q}_p},\mathcal{Q}_p(0,c))\otimes_{\mathbf{Z}}\mathcal{O}/\varpi^m$. One notes that, since the cohomology groups of $R\Gamma(X_{K(b,c)},\mathcal{V}_{(\lambda,\underline{\tau})}^{Q_p}/\varpi^m)$ are finite $\mathcal{O}/\varpi^m$-modules, the notions of taking $Q_p$-ordinary parts in the sense of \cite{KT17} and in the sense of \cite{Eme10} coincide. For an argument see the proof of \cite{ACC23}, Proposition 5.2.15. In other words, we have the following.
\begin{Lemma}\label{Lem3.10}    For any $m\geq 1$, $c\geq b\geq 0$ with $c\geq c_p$, there is a natural $\mathbf{T}^T$-equivariant isomorphism
    \begin{equation*}
        R\Gamma(X_{K(b,c)},\mathcal{V}_{(\lambda,\underline{\tau})}^{Q_p}/\varpi^m)^{Q_p\textnormal{-ord}}\xrightarrow{\sim}R\Gamma(\mathcal{Q}_p(b,c),\pi(K^p,\mathcal{V}_{(\lambda,\underline{\tau})}^{Q_p}/\varpi^m))^{Q_p\textnormal{-ord}}
    \end{equation*}
    in $D^+(\mathcal{O}/\varpi^m[M_p^0/M_p^b])$, induced by the natural isomorphism
    \begin{equation*}
        R\Gamma(X_{K(b,c)},\mathcal{V}_{(\lambda,\underline{\tau})}^{Q_p}/\varpi^m)\cong R\Gamma(\mathcal{Q}_p(b,c),\pi(K^p,\mathcal{V}_{(\lambda,\underline{\tau})}^{Q_p}/\varpi^m)).
    \end{equation*}
\end{Lemma}

We now highlight the two important features of Hida theory for Betti cohomology. The first is usually referred to as the \textit{independence of level} property. Consider completed cohomology
\begin{equation*}
\pi(K^{p},\mathcal{V}_{(\lambda,\underline{\tau})}^{Q_p}/\varpi^m)\in D^+_{\textnormal{sm}}(\mathcal{O}/\varpi^m[\Delta_{\mathcal{Q}_p}(c_p)]).
\end{equation*}
Recall that it is equipped with an action of $\mathbf{T}^T$.

\begin{Def}
We set $Q_p$-ordinary completed cohomology to be
\begin{equation*}
    \pi^{Q_p\textnormal{-ord}}(K^{p},\mathcal{V}_{(\lambda,\underline{\tau})}^{Q_p}/\varpi^m):=R\Gamma(N^0_{p},\pi(K^{p},\mathcal{V}_{(\lambda,\underline{\tau,N})}^{Q_p}/\varpi^m))^{Q_p\textnormal{-ord}}\in D^+_{\textnormal{sm}}(\mathcal{O}/\varpi^m[\Delta_{M_p}]).
\end{equation*}
\end{Def}
By Corollary~\ref{Cor3.9} and Lemma~\ref{Lem3.10}, for integers $m\geq 1$, $c\geq b\geq 0$ with $c\geq c_p$, we have
\begin{equation*}
    R\Gamma(M_p^b,\pi^{Q_p\textnormal{-ord}}(K^p,\mathcal{V}_{(\lambda,\underline{\tau})}^{Q_p}/\varpi^m))\cong R\Gamma(X_{K(b,c)},\mathcal{V}_{(\lambda,\underline{\tau})}^{Q_p}/\varpi^m)^{Q_p\textnormal{-ord}}
\end{equation*}
in $D^+(\mathcal{O}/\varpi^m[M^0_p/M_p^b])$.
As an immediate consequence, we can deduce the independence of level property.
\begin{Cor}\textnormal{(Independence of level)}\label{Cor3.12}
For integers $m\geq 1$, and $c\geq b\geq 0$ with $c\geq c_p$, the natural $\mathbf{T}^T$-equivariant morphism
\begin{equation*}
    R\Gamma(X_{K(b,\max\{c_p,b\})},\mathcal{V}_{(\lambda,\underline{\tau})}^{Q_p}/\varpi^m)^{Q_p\textnormal{-ord}}\to
\end{equation*}
\begin{equation*}
    R\Gamma(X_{K(b,c)},\mathcal{V}_{(\lambda,\underline{\tau})}^{Q_p}/\varpi^m)^{Q_p\textnormal{-ord}}
\end{equation*}
is an isomorphism in $D^+(\mathcal{O}/\varpi^m[M^0_{p}/M_{p}^b])$.
\end{Cor}
Note that the analogous statement also holds for compactly supported and boundary cohomology.

As a consequence of independence of level (or rather the statements behind its proof), we can deduce that the $Q_p$-ordinary part of the cohomology complexes $R\Gamma(X_{K(0,c)},\mathcal{V}_{(\lambda,\underline{\tau})}^{Q_p})$ is given by taking invariants of $Q_p$-ordinary completed cohomology with respect to the corresponding smooth types of the Levi subgroup.

\begin{Cor}\label{Cor3.13}
For every integer $c\geq c_p$, we have a natural $\mathbf{T}^T$-equivariant isomorphism 
\begin{equation*}
R\Gamma(X_{K(0,c)},\mathcal{V}_{(\lambda,\underline{\tau})}^{Q_p}/\varpi^m)^{Q_p\textnormal{-ord}}\cong 
\end{equation*}
\begin{equation*}
    R\Hom_{\mathcal{O}/\varpi^m[M^0_{p}]}(\sigma(\underline{\tau})^{\circ}/\varpi^m,\pi^{Q_p\textnormal{-ord}}(K^{p},\mathcal{V}_{\lambda}/\varpi^m))
\end{equation*}
in $D^+(\mathcal{O}/\varpi^m)$ induced by Corollary~\ref{Cor3.9}, Lemma~\ref{Lem3.10} and Lemma~\ref{Lemma2.2}.
\end{Cor}

\begin{proof}
By Corollary~\ref{Cor3.9} and Lemma~\ref{Lem3.10}, we have a natural isomorphism
\begin{equation*}
R\Gamma(X_{K(0,c)},\mathcal{V}_{(\lambda,\underline{\tau,N})}^{Q_p}/\varpi^m)^{Q_p\textnormal{-ord}}\cong
    R\Gamma(M^0_p,\pi^{Q_p\textnormal{-ord}}(K^p,\mathcal{V}_{(\lambda,\underline{\tau})}/\varpi^m))
\end{equation*}
in $D^+(\mathcal{O}/\varpi^m)$. Recall that by definition we have
\begin{equation}\label{eq3.3}
    \pi^{Q_p\textnormal{-ord}}(K^p,\mathcal{V}_{(\lambda,\underline{\tau})}^{Q_p}/\varpi^m)= R\Gamma(N^0_p,\pi(K^p,\mathcal{V}_{\lambda}/\varpi^m)\otimes_{\mathcal{O}/\varpi^m}\widetilde{\sigma(\underline{\tau})}^{\circ,\vee}/\varpi^m)^{Q_p\textnormal{-ord}}.
\end{equation}
Since $-\otimes_{\mathcal{O}/\varpi^m}\widetilde{\sigma(\underline{\tau})}^{\circ,\vee}/\varpi^m$ is exact and has $-\otimes_{\mathcal{O}/\varpi^m}\widetilde{\sigma(\underline{\tau})}^{\circ}/\varpi^m$ as an exact left adjoint, \cite{We94}, Corollary 10.8.3 applies proving that \ref{eq3.3} is naturally isomorphic to
\begin{equation*}
    \pi^{Q_p\textnormal{-ord}}(K^p,\mathcal{V}_{\lambda}/\varpi^m)\otimes_{\mathcal{O}/\varpi^m}\sigma(\underline{\tau})^{\circ,\vee}/\varpi^m=
\end{equation*}
\begin{equation*}
    \Hom_{\mathcal{O}/\varpi^m}(\sigma(\underline{\tau})^{\circ}/\varpi^m,\pi^{Q_p\textnormal{-ord}}(K^p,\mathcal{V}_{(\lambda,\underline{\tau})}^{Q_p}/\varpi^m))
\end{equation*}
in $D^+_{\textnormal{sm}}(\mathcal{O}/\varpi^m[\Delta_{M_p}])$ when $\sigma(\underline{\tau})^{\circ}$ is viewed as an $\mathcal{O}/\varpi^m[\Delta_{M_p}]$-module via inflation from $M^0_p$. Another application of \cite{We94}, Corollary 10.8.3 to 
\begin{equation*}
    \Hom_{\mathcal{O}/\varpi^m[M_p^0]}(\sigma(\underline{\tau})^{\circ}/\varpi^m,-)=\Gamma(M^0_p,-)\circ \Hom_{\mathcal{O}/\varpi^m}(\sigma(\underline{\tau})^{\circ}/\varpi^m,-)
\end{equation*} finishes the proof.
\end{proof}

We now turn to discussing the second feature of Hida theory called \textit{independence of weight.}
Recall the representation
\begin{equation*}
    \mathcal{V}_{w_0^{Q_p}\lambda}=\otimes_{v\in S_p}\mathcal{V}_{w_0^{Q_v}\lambda_v}\in \textnormal{Mod}(\mathcal{O}[\prod_{v\in S_p}M_v^0])
\end{equation*}
from Lemma~\ref{Lem2.14}. View it as a $\Delta_{M_p}$-module via inflation.
\begin{Prop}\label{Prop3.14}\textnormal{(\textnormal{Independence of weight})}
For any integers $m\geq 1$, $c\geq c_p$, subset $\overline{S}\subset \overline{S}_p$, and complex $\pi\in D^+_{\textnormal{sm}}(\mathcal{O}/\varpi^m[\Delta_{\mathcal{Q}_{\overline{S}}}(c)])$, the map introduced in Lemma~\ref{Lem2.14} induces an isomorphism
\begin{equation*}
    R\Gamma(N^0_{\overline{S}},\pi\otimes_{\mathcal{O}/\varpi^m}\mathcal{V}_{\lambda_{\overline{S}}}/\varpi^m)^{Q_{\overline{S}}\textnormal{-ord}}\xrightarrow{\sim} R\Gamma(N^0_{\overline{S}},\pi)^{Q_{\overline{S}}\textnormal{-ord}}\otimes_{\mathcal{O}/\varpi^m}\mathcal{V}_{w_0^{Q_{\overline{S}}}\lambda_{\overline{S}}}/\varpi^m
\end{equation*}
in $D^+_{\textnormal{sm}}(\mathcal{O}/\varpi^m[\Delta_{M_{\overline{S}}}])$.
In particular, we have a natural isomorphism
\begin{equation*}
    \pi^{Q_p\textnormal{-ord}}(K^p,\mathcal{V}_{\lambda}/\varpi^m)\xrightarrow{\sim}\pi^{Q_p\textnormal{-ord}}(K^p,\mathcal{O}/\varpi^m)\otimes_{\mathcal{O}/\varpi^m}\mathcal{V}_{w_0^{Q_p}\lambda}/\varpi^m
\end{equation*}
in $D^+_{\textnormal{sm}}(\mathcal{O}/\varpi^m[\Delta_{M_p}])$.
\end{Prop}
\begin{proof}
The same argument as in the proof of \cite{CN23}, Proposition 2.2.15 applies here.
\end{proof}

We now can deduce that $R\Gamma(X_{K(0,c_p)},\mathcal{V}_{(\lambda,\underline{\tau})}^{Q_p}/\varpi^m)^{Q_p\textnormal{-ord}}$ admits a natural Hecke action at $p$ corresponding to the data $(\lambda,\underline{\tau})$. Namely, we set
\begin{equation*}
    \sigma(\lambda,\underline{\tau})^{\circ}:= \mathcal{V}_{w_0^{Q_p}\lambda}^{\vee}\otimes_{\mathcal{O}}\sigma(\underline{\tau})^{\circ},
\end{equation*}
a locally algebraic $\mathcal{O}$-representation of $M^0_p$.
\begin{Cor}\label{Cor3.15}
    We have a natural isomorphism
    \begin{equation*}
R\Gamma(X_{K(0,c_p)},\mathcal{V}_{(\lambda,\underline{\tau})}^{Q_p}/\varpi^m)^{Q_p\textnormal{-ord}}\cong
    \end{equation*}
    \begin{equation*}
         R\Hom_{\mathcal{O}/\varpi^m[M^0_p]}(\sigma(\lambda,\underline{\tau})^{\circ}/\varpi^m,\pi^{Q_p\textnormal{-ord}}(K^p,\mathcal{O}/\varpi^m))
    \end{equation*}
    in $D^+(\mathcal{O}/\varpi^m)$ induced by Corollary~\ref{Cor3.13}, and Proposition~\ref{Prop3.14}.
    In particular, we have an induced algebra homomorphism
    \begin{equation*}
        \mathcal{H}(\sigma(\lambda,\underline{\tau})^{\circ,\vee})\otimes_{\mathcal{O}}\mathcal{O}/\varpi^m\to \textnormal{End}_{D^+(\mathcal{O}/\varpi^m)}(R\Gamma(X_{K(0,c_p)},\mathcal{V}_{(\lambda,\underline{\tau})}^{Q_p}/\varpi^m)^{Q_p\textnormal{-ord}}).
    \end{equation*}
\end{Cor}
\begin{proof}
The first part of the statement follows immediately from Corollary~\ref{Cor3.13} and Proposition~\ref{Prop3.14}.

For the second part, note that the formalism of \S\ref{sec3.1} with the choice $M^+_p$ for the role of $\Delta_{M_p}^+$ implies that $\pi^{Q_p\textnormal{-ord}}(K^p,\mathcal{O}/\varpi^m)$ can be viewed as an object in $D^+_{\textnormal{sm}}(\mathcal{O}/\varpi^m[M_p])$. We then have
\begin{equation*}
    R\Hom_{\mathcal{O}/\varpi^m[M^0_p]}(\sigma(\lambda,\underline{\tau})^{\circ}/\varpi^m,\pi^{Q_p\textnormal{-ord}}(K^p,\mathcal{O}/\varpi^m))\in D^+(\mathcal{H}(\sigma(\lambda,\underline{\tau})^{\circ,\vee}/\varpi^m)).
\end{equation*}
Moreover, the forgetful functor 
\begin{equation*}
    \textnormal{Mod}_{\textnormal{sm}}(\mathcal{O}/\varpi^m[G_p])\to \textnormal{Mod}_{\textnormal{sm}}(\mathcal{O}/\varpi^m[\Delta_{M_p}^+\ltimes N^0_p])
\end{equation*} sends injectives to $\Gamma(N^0_p,-)$-acyclics by \cite{Eme10b}, Proposition 2.1.11. Consequently, an application of \cite{We94}, Corollary 10.8.3 shows that the ordinary completed cohomology complex $\pi^{Q_p\textnormal{-ord}}(K^p,\mathcal{O}/\varpi)\in D^+(\textnormal{Mod}_{\textnormal{sm}}[\Delta_{M_p}])$ can be computed by applying $R\Gamma(N^0_p,-)^{Q_p\textnormal{-ord}}$ to the completed cohomology complex $\pi(K^p,\mathcal{O}/\varpi^m)\in D^+_{\textnormal{sm}}(\mathcal{O}/\varpi^m[G_p])$ followed by an application of the forgetful functor $\textnormal{Mod}_{\textnormal{sm}}(\mathcal{O}/\varpi^m[M_p])\to \textnormal{Mod}_{\textnormal{sm}}(\mathcal{O}/\varpi^m[\Delta_{M_p}])$. On the other hand, \cite{Eme10b}, Proposition 2.1.2 shows that the forgetful functor
\begin{equation*}
    \textnormal{Mod}_{\textnormal{sm}}(\mathcal{O}/\varpi^m[M_p])\to \textnormal{Mod}_{\textnormal{sm}}(\mathcal{O}/\varpi^m[M^0_p])
\end{equation*} preserves injectives.
In particular, another application of \cite{We94}, Corollary 10.8.3 yields an algebra homomorphism
\begin{equation*}
    \mathcal{H}(\sigma(\lambda,\underline{\tau})^{\circ,\vee}/\varpi^m)\to \textnormal{End}_{D^+(\mathcal{O}/\varpi^m)}(R\Gamma(X_{K(0,c_p)},\mathcal{V}_{(\lambda,\underline{\tau})}^{Q_p}/\varpi^m)^{Q_p\textnormal{-ord}}).
\end{equation*}
Moreover, we have a natural morphism
\begin{equation*}
    \mathcal{H}(\sigma(\lambda,\underline{\tau})^{\circ,\vee})\to\mathcal{H}(\sigma(\lambda,\underline{\tau})^{\circ,\vee}/\varpi^m)
\end{equation*}
induced by the short exact sequence
\begin{equation*}
    0\to \textnormal{c-Ind}_{M^0_p}^{M_p}\sigma(\lambda,\underline{\tau})^{\circ,\vee}\xrightarrow{\varpi^m\cdot}\textnormal{c-Ind}_{M^0_p}^{M_p}\sigma(\lambda,\underline{\tau})^{\circ,\vee}\to \textnormal{c-Ind}_{M^0_p}^{M_p}(\sigma(\lambda,\underline{\tau})^{\circ,\vee}/\varpi^m)\to 0
\end{equation*}
and the proof is finished.
\end{proof}
We note that the action of $\mathcal{H}(\sigma,\underline{\tau})^{\circ,\vee}$ on $R\Gamma(X_{K(0,c_p)},\mathcal{V}_{(\lambda,\underline{\tau})}^{Q_p}/\varpi^m)^{Q_p\textnormal{-ord}}$ admits another description. Namely, by independence of level, we can pass to level $K(0,m)$ and then, by independence of weight, we get an identification
\begin{equation*}
    R\Gamma(X_{K(0,c_p)},\mathcal{V}_{(\lambda,\underline{\tau})}^{Q_p}/\varpi^m)^{Q_p\textnormal{-ord}}\cong R\Gamma(X_{K(0,m)},\widetilde{\sigma(\lambda,\underline{\tau})^{\circ,\vee}/\varpi^m})^{Q_p\textnormal{-ord}}
\end{equation*}
in $D^+(\mathcal{O}/\varpi^m)$. On the other hand, we have
\begin{equation*}
    R\Gamma(X_{K(0,m)},\widetilde{\sigma(\lambda,\underline{\tau})^{\circ,\vee}/\varpi^m})\cong R\Hom_{Q_p(0,m)}(\widetilde{\sigma(\lambda,\underline{\tau})^{\circ}/\varpi^m},\pi(K^p,\mathcal{O}/\varpi^m))
\end{equation*}
in $D^+(\mathcal{O}/\varpi^m)$ and the latter naturally lives in $D^+(\mathcal{H}(\sigma(\lambda,\underline{\tau})^{\circ,\vee}/\varpi^m)^+)$ according to \S\ref{sec3.1}. Therefore, we get a natural action of $\mathcal{H}(\sigma(\lambda,\underline{\tau})^{\circ,\vee}/\varpi^m)^+$ on the complex $R\Gamma(X_{K(0,m)},\widetilde{\sigma(\lambda,\underline{\tau})^{\circ,\vee}/\varpi^m})$ and, using Corollary~\ref{Cor3.9}, one sees that the induced action of $\mathcal{H}(\sigma(\lambda,\underline{\tau})^{\circ,\vee}/\varpi^m)$ on $R\Gamma(X_{K(0,c_p)},\mathcal{V}_{(\lambda,\underline{\tau})}^{Q_p}/\varpi^m)^{Q_p\textnormal{-ord}}$ is the one constructed in Corollary~\ref{Cor3.15}. An upshot of this observation is that this way the Hecke action can be described using the formalism of \S\ref{sec2.4}.

\subsection{Hida theory with dual coefficients}\label{sec3.3}
Given a tuple $(Q_p,\lambda,\underline{\tau})$ as in \S\ref{sec3.2}, and an integer $m\geq 1$, we also discuss $\overline{Q}_p$-ordinary Hida theory for $\mathcal{V}_{\lambda,\underline{\tau}}^{Q_p,\vee}/\varpi^m$ on $X_K$ where $K\subset G(\mathbf{A}_{F^+}^{\infty})$ is a good subgroup such that $K_p=\mathcal{Q}_p(0,\widetilde{c})$ with $\widetilde{c}\geq c_p$. This is developed by applying the formalism of \S\ref{sec3.1} with parabolic subgroup $\overline{Q}_p=M_p\ltimes \overline{N}_p\subset G_p$, $N^0:=\overline{N}_p^{\widetilde{c}}$, $M^b:= M^b_p$, $\overline{N}^c:=N_p^c$, $\Delta_M^+:=(\Delta_{M_p}^{+})^{-1}$, $z_p:=u_p^{Q_p,-1}$, and $\sigma:=\sigma(\lambda,\underline{\tau})^{\circ,\vee}$. We can then introduce $\overline{Q}_p$-ordinary parts of the level $K(0,\widetilde{c})$ cohomology of $\mathcal{V}_{(\lambda,\underline{\tau})}^{Q_p,\vee}/\varpi^m$ by inverting the Hecke operator attached to $u_p^{Q_p,-1}$. We can also introduce $\overline{Q}_p$-ordinary completed cohomology 
\begin{equation*}
    \pi^{\overline{Q}_p\textnormal{-ord}}(K^p,\mathcal{V}_{(\lambda,\tau)}^{Q_p,\vee}/\varpi^m):=R\Gamma(\overline{N}^{\widetilde{c}}_p,\pi(K^p,\mathcal{V}_{(\lambda,\tau)}^{Q_p,\vee}/\varpi^m))^{\overline{Q}_p\textnormal{-ord}}\in D^+_{\textnormal{sm}}(\mathcal{O}/\varpi^m[\Delta_{M_p}]).
\end{equation*}
Given this setup, the formalism of \S\ref{sec3.1} combined with the short exact sequence
\begin{equation*}
    0\to \mathcal{V}_{w_0^{Q_p}\lambda}^{\vee}\to \mathcal{V}_{\lambda}^{\vee}\to \mathcal{K}_{\lambda}^{\vee}\to 0
\end{equation*}
induced by taking duals of the surjection of Lemma~\ref{Lem2.14} yields the independence of weight property for dual coefficients.
\begin{Prop}\textnormal{(Independence of weight)}
    For any integer $m\geq 1$, there is a natural isomorphism
    \begin{equation*}
        \pi^{\overline{Q}_p\textnormal{-ord}}(K^p,\mathcal{V}_{\lambda}^{\vee}/\varpi^m)\xrightarrow{\sim}\pi^{\overline{Q}_p\textnormal{-ord}}(K^p,\mathcal{O}/\varpi^m)\otimes_{\mathcal{O}/\varpi^m}\mathcal{V}_{w_0^{Q_p}\lambda}^{\vee}/\varpi^m
    \end{equation*}
    in $D^+_{\textnormal{sm}}(\mathcal{O}/\varpi^m[\Delta_{M_p}])$.
\end{Prop}
\begin{Rem}\label{Rem3.17}
    Note that even though the definition of $\overline{Q}_p$-ordinary completed cohomology seems to depend on the choice of $\Tilde{c}\geq c_p$, it in fact is independent of this choice up to natural isomorphism. To see this, consider two integers $\Tilde{c}_1,\Tilde{c}_2\geq c_p$. Then, by independence of weight for $i=1,2$, we have natural isomorphisms
    \begin{equation*}
        R\Gamma(\overline{N}_p^{\Tilde{c}_i},\pi(K^p,\mathcal{V}_{(\lambda,\tau)}^{Q_p,\vee}/\varpi^m))^{\overline{Q}_p\textnormal{-ord}}\cong R\Gamma(\overline{N}_p^{\Tilde{c}_i},\pi(K^p,\mathcal{O}/\varpi^m))^{\overline{Q}_p\textnormal{-ord}}\otimes_{\mathcal{O}/\varpi^m}\sigma(\lambda,\underline{\tau})^{\circ}/\varpi^m
    \end{equation*}
    in $D^+_{\textnormal{sm}}(\mathcal{O}/\varpi^m[\Delta_{M_p}])$. In particular, we can reduce the question to one with trivial coefficients. In that case, $\pi(K^p,\mathcal{O}/\varpi^m)$ in fact lies in $D^+_{\textnormal{sm}}(\mathcal{O}/\varpi^m[\prod_{v\in S_p}\overline{Q}_v(F_v)])$ and the proof of \cite{Eme10}, Proposition 3.1.12 shows that there is a natural isomorphism
    \begin{equation}
        R\Gamma(\overline{N}_p^{\Tilde{c}_1},\pi(K^p,\mathcal{O}/\varpi^m))^{\overline{Q}_p\textnormal{-ord}}\cong R\Gamma(\overline{N}_p^{\Tilde{c}_2},\pi(K^p,\mathcal{O}/\varpi^m))^{\overline{Q}_p\textnormal{-ord}}
    \end{equation}
    in $D^+_{\textnormal{sm}}(\mathcal{O}/\varpi^m[M_p])$, showing the claim.
We note that the same argument shows that instead of $\overline{N}^{\Tilde{c}}_p$, we could have taken any compact open $\overline{N}_p^{\circ}\subset \overline{N}_p^{c_p}$ that is preserved under conjugation by $M^0_p$.
\end{Rem}
As a consequence of Remark~\ref{Rem3.17}, we obtain an independence of level property for dual coefficients.
\begin{Prop}[Independence of level]\label{Prop3.18}
    For integers $m\geq 1$, and $\Tilde{c}\geq c_p$, the natural $\mathbf{T}^T$-equivariant morphism
    \begin{equation*}
        R\Gamma(X_{K(0,c_p)},\mathcal{V}_{(\lambda,\underline{\tau})}^{Q_p,\vee}/\varpi^m)^{\overline{Q}_p\textnormal{-ord}}\to R\Gamma(X_{K(0,\Tilde{c})},\mathcal{V}_{(\lambda,\underline{\tau})}^{Q_p,\vee}/\varpi^m)^{\overline{Q}_p\textnormal{-ord}}
    \end{equation*}
    is an isomorphism in $D^+(\mathcal{O}/\varpi^m)$.
\end{Prop}
Due to the independence of weight property, and Corollary~\ref{Cor3.9}, we can introduce the relevant Hecke action at $p$.
\begin{Cor}\label{Cor3.19}
    We have a natural isomorphism
    \begin{equation*}
R\Gamma(X_{K(0,\widetilde{c})},\mathcal{V}_{(\lambda,\underline{\tau})}^{Q_p,\vee}/\varpi^m)^{\overline{Q}_p\textnormal{-ord}}\cong
    \end{equation*}
    \begin{equation*}
        \cong R\Hom_{\mathcal{O}/\varpi^m[M^0_p]}(\sigma(\lambda,\underline{\tau})^{\circ,\vee}/\varpi^m,\pi^{\overline{Q}_p\textnormal{-ord}}(K^p,\mathcal{O}/\varpi^m))
    \end{equation*}
    in $D^+(\mathcal{O}/\varpi^m)$.
    In particular, we have a natural algebra homomorphism
    \begin{equation*}
        \mathcal{H}(\sigma(\lambda,\underline{\tau})^{\circ})\otimes_{\mathcal{O}}\mathcal{O}/\varpi^m\to \textnormal{End}_{D^+(\mathcal{O}/\varpi^m)}(R\Gamma(X_{K(0,\widetilde{c})},\mathcal{V}_{(\lambda,\underline{\tau})}^{Q_p,\vee}/\varpi^m)^{\overline{Q}_p\textnormal{-ord}}).
    \end{equation*}
\end{Cor}
\begin{proof}
    The proof is identical to that of Corollary~\ref{Cor3.15}. Namely, by applying Corollary~\ref{Cor3.9}, and the obvious analogue of Lemma~\ref{Lem3.10}, we get an isomorphism
    \begin{equation*}
R\Gamma(X_{K(0,\widetilde{c})},\mathcal{V}_{(\lambda,\underline{\tau})}^{Q_p,\vee}/\varpi^m)^{\overline{Q}_p\textnormal{-ord}}\cong R\Gamma(M^0_p,\pi^{\overline{Q}_p\textnormal{-ord}}(K^p,\mathcal{V}_{(\lambda,\underline{\tau})}^{Q_p,\vee}/\varpi^m)).
    \end{equation*}
    Then the isomorphism we seek is induced by the dual of the surjection
    \begin{equation}
\mathcal{V}_{(\lambda,\underline{\tau})}^{Q_p}/\varpi^m\to \widetilde{\sigma(\lambda,\underline{\tau})^{\circ,\vee}/\varpi^m}.\end{equation}
\end{proof}
We mention that the introduced Hecke action once again has a slightly different description just as explained at the end of \S\ref{sec3.2}.

Finally, we deduce Hecke-equivariance of Poincar\'e duality for $Q_p$-ordinary cohomology with $\mathcal{V}_{(\lambda,\underline{\tau})}^{Q_p}/\varpi^m$-coefficients.
\begin{Prop}\label{Prop3.20}
    The Verdier duality isomorphism
    \begin{equation*}
        R\Hom_{\mathcal{O}/\varpi^m}(R\Gamma_c(X_{K(0,\widetilde{c})},\mathcal{V}_{(\lambda,\underline{\tau})}^{Q_p,\vee}/\varpi^m),\mathcal{O}/\varpi^m)\cong R\Gamma(X_{K(0,\widetilde{c})},\mathcal{V}_{(\lambda,\underline{\tau})}^{Q_p}/\varpi^m)[\dim_{\mathbf{R}}X_K]
    \end{equation*}
    induces an isomorphism
    \begin{equation*}
        R\Hom_{\mathcal{O}/\varpi^m}(R\Gamma_c(X_{K(0,\widetilde{c})},\mathcal{V}_{(\lambda,\underline{\tau})}^{Q_p,\vee}/\varpi^m)^{\overline{Q}_p\textnormal{-ord}},\mathcal{O}/\varpi^m)\cong 
    \end{equation*}
    \begin{equation*}
        \cong R\Gamma(X_{K(0,\widetilde{c})},\mathcal{V}_{(\lambda,\underline{\tau})}^{Q_p}/\varpi^m)^{Q_p\textnormal{-ord}}[\dim_{\mathbf{R}}X_K]
    \end{equation*}
    in $D^+(\mathcal{O}/\varpi^m)$. Moreover, the latter isomorphism is equivariant with respect to the natural left action of $\mathcal{H}(\sigma(\lambda,\underline{\tau})^{\circ,\vee})\otimes_{\mathcal{O}}\mathcal{O}/\varpi^m$ on the RHS and the one induced by the anti-isomorphism
    \begin{equation*}
        \mathcal{H}(\sigma(\lambda,\underline{\tau})^{\circ,\vee}/\varpi^m)\xrightarrow{\sim} \mathcal{H}(\sigma(\lambda,\underline{\tau})^{\circ}/\varpi^m),
    \end{equation*}
    \begin{equation*}
        [g,\psi]\mapsto [g^{-1},\psi^t]
    \end{equation*}
    on the LHS.
\end{Prop}
\begin{proof}
    The first part follows from applying Corollary~\ref{Cor2.8} with $\sigma=\mathcal{V}_{(\lambda,\underline{\tau})}^{Q_p}/\varpi^m$ and noting that $U_p^{Q_p}=[u_p^{Q_p},u_p^{Q_p}\cdot(-)]\in \mathcal{H}(\mathcal{V}_{(\lambda,\underline{\tau})}^{Q_p}/\varpi^m)$.

    To see the second part, we reduce the question to the case when $\Tilde{c}\geq m$ using independence of level (cf. Corollary~\ref{Cor3.12}, Proposition~\ref{Prop3.18}). In particular, $\widetilde{\sigma(\lambda,\underline{\tau})^{\circ,\vee}/\varpi^m}$ makes sense as a representation of $\mathcal{Q}(0,\widetilde{c})$! Then, by independence of weight and naturality of Verdier duality applied to the $\mathcal{Q}(0,\widetilde{c})$-equivariant surjection
    \begin{equation}\label{eq3.4}
\mathcal{V}_{(\lambda,\underline{\tau})}^{Q_p}/\varpi^m\to \widetilde{\sigma(\lambda,\underline{\tau})^{\circ,\vee}/\varpi^m},
    \end{equation}
    the mentioned Verdier duality isomorphism for ordinary parts is also induced by the Verdier duality isomorphism
    \begin{equation}\label{eq3.5}
        R\Hom_{\mathcal{O}/\varpi^m}(R\Gamma_c(X_{K(0,\widetilde{c})},\widetilde{\sigma(\lambda,\underline{\tau})^{\circ}/\varpi^m}),\mathcal{O}/\varpi^m)\cong R\Gamma(X_{K(0,\widetilde{c})},\widetilde{\sigma(\lambda,\underline{\tau})^{\circ,\vee}/\varpi^m}).
    \end{equation}
    However, the Verdier duality isomorphism \ref{eq3.5} satisfies the desired Hecke-equivariance by Corollary~\ref{Cor2.8} and Corollary~\ref{Cor3.9}.\footnote{Note that at this step we used that our Hecke action at $p$ matches up with the Hecke action constructed at the end of \S\ref{sec3.2}.}
\end{proof}

\subsection{$\widetilde{Q}$-ordinary Hida theory for $\widetilde{G}$}\label{Sec3.4}
As everything mentioned in the previous two subsections applies verbatim for the group $\widetilde{G}$ at split $p$-adic places of $F^+$, we will only set up the notations and explain the relevant results. We revisit the setup of the corresponding part of \S\ref{sec2.7}. In  particular, we remind the reader of Assumption~\ref{blanket}. Fix a subset of $p$-adic places $\overline{S}\subset \overline{S}_p$ and a lift $v\mid \Bar{v}$ for each $\Bar{v}\in \overline{S}_p$ in $S_p$. Consider a tuple $(Q_p,\lambda,\underline{\tau})$ as previously and consider a corresponding tuple $(\widetilde{Q}_{\overline{S}},\Tilde{\lambda}_{\overline{S}},\underline{\tau}):=(\widetilde{Q}_{\Bar{v}},\Tilde{\lambda}_{\Bar{v}}, \underline{\tau_{\Bar{v}}})_{\Bar{v}\in\overline{S}}$ as in \S\ref{sec2.7}.\footnote{Note that in particular we are implicitly assuming that $\widetilde{\lambda}_{\overline{S}}$ is dominant.} Further set $\Tilde{\lambda}$ to be some extension of $\Tilde{\lambda}$ to a $(\textnormal{Res}_{F^+/\mathbf{Q}}\widetilde{B})_E$-dominant weight of $(\textnormal{Res}_{F^+/\mathbf{Q}}\widetilde{G})_E$. We then similarly form the parahoric level subgroups
\begin{equation*}
    \widetilde{\mathcal{Q}}_{\overline{S}}(b,c)=\prod_{\Bar{v}\in \overline{S}}\widetilde{\mathcal{Q}}_{\Bar{v}}(b,c)\subset \prod_{\Bar{v}\in\overline{S}}\widetilde{G}(\mathcal{O}_{F^+_{\Bar{v}}})
\end{equation*}
for integers $0\leq b\leq c$ with $c\geq c_p$. For the rest of the subsection, we fix a prime-to-$\overline{S}$ good level subgroup $\widetilde{K}^{\overline{S}}\subset \widetilde{G}(\mathbf{A}_{F^+}^{\overline{S}\cup\{\infty\}})$ and set $\widetilde{K}(b,c)$ to be the good subgroup $\widetilde{K}^{\overline{S}}\widetilde{\mathcal{Q}}_{\overline{S}}(b,c)\subset \widetilde{G}(\mathbf{A}_{F^+}^{\infty})$. We then freely borrow the notation of \S\ref{sec2.8}. In particular, we have an open submonoid $\widetilde{\Delta}_{\widetilde{\mathcal{Q}}^{w_0}_{\overline{S}}}(c_p)\subset \widetilde{G}_{\overline{S}}$, and can set $\widetilde{\Delta}_{\widetilde{M}_{\overline{S}}^{w_0}}^+:=\widetilde{\Delta}_{\widetilde{\mathcal{Q}}^{w_0}_{\overline{S}}}(c_p)\cap \widetilde{M}^{w_0}_{\overline{S}}\subset \widetilde{M}_{\overline{S}}^{w_0,+},$ and $\widetilde{\Delta}_{\widetilde{M}_{\overline{S}}^{w_0}}:=\widetilde{\Delta}_{\widetilde{M}_{\overline{S}}^{w_0}}^+[u_{\Bar{v}}^{\widetilde{Q}^{w_0}_{v},\pm 1}\mid \Bar{v}\in \overline{S}]$.

For $c\geq b\geq 0$ with $c\geq c_p$, the Hecke algebra $\mathcal{H}(\widetilde{\Delta}_{\widetilde{\mathcal{Q}}_{\overline{S}}^{w_0}}(c),\widetilde{\mathcal{Q}}_{\overline{S}^{w_0}}(b,c))$ acts on $R\Gamma(\widetilde{X}_{\widetilde{K}(b,c)},\mathcal{V}_{(\Tilde{\lambda},\underline{\tau})}^{\widetilde{Q}_{\overline{S}}^{w_0}})$ via endomorphisms in $D^+(\mathcal{O})$. In particular, we introduce 
the $\widetilde{Q}^{w_0}_{\overline{S}}$-ordinary parts of the complex $R\Gamma(\widetilde{X}_{\widetilde{K}(b,c)},\mathcal{V}_{(\Tilde{\lambda},\underline{\tau})}^{\widetilde{Q}_{\overline{S}}^{w_0}})^{\widetilde{Q}_{\overline{S}}^{w_0}\textnormal{-ord}}$ as the maximal direct summand on which each $U_{\Bar{v}}^{\widetilde{Q}_{\Bar{v}}^{w_0}}$ acts invertibly.

On the other hand, the formalism of \S\ref{sec3.1} can be applied with the choices $Q=\widetilde{Q}_{\overline{S}}^{w_0}$, $N^0=\widetilde{N}_{\overline{S}}^{w_0,0}$, $M^b=\widetilde{M}^{w_0,b}_{\overline{S}}$, $\overline{N}^c=\overline{\widetilde{N}}_{\overline{S}}^{w_0,c}$, and $\sigma$ being the trivial $\widetilde{\Delta}_{\widetilde{M}_{\overline{S}}^{w_0}}^+$-module.  One can then compare the two constructions and see that the analogue of Lemma~\ref{Lem3.10} holds.

We can further introduce $\widetilde{Q}_{\overline{S}}^{w_0}$-ordinary completed cohomology
\begin{equation*}
    \pi^{\widetilde{Q}_{\overline{S}}^{w_0}\textnormal{-ord}}(\widetilde{K}^{\overline{S}},\mathcal{V}_{(\Tilde{\lambda},\underline{\tau})}^{\widetilde{Q}_{\overline{S}}^{w_0}}/\varpi^m):= 
\end{equation*}
\begin{equation*}
    R\Gamma(\widetilde{N}^{w_0,0}_{\overline{S}},\pi(\widetilde{K}^{\overline{S}},\mathcal{V}_{(\Tilde{\lambda},\underline{\tau,N})}^{\widetilde{Q}_{\overline{S}}^{w_0}}/\varpi^m))^{\widetilde{Q}_{\overline{S}}^{w_0}\textnormal{-ord}}\in D^+_{\textnormal{sm}}(\mathcal{O}/\varpi^m[\widetilde{\Delta}_{\widetilde{M}_{\overline{S}}^{w_0}}]).
\end{equation*}
The analogues of the independence of level and weight property hold once again. To emphasise the normalisations and introduce the necessary notations for later, we spell out the statement of the latter.
\begin{Prop}[\textnormal{Independence of weight}]
    For any integer $m\geq 1$, there is a natural isomorphism
    \begin{equation*}
        \pi^{\widetilde{Q}_{\overline{S}}^{w_0}\textnormal{-ord}}(\widetilde{K}^{\overline{S}},\mathcal{V}_{\Tilde{\lambda}}/\varpi^m)\xrightarrow{\sim}\pi^{\widetilde{Q}_{\overline{S}}^{w_0}\textnormal{-ord}}(\widetilde{K}^{\overline{S}},\mathcal{V}_{\Tilde{\lambda}^{\overline{S}}})\otimes_{\mathcal{O}/\varpi^m}\mathcal{V}_{w_0^{\widetilde{Q}_{\overline{S}}^{w_0}}\widetilde{\lambda}_{\overline{S}}}/\varpi^m
    \end{equation*}
    in $D^+_{\textnormal{sm}}(\mathcal{O}/\varpi^m[\widetilde{\Delta}_{\widetilde{M}_{\overline{S}}^{w_0}}])$.
\end{Prop}
Here we denoted by $\widetilde{\lambda}^{\overline{S}}$ the dominant weight for $\widetilde{G}$ that is trivial at $\overline{S}$ and coincides with $\widetilde{\lambda}$ outside $\overline{S}$. Similarly, $\widetilde{\lambda}_{\overline{S}}$ is the dominant weight for $\widetilde{G}$ that is trivial outside $\overline{S}$ and coincides with $\widetilde{\lambda}$ at $\overline{S}$.
Also, analogously to the previous subsection, we view $\mathcal{V}_{w_0^{\widetilde{Q}_{\overline{S}}^{w_0}}\widetilde{\lambda}_{\overline{S}}}/\varpi^m$ as an $\mathcal{O}/\varpi^m[\widetilde{\Delta}_{\widetilde{M}_{\overline{S}}^{w_0}}]$-module via inflation from $\widetilde{M}^{w_0,0}_{\overline{S}}$.

We can then also deduce the analogues of Corollary~\ref{Cor3.13} and Corollary~\ref{Cor3.15}. In particular, $R\Gamma(\widetilde{X}_{\widetilde{K}(0,c)},\mathcal{V}_{(\Tilde{\lambda},\underline{\tau})}^{\widetilde{Q}_{\overline{S}}^{w_0}}/\varpi^m)^{\widetilde{Q}_{\overline{S}}^{w_0}\textnormal{-ord}}$ admits a natural Hecke action at $p$. Namely, set
\begin{equation*}
    \Tilde{\sigma}(\Tilde{\lambda}_{\overline{S}},\underline{\tau}_{\overline{S}})^{\circ}:=\mathcal{V}_{w_0^{\widetilde{Q}_{\overline{S}}^{w_0}}\widetilde{\lambda}_{\overline{S}}}^{\vee}\otimes_{\mathcal{O}}\Tilde{\sigma}(\underline{\tau}_{\overline{S}})^{\circ}\in \textnormal{Mod}(\mathcal{O}[\widetilde{M}^{w_0,0}_{\overline{S}}]).
\end{equation*}
Then there is a natural algebra homomorphism
\begin{equation*}
    \mathcal{H}(\Tilde{\sigma}(\Tilde{\lambda}_{\overline{S}},\underline{\tau}_{\overline{S}})^{\circ,\vee})\otimes_{\mathcal{O}}\mathcal{O}/\varpi^m\to \textnormal{End}_{D^+(\mathcal{O}/\varpi^m)}(R\Gamma(\widetilde{X}_{\widetilde{K}(0,c)},\mathcal{V}_{(\Tilde{\lambda},\underline{\tau})}^{\widetilde{Q}_{\overline{S}}^{w_0}}/\varpi^m)^{\widetilde{Q}_{\overline{S}}^{w_0}\textnormal{-ord}}).
\end{equation*}
\begin{Rem}\label{Rem3.22}
    We note that if we set $\widetilde{S}:=\{v\mid \Bar{v}\in \overline{S}\}$, we have
$\mathcal{V}_{w_0^{\widetilde{Q}^{w_0}_{\overline{S}}}\Tilde{\lambda}_{\overline{S}}}=\mathcal{V}_{w_0^{Q_{\widetilde{S}}}\lambda_{\widetilde{S}}}\otimes_{\mathcal{O}}\mathcal{V}_{-w_0^{Q_{\widetilde{S}^c}}\lambda_{\widetilde{S}^c}},$ and, by the proof of \cite{CN23}, Lemma 2.2.14, we get
\begin{equation*}
    \Tilde{\sigma}(\Tilde{\lambda}_{\overline{S}},\underline{\tau}_{\overline{S}})^{\circ}=\sigma(\lambda_{\widetilde{S}},\underline{\tau}_{\widetilde{S}})^{\circ}\otimes_{\mathcal{O}}(\theta_n^{-1})^{\ast}\sigma(\lambda_{\widetilde{S}^c},\underline{\tau}_{\widetilde{S}^c})^{\circ}\in 
\end{equation*}
\begin{equation*}
     \textnormal{Mod}(\mathcal{O}[\left(\prod_{v\in \widetilde{S}}M_{v}(\mathcal{O}_{F_v})\right)\times \left(\prod_{v^c\in\widetilde{S}^c}(\theta_nM_{v^c})(\mathcal{O}_{F_v})\right)].
\end{equation*}
\end{Rem}

Finally, one can develop $\overline{\widetilde{Q}}_{\overline{S}}$-ordinary Hida theory for $\mathcal{V}_{(\Tilde{\lambda},\underline{\tau})}^{\widetilde{Q}_{\overline{S}},w_0^P,\vee}/\varpi^m$ (see the end of \S\ref{sec2.7} for the notation) just as in \S\ref{sec3.3}. In particular, one arrives to Hecke-equivariance of Poincar\'e duality at $p$ for ordinary cohomology of $\widetilde{X}_{\widetilde{K}}$.

\subsection{Ordinary parts of the Bruhat stratification}\label{sec3.5} Here we again closely follow \cite{CN23} on computing the ordinary parts of certain Bruhat strata of parabolic inductions in the derived category. As we often will not need serious changes in the proofs, we sometimes only state the results we need and indicate how to deduce them from the arguments of \cite{ACC23} and \cite{CN23}.

As in \cite{CN23}, 2.3.1, we restrict ourselves to a completely local setup. Let $L/\mathbf{Q}_p$ be a finite field extension with ring of integers $\mathcal{O}_L$ and a choice of uniformiser $\varpi_L$. Let $G/L$ be a split connected reductive group with a split maximal torus $T\subset G$ and Weyl group $W=W(G,T)$. Fix a Borel subgroup $T\subset B$ and two standard parabolic subgroups $B\subset Q_1,Q_2 \subset G$ with Levi decomposition $Q_1=M_1 \ltimes N_1$ and $Q_2=M_2\ltimes N_2$, respectively. We denote by $W_{Q_i}$ the Weyl group of $M_i$ and by $^{Q_1}W^{Q_2}\subset W$ the set of minimal length representatives of $W_{Q_1}\backslash W/W_{Q_2}$. For $w\in W$ we denote its length by $\ell(w)\in\mathbf{Z}_{\geq 0}$. Recall that $G(L)$ admits a stratification (with respect to its $p$-adic topology) called the Bruhat stratification\footnote{For a reference in this generality, see \cite{Hau18}, Lemma 2.1.2. However, note that \textit{loc. cit.} considers the opposite parabolic $\overline{Q}_1$ on the left instead, so the closure relations are reversed there.}
\begin{equation*}
    G(L)=\coprod_{w\in \prescript{Q_1}{}{W}^{Q_2}}Q_1(L)wQ_2(L)=\coprod_{w\in \prescript{Q_1}{}{W}^{Q_2}}S_w
\end{equation*}
with closure relations given by the Bruhat order
\begin{equation*}
    \overline{S_w}=\coprod_{w\geq w'\in \prescript{Q_1}{}{W}^{Q_2}}S_{w'}.
\end{equation*}
In particular, for $i\in \mathbf{Z}_{\geq 0}$, the subset
\begin{equation*}
    G_{\geq i}:=\coprod_{\ell(w)\geq i}S_w\subset G(L)
\end{equation*}
is open.

We first recall how the Bruhat stratification provides a "stratification" of the exact functor
\begin{equation*}
    \textnormal{Ind}_{Q_1(L)}^{G(L)}:D^+_{\textnormal{sm}}(\mathcal{O}/\varpi^m[Q_1(L)])\to D^+_{\textnormal{sm}}(\mathcal{O}/\varpi^m[G(L)])
\end{equation*}
in a suitable sense.  We define the functor
\begin{equation*}
    I_{\geq i}:\textnormal{Mod}_{\textnormal{sm}}(\mathcal{O}/\varpi^m[Q_1(L)]))\to \textnormal{Mod}_{\textnormal{sm}}(\mathcal{O}/\varpi^m[Q_2(L)])
\end{equation*}
which sends $\pi$ to the $Q_2(L)$-stable subspace of functions in $\textnormal{Ind}_{Q_1(L)}^{G(L)}\pi$ which are supported at $G_{\geq i}$.
For $w\in \prescript{Q_1}{}{W}^{Q_2}$, we can further define the functor
\begin{equation*}
    I_w:\textnormal{Mod}_{\textnormal{sm}}(\mathcal{O}/\varpi^m[Q_1(L)]))\to \textnormal{Mod}_{\textnormal{sm}}(\mathcal{O}/\varpi^m[Q_2(L)]))
\end{equation*}
which sends $\pi$ to the set of locally constant functions $f:S_w\to\pi$ which are compactly supported modulo $Q_1(L)$ and left invariant with respect to elements of $Q_1(L)$ equipped with the left $Q_2(L)$-action given by right multiplication on the source.
Finally, for $w\in \prescript{Q_1}{}{W}^{Q_2}$, consider also the open subspace $S_w^{\circ}:=Q_1(L)wM_2(L)N_2(\mathcal{O}_L)\subset S_w$. We then define the functor
$I_w^{\circ}:\textnormal{Mod}_{\textnormal{sm}}(\mathcal{O}/\varpi^m[Q_1(L)])\to \textnormal{Mod}_{\textnormal{sm}}(\mathcal{O}/\varpi^m[M_2(L)^+\ltimes N_2(\mathcal{O}_L)])$
by setting $I_w^{\circ}(\pi)\subset I_w(\pi)$ to be the subset of functions supported on $S^{\circ}_w$.
Each of these functors are exact as the argument of \cite{ACC23}, Proposition 5.3.1 shows.

For every $i\in \mathbf{Z}_{\geq 0}$ and  $\pi \in D^+_{\textnormal{sm}}(\mathcal{O}/\varpi^m[Q_1(L)])$, the natural inclusion of functors $I_{\geq i+1}\subset I_{\geq i}$ induces a distinguished triangle
\begin{equation}\label{eq3.7}
    I_{\geq i+1}(\pi)\to I_{\geq i}(\pi)\to \oplus_{\ell(w)=i}I_w(\pi)\to I_{\geq i+1}(\pi)[1]
\end{equation}
in $D^+_{\textnormal{sm}}(\mathcal{O}/\varpi^m[Q_2(L)])$.
For the proof of this, see \cite{Hau18}, Lemma 2.2.1.

Given a $\sigma\in \textnormal{Mod}_{\textnormal{sm}}(\mathcal{O}/\varpi^m[M_2(\mathcal{O}_L])$, finite free over $\mathcal{O}/\varpi^m$, we consider the functor
\begin{equation*}
    R\Hom_{\mathcal{O}/\varpi^m[M_2(\mathcal{O}_L)]}(\sigma,R\Gamma(N_2(\mathcal{O}_L),-)^{Q_2\textnormal{-ord}}):D^+_{\textnormal{sm}}(\mathcal{O}/\varpi^m[Q_2(L)])\to D^+(\mathcal{H}(\sigma^{\vee})).
\end{equation*}
We can apply this functor
to \ref{eq3.7} to get a distinguished triangle
\begin{equation*}
    R\Hom_{\mathcal{O}/\varpi^m[M_2(\mathcal{O}_L)]}(\sigma,R\Gamma(N_2(\mathcal{O}_L),I_{\geq i+1}(\pi))^{Q_2\textnormal{-ord}})\to
\end{equation*}
\begin{equation*}
 \to R\Hom_{\mathcal{O}/\varpi^m[M_2(\mathcal{O}_L)]}(\sigma,R\Gamma(N_2(\mathcal{O}_L),I_{\geq i}(\pi))^{Q_2\textnormal{-ord}})\to
\end{equation*}
\begin{equation*}
    \to  \oplus_{\ell(w)=i}R\Hom_{\mathcal{O}/\varpi^m[M_2(\mathcal{O}_L)]}(\sigma,R\Gamma(N_2(\mathcal{O}_L),I_{w}(\pi))^{Q_2\textnormal{-ord}})\to  
\end{equation*}
\begin{equation*}
    \to R\Hom_{\mathcal{O}/\varpi^m[M_2(\mathcal{O}_L)]}(\sigma,R\Gamma(N_2(\mathcal{O}_L),I_{\geq i+1}(\pi))^{Q_2\textnormal{-ord}})[1].
\end{equation*}
The argument of \cite{CN23}, Proposition 2.3.4 shows that taking long exact sequence of cohomology of the distinguished triangle above gives a $\mathcal{H}(\sigma^{\vee})$-equivariant short exact sequence
\begin{equation*}
    0\to R^j\Hom_{\mathcal{O}/\varpi^m[M_2(\mathcal{O}_L)]}(\sigma,R\Gamma(N_2(\mathcal{O}_L),I_{\geq i+1}(\pi))^{Q_2\textnormal{-ord}})\to 
\end{equation*}
\begin{equation*}
    \to R^j\Hom_{\mathcal{O}/\varpi^m[M_2(\mathcal{O}_L)]}(\sigma,R\Gamma(N_2(\mathcal{O}_L),I_{\geq i}(\pi))^{Q_2\textnormal{-ord}})\to
\end{equation*}
\begin{equation*}
    \to \oplus_{\ell(w)=i}R^j\Hom_{\mathcal{O}/\varpi^m[M_2(\mathcal{O}_L)]}(\sigma,R\Gamma(N_2(\mathcal{O}_L),I_{w}(\pi))^{Q_2\textnormal{-ord}})\to 0
\end{equation*}
for every $j\in \mathbf{Z}_{\geq 0}$. To be more precise, the proof of Proposition 2.3.4 in \textit{loc. cit.} relies on their Lemma 2.3.5. We state the obvious generalisation we need. However, we omit the proof as it can be proved just as the version in \textit{loc. cit.}
\begin{Lemma}
    For any $i\in \mathbf{Z}_{\geq 0}$, there are decompositions
    \begin{equation*}
        G_{\geq i}=U_1^m\coprod U_2^m
    \end{equation*}
    into open and closed subsets, indexed by $m\in \mathbf{Z}_{\geq 1}$, that are $Q_1(L)$-invariant on the left and $Q_2(\mathcal{O}_L)$-invariant on the right such that
    \begin{equation*}
        G_{\geq i+1}=\bigcup_{m\geq 1}U_1^m.
    \end{equation*}
\end{Lemma}

To apply this line of argument, we compute
\begin{equation*}
    R\Gamma(N_2(\mathcal{O}_L),I_w(\pi))^{Q_2\textnormal{-ord}}
\end{equation*}
for a class of $w\in {}^{Q_1}W^{Q_2}$. First we note the following lemma that is essentially \cite{ACC23}, Lemma 5.3.4 (see also \cite{CN23}, Lemma 2.3.6).
\begin{Lemma}\label{Lem3.23}
    Let $w\in {}^{Q_1}W^{Q_2}$ such that $wM_2(L)w^{-1}\subset M_1(L)$. Then:
    \begin{enumerate}
        \item $I_w^{\circ}$ takes injectives to $\Gamma(N_2(\mathcal{O}_L),-)$-acyclics.
        \item Let $\pi\in D^+_{\textnormal{sm}}(\mathcal{O}/\varpi^m)[Q_1(L)])$. Then there is a natural isomorphism
        \begin{equation*}
            R\Gamma(N_2(\mathcal{O}_L),I_w^{\circ}(\pi))^{Q_2\textnormal{-ord}}\xrightarrow[]{\sim}R\Gamma(N_2(\mathcal{O}_L),I_w(\pi))^{Q_2\textnormal{-ord}}.
        \end{equation*}
    \end{enumerate}
\end{Lemma}
\begin{proof}
    Note that the assumption on $w$ implies that
    \begin{equation*}
        S_w^{\circ}=Q_1(L)wN_2(\mathcal{O}_L).
    \end{equation*}
    Knowing this, the proof of \cite{ACC23}, Lemma 5.3.4 applies verbatim.
\end{proof}
If $w\in {}^{Q_1}W^{Q_2}$ such that $wM_2(L)w^{-1}\subset M_1(L)$, we define
\begin{equation*}
    N_{2,w}^{\circ}:=Q_1(L)\cap wN_2(\mathcal{O}_L)w^{-1},
\end{equation*}
a compact subgroup of $Q_1(L)$. We define the functor
\begin{equation*}
    \Gamma(N_{2,w}^{\circ},-):\textnormal{Mod}_{\textnormal{sm}}(\mathcal{O}/\varpi^m[Q_1(L)])\to \textnormal{Mod}_{\textnormal{sm}}(\mathcal{O}/\varpi^m[M_2(L)^+])
\end{equation*}
where an element $m\in M_2(L)^+$ acts on a $v\in \pi^{N_{2,w}^{\circ}}$ by the formula
\begin{equation*}
    m\cdot v:=\sum_{n\in N_{2,w}^{\circ}/wmw^{-1}N_{2,w}^{\circ}wm^{-1}w^{-1}}nwmw^{-1}\cdot v
\end{equation*}
where on the right we consider the natural action of $wmw^{-1}\in Q_1(L)$. One checks easily that our assumption on $w$ ensures that this formula makes sense meaning that we have $wmw^{-1}N_{2,w}^{\circ}wm^{-1}w^{-1}\subset N_{2,w}^{\circ}$.
\begin{Lemma}
    Let $w\in {}^{Q_1}W^{Q_2}$ such that $wM_2(L)w^{-1}\subset M_1(L)$ and consider $\pi\in D^+_{\textnormal{sm}}(\mathcal{O}/\varpi^m[Q_1(L)])$. Then we have a natural isomorphism
    \begin{equation*}
        R\Gamma(N_2(\mathcal{O}_L),I_w^{\circ}(\pi))\cong R\Gamma(N_{2,w}^{\circ},\pi)
    \end{equation*}
    in $D^{+}_{\textnormal{sm}}(\mathcal{O}/\varpi^m[M_2(L)^+])$.
\end{Lemma}
\begin{proof}
    Again, it can be proved by running the proof of the analogous statement \cite{ACC23}, Lemma 5.3.5, keeping in mind our assumption on $w$. For the reader's convenience, we recall the argument here.

    By the first part of Lemma~\ref{Lem3.23}, it suffices to give a natural isomorphism of underived functors
    \begin{equation*}
        \Gamma(N_2(\mathcal{O}_L),I_w^{\circ}(-))\cong \Gamma(N_{2,w}^{\circ},-).
    \end{equation*}
    For $\pi\in \textnormal{Mod}_{\textnormal{sm}}(Q_1(L))$, the map will send an $N_2(\mathcal{O}_L)$-invariant function
    \begin{equation*}
        f:Q_1(L)wN_2(\mathcal{O}_L)\to \pi
    \end{equation*}
    to $f(w)\in \pi^{N_{2,w}^{\circ}}$. This visibly gives an isomorphism of underlying $\mathcal{O}/\varpi^m$-modules with inverse sending $v\in \pi^{N_{2,w}^{\circ}}$ to $f_v:pwn\mapsto p\cdot v$.

    We are left with checking that this defines an $M_2(L)^+$-equivariant map. In other words, for $f\in \Gamma(N_2(\mathcal{O}_L),I_w^{\circ}(\pi))$ and $m\in M_2(L)^+$, we need to see that
    \begin{equation}\label{eq3.8}
        \sum_{n\in N_2(\mathcal{O}_L)/mN_2(\mathcal{O}_L)m^{-1}}f(wnm)=\sum_{\Tilde{n}\in N_{2,w}^{\circ}/wmw^{-1}N_{2,w}^{\circ}(wmw^{-1})^{-1}}\Tilde{n}wmw^{-1}f(w)
.    \end{equation}
Note that the association
\begin{equation}\label{eq3.9}
    N_{2,w}^{\circ}/wmw^{-1}N_{2,w}^{\circ}(wmw^{-1})^{-1}\to N_2(\mathcal{O}_L)/mN_2(\mathcal{O}_L)m^{-1},
\end{equation}
\begin{equation*}
    \Tilde{n}\mapsto w^{-1}\Tilde{n}w
\end{equation*}
is injective. Moreover, for $n=w^{-1}\Tilde{n}w$ lying in the image, we have $f(wnm)=\Tilde{n}wmw^{-1}f(w)$. In particular, to see that \ref{eq3.8} holds, it suffices to prove that $f(wnm)\neq 0$ only if $n$ lies in the image of \ref{eq3.9}. So assume that $wnm\in S^{\circ}_w=Q_1(L)wN_2(\mathcal{O}_L)$. Accordingly, we write it in the form $wnm=qwn'$ with $q\in Q_1(L)$ and $n'\in N_2(\mathcal{O}_L)$. On the other hand,
\begin{equation*}
    n=w^{-1}qwn'm^{-1}=(w^{-1}qwm^{-1})(mn'm^{-1})\in N_2(\mathcal{O}_L)
\end{equation*}
and our assumptions imply that $mn'm^{-1}\in N_2(\mathcal{O}_L)$. Therefore, we get
\begin{equation*}
    w^{-1}qwm^{-1}=w^{-1}(qwm^{-1}w^{-1})w\in w^{-1}Q_1(L)w\cap N_2(\mathcal{O}_L)=w^{-1}N_{2,w}^{\circ}w.
\end{equation*}
\end{proof}
\begin{Cor}
    Let $w\in {}^{Q_1}W^{Q_2}$ such that $wM_2(L)w^{-1}\subset M_1(L)$ and $\pi\in D^+_{\textnormal{sm}}(\mathcal{O}/\varpi^m[Q_1(L)])$. There is a natural isomorphism
    \begin{equation*}
        R\Gamma(N_2(\mathcal{O}_L),I_w(\pi))^{Q_2\textnormal{-ord}}\cong R\Gamma(N_{2,w}^{\circ},\pi)^{Q_2\textnormal{-ord}}
    \end{equation*}
    in $D^+_{\textnormal{sm}}(\mathcal{O}/\varpi^m[M_2(L)])$.
\end{Cor}

Finally, we compute $R\Gamma(N_{2,w}^{\circ},\pi)^{Q_2\textnormal{-ord}}$ as in \cite{ACC23}, Lemma 5.3.7. For the rest of the subsection, we assume that $G=\textnormal{GL}_n$ with $T=T_n$, the torus consisting of diagonal matrices and $B=B_n$ the Borel consisting of upper triangular matrices. We also introduce some further notation. Given $w\in {}^{Q_1}W^{Q_2}$ such that $wM_2(L)w^{-1}\subset M_1(L)$, set $Q_w:=wQ_2w^{-1}\cap M_1\subset M_1$ with Levi quotient $M_w=wM_2w^{-1}$ and unipotent radical $N_w=wN_2w^{-1}\cap M_1$.
Consider the character
\begin{equation*}
    \chi_w:M_2(L)\to \mathcal{O}^{\times},
\end{equation*}
\begin{equation*}
    m\mapsto \frac{\textnormal{Norm}_{L/\mathbf{Q}_p}\det_L(\textnormal{Ad}(m^w)\mid_{\textnormal{Lie}N_{w}(L)})^{-1}}{|\textnormal{Norm}_{L/\mathbf{Q}_p}\det_L(\textnormal{Ad}(m^w)\mid_{\textnormal{Lie}N_{w}(L)})|_p}
\end{equation*}
where we set $m^w:=wmw^{-1}$.
Introduce the equivalence of categories
\begin{equation*}
    \tau_w:\textnormal{Mod}_{\textnormal{sm}}(\mathcal{O}/\varpi^m[M_2(L)])\to \textnormal{Mod}_{\textnormal{sm}}(\mathcal{O}/\varpi^m[wM_2(L)w^{-1}])
\end{equation*}
sending $\pi$ to $\tau_w(\pi)$ with underlying $\mathcal{O}/\varpi^m$-module $\pi$ but with the twisted action $\tau_w(\pi)(m)=\pi(w^{-1}mw)$.
Finally, set $N_w^{\circ}=N_w(\mathcal{O}_L)$ and $N_{2,w,N_1}^{\circ}:=N_{2,w}^{\circ}\cap N_1(L)$.
\begin{Lemma}
    Assume that $G=\textnormal{GL}_n$ with $T=T_n$, $B=B_n$ and let $w\in {}^{Q_1}W^{Q_2}$ such that $wM_2(L)w^{-1}\subset M_1(L)$. Then, for any $\pi\in D^+_{\textnormal{sm}}(\mathcal{O}/\varpi^m[M_1(L)])$, there is a natural isomorphism between
    \begin{equation*}
        R\Gamma(N_{2,w}^{\circ},\textnormal{Inf}_{M_1(L)}^{Q_1(L)}\pi)^{Q_2\textnormal{-ord}}
    \end{equation*}
    and
    \begin{equation*}
\mathcal{O}/\varpi^m(\chi_w)\otimes_{\mathcal{O}/\varpi^m}\tau_w^{-1}R\Gamma(N_w^{\circ},\pi)^{Q_w\textnormal{-ord}}[-\textnormal{rk}_{\mathbf{Z}_p}N_{2,w,N_1}^{\circ}]
    \end{equation*}
    in $D^+_{\textnormal{sm}}(\mathcal{O}/\varpi^m[M_2(L)])$.
\end{Lemma}
\begin{proof}
    Before starting the proof, we note that the argument is just an obvious generalisation of the proof of \cite{ACC23}, Lemma 5.3.7 and so we kept the structure of their argument and sometimes will refer to it as \textit{loc. cit}.

    We set the monoid $M_2(L)^+\ltimes_w N_{2,w}^{\circ}$ to be $M_2(L)^+\times N_{2,w}^{\circ}$ with the action $(m,1)(1,n)=(1, m^wn(m^w)^{-1})(m,1)$.
    Consider the short exact sequence
    \begin{equation*}
        0\to N_{2,w,N_1}^{\circ}\to N_{2,w}^{\circ}\to N_w^{\circ}\to 0
    \end{equation*}
    and note that it is obviously equivariant for the $M_2(L)^+$-action on each groups via $m\mapsto wmw^{-1}$.
    Then just as in \textit{loc. cit.}, we basically write $R\Gamma(N_{2,w}^{\circ},\textnormal{Inf}_{M_1(L)}^{Q_1(L)}-)^{Q_2\textnormal{-ord}}$ as $"R\Gamma(N_{w}^{\circ},-)^{Q_w\textnormal{-ord}}\circ R\Gamma(N_{2,w,N_1}^{\circ},\textnormal{Inf}_{M_1(L)}^{Q_1(L)}-)^{Q_1\textnormal{-ord}}"$.
To make this precise, we need to introduce some functors. Denote by
\begin{equation*}
    \textnormal{Res}^w:\textnormal{Mod}_{\textnormal{sm}}(\mathcal{O}/\varpi^m[M_1(L)])\to \textnormal{Mod}_{\textnormal{sm}}(\mathcal{O}/\varpi^m[M_2(L)^+\ltimes_w N_{2,w}^{\circ}])
\end{equation*}
the composite of $\textnormal{Inf}_{M_1(L)}^{Q_1(L)}$ with the functor that sends $\pi\in \textnormal{Mod}_{\textnormal{sm}}(Q_1(L))$ to itself as an $\mathcal{O}/\varpi^m$-module with the action $\textnormal{Res}^w(\pi)(mn)=\pi(m^{w}n)$. Further set
\begin{equation*}
    \alpha: \textnormal{Mod}_{\textnormal{sm}}(\mathcal{O}/\varpi^m[M_2(L)^+\ltimes_wN_w^{\circ}])\to \textnormal{Mod}_{\textnormal{sm}}(\mathcal{O}/\varpi^m[M_w(L)^+\ltimes N_w^{\circ}])
\end{equation*}
to be the functor defined by composing the equivalence
\begin{equation*}
    \textnormal{Mod}_{\textnormal{sm}}(\mathcal{O}/\varpi^m[M_2(L)^+\ltimes_w N_w^{\circ}])\xrightarrow{\sim} \textnormal{Mod}_{\textnormal{sm}}[wM_2(L)^+w^{-1}\ltimes N_w^{\circ}])
\end{equation*}
with the localisation
\begin{equation*}
    \textnormal{Mod}_{\textnormal{sm}}(\mathcal{O}/\varpi^m[wM_2(L)^+w^{-1}\ltimes N_w^{\circ}])\to \textnormal{Mod}_{\textnormal{sm}}(\mathcal{O}/\varpi^m[M_w(L)^+\ltimes N_w^{\circ}])
\end{equation*}
induced by the inclusion $wM_2(L)^+w^{-1}\subset M_w(L)^+$. Analogous construction defines a functor $\beta$ such that they fit into a commutative diagram (up to natural equivalence)
$$\begin{tikzcd}
 \textnormal{Mod}_{\textnormal{sm}}(\mathcal{O}/\varpi^m[M_2(L)^+\ltimes_wN_w^{\circ}]) \arrow{r}{\alpha} \arrow{d}{\Gamma(N_w^{\circ},-)}
               & \textnormal{Mod}_{\textnormal{sm}}(\mathcal{O}/\varpi^m[M_w(L)^+\ltimes N_w^{\circ}]) \arrow{d}{\Gamma(N_w^{\circ},-)}\\
\textnormal{Mod}_{\textnormal{sm}}(\mathcal{O}/\varpi^m[M_2(L)^+])  \arrow{r}{\beta} &\textnormal{Mod}_{\textnormal{sm}}(\mathcal{O}/\varpi^m[M_w(L)^+])
\end{tikzcd}$$
where the vertical arrows are defined the usual way by considering the corresponding Hecke actions. A reasoning similar to the proof of Lemma~\ref{Lem3.5} shows that $\alpha$ takes injectives to $\Gamma(N_w^{\circ},-)$-acyclics. Therefore, by checking things on underived functors, we have
\begin{equation*}
R\Gamma(N_{2,w}^{\circ},\textnormal{Inf}_{M_1(L)}^{Q_1(L)}\pi)^{Q_2\textnormal{-ord}}\cong
\end{equation*}
\begin{equation*}
    \tau_w^{-1}(\beta R\Gamma(N_w^{\circ},R\Gamma(N_{2,w,N_1}^{\circ},\textnormal{Res}^w(\pi)))^{Q_w\textnormal{-ord}}\cong
\end{equation*}
\begin{equation*}
    \tau_w^{-1}(R\Gamma(N_w^{\circ},\alpha R\Gamma(N_{2,w,N_1}^{\circ},\textnormal{Res}^w\pi)))^{Q_w\textnormal{-ord}}.
\end{equation*}
Since $N_{2,w,N_1}^{\circ}$ acts trivially on $\textnormal{Res}^w\pi$, we get
\begin{equation*}
R\Gamma(N_{2,w,N_1}^{\circ},\textnormal{Res}^w\pi)\cong \textnormal{Res}^w\pi\otimes_{\mathcal{O}/\varpi^m}R\Gamma(N_{2,w,N_1}^{\circ},\mathcal{O}/\varpi^m)
\end{equation*}
in $D^+_{\textnormal{sm}}(\mathcal{O}/\varpi^m[M_2(L)^+\ltimes_w N_w^{\circ}])$.
In particular, to compute $\alpha R\Gamma(N_{2,w,N_1}^{\circ},\textnormal{Res}^w\pi)$, it suffices to compute $\alpha R\Gamma(N_{2,w,N_1}^{\circ},\mathcal{O}/\varpi^m)$.
\begin{claim}
We have natural isomorphisms
    \begin{equation*}
        \alpha R\Gamma(N_{2,w,N_1}^{\circ},\mathcal{O}/\varpi^m)\cong \tau_w\mathcal{O}/\varpi^m(\chi_w)[-\textnormal{rk}_{\mathbf{Z}_p}N_{2,w,N_1}^{\circ}]\cong
    \end{equation*}
    \begin{equation*}
        \alpha \mathcal{O}/\varpi^m(\chi_w)[-\textnormal{rk}_{\mathbf{Z}_p}N_{2,w,N_1}^{\circ}]
    \end{equation*}
    in $D^+_{\textnormal{sm}}(\mathcal{O}/\varpi^m[M_w(L)^+\ltimes N_w^{\circ}])$ where, by abuse of notation we consider $\chi_w$ as an $M_2(L)^+\ltimes_w N_w^{\circ}$-module by inflation.
\end{claim}
\begin{proof}[Proof of claim]
    Assume that $Q_1=P_{(n_1,...,n_t)}$. Then we can write
    \begin{equation*}
        N_1(\mathcal{O}_L)\cong \prod _{1\leq i<j \leq t}N_{i,j}^{\circ}
    \end{equation*}
    where $N_{i,j}^{\circ}\subset N_1(\mathcal{O}_L)$ corresponds to the entries lying in $[n_{i-1}+1,n_i]\times [n_{j-1}+1,n_j]$ with the convention that $n_0=0$. This induces an isomorphism
    \begin{equation*}
        N_{2,w,N_1}^{\circ}\cong \prod_{1\leq i<j\leq t}N_{2,w,ij}^{\circ}.
    \end{equation*}
    Set $r_{i,j}:=\textnormal{rk}_{\mathbf{Z}_p}N_{2,w,ij}^{\circ}$ and, for $1\leq k\leq t$,
    \begin{equation*}
        z_{p,k}:=\textnormal{diag}(p,...,p,1,...,1)
    \end{equation*}
    where the first $n_1+...+n_k$ entries in the diagonal are given by $p$ and the rest by $1$.
    Note that $z_{p,k}$ is in the centre of $M_1(L)$, so it lies in $M_w(L)^+$ and it is invertible there, in particular, its Hecke action on $\alpha R\Gamma(N_{2,w,N_1}^{\circ},\mathcal{O}/\varpi^m)$ is invertible.

    Moreover, K\"unneth formula gives
    \begin{equation*}
        R\Gamma(N_{2,w,N_1}^{\circ},\mathcal{O}/\varpi^m)\cong \bigotimes_{1\leq i<j\leq t}R\Gamma(N_{2,w,ij}^{\circ},\mathcal{O}/\varpi^m)
    \end{equation*}
    and the Hecke action of $z_{p,k}$ is given by multiplying by the scalar
    \begin{equation*}
        [N_{2,w,N_1}^{\circ}:z_{p,k}N_{2,w}^{\circ}z_{p,k}^{-1}]=p^{\sum_{i<k<j}r_{i,j}}
    \end{equation*}
    the tensor product of maps
    \begin{equation*}
        m_{i,j}^k:R\Gamma(N_{2,w,ij}^{\circ},\mathcal{O}/\varpi^m)\to R\Gamma(N_{2,w,ij}^{\circ},\mathcal{O}/\varpi^m)
    \end{equation*}
    induced by multiplication by $p$ on $N_{2,w,ij}^{\circ}$ if $i<k<j$ and by $\textnormal{id}:N_{2,w,ij}^{\circ}\to N_{2,w,ij}^{\circ}$ otherwise. In the first case this means that, for $0\leq d_{i,j}\leq r_{i,j}$, $H^{d_{i,j}}(m_{i,j}^k)$ is multiplication by $p^{-d_{i,j}}$ and it is multiplication by $1$ otherwise. In particular, $\alpha(\bigotimes_{i<j}H^{d_{i,j}}(N_{2,w,ij}^{\circ},\mathcal{O}/\varpi^m))\neq 0$ only when $d_{i,j}=r_{i,j}$ for each $1\leq i<j\leq t$.
    Note that $\sum_{i<j}r_{i,j}=\textnormal{rk}_{\mathbf{Z}_p}N_{2,w,N_1}^{\circ}$.
    Therefore, we have
    \begin{equation*}
        \alpha R\Gamma(N_{2,w,N_1}^{\circ},\mathcal{O}/\varpi^m)\cong \alpha H^{\textnormal{rk}_{\mathbf{Z}_p}N_{2,w,N_1}^{\circ}}(N_{2,w,N_1}^{\circ},\mathcal{O}/\varpi^m)[-\textnormal{rk}_{\mathbf{Z}_p}N_{2,w,N_1}^{\circ}].
    \end{equation*}
    Moreover, just as in the proof of \cite{Hau16}, Proposition 3.1.8, using \cite{Eme10b}, Proposition 3.5.6 and the description of the corestriction map on top degree cohomology (cf. \cite{Eme10b}, Lemma 3.5.10), we see that the latter is, as an $\mathcal{O}/\varpi^m[M_w(L)^+]$-module, given by $\alpha\mathcal{O}/\varpi^m(\chi_w)[-\textnormal{rk}_{\mathbf{Z}_p}N_{2,w,N_1}^{\circ}]$.
\end{proof}
Putting everything together, we get
\begin{equation*}
    R\Gamma(N_{2,w}^{\circ},\textnormal{Inf}_{M_1(L)}^{Q_1(L)}\pi)^{Q_2\textnormal{-ord}}\cong
\end{equation*}
\begin{equation*}
    \tau_w^{-1}R\Gamma(N_w^{\circ},\alpha(\mathcal{O}/\varpi^m(\chi_w)\otimes \textnormal{Res}^w\pi))^{Q_w\textnormal{-ord}}[-\textnormal{rk}_{\mathbf{Z}_p}N_{2,w,N_1}^{\circ}]\cong
\end{equation*}
\begin{equation*}
    \tau_w^{-1}(\tau_w(\mathcal{O}/\varpi^m(\chi_w)))\otimes_{\mathcal{O}/\varpi^m}\tau_w^{-1}R\Gamma(N_w^{\circ},\pi)^{Q_w\textnormal{-ord}}[-\textnormal{rk}_{\mathbf{Z}_p}N_{2,w,N_1}^{\circ}]\cong
\end{equation*}
\begin{equation*}
\mathcal{O}/\varpi^m(\chi_w)\otimes_{\mathcal{O}/\varpi^m}\tau_w^{-1}R\Gamma(N_w^{\circ},\pi)^{Q_w\textnormal{-ord}}[-\textnormal{rk}_{\mathbf{Z}_p}N_{2,w,N_1}^{\circ}].
\end{equation*}
\end{proof}
Combining the results of the subsection, we obtain the following.
\begin{Cor}\label{Cor3.28}
    Assume that $G=\textnormal{GL}_n$, $T=T_n$ and $B=B_n$. Let $w\in {}^{Q_1}W^{Q_2}$ such that $wM_2(L)w^{-1}\subset M_1(L)$. Then, for every $\pi \in D^+_{\textnormal{sm}}(\mathcal{O}/\varpi^m[M_1(L)])$, $\sigma\in \textnormal{Mod}_{\textnormal{sm}}(\mathcal{O}/\varpi^m[M_2(\mathcal{O}_L)])$ finite free as an $\mathcal{O}/\varpi^m$-module, and $j\in \mathbf{Z}_{\geq 0}$, the group
    \begin{equation*}
        R^j\Hom_{\mathcal{O}/\varpi^m[M_2(\mathcal{O}_L)]}(\sigma,R\Gamma(N_2(\mathcal{O}_L),\textnormal{Ind}_{Q_1(L)}^{G(L)}\pi)^{Q_2\textnormal{-ord}})
    \end{equation*}
    admits
    \begin{equation*}
        R^{j-\textnormal{rk}_{\mathbf{Z}_p}N_{2,w,N_1}^{\circ}}\Hom_{\mathcal{O}/\varpi^m[M_2(\mathcal{O}_L)]}(\sigma,\mathcal{O}/\varpi^m(\chi_w)\otimes_{\mathcal{O}/\varpi^m}\tau_w^{-1}R\Gamma(N_w^{\circ},\pi)^{Q_w\textnormal{-ord}})
    \end{equation*}
    as a $\mathcal{H}(\sigma)$-equivariant subquotient.
\end{Cor}

For the reader's convenience, we spell out the case of interest for our application in proving local-global compatibility. For this we introduce some notation that hopefully makes it easier to motivate how Corollary~\ref{Cor3.28} will be applied. In particular, consider $\widetilde{G}=\textnormal{GL}_{2n}$, and set $P=P_{(n,n)}\subset \widetilde{G}$ to be the Siegel parabolic with Levi decomposition $P=G\ltimes U$. Moreover, set $\widetilde{Q}\subset P$ to be any standard\footnote{Standard with respect to the Borel of upper triangular matrices.} parabolic subgroup with Levi decomposition $\widetilde{M}\ltimes \widetilde{N}$. Set $\widetilde{Q}\cap G=Q_c\times Q\subset G=\textnormal{GL}_n\times \textnormal{GL}_n$ and denote their Levi decompositions by $M_c\ltimes N_c$, and $M\ltimes N$, respectively. Note that $\widetilde{M}=M_c\times M$. Denote by $\widetilde{Q}^{w_0}\subset P \subset \widetilde{G}$ the standard parabolic subgroup with Levi decomposition $\widetilde{M}^{w_0}\ltimes \widetilde{N}^{w_0}$ where $\widetilde{M}^{w_0}=M\times M_c\subset \textnormal{GL}_n\times \textnormal{GL}_n$. Pick $\widetilde{\sigma}=\sigma\otimes \sigma_c\in \textnormal{Mod}_{\textnormal{sm}}(\mathcal{O}/\varpi^m[\widetilde{M}^{w_0}(\mathcal{O}_L])=\textnormal{Mod}_{\textnormal{sm}}(\mathcal{O}/\varpi^m[M(\mathcal{O}_L)\times M_c(\mathcal{O}_L)])$, finite free as an $\mathcal{O}/\varpi^m$-module. Finally, note that $w_0^P\in {}^PW^{\widetilde{Q}^{w_0}}$. Then a direct application of Corollary~\ref{Cor3.28} with $Q_1=P$, $Q_2=\widetilde{Q}^{w_0}$, and $w=w_0^P$ gives.
\begin{Cor}\label{Cor3.29}
    For every $\pi\in D^+_{\textnormal{sm}}(\mathcal{O}/\varpi^m[G(L)])$, and $j\in \mathbf{Z}_{\geq 0}$, the group
    \begin{equation*}
        R^j\Hom_{\mathcal{O}/\varpi^m[\widetilde{M}^{w_0}(\mathcal{O}_L)]}(\widetilde{\sigma}, R\Gamma(\widetilde{N}^{w_0}(\mathcal{O}_L),\textnormal{Ind}_{P(L)}^{\widetilde{G}(L)}\pi)^{\widetilde{Q}^{w_0}\textnormal{-ord}})
    \end{equation*}
    admits
    \begin{equation*}
        R^j\Hom_{\mathcal{O}/\varpi^m[M(\mathcal{O}_L)\times M_c(\mathcal{O}_L)]}(\sigma\otimes \sigma_c, R\Gamma(N(\mathcal{O}_L)\times N_c(\mathcal{O}_L),\tau_{w_0^P}^{-1}\pi)^{Q\times Q_c\textnormal{-ord}})
    \end{equation*}
    as a $\mathcal{H}(\sigma)\times \mathcal{H}(\sigma_c)$-equivariant subquotient.
\end{Cor}
We also deduce a dual statement computing $\overline{\widetilde{Q}}$-ordinary parts of a Bruhat stratum of $\textnormal{Ind}_{P(L)}^{\widetilde{G}(L)}\pi$. 
\begin{Cor}\label{Cor3.30}
    For every $\pi\in D^+_{\textnormal{sm}}(\mathcal{O}/\varpi^m[G(L)])$, and integer $j\in\mathbf{Z}_{\geq 0}$, the group
    \begin{equation*}
        R^j\Hom_{\mathcal{O}/\varpi^m[\widetilde{M}(\mathcal{O}_L)]}(\tau_{w_0^P}\widetilde{\sigma}, R\Gamma(\overline{\widetilde{N}}^1,\textnormal{Ind}_{P(L)}^{\widetilde{G}(L)}\pi)^{\overline{\widetilde{Q}}\textnormal{-ord}})
    \end{equation*}
    admits
    \begin{equation*}
        R^j\Hom_{\mathcal{O}/\varpi^m[M_c(\mathcal{O}_L)\times M(\mathcal{O}_L)]}(\sigma_c\otimes \sigma, R\Gamma(\overline{N}^1_c\times \overline{N}^1,\pi)^{\overline{Q_c}\times \overline{Q}\textnormal{-ord}})
    \end{equation*}
    as a $\mathcal{H}(\sigma_c)\times \mathcal{H}(\sigma)$-equivariant subquotient.
\end{Cor}
\begin{proof}
    We reduce it to Corollary~\ref{Cor3.28}. Set $\Tilde{z}_p:=u_p^{\widetilde{Q}}\in Z_{\widetilde{M}}^+$ to be the usual element defining the $U_p$-operator with respect to $\widetilde{Q}$ and $\widetilde{w}_0:=w_0^{\widetilde{Q}}$. Set $\widetilde{Q}^{\Tilde{w}_0}=\widetilde{M}^{\Tilde{w}_0}\ltimes \widetilde{N}^{\Tilde{w}_0}$ to be the standard parabolic subgroup with Levi subgroup $\Tilde{w}_0^{-1}\widetilde{M}\Tilde{w}_0$. Note that $\overline{N}^{\circ}:=\Tilde{z}_p\widetilde{w}_0\widetilde{N}^{\Tilde{w}_0}(\mathcal{O}_L)(\Tilde{z}_p\widetilde{w}_0)^{-1}\subset \overline{\widetilde{N}}^1$ and, by Remark~\ref{Rem3.17}, we have a natural ($\mathcal{H}(\tau_{w_0^P}\widetilde{\sigma})$-equivariant) isomorphism
    \begin{equation*}
        R^j\Hom_{\mathcal{O}/\varpi^m[\widetilde{M}(\mathcal{O}_L)]}\left(\tau_{w_0^P}\widetilde{\sigma},R\Gamma(\overline{\widetilde{N}}^1,\textnormal{Ind}_{P(L)}^{\widetilde{G}(L)}\pi)^{\overline{\widetilde{Q}}\textnormal{-ord}}\right)\cong 
    \end{equation*}
        
    \begin{equation*}
         R^j\Hom_{\mathcal{O}/\varpi^m[\widetilde{M}(\mathcal{O}_L)]}\left(\tau_{w_0^P}\widetilde{\sigma},R\Gamma(\overline{N}^{\circ},\textnormal{Ind}_{P(L)}^{\widetilde{G}(L)}\pi)^{\overline{\widetilde{Q}}\textnormal{-ord}}\right).
    \end{equation*}
    Moreover, multiplication by $\Tilde{z}_p\widetilde{w}_0$ sets up a $\mathcal{H}(\tau_{w_0}^P\widetilde{\sigma})$-equivariant isomorphism between the latter and
    \begin{equation}\label{eq3.10}
        R^j\Hom_{\mathcal{O}/\varpi^m[\widetilde{M}^{\Tilde{w}_0}(\mathcal{O}_L)]}\left(\tau_{w_0^{Q\times Q_c}}\widetilde{\sigma},R\Gamma(\widetilde{N}^{\Tilde{w}_0}(\mathcal{O}_L), \textnormal{Ind}_{P(L)}^{\widetilde{G}(L)}\pi)^{\widetilde{Q}^{\Tilde{w}_0}\textnormal{-ord}}\right)
    \end{equation}
    where the Hecke action on \ref{eq3.10} is through the isomorphism $\mathcal{H}(\tau_{w_0^P}\widetilde{\sigma})\cong\mathcal{H}(\tau_{w_0^{Q\times Q_c}}\widetilde{\sigma})$ induced by $\Tilde{w}_0$. Set $N^{\Tilde{w}_0}\times N^{\Tilde{w}_0}_c:= \widetilde{N}^{\Tilde{w}_0}\cap G$ and note that $N^{\Tilde{w}_0}$, respectively $N^{\Tilde{w}_0}_c$ is the unipotent radical corresponding to $M^{\Tilde{w}_0}:=w_0^{Q}Mw_0^{Q,-1}$, respectively $M^{\Tilde{w}_0}_c:=w_0^{Q_c}M_cw_0^{Q_c,-1}$. Then Corollary~\ref{Cor3.28} shows that \ref{eq3.10} admits
    \begin{equation}\label{eq3.11}
        R^j\Hom_{\mathcal{O}/\varpi^m[M^{\Tilde{w}_0,0}_c\times M^{\Tilde{w}_0,0}]}\left(\tau_{w_0^{Q_c}}\sigma_c\times \tau_{w_0^{Q}}\sigma,R\Gamma(N^{\Tilde{w}_0}_c(\mathcal{O}_L)\times N^{\Tilde{w}_0}(\mathcal{O}_L), \pi)^{Q^{\Tilde{w}_0}_c\times Q^{\Tilde{w}_0}\textnormal{-ord}}\right)
    \end{equation}
    as a $\mathcal{H}(\tau_{w_0^{Q\times Q_c}}\widetilde{\sigma})$-equivariant subquotient. Note that $\widetilde{w}_0=(w_0^{Q_c},w_0^Q)w_0^P=w_0^P(w_0^Q,w_0^{Q_c})$. Therefore, multiplication by $(w_0^{Q_c}u_p^{Q_c},w_0^{Q}u_p^{Q})$ and another application of Remark~\ref{Rem3.17} sets up a $\mathcal{H}(\widetilde{\sigma})\cong \mathcal{H}(\tau_{\Tilde{w}_0}\widetilde{\sigma})$-equivariant isomorphism between \ref{eq3.11} and
    \begin{equation*}
    R^j\Hom_{\mathcal{O}/\varpi^m[M_c(\mathcal{O}_L)\times M(\mathcal{O}_L)]}\left(\tau_{w_0^P}\widetilde{\sigma}, R\Gamma(\overline{N}^1_c\times \overline{N}^1,\pi)^{\overline{Q}_c\times \overline{Q}\textnormal{-ord}}\right)\cong
    \end{equation*}
    \begin{equation*}
     R^j\Hom_{\mathcal{O}/\varpi^m[M_c(\mathcal{O}_L)\times M(\mathcal{O}_L)]}(\sigma_c\otimes \sigma, R\Gamma(\overline{N}^1_c\times \overline{N}^1,\pi)^{\overline{Q_c}\times \overline{Q}\textnormal{-ord}}).
    \end{equation*}
\end{proof}

\section{$Q$-ordinary parts in characteristic $0$}
In this section, we discuss the notion of taking ordinary parts of smooth admissible representations of a $p$-adic reductive group that arise as local components of cohomological automorphic representations (see \cite{Ger18}, \S5.1 for instance). In fact, we take a slightly more involved approach and define ordinary parts of locally algebraic representations in terms of Emerton's slope $0$ part. It will allow us to compare it to the corresponding Jacquet module and to our previous notion of ordinary parts. The latter will be useful in the endgame of proving our local-global compatibility results.  We will then prove our main characteristic $0$ result regarding $Q$-ordinary parts of $Q$-ordinary locally algebraic representations of $\textnormal{GL}_n$. Finally, we close the section with deducing $Q$-ordinary local-global compatibility for regular algebraic cuspidal automorphic representations of $\textnormal{GL}_n$ in the conjugate self-dual case.
\subsection{Ordinary parts of locally algebraic representations}\label{sec4.1} We briefly summarise the notion of ordinary parts for locally algebraic representations as introduced in \cite{BD20}, \S4.3 (see also \cite{Eme11}, \S5.6) and how it compares to Emerton's ordinary part functor from \cite{Eme10} when the representation admits an invariant lattice. We revisit the setup of \S\ref{sec3.1} and without further notice will use the introduced notation. However, we further assume that $G$ is split over $L$ and fix a choice of maximal torus $T$ and a Borel subgroup $B$ containing it. Finally, $Q$ is now assumed to be standard with respect to $B$. Recall that our coefficient field is $E/\mathbf{Q}_p$, a finite extension, large enough, so that $[L:\mathbf{Q}_p]=\Hom(L,E)$. Then, given a $(\textnormal{Res}_{L/\mathbf{Q}_p}B)_E$-dominant weight $\lambda\in \Hom((\textnormal{Res}_{L/\mathbf{Q}_p}T)_E,\mathbf{G}_{m,E})\cong \oplus_{\iota:L\hookrightarrow E}\Hom(T_E,\mathbf{G}_{m,E})$, set $V_{\lambda}=\otimes_{\iota}V_{\lambda_{\iota}}$ to be the corresponding absolutely irreducible algebraic $E$-representation of $G(L)$. Note that the dual $V_{\lambda}^{\vee}$ is isomorphic to $V_{\lambda^{\vee}}$ where $\lambda^{\vee}:=-w_0^G\lambda$. 
\begin{Def}
    An $E$-representation $\Pi$ of $G(L)$ is called locally algebraic of weight $\lambda$ if it is locally $V_{\lambda^{\vee}}$-algebraic in the sense of \cite{Eme17}, Definition 4.2.1 such that the smooth vectors $\Hom(V_{\lambda^{\vee}},\Pi)_{\textnormal{sm}}=\varinjlim_{K\to 1}\Hom_K(V_{\lambda^{\vee}},\Pi)$ form an admissible smooth representation of $G(L)$.
\end{Def}
We note that our definition admits a more intrinsic formulation (cf. \cite{Eme17} Definition 6.3.9, Proposition 6.3.10).
\begin{Rem}
    Note that any locally algebraic $E$-representation $\Pi$ of weight $\lambda$ is of the form $\pi\otimes V_{\lambda^{\vee}}$ with $\pi$ a smooth admissible $E$-representation (cf. \cite{Eme17}, Proposition 4.2.4). The functor $\Hom(V_{\lambda^{\vee}},-)_{\textnormal{sm}}$ sets up a natural equivalence between the category of locally algebraic $E$-representations of $G(L)$ of weight $\lambda$ and the category of smooth admissible $E$-representations of $G(L)$. However, the reason one might want to appeal to the notion of locally algebraic representations is that it yields a different notion of $p$-adic integrality. Indeed, this is our motivation to work with locally algebraic representations.
\end{Rem}
\begin{Rem}\label{Rem4.3}
    Note that, $V_{\lambda}$ being an algebraic $E$-representation of $\textnormal{Res}_{L/\mathbf{Q}_p}G$,
\begin{equation*}
    V_{\lambda}^{N^0}\cong V_{\lambda}^{N(L)}
\end{equation*} as subspaces of $V_{\lambda}$. Moreover, the latter is the absolutely irreducible representation of $M(L)$ associated with $\lambda$ viewed as a $(\textnormal{Res}_{L/\mathbf{Q}_p}B\cap M)_E$-dominant weight for $(\textnormal{Res}_{L/\mathbf{Q}_p}M)_E$ (see \cite{Cab84}). In particular, the former is naturally an $M(L)$-representation. As a corollary, one sees that for a locally algebraic $E$-representation $\Pi=\pi\otimes V_{\lambda^{\vee}}$ of weight $\lambda$, we have an induced identification $\Pi^{N^0}\cong \pi^{N^0}\otimes V_{\lambda^{\vee}}^{N(L)}$. Moreover, under this isomorphism, the Hecke action of $M^+$ on $\Pi^{N^0}$ coincides with the $M^+$-action on $\pi^{N^0}\otimes V_{\lambda^{\vee}}^{N(L)}$ given by the usual Hecke action on the first factor and the natural action on the algebraic part (see the proof of \cite{Eme17}, Proposition 4.3.6).
\end{Rem}
We now introduce the notion of finite slope and slope $0$ parts of $\Pi^{N^0}$. For $b\geq 1$, consider the \textit{finite dimensional} $E$-vector space
\begin{equation*}
    \Pi_b:=\pi^{\mathcal{Q}(b,b)}\otimes V_{\lambda^{\vee}}^{N(L)}\subset \Pi^{N^0}.
\end{equation*}
By Hypothesis~\ref{hyp3.2}, it is a $Z_M^+$-invariant subspace. Denote by $B_b$ the $E$-subalgebra of $\textnormal{End}_E(\Pi_b)$ generated by $Z_M^+$. This is an Artinian $E$-algebra and as such, it decomposes into a product of local Artinian $E$-algebras indexed by its maximal ideals. We then say that a maximal ideal $\mathfrak{m}\subset B_b$ is of \textit{finite slope} if the image of $Z_M^+$ in $B_b$ is disjoint from $\mathfrak{m}$. For such an $\mathfrak{m}$, the composition
\begin{equation*}
    Z_M^+\to B_b\twoheadrightarrow B_b/\mathfrak{m}\hookrightarrow \overline{\mathbf{Q}}_p
\end{equation*}
lands in $\overline{\mathbf{Q}}_p^{\times}$. We further say that a finite slope maximal ideal $\mathfrak{m}\subset B_b$ is of \textit{slope zero} if the composition lands in $\overline{\mathbf{Z}}_p^{\times}$. Considering the corresponding factors of $B_b$ induces $Z_M^+$-equivariant decompositions
\begin{equation}\label{eq4.1}
    \Pi_b\cong (\Pi_b)_{\textnormal{fs}}\oplus (\Pi_b)_{\textnormal{null}}\cong (\Pi_b)_0\oplus (\Pi_b)_{>0}
\end{equation}
into finite slope and slope $0$ parts and their complements. Note that $(\Pi_b)_0\subset (\Pi_b)_{\textnormal{fs}}$ and that the $Z_M^{+}$-action uniquely extends to a $Z_M(L)$-action on both by definition. The decompositions in \ref{eq4.1} are easily checked to be compatible when we vary $b$. Therefore, by passing to the colimit over $b\geq 1$, we get $Z_M^+$-equivariant decompositions
\begin{equation*}
    \Pi^{N^0}\cong (\Pi^{N^0})_{\textnormal{fs}}\oplus (\Pi^{N^0})_{\textnormal{null}}\cong (\Pi^{N^0})_0\oplus (\Pi^{N^0})_{>0}.
\end{equation*}
Moreover, in the colimit, the decompositions are preserved by the $M^+$-action. In particular, Lemma~\ref{Lem3.3} and Remark~\ref{Rem4.3} shows that $(\Pi^{N^0})_{\textnormal{fs}}$ and $(\Pi^{N^0})_0$ become locally algebraic $E$-representations of $M(L)$ of weight $w_0^Mw_0^G\lambda=w_0^Q\lambda$.

First, we compare the finite slope part with the classical (unnormalised) Jacquet functor. This is essentially the theory of canonical liftings that goes back to Casselman (cf. \cite{Cas95}).
\begin{Prop}\label{Prop4.4}
    We have a natural $M(L)$-equivariant isomorphism
    \begin{equation*}
        (\Pi^{N^0})_{\textnormal{fs}}\otimes \delta_Q\cong J_Q(\pi)\otimes V_{\lambda^{\vee}}^{N(L)}
    \end{equation*}
    induced by the natural $M^+$-equivariant (surjective) map
    \begin{equation*}
        \Pi^{N^0}\otimes \delta_Q\to J_Q(\pi)\otimes V_{\lambda^{\vee}}^{N(L)}.
    \end{equation*}
\end{Prop}
\begin{proof}
    This can be deduced from \cite{Eme06c}, \S4.3. We sketch an argument here. It is an easy consequence of the definitions that $(\Pi^{N^0})_{\textnormal{fs}}=(\pi^{N^0})_{\textnormal{fs}}\otimes V_{\lambda^{\vee}}^{N(L)}$ and $(\Pi^{N^0})_{\textnormal{null}}=(\pi^{N^0})_{\textnormal{null}}\otimes V_{\lambda^{\vee}}^{N(L)}$. Moreover, one sees that $(\pi^{\mathcal{Q}(b,b)})_{\textnormal{null}}\subset \pi^{\mathcal{Q}(b,b)}$ is the subspace of vectors on which, for some choice of $z_p\in Z_M^+$\footnote{Hence for any choice of such $z_p$.} as in Lemma~\ref{Lem3.3}, the action of $z_p$ is nilpotent. Using this description of $(\pi^{N^0})_{\textnormal{null}}$, one sees that $(\pi^{N^0})_{\textnormal{null}}$ is exactly the kernel of the natural $\pi^{N^0}\to J_Q(\pi)$. Moreover, \cite{Eme06c}, Proposition 4.3.4 i) shows that $\pi^{N^0}\otimes \delta_Q\to J_Q(\pi)$ is an $M^+$-equivariant surjection (see also \cite{Cas95}, Theorem 3.3.3 and Lemma 4.1.1). Note that in \cite{Eme06c} we don't see the appearance of $\delta_Q$. This is due to the fact that they work with a different normalisation of the Hecke action on $\pi^{N^0}$ (see \textit{loc. cit.} Definition 3.4.1). Combining these observations, we get the proposition.
\end{proof}
\begin{Rem}
    To avoid confusion, we point out that what is denoted by $J_Q$ in \cite{Eme06c} is Emerton's locally analytic Jacquet functor. It can be applied to a certain class of locally analytic representations $\Pi$ of $G(L)$ and defined by the analogous formula $(\Pi^{N^0})_{\textnormal{fs}}$ for a suitable notion of finite slope parts in this setup. Without elaborating on it any further, we just note that it is easily extracted from \cite{Eme06c}, \S4.3 that the constructions of \cite{BD20}, \S4.3 are all compatible with the ones of Emerton. In other words, for $\Pi$ a locally algebraic $E$-representation of $G(L)$ we have $J_Q(\Pi)\cong (\Pi^{N^0})_{\textnormal{fs}}\otimes \delta_Q$ where the former is in the sense of Emerton and the latter is in the sense of Breuil--Ding. Again, the character $\delta_Q$ appears because of the different normalisations of the Hecke action. From now on, we freely use the notation $J_Q(\Pi)$ for $\Pi$ a locally algebraic representation to denote its Jacquet module in the sense of Emerton. In light of Proposition~\ref{Prop4.4}, it recovers the classical Jacquet functor.
\end{Rem}
We now turn to discussing $(\Pi^{N^0})_0$. We would like to compare it with Emerton's ordinary parts for smooth admissible $\mathcal{O}/\varpi^m[G(L)]$-modules (cf. \cite{Eme10}). In order to have a chance to compare the two notions, we assume that our $\Pi$ is also a unitary representation of $G(L)$. By this we mean that there is an $\mathcal{O}$-lattice $\Pi^{\circ}\subset \Pi$\footnote{In other words, an $\mathcal{O}$-submodule that spans $\Pi$ over $E$ and contains no $E$-line.} that is invariant under the action of $G(L)$. The given lattice induces $\mathcal{O}$-lattices $\Pi_b^{\circ}:=\Pi^{\circ}\cap \Pi_b\subset \Pi_b$ which are necessarily finite and free as $\mathcal{O}$-modules. Then, for an integer $b\geq 1$, set $A_b$ to be the $\mathcal{O}$-subalgebra of $\textnormal{End}_{\mathcal{O}}(\Pi_b^{\circ})$ generated by $Z_M^+$. This is a finite $\mathcal{O}$-algebra and, in particular, we have $A_b\cong \prod_{\mathfrak{n}}(A_b)_{\mathfrak{n}}$, where the propduct runs over maximal ideals in $A_b$. Call a maximal ideal $\mathfrak{n}\subset A_b$ \textit{ordinary} if the image of $Z_M^+$ in $A_b$ is disjoint from $\mathfrak{n}$. This gives a $Z_M^+$-equivariant decomposition
\begin{equation*}
    \Pi_b^{\circ}\cong (\Pi_b^{\circ})_{\textnormal{ord}}\oplus (\Pi_b^{\circ})_{\textnormal{nonord}}
\end{equation*}
and, just as before, the $Z_M^+$-action on $(\Pi_b^{\circ})_{\textnormal{ord}}$ extends uniquely to a $Z_M(L)$-action. By passing to the colimit over $b\geq 1$, we obtain an $M^+$-equivariant decomposition
\begin{equation*}
    \Pi^{\circ,N^0}\cong \textnormal{Ord}_Q^{\textnormal{lalg}}(\Pi^{\circ})\oplus \textnormal{NOrd}_Q^{\textnormal{lalg}}(\Pi^{\circ}).
\end{equation*}
Therefore, $\textnormal{Ord}_Q^{\textnormal{lalg}}(\Pi^{\circ})$ is naturally a representation of $M(L)$. Then \cite{BD20}, Lemma 4.10 shows that we have a natural isomorphism
\begin{equation*}
    \textnormal{Ord}_Q^{\textnormal{lalg}}(\Pi^{\circ})\otimes_{\mathcal{O}}E\cong (\Pi^{N^0})_0
\end{equation*}
of locally algebraic $E$-representations of $M(L)$. In particular, the former is independent of the choice of lattice $\Pi^{\circ}\subset \Pi$ and $N^0$ and the latter is unitary. Motivated by this, introduce the following notation.
\begin{Def}
    Let $\Pi$ be a locally algebraic $E$-representation of $G(L)$ of weight $\lambda$. We then define its $Q$-ordinary part
    \begin{equation*}
        \textnormal{Ord}_Q^{\textnormal{lalg}}(\Pi):=(\Pi^{N^0})_0,
    \end{equation*}
    a locally algebraic $E$ representation of $M(L)$ of weight $w_0^Q\lambda$.
\end{Def}
In the rest of the subsection, we will justify that when $\Pi$ is assumed to be unitary our locally algebraic $Q$-ordinary parts does behave like an ordinary part. In the next subsection we will see that in fact under some a priori milder assumption it still behaves like an ordinary part functor.

From now on, for the rest of the subsection, we assume that $\Pi$ is unitary and choose an $\mathcal{O}$-lattice $\Pi^{\circ}$. For an integer $m\geq 1$, we can also introduce analogous constructions for $\Pi^{\circ}_b/\varpi^m$ to get a $Z_M^+$-equivariant decomposition $\Pi_b^{\circ}/\varpi^m\cong (\Pi_b^{\circ}/\varpi^m)_{\textnormal{ord}}\oplus(\Pi_b^{\circ}/\varpi^m)_{\textnormal{nonord}} $. The natural map then induces an isomorphism (cf. \cite{BD20}, 4.16)
\begin{equation}\label{eq4.2}
    (\Pi_b^{\circ})_{\textnormal{ord}}/\varpi^m\xrightarrow{\sim}(\Pi_b^{\circ}/\varpi^m)_{\textnormal{ord}}.
\end{equation}
On the other hand, $\Pi^{\circ}/\varpi^m$ is a smooth admissible $\mathcal{O}/\varpi^m[G(L)]$-module so one can take its $Q$-ordinary part in the sense of Emerton. Then \cite{BD20}, Lemma 4.14 shows that we have a natural injection
\begin{equation}\label{eq4.3}
    \varinjlim_{b\geq 1}(\Pi_b^{\circ}/\varpi^m)_{\textnormal{ord}}\hookrightarrow \textnormal{Ord}_Q(\Pi^{\circ}/\varpi^m)
\end{equation}
where the latter is Emerton's $Q$-ordinary part (cf. \cite{Eme10}, Definition 3.1.7).\footnote{We note here that since we are in an admissible situation, Emerton's ordinary part simplifies to localising $(\Pi^{\circ}/\varpi^m)^{N^0}$ along $\mathcal{O}[z_p]\hookrightarrow \mathcal{O}[z_p^{\pm 1}]$ for any $z_p$ as in Lemma~\ref{Lem3.3} (see \cite{Eme10b}, Lemma 3.2.1).}
However, due to the presence of group cohomology, \ref{eq4.3} is not necessarily surjective. Combining \ref{eq4.2} and \ref{eq4.3}, one gets (cf. \cite{BD20}, Lemma 4.16) a natural $M(L)$-equivariant injection
\begin{equation}\label{eq4.4}
    \textnormal{Ord}_Q^{\textnormal{lalg}}(\Pi^{\circ})\hookrightarrow \textnormal{Ord}_Q(\Pi^{\circ}):=\varprojlim_m\textnormal{Ord}_Q(\Pi^{\circ}/\varpi^m).
\end{equation}
\begin{Rem}\label{Rem4.7} The map \ref{eq4.4} is clearly not surjective in general: The source is a locally algebraic representation, the target is a lattice in an $E$-Banach space representation. A more sensible question to ask is whether it has dense image or not.  One notes that it does have dense image as long as \ref{eq4.3} is surjective, see \cite{BD20}, Remark 4.15 and Remark 4.17.
\end{Rem}
Given a unitary admissible $E$-Banach space representation $\Pi$ of $G(L)$, with a unit ball $\Pi^{\circ}$, then $\varprojlim_m\textnormal{Ord}_Q(\Pi^{\circ}/\varpi^m)$ is what Emerton calls $\textnormal{Ord}_Q(\Pi^{\circ})$ in \cite{Eme10}. Moreover, one can then set $\textnormal{Ord}_Q(\Pi):=\textnormal{Ord}_Q(\Pi^{\circ})\otimes_{\mathcal{O}}E$.\footnote{One checks easily that it is independent of the choice of unit ball.} We borrow his notation. We conclude the subsection with discussing a corollary of the adjunction property of \cite{BD20}, \S4.4. Before stating the corollary, we introduce the following notation. Let $\Pi$ be any $E$-representation of $G(L)$ and $W$ be an algebraic $E$-representation of $G(L)$. We then denote by $\Pi^{W\textnormal{-lalg}}\subset \Pi$ the $G(L)$-invariant subspace of locally $W$-algebraic vectors in $\Pi$ in the sense of \cite{Eme17}, Proposition-Definition 4.2.6.
\begin{Prop}\label{Prop4.8}
    Let $\Pi$ be a unitary admissible\footnote{For the notion of admissible $E$-Banach space representations, see \cite{ST06}, and \cite{Eme17}, \S6.} $E$-Banach space representation of $G(L)$. Then, for any choice of $(\textnormal{Res}_{L/\mathbf{Q}_p}B)_E$-dominant weight $\lambda$, we have a natural $M(L)$-equivariant identification
    \begin{equation*}
        \textnormal{Ord}_Q(\Pi)^{V_{\lambda^{\vee}}^{N(L)}\textnormal{-lalg}}\cong \textnormal{Ord}_Q^{\textnormal{lalg}}(\Pi^{V_{\lambda^{\vee}}\textnormal{-lalg}}).
    \end{equation*}
\end{Prop}
\begin{proof} Before proving the proposition, we first remark that the statement indeed makes sense. In other words, we need to check that $\Pi^{V_{\lambda^{\vee}}\textnormal{-lalg}}$ is indeed a locally algebraic representation (of weight $\lambda$) in the sense of this text. Note that the functor of taking locally $V_{\lambda^{\vee}}$-algebraic vectors factors through the functor $\Pi\mapsto \Pi^{\textnormal{la}}$ of taking locally analytic vectors. By \cite{Eme17}, Proposition 6.2.2 and Proposition 6.2.4, we see that $\Pi^{\textnormal{la}}$ is an admissible locally analytic representation of $G(L)$. By \textit{loc. cit.} Proposition 6.3.6, we see that $\Pi^{\textnormal{la},V_{\lambda^{\vee}}\textnormal{-lalg}}\cong \Pi^{V_{\lambda^{\vee}}\textnormal{-lalg}}$ is again an admissible locally analytic representation. Finally, the claim follows from \textit{loc. cit.} Proposition 6.3.10 by noting that $\textnormal{End}_{E[G(L)]}(\textnormal{V}_{\lambda^{\vee}})=E$.

Now the proposition easily follows from \cite{BD20} \S4 as we explain now. We first apply \cite{BD20}, Lemma 4.18 with the choice $V=\Pi^{V_{\lambda^{\vee}}\textnormal{-lalg}}\subset W=\Pi$, $W^{\circ}=\Pi^{\circ}$ and $V^{\circ}=V\cap W^{\circ}$. It gives an injection
    \begin{equation}\label{eq4.5}
        \textnormal{Ord}_Q^{\textnormal{lalg}}(\Pi^{V_{\lambda^{\vee}}\textnormal{-lalg}})\hookrightarrow \textnormal{Ord}_Q(\Pi).
    \end{equation}
    As $\textnormal{Ord}_Q^{\textnormal{lalg}}(\Pi^{V_{\lambda^{\vee}}\textnormal{-lalg}})$ is a locally algebraic representation of $M(L)$ of weight $w_0^Q\lambda$, it necessarily lands in $\textnormal{Ord}_Q(\Pi)^{V_{\lambda^{\vee}}^{N(L)}\textnormal{-lalg}}$. 
    
    For the other inclusion, note that $\textnormal{Ord}_Q(\Pi)$ is again a unitary admissible $E$-Banach space representation by \cite{Eme10}, Theorem 3.4.8 (and \cite{Eme17}, Proposition 6.5.7). Therefore, we see that $\textnormal{Ord}_Q(\Pi)^{V_{\lambda^{\vee}}^{N(L)}\textnormal{-lalg}}\cong \pi\otimes V_{\lambda^{\vee}}^{N(L)}$ is a locally algebraic representation of weight $w_0^Q\lambda$. Consider the natural inclusion
    \begin{equation*}
        \iota_v:\pi\otimes V_{\lambda^{\vee}}^{N(L)}\cong\textnormal{Ord}_Q(\Pi)^{V_{\lambda^{\vee}}^{N(L)}\textnormal{-lalg}}\hookrightarrow\textnormal{Ord}_Q(\Pi).
    \end{equation*}
    Then \cite{BD20}, Proposition 4.21 applied to $f:=\iota_v$ yields a map
    \begin{equation*}
        \textnormal{Ind}_{\overline{Q}(L)}^{G(L)}\pi\otimes V_{\lambda^{\vee}}\to \Pi
    \end{equation*}
    and $\iota_v$ can be reconstructed by precomposing the map
    \begin{equation}\label{eq4.6}
        \textnormal{Ord}_Q^{\textnormal{lalg}}(\textnormal{Ind}_{\overline{Q}(L)}^{G(L)}\pi\otimes V_{\lambda^{\vee}})\to \textnormal{Ord}_Q(\Pi)
    \end{equation}
    induced by \textit{loc. cit.} Lemma 4.18 with the inclusion 
    \begin{equation*}
        \pi\otimes V_{\lambda^{\vee}}^{N(L)}\hookrightarrow \textnormal{Ord}_Q^{\textnormal{lalg}}(\textnormal{Ind}_{\overline{Q}(L)}^{G(L)}\pi\otimes V_{\lambda^{\vee}})
    \end{equation*} of \textit{loc. cit.} Lemma 4.19. But \ref{eq4.6} factors through \ref{eq4.5} and the proof is finished.
\end{proof}

\subsection{Ordinary parts of weakly admissible representations}\label{sec4.2} We introduce the notion of weakly admissible locally algebraic representations, an a priori weaker integrality notion than unitariness. We then prove that the notion of being $Q$-ordinary behaves well for such representations. For the rest of the subsection, we assume that $G=\textnormal{GL}_n$, $T$ is the torus of diagonal matrices and $B$ is the Borel subgroup of upper-triangular matrices. In particular, $Q=P_{(n_1,...,n_t)}$ for a partition $n_1+...+n_t=n$. Fix a choice of uniformiser $\varpi_L \in \mathcal{O}_L$ and recall that $e$ denotes the ramification index of $L$. Moreover, for integers $0\leq b\leq c$ with $c\geq 1$, let $\mathcal{Q}(b,c)\subset G(\mathcal{O}_L)$ be the corresponding parahoric subgroup associated with $Q$. Then, by the Iwahori decomposition, $\mathcal{Q}(b,c)=\overline{N}^cM^bN^0$. For $1\leq i\leq n$, set
\begin{equation*}
    u_p^{(i)}:=\textnormal{diag}(\varpi_L,...,\varpi_L,1,...,1)
\end{equation*}
where the first $i$ elements are given by $\varpi_L$ and the rest by $1$. Then $u_p^{(n_1)}\cdot....\cdot u_p^{(n_1+...+n_{t-1})}\in Z_M^+$ is our choice for the role of $z_p$ from the proof of Lemma~\ref{Lem3.3}. For $x\in L^{\times}$, set $\langle x\rangle:=\textnormal{diag}(x,...,x)\in G(L)$.
\begin{Def}
    Let $\Pi$ be a locally algebraic $E$-representation of $G(L)$ of weight $\lambda$. We say that $\Pi$ is weakly admissible (of weight $\lambda$) if, for any standard parabolic subgroup $Q'=M'\ltimes N'\subset G$ and locally algebraic character $\chi:Z_{M'}(L)\to \overline{\mathbf{Q}}_p^{\times}$, such that $\Hom_{Z_{M'}(L)}(\chi,J_{Q'}(\Pi))\neq 0$, we have
    \begin{equation*}
        \mid \chi(z)\delta_{Q'}^{-1}(z)\mid_p\leq 1,
    \end{equation*}
    for every $z\in Z_{M'}^+$. We denote the category of such representations by $\textnormal{Mod}_{E}^{\textnormal{wa},\lambda}(G(L))$.
\end{Def}
\begin{Rem}
    This is what Hu refers to as Emerton's condition in \cite{Hu09}, where he proves the "easy" direction of the Breuil--Schneider conjecture.\footnote{In fact, he asks for this condition to be satisfied for \textit{any} parabolic subgroup $Q'\subset G$, but as it is explained in \textit{loc. cit.} Remarque 2.11, it suffices to check it for standard parabolic subgroups.} The main result (Th\'eor\`eme 1.2) of \cite{Hu09} shows that, at least when $\Pi$ is of the form $\pi\otimes V_{\lambda^{\vee}}$ with $\pi$ irreducible and generic, $\Pi$ is weakly admissible if and only if $\textnormal{rec}^T(\pi)$ comes from a de Rham Galois representation with Hodge--Tate weights given by the usual $\rho$-shift of $\lambda$. The latter means that the $(\varphi,N,G_L)$-module associated with $\textnormal{rec}^T(\pi)$ admits a filtration with jumps given by the $\rho$-shift of $\lambda$ that makes it a weakly admissible filtered $(\varphi,N,G_L)$-module.\footnote{To avoid confusion, note that the original formulation of the Breuil--Schneider conjecture uses different normalisations. Namely, one instead asks whether the representation $\pi(n-1)\otimes V_{\lambda^{\vee}+n-1}$ admits a unitary completion. As the character $(|\det|\det)^{n-1}$ is unitary, regardless of the mismatch, the two different normalisations yield equivalent conjectures.} The way Hu applies it to the Breuil--Schneider conjecture is that if $\Pi$ is unitary then one deduces easily that it must be weakly admissible (see \cite{Eme06c}, Lemma 4.4.2).
\end{Rem}
Using Emerton's condition to define weak admissibility has the advantage of being easy to keep track of while proving results regarding weakly admissible representations. However, we record here an equivalent condition that has the advantage that in the case of our examples appearing in this text it will clearly be satisfied.

\begin{Lemma}\label{Lem4.11}
    Let $\Pi=\pi\otimes V_{\lambda^{\vee}}$ be a locally algebraic $E$-representation of G(L) of weight $\lambda$ and $P_{(n_1',...,n_h')}=Q'=M'\ltimes N'\subset G$ a standard parabolic subgroup. The following two conditions are equivalent:
    \begin{enumerate}
        \item For each $1\leq i\leq h$, the generalised eigenvalues of the double coset operator $[N'^0u_p^{(n_1+...+n_i)}N'^0]$ acting on $\Pi^{N'^0}$ lie in $\overline{\mathbf{Z}}_p$.
        \item For any locally algebraic character $\chi:Z_M(L)\to \overline{\mathbf{Q}}_p^{\times}$ with
        \begin{equation*}
            \Hom_{Z_M(L)}(\chi,J_{Q'}(\Pi))\neq 0,
        \end{equation*} we have
        \begin{equation*}
            |\chi(z)\delta_{Q'}^{-1}(z)|_p\leq 1
        \end{equation*}
        for every $z\in Z_M^+$.
    \end{enumerate}
    In particular, $\Pi$ is weakly admissible if and only if it satisfies $i)$ for every choice of standard parabolic subgroup $Q'\subset G$. 
\end{Lemma}
\begin{proof}
    We first make a few observations about condition $ii)$. Note that $\chi$ is of the form $\chi_{\textnormal{sm}}\otimes \chi_{\textnormal{alg}}$. Moreover, $J_{Q'}(\Pi)=J_{Q'}(\pi)\otimes V_{(w_0^Q\lambda)^{\vee}}$ and $V_{(w_0^Q\lambda)^{\vee}}$ has constant central character sending $u_p^{(n'_1+...+n'_i)}$ to \begin{equation*}
    -w_0^G\lambda(u_p^{(n'_1+...+n'_i)})=\prod_{\iota:L\hookrightarrow E}\iota(\varpi_L)^{-(\lambda_{\iota,n}+...+\lambda_{\iota,n+1-(n_1+...+n_i)})}.
    \end{equation*} Therefore, $\chi$ is forced to be a smooth twist of $-w_0^G\lambda$.

    Moreover, after extending $\pi$ to a $\overline{\mathbf{Q}}_p$-representation, we can apply \cite{Cas95}, Proposition 2.1.9 to get a direct sum decomposition
    \begin{equation}\label{eq4.7}
        J_{Q'}(\pi)=\bigoplus_{\chi_{\textnormal{sm}}:Z_{M'}(L)\to \overline{\mathbf{Q}}_p^{\times}}J_{Q'}(\pi)_{\chi_{\textnormal{sm}}}
    \end{equation}
of the $M(L)$-representation into generalised eigenspaces with respect to all the possible central characters. Here we used that both base change to $\overline{\mathbf{Q}}_p$ and $J_{Q'}$ preserves admissibility. In particular, we see that $J_{Q'}(\pi)_{\chi_{\textnormal{sm}}}\neq 0$ if and only if $\Hom_{Z_{M'}(L)}(\chi_{\textnormal{sm}}(-w_0^G\lambda),J_{Q'}(\Pi))\neq 0$.

We also note that the decomposition \ref{eq4.7} is a refinement of the decomposition of $J_{Q'}(\pi)$ into generalised eigenspaces with respect to the actions of the $u_p^{(n'_1+...+n'_i)}$'s. On the other hand, it is easy to check that $\chi|_{Z_{M'}^0}$ and $\delta_{Q'}|_{Z_{M'}^0}$ necessarily land in $\overline{\mathbf{Z}
}_p^{\times}$. So the  inequality $|\chi(z)\delta_{Q'}^{-1}(z)|_p\leq 1$ is satisfied for all $z\in Z_{M'}^+$ if and only if it is satisfied for $z=u_p^{(n'_1+...+n'_i)}$ for all $1\leq i\leq h$.

Finally, we finish the proof by recalling that the natural map $\Pi^{N'^0}\otimes \delta_{Q'}\to J_{Q'}(\Pi)$ is $Z_{M'}^+$-equivariant, and its kernel is the subspace on which the $U_p$-operators act nilpotently.
\end{proof}
\begin{Rem}\label{Rem4.12}
    As was promised, using Lemma~\ref{Lem4.11}, we can find a large source of examples of weakly admissible representations. Namely, given a number field $F$, we can consider any regular algebraic cuspidal automorphic representation $\pi$ of $\textnormal{GL}_n(\mathbf{A}_F)$ of weight $\lambda$. Given a $p$-adic place $v\in S_p(F)$, set $L=F_v$, and look at the component $\pi_v$, a smooth admissible $\mathbf{C}$-representation of $G(L)$. After fixing an identification $t:\overline{\mathbf{Q}}_p\cong \mathbf{C}$, we can find a model of $t^{-1}\pi_v$ over some large enough field extension $E/\mathbf{Q}_p$. Since $t^{-1}\pi_v$ can then be found as a $\textnormal{GL}_n(F_v)$-equivariant direct summand in $(\varinjlim_{K_v\subset \textnormal{GL}_n(\mathcal{O}_{F_v})}H^{\ast}(X_{K^{v}K_v},\mathcal{V}_{\lambda}))\otimes_{\mathcal{O}}E$ and the \textit{normalised} $U_p$-operators already act integrally on the latter by \cite{CN23}, Lemma 2.1.17, we see that Lemma~\ref{Lem4.11}, i) is satisfied for any standard parabolic subgroup, making $t^{-1}\pi_v\otimes_E V_{\lambda^{\vee}}$ into a weakly admissible $E$-representation of $G(L)$.
\end{Rem}

\begin{Def}\label{Def4.13}
    Let $\Pi$ be a weakly admissible $E$-representation of $G(L)$ of weight $\lambda$. We say that $\Pi$ is $Q$-ordinary if $\textnormal{Ord}_Q^{\textnormal{lalg}}(\Pi)\neq 0$. We denote the subcategory of such by $\textnormal{Mod}^{Q\textnormal{-ord},\lambda}_E(G(L))\subset \textnormal{Mod}_{E}^{\textnormal{wa},\lambda}(G(L))$.
\end{Def}
\begin{Rem}
     We also note that both the notion of weakly admissible and the notion of ordinary generalises to any split reductive group over $L$ and we will be speaking of weakly admissible and ordinary representations of Levi subgroups $M(L)$ of $G(L)$ without further explanation.
\end{Rem}

We would also like to talk about parabolic induction for locally algebraic representations.
\begin{Def}
    Let $\Pi_M=\pi_M\otimes V_{(w_0^G\lambda)^{\vee}}$ be a weakly admissible $E$-representation of $M(L)$ of weight $w_0^Q\lambda$. We then set its parabolic induction to be
    \begin{equation*}
        \textnormal{Ind}_Q^G(\Pi_M):=(\textnormal{Ind}_{Q(L)}^{G(L)}\pi_M)\otimes V_{\lambda^{\vee}},
    \end{equation*}
    a locally algebraic $E$-representation of $G(L)$ of weight $\lambda$.
\end{Def}
Note that $\textnormal{Ind}_Q^G$ can only be applied to locally algebraic representations with weight of the form $w_0^Q\lambda$.

The key result of the subsection is the following theorem.
\begin{Th}\label{Th4.16}
    Let $\Pi_M$ be a weakly admissible $E$-representation of $M(L)$ of weight $w_0^Q\lambda$. Then $\textnormal{Ind}_Q^G(\delta_Q\otimes\Pi_M)$ is a weakly admissible $E$-representation of $G(L)$ of weight $\lambda$ and we have a natural isomorphism
    \begin{equation}\label{eq4.8}
        \textnormal{Ord}_Q^{\textnormal{lalg}}(\textnormal{Ind}_Q^G(\delta_Q\otimes\Pi_M))\cong \textnormal{Ord}_M^{\textnormal{lalg}}(\Pi_M).
    \end{equation}
    In particular, if $\Pi_M$ is $M$-ordinary, then $\textnormal{Ind}_Q^G(\delta_Q\otimes\Pi_M)$ is $Q$-ordinary and if $\Pi_M$ is also absolutely irreducible, we have $\textnormal{Ord}_Q^{\textnormal{lalg}}(\textnormal{Ind}_Q^G(\delta_Q\otimes\Pi_M))\cong \Pi_M.$
\end{Th}
Before starting the proof, we prove some preliminary technical lemmas. The proof of these lemmas and the proof of Theorem~\ref{Th4.16} was intentionally written in a way so that it clearly generalises to split reductive groups other than $\textnormal{GL}_n$. However, we kept the assumption $G=\textnormal{GL}_n$ for the whole subsection as later we will appeal to the specific features of the representation theory of $\textnormal{GL}_n(L)$, such as the Bernstein--Zelevinsky classification.

Introduce the following notation. Denote by $\Sigma_G$ the set of roots of $(G,T)$ and by $\Sigma_G^{+}\subset \Sigma_G$ the subset of $B$-dominant roots. Given a standard parabolic subgroup $M'\ltimes N'=Q'\subset G$,  analogous notations apply to $M'$. Further denote by $\Sigma_{N'}\subset \Sigma_G^+$ the roots corresponding to $N'$. Note that $\Sigma_G^+=\Sigma_{M'}^+\coprod \Sigma_{N'}$ and $\Sigma_G=-\Sigma_N\coprod \Sigma_{M'}\coprod \Sigma_{N'}$.
\begin{Lemma}\label{Lem4.17}
    Let $M\ltimes N=Q,M'\ltimes N'=Q'\subset G$ be two standard parabolic subgroups, $\nu\in X_{\ast}(Z_{M'})$ be a cocharacter that is moreover $B$-dominant. Consider an element $w\in {}^{Q'}W^{Q}$ and define the standard parabolic subgroups $M_{w^{-1}}\ltimes N_{w^{-1}}=Q_{w^{-1}}^M:=M\cap w^{-1}Q'w^{}\subset M$ and $M_{w}'\ltimes N_w'=Q_w^{M'}:=M'\cap w^{}Qw^{-1}\subset M'$. Then $w^{-1}\nu$ is in $X_{\ast}(Z_{M_{w^{-1}}})$ and is a $B\cap M$-dominant cocharacter. In particular, the isomorphism $M_w'(L)\xrightarrow{\sim}M_{w^{-1}}(L)$, $m\mapsto w^{-1}mw$ restricts to a map $Z_{M'}^+\to Z_{M_{w^{-1}}}^+$.\footnote{Note that $M_{w^{-1}}$ is considered as a Levi subgroup in $M$ and $Z_{M_{w^{-1}}}^+$ has its meaning accordingly.}
\end{Lemma}
\begin{proof}
    To see that $w^{-1}\nu$ lies in $X_{\ast}(Z_{M_{w^{-1}}})$, note that, by \cite{Hau18}, Lemma 2.1.1, (ii), we have $M_{w^{-1}}=M\cap w^{-1}M'w^{}$ and $M_{w}'=M'\cap w^{}Mw^{-1}$. In particular, conjugation by $w^{-1}$ restricts to an isomorphism $M_w'\xrightarrow{\sim}M_{w^{-1}}$ from which the claim follows.

    To verify that $w^{-1}\nu$ is $B\cap M$-dominant, we have to check that for any $\alpha\in \Sigma_M^+$, we have $\langle w^{-1}\nu,\alpha\rangle\geq 0$. Equivalently, we prove that $\langle\nu, w\alpha\rangle\geq 0$. Since $w\in {}^{Q'}W^{Q}$, we know that $w\alpha\in \Sigma_G^+$ so the inequality follows from the fact that $\nu$ is $B$-dominant.

    For the last claim of the lemma, note that any $z\in Z_{M'}^+$ must be of the form $\nu(\varpi_L)\cdot z'$ for some $B$-dominant cocharacter $\nu\in X_{\ast}(Z_{M'})$ and $z'\in Z_{M'}^0$.
\end{proof}
\begin{Lemma}\label{Lem4.18}
    Let $M\ltimes N=Q$, $ M'\ltimes N'=Q'$, $w\in {}^{Q'}W^Q$, $Q_{w^{-1}}^M=M_{w^{-1}}\ltimes N_{w^{-1}}$ and $Q_{w}^{M'}=M'_w\ltimes N'_w$ as in Lemma~\ref{Lem4.17}. Further consider the standard parabolic subgroups $Q_{w^{-1}}=M_{w^{-1}}\ltimes (N_{w^{-1}}N)$, $Q'_w=M'_{w}\ltimes (N'_wN')\subset G$. Then, for any $z\in Z_{M'}^+$, we have
    \begin{equation*}
        \textnormal{val}_p(\delta_{Q'_w}(z))\leq \textnormal{val}_p(\delta_{Q_{w^{-1}}}(w^{-1}zw)).
    \end{equation*}
    Moreover, if $w\neq 1$, we can find $z\in Z_{M'}^+$ such that $\textnormal{val}_p(\delta_{Q'_w}(z))<\textnormal{val}_p(\delta_{Q_{w^{-1}}}(w^{-1}zw))$.
\end{Lemma}
\begin{proof}
    Recall that we have
    \begin{equation*}
        \delta_{Q'_w}(z)=\left| \prod_{\alpha\in \Sigma_{N'_w}\coprod \Sigma_{N'}}\alpha(z)\right|_L
    \end{equation*}
    and
    \begin{equation*}
        \delta_{Q_{w^{-1}}}(w^{-1}zw)=\left|\prod_{\alpha\in \Sigma_{N_{w^{-1}}}\coprod \Sigma_N}\alpha(w^{-1}zw)\right|_L= \left|\prod_{\alpha\in \Sigma_{N_{w^{-1}}}\coprod \Sigma_N}w\alpha(z)\right|_L.
    \end{equation*}
We also note that $w\in {}^{Q'}W^Q$ means that we have
\begin{enumerate}
    \item $w^{-1}(-\Sigma_G^{+})\cap \Sigma_G^+\subset \Sigma_N$, and
    \item$(-\Sigma_G^+)\cap w(\Sigma_G^{+})\subset -\Sigma_{N'}$.
\end{enumerate}
Moreover, by \cite{Hau18}, Lemma 2.1.1, (ii), $w$ sets up a bijection between $\Sigma_{M_{w^{-1}}}$ and $\Sigma_{M_w'}$. In particular, we have
\begin{equation*}
    w(\Sigma_{N_{w^{-1}}}\coprod \Sigma_{N})\subset \Sigma_{N'_w}\coprod \Sigma_{N'}\coprod -\Sigma_{N'}.
\end{equation*}
From this the first part of the lemma follows.

For the last part, we assume that $1\neq w$ and would like to find $z\in Z_{M'}^+$ such that $\delta_{Q'_w}(z)<\delta_{Q_{w^{-1}}}(w^{-1}zw)$. Note that if we found an $\alpha\in \Sigma_N$ such that $w\alpha\in-\Sigma_{N'}$, then for any choice of a $B$-dominant cocharacter $\nu\in X_{\ast}(Z_{M'})$ with $\langle \nu, w\alpha\rangle< 0$, $\nu(\varpi)$ would be a suitable choice for $z$. The existence of such an $\alpha$ is the content of \cite{AHV19}, Lemma 5.13.
\end{proof}

\begin{proof}[Proof of Theorem~\ref{Th4.16}]
    We first prove that $\textnormal{Ind}_Q^G(\delta_Q\otimes \Pi_M)$ is weakly admissible. To do so, we fix a standard parabolic subgroup $Q'=M'\ltimes N'\subset G$ and write $\Pi_M=\pi_M\otimes V_{-w_0^G\lambda}$. We would like to understand the $Z_{M'}^+$-action on $J_{Q'}(\textnormal{Ind}_Q^G(\delta_Q\otimes \Pi_M))$. By applying the geometric lemma (cf. \cite{BZ77}, Lemma 2.12) to $J_{Q'}(\textnormal{Ind}_{Q(L)}^{G(L)}(\delta_Q\otimes \pi_M))$, we obtain an $M'(L)$-equivariant filtration of $J_{Q'}(\textnormal{Ind}_Q^G(\delta_Q\otimes \Pi_M))$ with subquotients given by
    \begin{equation*}
        I_w^{Q'}:=\biggl(\delta_{Q'}^{1/2}\textnormal{Ind}_{Q_w ^{M'}}^{M'}\Bigl(\delta_{Q_{w} ^{M'}}^{1/2}w^{\ast}\bigl(\delta_{Q_{w^{-1}}^M}^{-1/2}J_{Q_{w^{-1}}^M}(\delta_Q^{1/2}\pi_M)\bigr)\Bigr)\biggr)\otimes V_{(w_0^Q\lambda)^{\vee}}\footnote{We warn the reader that we consider unnormalised versions of parabolic induction and the Jacquet functor as opposed to \cite{BZ77}, where each of the functors are normalised, hence the appearance of several modulus characters in the definition of $I_w^{Q'}$.}
    \end{equation*}
    for $w\in {}^{Q'}W^Q$ where $M'_w\ltimes N_w'=Q_w^{M'}:=M'\cap wQw^{-1}\subset M'$, $M_{w^{-1}}\ltimes N_{w^{-1}}=Q_{w^{-1}}^M:=M\cap w^{-1}Q'w$, and $w^{\ast}(-)$ denotes the pullback along the isomorphism $M_{w^{-1}}(L) \xrightarrow{\sim} M_{w}'(L)$, $m\mapsto wmw^{-1}$.\footnote{For the fact that it is indeed an isomorphism, see \cite{Hau18}, Lemma 2.1.1, (ii).} To check weak admissibility, it suffices to see that for any $w\in {}^{Q'}W^Q$, any locally algebraic character $\chi=\chi_{\textnormal{sm}}(-w_0^G\lambda):Z_{M'}(L)\to \overline{\mathbf{Q}}_p^{\times}$ with a $Z_{M'}(L)$-equivariant embedding $\chi\hookrightarrow I_w^{Q'}$ and any $z\in Z_{M'}^+$, we have
    \begin{equation}
        \textnormal{val}_p(\chi_{\textnormal{sm}}(z))\geq\textnormal{val}_p(w_0^G\lambda(z))+\textnormal{val}_p(\delta_{Q'}(z)).
    \end{equation}
    Fix such a Weyl group element $w$, character $\chi$, and embedding $\chi\hookrightarrow I_w^{Q'}$.
Set $I_w^{Q',\textnormal{sm}}$ to be the smooth part of $I_{w}^{Q'}$. Fix a $z\in Z_{M'}^+$ and compute
\begin{equation*}
    I_w^{Q',\textnormal{sm}}(z)=\delta_{Q'}^{1/2}(z)\delta_{Q_w^{M'}}^{1/2}(z)\delta_{Q_{w^{-1}}^M}^{-1/2}(w^{-1}zw)J_{Q_{w^{-1}}^M}(\delta_Q^{1/2}\pi_M)(w^{-1}zw)=
\end{equation*}
\begin{equation*}
    =\delta_{Q'}^{1/2}(z)\delta_{Q_{w^{-1}}^M}^{-1/2}(w^{-1}zw)\delta_{Q}^{1/2}(w^{-1}zw)J_{Q_{w^{-1}}^M}(\pi_M)(w^{-1}zw).
\end{equation*}
Here we used that $z$ is a central element in $M'(L)$. Since $w\in {}^{Q'}W^Q$, Lemma~\ref{Lem4.17} shows that $w^{-1}zw\in Z_{M_{w^{-1}}}^+$. In particular, by weak admissibility of $\Pi_M$, we see that for any $Z_{M_{w^{-1}}}(L)$-equivariant embedding $\widetilde{\chi}=\widetilde{\chi}_{\textnormal{sm}}(-w_0^G\lambda)\hookrightarrow J_{Q_{w^{-1}}^M}(\Pi_M)=J_{Q_{w^{-1}}^M}(\pi_M)\otimes V_{-w_0^G\lambda}$, we have
\begin{equation*}
\textnormal{val}_p(\widetilde{\chi}_{\textnormal{sm}}(w^{-1}zw))\geq \textnormal{val}_p(\delta_{Q_{w^{-1}}^M}(w^{-1}z w))+\textnormal{val}_p(w_0^G\lambda(w^{-1}zw)).
\end{equation*}
Our chosen embedding $\chi\hookrightarrow I_w^{Q'}$ then induces a $w^{-1}Z_{M'}(L)w$-equivariant embedding $\widetilde{\chi}'_{\textnormal{sm}}:=(w^{-1})^{\ast}(\chi_{\textnormal{sm}}\delta_{Q'}^{-1/2})\delta_{Q_{w^{-1}}^M}^{1/2}\delta_Q^{-1/2}\hookrightarrow J_{Q_{w^{-1}}^M}(\pi_M)$. As a consequence, there exists a character $\widetilde{\chi}:=\widetilde{\chi}_{\textnormal{sm}}(-w_0^G\lambda):Z_{M_{w^{-1}}}(L)\to \overline{\mathbf{Q}}_p^{\times}$ with an embedding $\widetilde{\chi}\hookrightarrow J_{Q_{w^{-1}}^M}(\Pi_M)$ such that $\widetilde{\chi}_{\textnormal{sm}}|_{w^{-1}Z_{M'}(L)w}=\widetilde{\chi}_{\textnormal{sm}}'$. Pick any such $\widetilde{\chi}$ and compute
\begin{equation}\label{eq4.10}
    \textnormal{val}_p\bigl(\chi_{\textnormal{sm}}(z)\delta_{Q'}^{-1/2}(z)\delta_{Q_{w^{-1}}^M}^{1/2}(w^{-1}zw)\delta_{Q}^{-1/2}(w^{-1}zw)\bigr)=\textnormal{val}_p(\widetilde{\chi}_{\textnormal{sm}}(w^{-1}zw))
\end{equation}
\begin{equation*}
    \geq \textnormal{val}_p(\delta_{Q_{w^{-1}}^M}(w^{-1}zw))+\textnormal{val}_p(w_0^G\lambda(w^{-1}zw))\geq \textnormal{val}_p(\delta_{Q_{w^{-1}}^M}(wzw^{-1}))+\textnormal{val}_p(w_0^G\lambda(z)).
\end{equation*}
Here the last inequality follows from the fact that $\lambda$ is $(\textnormal{Res}_{L/\mathbf{Q}_p}B)_E$-dominant. After rearranging \ref{eq4.10}, we get
\begin{equation*}
    \textnormal{val}_p(\chi_{\textnormal{sm}}(z))\geq \frac{1}{2}\Bigl(\textnormal{val}_p(\delta_{Q_{w^{-1}}^M}(w^{-1}zw))+\textnormal{val}_p(\delta_Q(w^{-1}zw))-\textnormal{val}_p(\delta_{Q'}(z))\Bigr)+
\end{equation*}
\begin{equation*}
    +\textnormal{val}_p(\delta_{Q'}(z))+\textnormal{val}_p(w_0^G\lambda(z)).
\end{equation*}
In particular, it suffices to prove that
\begin{equation*}
    \textnormal{val}_p\Bigl(\delta_{Q_{w^{-1}}^M}(w^{-1}zw)\delta_Q(w^{-1}zw)\delta_{Q'}^{-1}(z)\Bigr)\geq 0.
\end{equation*}
To see this, denote by $Q_{w^{-1}}\subset G$ the standard parabolic subgroup $M_{w^{-1}}\ltimes (N_{w^{-1}}N)$ and by $Q'_w\subset G$ the standard parabolic subgroup $M_{w}'\ltimes(N'_{w}N')$. Then one has $\delta_{Q'}(z)=\delta_{Q'_w}(z)$ and $\delta_{Q_{w^{-1}}^M}(w^{-1}zw)\delta_Q(w^{-1}zw)=\delta_{Q_{w^{-1}}}(w^{-1}zw)$. In particular, we only need to see that $\textnormal{val}_p(\delta_{Q'_w}(z))\leq \textnormal{val}_p(\delta_{Q_{w^{-1}}}(w^{-1}zw))$. This is the content of the first part of Lemma~\ref{Lem4.18}. We conclude that $\textnormal{Ind}_Q^G(\delta_Q\otimes \Pi_M)$ is indeed weakly admissible.

We note that we essentially already proved the rest of the theorem, too. Namely, recall that we have an $M(L)$-equivariant split injection
\begin{equation}\label{eq4.11}
        \textnormal{Ord}_Q^{\textnormal{lalg}}(\textnormal{Ind}_{Q}^{G}(\delta_Q\otimes \Pi_M))\xhookrightarrow{\oplus} J_Q(\textnormal{Ind}_{Q}^{G}(\delta_Q\otimes\Pi_M))\otimes \delta_Q^{-1}
\end{equation}
with image given by the slope $0$ part (cf. \S\ref{sec4.1}). By the geometric lemma applied to the RHS of \ref{eq4.11}, we get a surjection
\begin{equation}\label{eq4.12}
    J_Q(\textnormal{Ind}_{Q}^{G}(\delta_Q\otimes\Pi_M))\otimes \delta_Q^{-1}\twoheadrightarrow  I_1^{Q}\otimes \delta_Q^{-1} =\Pi_M
\end{equation}
with its kernel $K$ admitting a filtration with subquotients given by $I_{w}^Q\otimes \delta_Q^{-1}$ for $1\neq w\in {}^QW^Q$. In particular, if we prove that, for any $1\neq w\in {}^QW^Q$, the slope zero part $(\delta_Q^{-1}\otimes I_w^Q)_0$ is trivial, then we immediately obtain that \ref{eq4.8} holds. Indeed,
we then can conclude by applying slope $0$ part to the short exact sequence
\begin{equation*}
    0\to K\to J_Q(\textnormal{Ind}_{Q}^{G}(\delta_Q\otimes\Pi_M))\otimes \delta_Q^{-1}\to \Pi_M\to 0.
\end{equation*}
 By the proof of Lemma~\ref{Lem4.11}, to prove the claimed description of $K$, it is enough to check that for any $1\neq w\in {}^QW^Q$ and any locally algebraic character $\chi=\chi_{\textnormal{sm}}(-w_0^G\lambda):Z_M(L)\to \overline{\mathbf{Q}}_p^{\times}$ with an embedding $\chi \hookrightarrow I_w\otimes \delta_Q^{-1}$, there is a $z\in Z_M^+$ such that
\begin{equation*}
    \textnormal{val}_p(\chi_{\textnormal{sm}}(z))>\textnormal{val}_p(w_0^G\lambda(z)).
\end{equation*}
Based on the previous paragraph of the proof, to see this, we only need to prove that we can find $z\in Z_M^+$ such that $\textnormal{val}_p(\delta_{Q_{w}}(z))<\textnormal{val}_p(\delta_{Q_{w^{-1}}}(w^{-1}zw))$. This is the content of the second part of Lemma~\ref{Lem4.18}.
\end{proof}
Theorem~\ref{Th4.16} shows that we obtain a functor
\begin{equation*}
    \textnormal{Ind}_Q^G(-\otimes \delta_Q):\textnormal{Mod}_E^{\textnormal{wa},w_0^Q\lambda}(M(L))\to \textnormal{Mod}_E^{\textnormal{wa},\lambda}(G(L))
\end{equation*}
that restricts to a functor
\begin{equation*}
    \textnormal{Mod}_E^{M\textnormal{-ord},w_0^Q\lambda}(M(L))\to \textnormal{Mod}_E^{Q\textnormal{-ord},\lambda}(G(L)).
\end{equation*}
Furthermore, we already constructed a functor
\begin{equation*}
    \textnormal{Ord}_Q^{\textnormal{lalg}}(-):\textnormal{Mod}_E^{\textnormal{wa},\lambda}(G(L))\to\textnormal{Mod}_E^{\textnormal{wa},w_0^Q\lambda}(M(L)).
\end{equation*}
We introduce the notion of \textit{$Z$-integral} weakly admissible representations which, by Theorem~\ref{Th4.16}, form the essential image of $\textnormal{Ord}_{Q}^{\textnormal{lalg}}$.
\begin{Def}
    Call a weakly admissible $E$-representation $\Pi$ of $G(L)$ (of  weight $\lambda$) $Z$-integral if $\textnormal{Ord}_G^{\textnormal{lalg}}(\Pi)=\Pi$. We denote the corresponding subcategory by $\textnormal{Mod}_E^{Z\textnormal{-int},\lambda}(G(L))\subset \textnormal{Mod}_E^{\textnormal{wa},\lambda}(G(L))$.
\end{Def}

\begin{Lemma}\label{Lem4.20}
    The pair of functors $(\textnormal{Ord}_Q^{\textnormal{lalg}},\textnormal{Ind}_Q^G(-\otimes \delta_Q))$, between the categories $\textnormal{Mod}_{E}^{\textnormal{wa},\lambda}(G(L))$ and $\textnormal{Mod}_E^{Z\textnormal{-int}, w_0^Q\lambda}(M(L))$, forms an adjoint pair. The association in one direction sends, for $\Pi\in\textnormal{Mod}_E^{\textnormal{wa},\lambda}(G(L)) $, $\Pi_M\in \textnormal{Mod}_E^{\textnormal{wa},w_0^Q\lambda}(M(L))$, a map $f:\Pi\to \textnormal{Ind}_Q^G(\Pi_M\otimes \delta_Q)$ to the composition $\textnormal{Ord}_Q^{\textnormal{lalg}}(\Pi)\xrightarrow{g_1}\textnormal{Ord}_Q^{\textnormal{lalg}}(\textnormal{Ind}_Q^G(\Pi_M\otimes\delta_Q))\xrightarrow{g_2} \Pi_M$ where $g_1$ is induced by applying $\textnormal{Ord}_Q^{\textnormal{lalg}}$ to $f$ and $g_2$ is the isomorphism of Theorem~\ref{Th4.16}.  
\end{Lemma}
\begin{proof}
    By Frobenius reciprocity, maps of the form $f:\Pi\to \textnormal{Ind}_Q^G(\Pi_M\otimes \delta_Q)$ are in natural bijection with maps of the form $g:J_Q(\Pi)\to \Pi_M\otimes \delta_Q$. Since $\Pi_M$ is $Z$-integral, any such $g$ must factor through the split surjection
    \begin{equation}\label{eq4.13}
        J_Q(\Pi)\twoheadrightarrow \textnormal{Ord}_Q^{\textnormal{lalg}}(\Pi)\otimes \delta_Q.
    \end{equation}
    Moreover, from the obtained map $\textnormal{Ord}_Q^{\textnormal{lalg}}(\Pi)\otimes \delta_Q\to \Pi_M\otimes \delta_Q$, $g$ can be recovered by precomposing it with \ref{eq4.13}. Finally, twisting by $\delta_Q^{-1}$ certainly sets up a bijection of Hom sets.

    That this association coincides with the one in the statement follows from an easy diagram chase, the definition of Frobenius reciprocity and the definition of the isomorphism \ref{eq4.8}.
\end{proof}

We record an immediate corollary that was the main motivation to prove Theorem~\ref{Th4.16}.
\begin{Cor}\label{Cor4.21}
    Let $\Pi$ be an absolutely irreducible $Q$-ordinary $E$-representation of $G(L)$ of weight $\lambda$. Then $\textnormal{Ord}_Q^{\textnormal{lalg}}(\Pi)$ is an absolutely irreducible locally algebraic $E$-representation of $M(L)$ of weight $w_0^Q\lambda$.
\end{Cor}
\begin{proof}
    Since $J_Q(\Pi)$ is of finite finite length, at the very least, so is $\textnormal{Ord}_Q^{\textnormal{lalg}}(\Pi)$. 
    After possibly enlarging $E$, we can pick an absolutely irreducible quotient
    \begin{equation}\label{eq4.14}
        \textnormal{Ord}_Q^{\textnormal{lalg}}(\Pi)\twoheadrightarrow \Pi_M.
    \end{equation}
    In particular, $\Pi_M$ is $Z$-integral.
    Adjunction then gives a non-trivial map $g:\Pi\to \textnormal{Ind}_Q^G(\Pi_M\otimes \delta_Q)$ that must then be an embedding. By Lemma~\ref{Lem4.20}, we recover \ref{eq4.14} by applying $\textnormal{Ord}_Q^{\textnormal{lalg}}(-)$ to $g$ and postcomposing it with \ref{eq4.8}. In particular, due to left exactness of $\textnormal{Ord}_Q^{\textnormal{lalg}}$, \ref{eq4.14} must also be an injection.
\end{proof}

As a consequence, for any absolutely irreducible object $\Pi\in\textnormal{Mod}_E^{Q\textnormal{-ord}, \lambda}(G(L))$, there is a unique absolutely irreducible $\Pi_M\in \textnormal{Mod}_E^{\textnormal{wa}, w_0^Q\lambda}(M(L))$ such that $\Pi$ embeds into $\textnormal{Ind}_Q^G(\Pi_M\otimes \delta_Q)$. We can then make the following definition.
\begin{Def}
    Given an absolutely irreducible $\Pi\in \textnormal{Mod}_E^{Q\textnormal{-ord},\lambda}(G(L))$, we say that $\Pi_M\in \textnormal{Mod}_E^{\textnormal{wa},w_0^Q\lambda}(M(L))$ is its \textit{$Q$-ordinary support} if it is the (necessarily unique) absolutely irreducible representation in $\textnormal{Mod}_E^{\textnormal{wa},w_0^Q\lambda}(M(L))$ with an embedding $\Pi\hookrightarrow \textnormal{Ind}_Q^G(\Pi_M\otimes \delta_Q)$.
\end{Def}
A very useful feature of the notion of the $Q$-ordinary support is that it can be read off from $\Pi$ by applying $\textnormal{Ord}_Q^{\textnormal{lalg}}(-)$.

For the rest of the subsection, we will be occupied with understanding the relation between $\Pi$ and its $Q$-ordinary support in terms of the Bernstein--Zelevinsky and Langlands classifications. In particular, it is this point from which we make use of the assumption that $G=\textnormal{GL}_n$. For a quick review on the Bernstein--Zelevinsky classification and the relevant notations used in the rest of the subsection, we refer to the appendix.

\begin{Lemma}\label{Lemma4.23}
    Let $\Pi=\pi\otimes V_{\lambda^{\vee}}$ be an absolutely irreducible weakly admissible $E$-representation of $G(L)$ of weight $\lambda$ that is $Z$-integral.  Then, for any $\pi_{\textnormal{sc}}\in \textnormal{SC}(\pi)$\footnote{Recall that we denote by $\textnormal{SC}(\pi)$ the set $\{\pi_1,...,\pi_k\}$ where $(\textnormal{GL}_{n_1}(L)\times...\times \textnormal{GL}_{n_k}(L),\pi_1\otimes...\otimes \pi_k)$ is the supercuspidal support of $\pi$.}, we have\footnote{We remind the reader to the notation $\langle\varpi\rangle=\textnormal{diag}(\varpi,...,\varpi)\in G(L)$.}
    \begin{equation}\label{eq4.15}
        \frac{1}{e}\sum_{\iota:L\hookrightarrow E}(\lambda_{\iota, n}+\frac{1-n}{2})\leq \frac{\textnormal{val}_p(\pi_{\textnormal{sc}}(\langle\varpi\rangle))}{\textnormal{deg}(\pi_{\textnormal{sc}})}\leq \frac{1}{e}\sum_{\iota:L\hookrightarrow E}(\lambda_{\iota,1}+\frac{n-1}{2}).
    \end{equation}
\end{Lemma}
\begin{proof}
    Assume that $\pi = Z(\Delta_1,...\Delta_k)$ for an ordered multiset of segments $\underline{\Delta}:=(\Delta_1,...,\Delta_k)$ with $\Delta_i:=\Delta(\pi_i,r_i)$ for $i=1,...,k$. Denote by $Q_{\underline{\Delta}}\subset G$ the standard parabolic subgroup attached to the corresponding ordering of the \textit{supercuspidal support} of $\pi$. For $\Delta=\Delta(\sigma,r)$, set
    \begin{equation*}
        v_{\Delta}:=\frac{\sum_{i=0}^{r-1}\textnormal{val}_p(\sigma(i)(\langle \varpi\rangle))}{r\deg(\sigma)}=\frac{\textnormal{val}_p(\sigma(\langle\varpi\rangle))}{\deg(\sigma)}+\frac{[L:\mathbf{Q}_p]}{e}\frac{1-r}{2},
    \end{equation*}
    the arithmetic mean value of the numbers $\frac{\textnormal{val}_p(\pi_{\textnormal{sc}}(\langle\varpi\rangle))}{\deg(\pi_{\textnormal{sc}})}$ for $\pi_{\textnormal{sc}}\in \textnormal{SC}(Z(\Delta))$.

Note that we can assume that $\underline{\Delta}$ is ordered so that
\begin{equation}\label{eq4.16}
    (v_{\Delta_1},-r_1)\leq...\leq (v_{\Delta_k},-r_k)
\end{equation}
with respect to the lexicographic ordering. Similarly, we can assume that $\underline{\Delta}$ is ordered so that
\begin{equation}\label{eq4.17}
    (v_{\Delta_1},r_1)\leq...\leq (v_{\Delta_k},r_k).
\end{equation}
Indeed, if $\Delta_i$ and $\Delta_j$ are linked for some $1\leq i\neq j\leq k$, say $\Delta_i$ precedes $\Delta_j$, then $v_{\Delta_j}<v_{\Delta_i}$ and, by well-orderedness of the multiset of segments in the sense of Bernstein--Zelevinsky, we must also have $j<i$. Otherwise, if $\Delta_i$ and $\Delta_j$ are not linked, for instance if $v_{\Delta_i}=v_{\Delta_j}$, we have the freedom of choosing the order.

If we choose $\underline{\Delta}$ so that \ref{eq4.16} is satisfied, we see that
\begin{equation*}
    v_{\Delta_1}+\frac{[L:\mathbf{Q}_p]}{e}\frac{1-r_1}{2}=\frac{\textnormal{val}_p(\pi_1(r_1-1)(\langle \varpi\rangle))}{\deg(\pi_1)}\leq \frac{\textnormal{val}_p(\pi_{\textnormal{sc}}(\langle \varpi\rangle))}{\deg(\pi)}
\end{equation*}
for any $\pi_{\textnormal{sc}}\in\textnormal{SC}(\pi)$.
Similarly, if we choose $\underline{\Delta}$ so that \ref{eq4.17} is satisfied, we see that
\begin{equation*}
    \frac{\textnormal{val}_p(\pi_{\textnormal{sc}}(\langle \varpi\rangle))}{\deg (\pi_{\textnormal{sc}})}\leq v_{\Delta_k}+\frac{r_k-1}{2}=\frac{\textnormal{val}_p(\pi_k)(\langle\varpi\rangle))}{\deg(\pi_k)}
\end{equation*}
for any $\pi_{\textnormal{sc}}\in\textnormal{SC}(\pi)$.
In particular, it suffices to prove that, for any choice of ordering $\underline{\Delta}$, we have
\begin{equation}\label{eq4.18}
    \frac{\textnormal{val}_p((\delta_{Q_{\underline{\Delta}}}^{1/2}w_0^G\lambda)(u_p^{(\deg(\Delta_1))}))}{\deg(\Delta_1)}\leq v_{\Delta_1}+\frac{[L:\mathbf{Q}_p]}{e}\frac{1-r_1}{2}
\end{equation}
and
\begin{equation}\label{eq4.19}
    v_{\Delta_k}+\frac{[L:\mathbf{Q}_p]}{e}\frac{r_k-1}{2}\leq \frac{\textnormal{val}_p((\delta_{Q_{\underline{\Delta}}}^{1/2}w_0^G\lambda)(u_p^{(n)}/u_p^{(n-\deg(\Delta_k))}))}{\deg(\Delta_k)}.
\end{equation}
Indeed, an easy computation, combined with the regularity of $\lambda$ and the equality
\begin{equation*}
    \textnormal{val}_p(\delta_{Q_{\underline{\Delta}}}^{1/2}(u_p^{(m)}))=\frac{[L:\mathbf{Q}_p]}{e}\sum_{i=1}^{m}(\frac{1-n}{2}+i-1)\textnormal{ for every }1\leq m\leq n,
\end{equation*}
shows that the LHS of \ref{eq4.15} is bounded by the LHS of \ref{eq4.18} and the RHS of \ref{eq4.19} is bounded by the RHS of \ref{eq4.15}.

To prove these inequalities, we note that, by Lemma~\ref{LemA.4}, $J_{Q_{\underline{\Delta}}}(\pi)$ admits 
$\delta_{Q_{\underline{\Delta}}}^{1/2}\Delta_1\otimes...\otimes\Delta_k$ as a quotient. In particular, weak admissibility of $\Pi$ gives that
\begin{equation*}
    \textnormal{val}_p((\delta_{Q_{\underline{\Delta}}}w_0^G\lambda))(u_p^{(\deg(\Delta_1))})\leq\textnormal{val}_p(\Delta_1(\langle\varpi\rangle)\delta_{Q_{\underline{\Delta}}}^{1/2}(u_p^{(\deg(\Delta_1))}))
\end{equation*}
and \ref{eq4.18} is proved.

To get \ref{eq4.19}, one similarly applies weak admissibility with the choice $z=u_p^{(n-\deg(\Delta_k))}$ and combines it with the equality
\begin{equation*}
\textnormal{val}_p((\delta_{Q_{\underline{\Delta}}}w_0^G\lambda)(\langle\varpi\rangle))=\textnormal{val}_p(J_{Q_{\underline{\Delta}}}(\pi)(\langle\varpi\rangle))=\textnormal{val}_p((\delta_{Q_{\underline{\Delta}}}^{1/2}\Delta_1\otimes ...\otimes \Delta_k)(\langle\varpi\rangle))
\end{equation*}
coming from the assumption that $\Pi$ is absolutely irreducible and $Z$-integral, so it has a central character that must be integral, and the fact that $\delta_{Q_{\underline{\Delta}}}(\langle\varpi\rangle)=1$.
\end{proof}
\begin{Lemma}\label{Lem4.24}
    Let $\Pi_M=\pi_M\otimes V_{(w_0^Q\lambda)^{\vee}}$ be an absolutely irreducible weakly admissible $E$-representation of $M(L)$ of weight $w_0^Q\lambda$ that is $Z$-integral. If $Q=M\ltimes N=P_{(n_1,...,n_t)}$ and $\delta_Q^{1/2}\otimes \pi_M=\pi_1\otimes...\otimes \pi_t$, then, for $1\leq i< j\leq t$ and $\pi_{\textnormal{sc},i}\in \textnormal{SC}(\pi_i)$, $\pi_{\textnormal{sc},j}\in \textnormal{SC}(\pi_j)$, we have
    \begin{equation*}
        \frac{\textnormal{val}_p(\pi_{\textnormal{sc},i}(\langle \varpi\rangle))}{\deg(\pi_{\textnormal{sc},i})}\leq \frac{\textnormal{val}_p(\pi_{\textnormal{sc},j}(\langle \varpi\rangle))}{\deg(\pi_{\textnormal{sc},j})}-\frac{[L:\mathbf{Q}_p]}{e}=\frac{\textnormal{val}_p(\pi_{\textnormal{sc},j}(\langle \varpi\rangle))}{\deg(\pi_{\textnormal{sc},j})}+\textnormal{val}_p(|\varpi|_L).
    \end{equation*}
\end{Lemma}
\begin{proof}
    Write $\Pi_M=\Pi_1\otimes...\otimes \Pi_t$. Then the lemma follows easily from applying Lemma~\ref{Lemma4.23} to each of the $\Pi_i$'s for $1\leq i\leq t$ and noting that, for $1\leq i\leq t$,
    \begin{equation*}
        \delta_Q^{1/2}|_{\textnormal{GL}_{n_i}(L)}=|\det|^{\frac{n-n_i}{2}-(n_1+...+n_{i-1})}.
    \end{equation*}
\end{proof}

Combining Lemma~\ref{LemA.3} with the results of the subsection, we are now ready to understand the relation between the Langlands classification of $\Pi$ and its $Q$-ordinary support. This is best stated using the local Langlands correspondence.
\begin{Cor}\label{Cor4.25}
    Let $\Pi$ be an absolutely irreducible $P_{(n_1,...,n_t)}=Q$-ordinary  $E$-representation of $G(L)$ of weight $\lambda$ with $Q$-ordinary support $\Pi_M$. Write $\Pi\otimes_E\overline{\mathbf{Q}}_p=\pi\otimes_{\overline{\mathbf{Q}}_p}V_{\lambda^{\vee}}$ and $\Pi_M\otimes_E\overline{\mathbf{Q}}_p=(\pi_1\otimes...\otimes \pi_t)\otimes_{\overline{\mathbf{Q}}_p} V_{-w_0^G\lambda}$. Fix an identification $t:\overline{\mathbf{Q}}_p\cong \mathbf{C}$ and assume further that $\pi$ is $t$-preunitary (see the appendix). Then $\textnormal{rec}^T(\pi)$ admits a flag $0=F_0\subset F_1\subset ...\subset F_t=\textnormal{rec}^T(\pi)$
of sub-Weil--Deligne representations such that, for $1\leq j\leq t$, we have isomorphisms
\begin{equation*}
    F_j/F_{j-1}\cong \textnormal{rec}^T(\pi_j\otimes|\cdot |^{-\sum_{j-1}})
\end{equation*}
where $\sum_{j}:=n_1+...+n_{j}$.
In other words,  there is an isomorphism of Weil--Deligne representations
\begin{equation*}
    \textnormal{rec}^T(\pi)\sim \begin{pmatrix}
\textnormal{rec}^T(\pi_1) & \ast & ... & \ast\\
0 & \textnormal{rec}^T(\pi_2\otimes|\cdot|^{-n_1}) & ... & \ast\\
. & . & . & .\\
. & . & . & .\\
0 & ... & 0 & \textnormal{rec}^T(\pi_t\otimes|\cdot|^{n_t-n})
\end{pmatrix}.
\end{equation*}
\end{Cor}
\begin{proof}
Note that adjunction applied to $\textnormal{Ord}_Q^{\textnormal{lalg}}(\Pi)\otimes|\cdot|^{\frac{1-n}{2}} =\Pi_M\otimes|\cdot|^{\frac{1-n}{2}}\xrightarrow{\textnormal{id}}\Pi_M\otimes|\cdot|^{\frac{1-n}{2}}$ gives an embedding
\begin{equation}\label{eq4.20}
    \pi\otimes|\cdot |^{\frac{1-n}{2}}\hookrightarrow 
    \textnormal{n-Ind}_{Q(L)}^{G(L)}(\delta_Q^{1/2}|_{\textnormal{GL}_{n_1}(L)}\otimes|\cdot|^{\frac{1-n}{2}}\otimes\pi_1)\otimes...\otimes (\delta_Q^{1/2}|_{\textnormal{GL}_{n_t}(L)}\otimes|\cdot|^{\frac{1-n}{2}}\otimes\pi_t).
\end{equation}
One computes that, for $1\leq j\leq t$, we have
\begin{equation*}
    \delta_Q^{1/2}|_{\textnormal{GL}_{n_j}(L)}=|\cdot|^{\frac{n-n_j}{2}-\sum_{j-1}}.
\end{equation*}
Therefore, we compute
\begin{equation*}
    \delta_Q^{1/2}|_{\textnormal{GL}_{n_j}(L)}\otimes|\cdot|^{\frac{1-n}{2}}\otimes\pi_j=(\pi_j\otimes|\cdot|^{\frac{1-n_j}{2}})\otimes |\cdot|^{\frac{-(n-n_j)}{2}}\otimes \delta_Q^{1/2}=(\pi_j\otimes|\cdot |^{\frac{1-n_j}{2}})\otimes|\cdot|^{-\sum_{j-1}}.
\end{equation*}
This implies the existence of a filtration
\begin{equation*}
    0=F_0^{\textnormal{ss}}\subset...\subset F_t^{\textnormal{ss}}=\textnormal{rec}^T(\pi)^{\textnormal{ss}}
\end{equation*}
of representations of the Weil group with subquotients given by
\begin{equation*}
    \textnormal{rec}^T(\pi_j\otimes|\cdot|^{-\sum_{j-1}})^{\textnormal{ss}}.
\end{equation*}In fact, this really is a direct sum decomposition as the underlying Weil group representation of $\textnormal{rec}^T(\pi)$ is semisimple.

By combining Lemma~\ref{Lem4.24} and our unitariness assumption, we can apply Lemma~\ref{LemA.3} to \ref{eq4.20}. After unravelling the construction of the reduction of the local Langlands correspondence to supercuspidal representations, this exactly says that the monodromy on the subquotients does not change. In particular, we can upgrade $F^{\textnormal{ss}}_{\bullet}$ to the desired filtration of Weil--Deligne representations of $\textnormal{rec}^T(\pi)$.
\end{proof}
To ease the notation in the upcoming sections, we introduce the following notation.
\begin{Def}\label{Def4.26}
    Denote by $\Pi$ an absolutely irreducible $Q$-ordinary $E$-representation of $G(L)$ of weight $\lambda$ and write $\Pi\otimes_E\overline{\mathbf{Q}}_p=\pi\otimes_E V_{\lambda^{\vee}}$. Then we set $\pi^{Q\textnormal{-ord}}$ to be the smooth $\overline{\mathbf{Q}}_p$-representation of $M(L)$ such that $\textnormal{Ord}^{\textnormal{lalg}}(\Pi)\otimes_E\overline{\mathbf{Q}}_p=\pi^{Q\textnormal{-ord}}\otimes_{\overline{\mathbf{Q}}_p} V_{-w_0^G\lambda}$.
\end{Def}

\subsection{A $Q$-ordinary local-global compatibility result}

Finally, we use our observations about $Q$-ordinary representations to deduce $Q$-ordinary local-global compatibility for regular algebraic conjugate self-dual cuspidal automorphic representations (RACSDCAR). We fix an integer $n\geq 1$ with a partition $n_1+...+n_t$ and denote by $Q=M\ltimes N$ the corresponding parabolic subgroup of $\textnormal{GL}_n$. We further consider the appropriate global setup i.e., $F$ will denote a CM number field and $v|p$ is a fixed $p$-adic place. If $\pi$ is a RACAR of $\textnormal{GL}_n(\mathbf{A}_F)$ of weight $\lambda\in (\mathbf{Z}^n_+)^{\Hom(F,\mathbf{C})}$ and $t:\overline{\mathbf{Q}}_p\cong \mathbf{C}$ is a fixed isomorphism then $t^{-1}\pi_v$ can be realized over a finite extension $E/\mathbf{Q}_p$. Moreover, $t^{-1}\pi_v\otimes_EV_{t^{-1}\lambda^{\vee}}$ becomes a weakly admissible $E$-representation of $\textnormal{GL}_n(F_v)$ of weight $t^{-1}\lambda_v$ (see Remark~\ref{Rem4.12}). Therefore, we are in the situation of \S\ref{sec4.2}. We are then interested in proving the following local-global compatibility result. 
\begin{Th}\label{Th4.27}
Let $\pi$ be a RACSDCAR of $\textnormal{GL}_n(\mathbf{A}_F)$ of weight $\lambda\in (\mathbf{Z}_+^n)^{\Hom(F,\mathbf{C})}$, $t:\overline{\mathbf{Q}}_p\cong \mathbf{C}$ be a fixed isomorphism and $v|p$ be a $p$-adic place of $F$. Assume that $t^{-1}\pi_v\otimes_E V_{\lambda^{\vee}}$ is $Q$-ordinary. Write $\pi^{Q\textnormal{-ord}}=\pi_1\otimes...\otimes \pi_t$ (see Definition~\ref{Def4.26}). Then there is an isomorphism
\begin{equation*}
    r_t(\pi)|_{G_{F_{v}}}\sim \begin{pmatrix}
\rho_1 & \ast & ... & \ast\\
0 & \rho_2 & ... & \ast\\
. & . & . & .\\
. & . & . & .\\
0 & ... & 0 & \rho_t
\end{pmatrix}
\end{equation*}
where, for $1\leq j\leq t$,
\begin{equation*}
    \rho_j:G_{F_v}\to \textnormal{GL}_{n_j}(\overline{\mathbf{Q}}_p)
\end{equation*}
is potentially semistable such that, for every embedding $\iota:F_v\hookrightarrow \overline{\mathbf{Q}}_p$, the labelled $\iota$-Hodge--Tate weights of $\rho_j$ are given by
\begin{equation*}
    \lambda_{t\circ \iota,n+1-(n_1+...+n_j)}+n_1+...+n_{j-1}+n_j-1>...>\lambda_{t\circ \iota,n+1-(n_1+...+n_{j-1}+1)}+n_1+...+n_{j-1}
\end{equation*}
and we have
\begin{equation*}
    \textnormal{WD}(\rho_j)^{F-ss}\cong \textnormal{rec}^T(\pi_j\otimes|\cdot|^{-\sum_{j-1}})
\end{equation*}
where $\sum_j:=n_1+...+n_j$.
\end{Th}
We highlight the following simple observation.
\begin{Lemma}\label{Lem4.28}
Let $\pi_v$ and $(\pi_1,...,\pi_t)$ as in Theorem~\ref{Th4.27}. Then, for any lift of (geometric) Frobenius $\varphi_{v}\in W_{F_{v}}$ and $1\leq j\leq t$, we have
\begin{equation*}
    \textnormal{val}_p(\textnormal{det}(\textnormal{rec}^T(\pi_j\otimes |\cdot |^{-\sum_{j-1}})(\varphi_{v})))=\frac{1}{e_{v}}\sum_{\iota:F_{v}\hookrightarrow \overline{\mathbf{Q}}_p}\sum_{i=n_1+...+n_{j-1}+1}^{n_1+...+n_j}(\lambda_{\iota,n+1-i}+i-1).
\end{equation*}
\end{Lemma}
\begin{proof}
Recall that for $\sigma$ a smooth admissible $\overline{\mathbf{Q}}_p$-representation of $\textnormal{GL}_n(F_v)$, $\textnormal{det}\circ\textnormal{rec}(\sigma)=\textnormal{rec}(\omega_{\sigma})$, where $\omega_{\sigma}$ denotes the central character of $\sigma$.

In particular, we have 
\begin{equation*}
\textnormal{val}_p(\det(\textnormal{rec}^T(\pi_j\otimes|\cdot |^{-\sum_{j-1}})(\varphi_v)))= \textnormal{val}_p((\pi_j\otimes |\cdot|^{\frac{1-n_j}{2}-\sum_{j-1}})(\textnormal{diag}(\varpi_v,...,\varpi_v)))
\end{equation*} for $\varpi_v$ the image of $\varphi_v$ under the local Artin map. To see that the RHS of the equation above is the desired number, we use that $(\pi_1\otimes...\otimes \pi_t)\otimes_{E}V_{-w_0^G\lambda}$ has ($p$-adically) integral central character. Namely, it implies that we have
\begin{equation*}
    \textnormal{val}_p(\pi_j(\textnormal{diag}(\varpi_v,...,\varpi_v)))=\sum_{\iota:F_v\hookrightarrow E}\sum_{i=n_1+...+n_{j-1}+1}^{n_1+...+n_j}\lambda_{\iota,n+1-i}\cdot\textnormal{val}_p(\iota(\varpi_v))=
\end{equation*}
\begin{equation}\label{easyeq1}
    \frac{1}{e_v}\sum_{\iota:F_v\hookrightarrow E}\sum_{i=n_1+...+n_{j-1}+1}^{n_1+...+n_j}\lambda_{\iota,n+1-i}.
\end{equation}

Moreover, we have
\begin{equation*}
    \textnormal{val}_p(\mid\det(\textnormal{diag}(\varpi_v,...,\varpi_v))\mid^{\frac{1-n_j}{2}-\sum_{j-1}})= (\frac{1-n_j}{2}-\sum_{j-1})\cdot n_j\cdot \textnormal{val}_p(|\varpi_v|_v)=
\end{equation*}
\begin{equation}\label{easyeq2}
    f_v\cdot\left(\sum_{i=n_1+...+n_{j-1}+1}^{n_1+...+n_j}(i-1)\right)=  \frac{1}{e_v}|\Hom(F_v,\overline{\mathbf{Q}}_p)|\cdot \left(\sum_{i=n_1+...+n_{j-1}+1}^{n_1+...+n_j}(i-1)\right)
\end{equation}
where in the last equality we used that $f_v=\frac{[F_v:\mathbf{Q}_p]}{e_v}$.

We conclude by combining \ref{easyeq1} and \ref{easyeq2}.

\end{proof}

Before starting the proof of Theorem~\ref{Th4.27}, we recall some constructions from $p$-adic Hodge theory which we will make use of in the proof.
Consider a potentially semistable $p$-adic Galois representation $\rho:G_{F_{v}}\to \textnormal{GL}_d(E)$ for some finite extension $E/\mathbf{Q}_p$ and for simplicity assume that it has distinct Hodge--Tate weights. Let $K/F_{v}$ be a finite Galois extension such that $\rho|_{G_K}$ is semistable and enlarge $E$ so that it contains the images of all embeddings $K\hookrightarrow \overline{\mathbf{Q}}_p$. Set $K_0:=W(\mathcal{O}_K/\mathfrak{m}_K)[1/p]$ and denote by $\sigma$ its arithmetic Frobenius.

We can then apply Fontaine's construction which associates with $\rho$ a filtered $(\varphi,N,K/F_{v}, E)$-module $D:=D_{\textnormal{st},K}(\rho)$. This, by definition, is a free $K_0\otimes_{\mathbf{Q}_p}E$-module with
\begin{enumerate}
    \item a $\sigma\otimes 1$-semilinear automorphism $\varphi$ of $D$;
    \item a $K_0\otimes_{\mathbf{Q}_p}E$-linear endomorphism $N$ of $D$ such that $N\varphi=p\varphi N$;
    \item a $K$-semilinear, $E$-linear action of $\textnormal{Gal}(K/F_{v})$ commuting with $\varphi$ and $N$;
    \item and a filtration $\textnormal{Fil}_{\bullet}D_K$ of $D_K:=D\otimes_{K_0}K$ by $K\otimes_{\mathbf{Q}_p}E$-submodules.
\end{enumerate}
Moreover, we know that $D$ is weakly admissible. In other words, "its Newton polygon lies over its Hodge polygon", i.e. in the notations of \cite{Fon94}, $t_N(D)=t_H(D)$ and for any sub-$(\varphi,N,K/F_{v}, E)$-module $D'\subset D$ equipped with the induced filtration from $D$ we have $t_N(D')\geq t_H(D')$. See \cite{Fon94}, 4.4.1 and Definition 4.4.3.

We have identifications
\begin{equation*}
    D_K\cong \prod_{\iota:K\hookrightarrow E}D_{\iota},\; \textnormal{Fil}_{\bullet}D_K=\prod_{\iota:K\hookrightarrow E}\textnormal{Fil}_{\bullet}D_{\iota}.
\end{equation*}
Then for every embedding $\iota:K\hookrightarrow E$, we have
\begin{equation*}
    \dim_E\textnormal{gr}^i\textnormal{Fil}_{\bullet}D_{\iota}=\begin{cases}
    1,& \text{if } i=\lambda_{\iota|_{F_{v}},j}\text{for }j\in\{1,...,d\}\\
    0,              & \text{otherwise}
\end{cases}
\end{equation*}
where $\lambda_{\iota|_{F_{v}},1}> ...> \lambda_{\iota|_{F_{v}},d}$ are the labelled $\iota|_{F_{v}}$-Hodge--Tate-weights of $\rho$.
In other words, the Hodge--Tate weights are encoded in the Hodge-filtration of $D$.

We can further associate to $D$ a Weil--Deligne representation as follows. Given $g\in W_{F_{v}}$, let $g$ act on $D$ by $(g\mod W_K)\circ \varphi^{-\alpha(g)}$ where $\alpha(g)$ is given by the power of the arithmetic Frobenius given by the action of $g$ on the residue field of $\overline{F}_{v}$. Note that if $f_{K}$ resp. $f_{v}$ denotes the inertia degree of $K/\mathbf{Q}_p$ resp. $F_{v}/\mathbf{Q}_p$, then we have that, for any lift of geometric Frobenius $\varphi_{v}\in W_{F_{v}}$, $\varphi_{v}^{f_{K}/f_{v}}$ acts on $D$ by $\varphi^{f_K}$. This action is then $K_0\otimes_{\mathbf{Q}_p}E$-linear by definition and we can consider the $W_{F_{v}}$- and $N$-invariant decomposition
\begin{equation*}
    D=\prod_{\iota_0:K_0\hookrightarrow E}D_{\iota_0}.
\end{equation*}
We then set $\textnormal{WD}(\rho):=D_{\iota_0}$ to be the associated Weil--Deligne representation over $E$ for a choice of $\iota_0$. As the notation suggests, it is independent of the choice of $\iota_0$ and $K$ up to isomorphism.
\begin{proof}[Proof of Theorem~\ref{Th4.27}]
We argue as in \cite{Tho15}, Theorem 2.4 which was based on \cite{Ger18}, Corollary 2.7.8. In particular, the key ingredients are local-global compatibility at $v$ for $r_t(\pi)$, Fontaine's theory of weakly admissible modules and Corollary~\ref{Cor4.25}.

Let $E\subset \overline{\mathbf{Q}}_p$ be a finite extension of $\mathbf{Q}_p$ such that $r_t(\pi)|_{G_{F_v}}$ lands in $\textnormal{GL}_n(E)$. First note that by Theorem~\ref{Th2.17} we have that $r_t(\pi)$ is potentially semistable at $v$. Let $K/F_{v}$ be a finite Galois extension such that $r_t(\pi)|_{G_K}$ is semistable and enlarge $E$ if necessary to assume that it contains all images of embeddings $K\hookrightarrow \overline{\mathbf{Q}}_p$. Set $D$ to be the weakly admissible filtered $(\varphi,N,K/F_{v},E)$-module associated with $r_t(\pi)|_{G_{F_v}}$. By Theorem~\ref{Th2.17} again, we further have an identification
\begin{equation*}
    \textnormal{WD}(r_t(\pi)|_{G_{F_v}})^{F-ss}\cong \textnormal{rec}^T(t^{-1}\pi_v).
\end{equation*}
Therefore, by Corollary~\ref{Cor4.25}, we have an isomorphism
\begin{equation*}
    \textnormal{WD}(r_t(\pi)|_{G_{F_{v}}})^{F-ss}\sim \begin{pmatrix}
\textnormal{rec}^T(\pi_1) & \ast & ... & \ast\\
0 & \textnormal{rec}^T(\pi_2\otimes |\cdot |^{-n_1}) & ... & \ast\\
. & . & . & .\\
. & . & . & .\\
0 & ... & 0 & \textnormal{rec}^T(\pi_t\otimes|\cdot|^{n_t-n})
\end{pmatrix}.
\end{equation*}
In particular, we obtain an isomorphism
\begin{equation*}
    \textnormal{WD}(r_t(\pi)|_{G_{F_{v}}})\sim \begin{pmatrix}
\textnormal{WD}_1 & \ast & ... & \ast\\
0 & \textnormal{WD}_2 & ... & \ast\\
. & . & . & .\\
. & . & . & .\\
0 & ... & 0 & \textnormal{WD}_t,
\end{pmatrix}
\end{equation*}
such that, for $1\leq j\leq t$, $\textnormal{WD}_j^{F-ss}\cong \textnormal{rec}^T(\pi_j\otimes |\cdot|^{-\sum_{j-1}})$.
Consequently, $D=D_{\textnormal{st},K}(r_t(\pi)|_{G_{F_v}})$ admits a flag
\begin{equation*}
    F_0=0\subset F_1\subset ...\subset F_t=D
\end{equation*}
of sub-$(\varphi,N,K/F_v,E)$-modules with $F_j/F_{j-1}$ corresponding to $\textnormal{rec}^T(\pi_j\otimes |\cdot|^{-\sum_{j-1}})$.

We now can apply Fontaine's theorem about classifying weakly admissible filtered $(\varphi,N)$-modules. Namely, if we combine Lemma~\ref{Lem4.28} with \cite{Fon94}, Theorem 5.6.7 and use that $\varphi_v^{f_K/f_v}$ acts on $D$ by $\varphi^{f_K}$, we get that, for each $1\leq j\leq t$, $F_j\subset D$ comes from a sub-$G_{F_v}$-representation $\widetilde{\rho}_j\subset r_t(\pi)|_{G_{F_v}}$. Moreover, their subquotients $\rho_j:=\widetilde{\rho}_j/\widetilde{\rho}_{j-1}$ are clearly potentially semistable and have the right associated Weil--Deligne representations as in the statement of the theorem.

Finally, to find the Hodge--Tate weights, we argue by induction on $1\leq j\leq t$. Assume that the claim holds for indices smaller than $j$. We can combine the fact that $t_N(F_j)=t_H(F_j)$, $\textnormal{WD}(\rho_j)^{F-ss}\cong \textnormal{rec}^T(\pi_j\otimes |\cdot|^{-\sum_{j-1}})$, Lemma~\ref{Lem4.28} and the induction hypothesis to see that the sum of the Hodge--Tate weights of $\rho_j$ is given by
\begin{equation*}
    \sum_{\iota:F_{v}\hookrightarrow \overline{\mathbf{Q}}_p}\sum_{i=n_1+...+n_{j-1}+1}^{n_1+...+n_j}(\lambda_{t\circ \iota,n+1-i}+i-1).
\end{equation*}
Now the regularity of $\lambda$ combined with Theorem~\ref{Th2.17} forces $\rho_j$ to have the right Hodge--Tate weights.
\end{proof}

\begin{Rem}\label{Rem4.29}
Note that the proof of Theorem~\ref{Th4.27} already works for automorphic representation $\pi$ with $\pi_v$ being pre-unitary (in the sense of the appendix) and admitting an associated Galois representation satisfying local-global compatibility at $v$. Therefore, our results will already hold for automorphic representations $\pi$ which are pre-unitary at $v$ and are isobaric sums of regular algebraic discrete conjugate self-dual automorphic representations. Indeed, this follows from Moeglin--Waldspurger's classification of discrete automorphic representations. This observation combined with Shin's base change result \cite{Shi14} leads to the application below, which is one of the crucial ingredients in proving our main local-global compatibility results.
\end{Rem}

To be able to appeal to \cite{Shi14}, we assume for the rest of the subsection that $F$ contains an imaginary quadratic field in which $p$ splits. Fix a $p$-adic place $\Bar{v}$ of $F^+$ and fix a place $v$ of $F$ above $\Bar{v}$. Let $\widetilde{Q}_{\Bar{v}}\subset P_{F^+_{\Bar{v}}}\subset \widetilde{G}_{F^+_{\Bar{v}}}$ be a parabolic subgroup with Levi decomposition $\widetilde{Q}_{\Bar{v}}=\widetilde{M}_{\Bar{v}}\ltimes \widetilde{N}_{\Bar{v}}$. Then $\widetilde{Q}_{\Bar{v}}(F^+_{\Bar{v}})$ is identified under $\iota_v$ with $P_{(n_1,...,n_t)}(F_v)\subset \textnormal{GL}_{2n}(F_v)$ for some tuple of integers $(n_1,...,n_k,n_{k+1},...,n_t)$, refining $(n,n)$.

\begin{Th}\label{Th4.30}
Let $\widetilde{\pi}$ be a $\xi$-cohomological cuspidal automorphic representation of $\widetilde{G}(\mathbf{A}_{F^+})$ as in Theorem~\ref{Th2.19}. Denote by $\Tilde{\lambda}$ the highest weight of the representation $(\xi\otimes \xi)^{\vee}$. Assume further that, for a fixed $t:\overline{\mathbf{Q}}_p\cong \mathbf{C}$, $t^{-1}(\widetilde{\pi}_{\Bar{v}}\circ \iota_v^{-1})\otimes V_{t^{-1}\Tilde{\lambda}_{\Bar{v}}}$ is $\widetilde{Q}_{\Bar{v}}$-ordinary of weight $t^{-1}\Tilde{\lambda}_{\Bar{v}}$. Write $(t^{-1}(\widetilde{\pi}\circ \iota_v^{-1}))^{\widetilde{Q}_{\Bar{v}}\textnormal{-ord}}=\widetilde{\pi}_1\otimes...\otimes\widetilde{\pi}_t$.
Then there is an isomorphism
\begin{equation*}
    r_t(\widetilde{\pi})|_{G_{F_{v}}}\sim \begin{pmatrix}
\rho_{1} & \ast & . &  . & . & . & \ast &  \ast\\
0 & . &  &   &   & &   & \ast\\
. &  & . &   &   &  &   &  .&\\
. &  & &  \rho_{k}  & &  &   &  .&\\
. &  &  &  & \rho_{k+1} &  &   & . &\\
. &  &  &   &   & .  &  & . &\\

0 &   &  &   &  &  & . & \ast\\
0 & 0 & . &  . & . & . &  0 & \rho_{t}
\end{pmatrix}
\end{equation*}
with $\rho_{j}$ being potentially semistable and, for every $\iota:F_v\hookrightarrow\overline{\mathbf{Q}}_p$, it has labelled $\iota$-Hodge--Tate weights
\begin{equation*}
   \Tilde{\lambda}_{\iota,2n+1-n_1+...+n_j}+n_1+...+n_j-1>...>\Tilde{\lambda}_{\iota,2n+1-n_1+...+n_{j-1}+1}+n_1+...+n_{j-1}.
\end{equation*}
Moreover, we have
\begin{equation*}
    \textnormal{WD}(\rho_j)^{F-ss}\cong \textnormal{rec}^T(\widetilde{\pi}_j\otimes |\cdot |^{-\sum_{j-1}})
\end{equation*}
where $\sum_j:=n_1+...+n_j$.
\end{Th}
\begin{proof}
This follows from Remark~\ref{Rem4.29}, and Theorem~\ref{Th2.19}.
\end{proof}

\section{A torsion local-global compatibility conjecture}\label{sec5}
In this section, we formulate integral local-global compatibility conjectures in great generality and state our progress on proving them. Most of this, although only in more restrictive setups, was already formulated in \cite{CEGGPS16}, and \cite{GN20}.

For the rest of the section, set $G:={\textnormal{GL}_n}_{/F}$ for a fixed number field $F$ and an integer $n\geq 2$. Pick a subfield $E\subset \overline{\mathbf{Q}}_p$, finite over $\mathbf{Q}_p$, with ring of integers $\mathcal{O}\subset E$ and a choice of uniformiser $\varpi\in \mathcal{O}$. We assume that $E$ is large enough so that $[F:\mathbf{Q}]=|\Hom(F,E)|$.
\subsection{Local deformation rings}
To formulate local-global compatibility that also keeps track of torsion, we make use of Kisin's potentially semistable local deformation rings.

Let $\textnormal{CNL}_{\mathcal{O}}$ denote the category of complete local Noetherian $\mathcal{O}$-algebras with residue field $k:=\mathcal{O}/\varpi$. Set $L:=F_v$ for some $v\in S_p(F)$. Denote by $L_0$ its maximal subfield that is unramified over $\mathbf{Q}_p$, and by $\overline{L}_0$ the maximal unramified extension of $L_0$. Similar notations apply to any finite field extension $L'/L$. Given a continuous Galois representation $\overline{\rho}:G_L\to \textnormal{GL}_n(k)$, we can form the framed deformation functor $D_{\overline{\rho}}^{\Box}:\textnormal{CNL}_{\mathcal{O}}\to \textnormal{Sets}$, sending a test object $(A,\mathfrak{m}_A)$ to the set of continuous homomorphisms $\rho:G_L\to \textnormal{GL}_n(A)$ with $\rho\otimes_AA/\mathfrak{m}_A= \overline{\rho}$. It is known to be represented by an object $R^{\Box}_{\overline{\rho}}\in \textnormal{CNL}_{\mathcal{O}}$, that is called its \textit{framed deformation ring}. We denote by $\rho^{\textnormal{univ}}:G_L\to \textnormal{GL}_n(R^{\Box}_{\overline{\rho}})$ the universal lift, and, for $x:R^{\Box}_{\overline{\rho}}\to A$ in $\textnormal{CNL}_{\mathcal{O}}$, by $\rho_x$ its specialisation $\rho^{\textnormal{univ}}\otimes_{R^{\Box}_{\overline{\rho}},x}A$.

Consider an $n$-dimensional Weil--Deligne inertial type $\tau=(\rho_{\tau},N_{\tau})$ for $L$ over $E$ and a highest weight vector $\lambda\in (\mathbf{Z}_+^n)^{\Hom(L,E)}$ for $\textnormal{Res}_{L/\mathbf{Q}_p}\textnormal{GL}_n$. From now on we will make the following technical assumption.
\begin{assumption}
    Assume that $E$ is large enough so that there is a finite Galois extension $L_{\tau}/L$ such that $\rho_{\tau}|_{I_{L_{\tau}}}$ is trivial with $|\Hom(L_{\tau},E)|=|L_{\tau}:\mathbf{Q}_p|$. Fix such an  $L_{\tau}/L$.
\end{assumption}
In \cite{Kis07}, Kisin constructs a reduced $\mathcal{O}$-flat quotient $R^{\lambda,\rho_{\tau}}_{\overline{\rho}}$ of $R^{\Box}_{\overline{\rho}}$ whose $E'$-points, for any finite extension $E'/E$, correspond to lifts $\rho: G_L\to \textnormal{GL}_n(E')$ of $\overline{\rho}$ such that $\rho$ is potentially semistable with labelled Hodge--Tate weights $(\lambda_{\iota,1}+n-1>...>\lambda_{\iota,n})_{\iota\in \Hom(L,E)}$ and $\textnormal{WD}(\rho)^{ss}|_{I_L}\cong \rho_{\tau}\otimes_EE'$. After introducing some terminology, we will further describe its $B$-valued points for a general finite $E$-algebra $B$. 

To a given $\lambda$, we associate a $p$-adic Hodge type $\mathbf{v}_{\lambda}$ in the sense of \cite{Kis07}. Consider an $n$-dimensional $E$-vector space $D_E$ and write 
\begin{equation*}
    D_{E,L}:=D_E\otimes_{\mathbf{Q}_p}L\cong \oplus_{\iota:L\hookrightarrow E}D_{E,\iota}.
\end{equation*}
For each $\iota: L\hookrightarrow E$, we pick an (arbitrary) decreasing filtration $\textnormal{Fil}^{\bullet}D_{E,\iota}$ by sub-$E$-vector spaces so that $\dim_E\textnormal{gr}^iD_{E,\iota}\neq 0$ if and only if $i=\lambda_{\iota,n+j-1}+j-1$ for $1\leq j\leq n$ in which case the $E$-dimension is exactly $1$. By taking the direct sum of the filtrations, we get a decreasing a decreasing filtration $\textnormal{Fil}^{\bullet}D_{E,L}$ on $D_{E,L}$ by $E\otimes_{\mathbf{Q}_p}L$-submodules. We set $\mathbf{v}_{\lambda}=\{D_E,\textnormal{Fil}^{\bullet}D_{E,L}\}$. For a finite $E$-algebra $B$, and a continuous de Rham representation $\rho_B$ of $G_L$ on a finite free rank $n$ $B$-module $V_B$, we say that $\rho_B$ has $p$-adic Hodge type $\mathbf{v}_{\lambda}$ if, for each $i\in \mathbf{Z}$, there is an isomorphism
\begin{equation*}
\textnormal{gr}^i(V_B\otimes_{\mathbf{Q}_p}B_{\textnormal{dR}})^{G_L}\cong B\otimes_E\textnormal{gr}^iD_{E,L}
\end{equation*}
of $B\otimes_{\mathbf{Q}_p}L$-modules. In particular, any de Rham Galois representation $\rho:G_L\to \textnormal{GL}_n(E)$ with labelled Hodge--Tate weights $(\lambda_{\iota,1}+n-1>...>\lambda_{\iota,n})_{\iota\in \Hom(L,E)}$ has $p$-adic Hodge type $\mathbf{v}_{\lambda}$.

Consider a finite $E$-algebra $B$, and a continuous potentially semistable representation $\rho_B$ of $G_L$ on a finite free $B$-module $V_B$ of rank $n$. We explain the notion of $\rho_{B}$ having inertial type $\rho_{\tau}$. Assume for a second that $B$ is local with residue field $E'$. Further assume that $V_B$ becomes semistable as a representation of $G_{L_{\tau}}$. Set
\begin{equation*}
    D^{L_{\tau}}_{\textnormal{st}}(V_B)=(V_B\otimes_{\mathbf{Q}_p}B_{\textnormal{st}})^{G_{L_{\tau}}},
\end{equation*}
a finite free $B\otimes_{\mathbf{Q}_p}L_{\tau,0}$-module, forming a filtered $(\varphi,N,L_{\tau}/L,E)$-module. In particular, it admits a $B\otimes_{\mathbf{Q}_p}L_{\tau,0}$-linear action of $I_{L_{\tau}/L}$ that commutes with $\varphi$ and $N$. Since higher cohomology of finite groups is trivial in characteristic $0$, deformation theory tells us that the $I_{L_{\tau}/L}$-action on $D^{L_{\tau}}(V_B)$ comes as extension of scalars along $E'\otimes_{\mathbf{Q}_p}L_{\tau,0}\hookrightarrow B\otimes_{\mathbf{Q}_p}L_{\tau,0}$ of a representation over a rank $n$ free $E'\otimes_{\mathbf{Q}_p}L_{\tau,0}$-module. Moreover, since the $I_{L_{\tau}/L}$-action commutes with $\varphi$, it further descends to a representation of $I_{L_{\tau}/L}$ on some $n$-dimensional $E'$-vector space $P_{\rho_{B}}$. We say that $\rho_B$ is of inertial type $\rho_{\tau}$ if its restriction to $G_{L_{\tau}}$ is semistable and $P_{\rho_B}$ with its $I_{L_{\tau}/L}$-action is equivalent to $\rho_{\tau}$. One then easily extends the definition to general finite $E$-algebras using that every such algebra is a product of finite \textit{local} $E$-algebras.

Finally, we prepare the upcoming result by making the following definition. Let $A$ be an $E$-algebra. We then define a $(\varphi, N,L_{\tau}/L,A)$-module of rank $n$ to be a projective $A \otimes_{\mathbf{Q}_p} L_{\tau,0}$-module of rank $n$ equipped with
\begin{enumerate}
    \item a Frobenius semilinear automorphism $\varphi$;
    \item an $A \otimes_{\mathbf{Q}_p} L_{\tau,0}$-linear endomorphism $N$ satisfying $N\varphi =p\varphi N$;
    \item and an $L_{\tau,0}$-semilinear, $A$-linear $\textnormal{Gal}(L_{\tau}/L)$-action commuting with $\varphi$ and $N$.
\end{enumerate}
\begin{Th}\label{Th5.1}
    There is a unique $\mathcal{O}$-flat quotient $R^{\Box}_{\overline{\rho}}\to R^{\lambda,\rho_{\tau}}_{\overline{\rho}}$ characterised by the property that an arbitrary map $x:R^{\Box}_{\overline{\rho}}\to B$ to a finite $E$-algebra factors through $R^{\lambda,\rho_{\tau}}_{\overline{\rho}}$ if and only if $\rho_x$ is potentially semistable of $p$-adic Hodge type $\mathbf{v}_{\lambda}$ and inertial type $\rho_{\tau}$. Moreover, $R^{\lambda,\rho_{\tau}}_{\overline{\rho}}$ is reduced.

    Finally, there exists a rank $n$ $(\varphi,N,L_{\tau}/L,R^{\lambda,\rho_{\tau}}_{\overline{\rho}}[1/p])$-module $D^{\lambda,\rho_{\tau}}_{\textnormal{st},L_{\tau},\overline{\rho}}$ such that, for any finite $E$-algebra $B$, and map $x:R^{\lambda,\rho_{\tau}}_{\overline{\rho}}\to B$, we have a canonical isomorphism of $(\varphi,N,L_{\tau}/L,B)$-modules
    \begin{equation*}
        D^{\lambda,\rho_{\tau}}_{\textnormal{st},L_{\tau},\overline{\rho}}\otimes_{R^{\lambda,\rho_{\tau}}_{\overline{\rho}}[1/p],x}B\cong (\rho_x\otimes_{\mathbf{Q}_p}B_{\textnormal{st}})^{G_{L_{\tau}}}=D_{\textnormal{st}}^{L_{\tau}}(\rho_x).
    \end{equation*}
\end{Th}
\begin{proof}
    The first part, besides the reducedness, is \cite{Kis07}, Theorem 2.7.6. The reducedness follows from \cite{BG19}, Theorem 3.3.3.

The second part is implicit in \cite{Kis07} and follows from \textit{loc. cit.}, Theorem 2.5.5 and Proposition 2.7.2. Namely, by \textit{loc. cit.} Theorem 2.5.5, there is a projective $R^{\lambda,\rho_{\tau}}_{\overline{\rho}}[1/p] \otimes_{\mathbf{Q}_p}  L_{\tau,0}$-module $D^{\lambda,\rho_{\tau}}_{\textnormal{st},L_{\tau},\overline{\rho}}$ of rank $n$ with a Frobenius semilinear automorphism $\varphi$, and an $R^{\lambda,\rho_{\tau}}_{\overline{\rho}}[1/p] \otimes_{\mathbf{Q}_p}  L_{\tau,0}$-linear endomorphism $N$ such that $N\varphi =p\varphi N$. Moreover, for any $x:R^{\lambda,\rho_{\tau}}_{\overline{\rho}}\to B$ as in the statement, there is a canonical isomorphism
\begin{equation}\label{eq5.1}
D^{\lambda,\rho_{\tau}}_{\textnormal{st},L_{\tau},\overline{\rho}}\otimes_{R^{\lambda,\rho_{\tau}}_{\overline{\rho}}[1/p],x}B\cong (\rho_x\otimes_{\mathbf{Q}_p}B_{\textnormal{st}})^{G_{L_{\tau}}}=D_{\textnormal{st}}^{L_{\tau}}(\rho_x).
    \end{equation}
    respecting $\varphi$ and $N$. Furthermore, the isomorphism above is induced by the $R^{\lambda,\rho_{\tau}}_{\overline{\rho}}[1/p] \otimes_{\mathbf{Q}_p}  L_{\tau,0}$-linear isomorphism
    \begin{equation}\label{eq5.2}
        D_{\textnormal{st},L_{\tau},\overline{\rho}}^{\lambda,\rho_{\tau}}\xrightarrow{\sim}(\rho^{\textnormal{univ}}\otimes_{R^{\lambda,\rho_{\tau}}_{\overline{\rho}}[1/p]}B_{\textnormal{st},R^{\lambda,\rho_{\tau}}_{\overline{\rho}}[1/p]})^{G_{L_{\tau}}}
    \end{equation}
    of \cite{Kis07}, Proposition 2.7.2. The RHS of \ref{eq5.2} admits a natural $L_{\tau,0}$-semilinear and $R^{\lambda,\rho_{\tau}}_{\overline{\rho}}$-linear action of $\textnormal{Gal}(L_{\tau}/L)$, and this action commutes with $\varphi$ and $N$, equipping $D_{\textnormal{st},L_{\tau},\overline{\rho}}^{\lambda,\rho_{\tau}}$ with the structure of a $(\varphi,N,L_{\tau}/L,R^{\lambda,\rho_{\tau}}_{\overline{\rho}}[1/p])$-module. Therefore, by definition, \ref{eq5.1} will also respect the $\textnormal{Gal}(L_{\tau}/L)$-action.
\end{proof}

Finally, we consider a slight refinement of Kisin's deformation rings. Set $R^{\lambda,\preceq \tau}_{\overline{\rho}}$ be the the $\mathcal{O}$-flat reduced quotient of $R^{\lambda,\rho_{\tau}}_{\overline{\rho}}$ corresponding to the Zariski closure of
\begin{equation*}
    S_{\preceq \tau}:=\{x\in \textnormal{m-Spec}(R^{\lambda,\rho_{\tau}}_{\overline{\rho}}[1/p])\mid \textnormal{WD}(\rho_x)|_{I_L}\preceq \tau\}=
\end{equation*}
\begin{equation*}
    \{x:R^{\lambda,\rho_{\tau}}_{\overline{\rho}}[1/p]\to E_x\mid E_x/E\textnormal{ finite extension, } \textnormal{WD}(\rho_x)|_{I_L}\preceq \tau\}
\end{equation*}
in $\textnormal{Spec}(R^{\lambda,\rho_{\tau}}_{\overline{\rho}}[1/p])$. We check the only property of $R^{\lambda,\preceq \tau}_{\overline{\rho}}$ that  we will need.
\begin{Prop}
    An $E$-algebra map $x:R^{\lambda,\rho_{\tau}}_{\overline{\rho}}\to E_x$ for a finite field extension $E_x/E$ factors through $R^{\lambda,\preceq \tau}_{\overline{\rho}}$ if and only if $\textnormal{WD}(\rho_x)|_{I_L}\preceq \tau$.
\end{Prop}
\begin{proof}
    Note that it suffices to check that $S_{\preceq\tau}\subset \textnormal{m-Spec}(R^{\lambda,\rho_{\tau}}_{\overline{\rho}}[1/p])$ is Zariski closed. Indeed, as then, $R^{\lambda,\rho_{\tau}}_{\overline{\rho}}$ being a Jacobson ring, the set of closed points of the Zariski closure of $S_{\preceq\tau}$ in $\textnormal{Spec}(R^{\lambda,\rho_{\tau}}_{\overline{\rho}}[1/p])$ must be $S_{\preceq \tau}$ itself (cf. \cite[\href{https://stacks.math.columbia.edu/tag/005Z}{Lemma 005Z}]{stacks-project}).

    We further note that we are free to enlarge $E$ as one checks that $R_{\overline{\rho},\mathcal{O}}^{\lambda,\rho_{\tau}}\otimes_{\mathcal{O}}\mathcal{O}_{E'}\cong R^{\lambda,\rho_{\tau}}_{\overline{\rho},\mathcal{O}_{E'}}$ for any finite extension $E'/E$. In particular, we may assume that the isotypic decomposition
    \begin{equation*}
        V_{\rho_{\tau}}\cong\oplus_{\theta}V_{\rho_{\tau}}[\theta]
    \end{equation*}
    is defined over $E$ where $V_{\rho_{\tau}}$ is the representation space of $\rho_{\tau}$ and the sum runs over absolutely irreducible $E$-representations of $I_{L_{\tau}/L}$.

    According to Theorem~\ref{Th5.1}, we have a rank $n$ $(\varphi,N,L_{\tau}/L,R^{\lambda,\rho_{\tau}}_{\overline{\rho}}[1/p])$-module $D:=D_{\textnormal{st},L_{\tau},\overline{\rho}}^{\lambda,\rho_{\tau}}$. In particular, for a fixed $\iota_0:L_{\tau,0}\hookrightarrow E$, we can introduce
    \begin{equation*}
        W:=D\otimes_{R^{\lambda,\rho_{\tau}}_{\overline{\rho}}[1/p]\otimes_{\mathbf{Q}_p}L_{\tau,0},\iota_0}R^{\lambda,\rho_{\tau}}_{\overline{\rho}}[1/p],
    \end{equation*}
    a projective rank $n$ $R^{\lambda,\rho_{\tau}}_{\overline{\rho}}[1/p]$-module with a linear endomorphism $N$, and a commuting linear action of $I_{L_{\tau}/L}$. For any ring map $R^{\lambda,\rho_{\tau}}_{\overline{\rho}}[1/p]\to R$, we use the abbreviation $W_R:=W\otimes_{R^{\lambda,\rho_{\tau}}_{\overline{\rho}}[1/p]}R$. For a closed point $x:R^{\lambda,\rho_{\tau}}_{\overline{\rho}}[1/p]\to E_x$, we further denote by $(W_x,N_x)$ the corresponding specialisation and note that it is isomorphic to $\textnormal{WD}(\rho_x)|_{I_L}$.

    We now pick an irreducible component $Z=\textnormal{Spec}(A)\hookrightarrow \textnormal{Spec}(R^{\lambda,\rho_{\tau}}_{\overline{\rho}}[1/p])$ and prove the following lemma from which the proposition follows.

\begin{Lemma}
    There is a Weil--Deligne inertial type of the form $\tau'=(\rho_{\tau},N_{\tau'})$ such that the locus on $\textnormal{m-Spec}(A)$ where $\textnormal{WD}(\rho_x)|_{I_L} \sim  \tau'$ is open and dense. Moreover, the locus on $\textnormal{m-Spec}(A)$ where $\textnormal{WD}(\rho_x)|_{I_L}\preceq\tau$  is Zariski closed and only non-empty if $\tau\preceq \tau'$.
\end{Lemma}
\begin{proof}
    To prove the first statement, set $F=\textnormal{Frac}A$, and consider the base change $W_F$, a finite free $F$-module of rank $n$. Consider the decomposition into isotypic components
    \begin{equation*}
        W_F=\oplus_{\theta} W_F[\theta]
    \end{equation*}
    with respect to the action of the inertia subgroup. Set $d_{\theta}=\dim_FW_F[\theta]=\dim_EV_{\rho}[\theta]$.

    Consider the intersection $W_A[\theta]:=W_A\cap W_F[\theta]$ and note that it is preserved by $N$ and the action of $I_{L_{\tau}/L}$ as it holds for both members of the intersection. Moreover, we have the equality $W_A[\theta]\otimes_A F=W_F[\theta]$. In particular, $W_A[\theta]$ is generically both a free module of rank $d_{\theta}$ and a direct summand of $W_A$. In other words, we can find $f'_{\theta}$ such that the base change $W_{A[1/f'_{\theta}]}[\theta]=W_A[\theta]\otimes_{A}A[1/f'_{\theta}]$ is finite free over $A[1/f'_{\theta}]$ and is a direct summand of $W_{A[1/f'_{\theta}]}$. Again, it is preserved by the monodromy and the action of the inertia subgroup. Denote the restriction of $N$ to the obtained subspace by $N[\theta]$. 
    
    Now we can apply Lemma 7.8.10 of \cite{BC09} to get $f''_{\theta}\in A$ such that, for $f_{\theta}=f'_{\theta}f''_{\theta}$, the equivalence class of $N[\theta]_x$ is independent of $x\in \textnormal{m-Spec}(A[1/f_{\theta}])$. The first part of the lemma follows by noting that $N[\theta]_x=N_x[\theta]$, the latter denoting the restriction of the specialisation of $N_x$ to $W_x[\theta]$ where $W_x\sim \textnormal{WD}(\rho_x)|_{I_L}$. Denote the resulting Weil--Deligne inertial type by $\tau'=(\rho_{\tau},N_{\tau'})$.

    \vspace{2mm}

    Now consider an arbitrary point $x\in \textnormal{m-Spec}(A)$ with corresponding maximal ideal $\mathfrak{m}_x\subset A$, Galois representation $\rho_x$ and residue field $E_x$, a finite extension of $E$. We prove that $\textnormal{WD}(\rho_x)|_{I_L}\preceq \tau'$. Consider the localisation $W_{\mathfrak{m}_x}:=W_{A_{\mathfrak{m}_x}}$, and its base change $W_{\mathfrak{m}_x}^{\wedge}$ along $A_{\mathfrak{m}_x}\to A_{\mathfrak{m}_x}^{\wedge}$. Write $F_x:= \textnormal{Frac}A^{\wedge}_{\mathfrak{m}_x}$ and $W_{\mathfrak{m}_x}^{\wedge}[\theta]:= W_{\mathfrak{m}_x}^{\wedge}\cap W_{F_x} [\theta]$.
    
    Since $I_{L_{\tau}/L}$ is a finite group, it has trivial higher cohomology in characteristic $0$. In particular, the deformation $\rho_{A_{\mathfrak{m}_x}^{\wedge}}:I_{L_{\tau}/L}\to \textnormal{End}_{A_{\mathfrak{m}_x}^{\wedge}}(W_{\mathfrak{m}_x}^{\wedge})$ of $\textnormal{WD}(\rho_x)^{ss}\mid_{I_{L}}$ must be isomorphic to the extension of scalars of the latter along $E_x\to A_{\mathfrak{m}_x}^{\wedge}$. In other words, we have an isomorphism $W_{\mathfrak{m}_x}^{\wedge}\cong W_x\otimes_{E_x}A_{\mathfrak{m}_x}^{\wedge}$ of $I_{L_{\tau}/L}$-representations. In particular, $W_{\mathfrak{m}_x}^{\wedge}[\theta]\subset W_{\mathfrak{m}_x}^{\wedge}$ is a free direct summand for all isotypic component. 

    Set $N_{A_{\mathfrak{m}_x}^{\wedge}}[\theta]=N_{A_{\mathfrak{m}_x}^{\wedge}}|_{W_{\mathfrak{m}_x}^{\wedge}[\theta]}$. Note that $N_{A_{\mathfrak{m}_x}^{\wedge}}[\theta]\otimes_{A_{\mathfrak{m}_x}^{\wedge}}F_x=N_{F_x}[\theta]\sim N_F[\theta]\sim N_{\tau'}[\theta]$. Moreover, $N_{A_{\mathfrak{m}_x}^{\wedge}}[\theta]\otimes_{A_{\mathfrak{m}_x}^{\wedge}}E_x=N_x[\theta]$. Therefore, using Proposition~\ref{Prop2.10} one sees that $N_x[\theta]\preceq N_{\tau'}[\theta]$ (cf. \cite{Dot21}, Lemma 4.2 and Remark 4.4). Running over all isotypic components, we obtain $\textnormal{WD}(\rho_x)|_{I_L}(\sim(W_x,N_x))\preceq \tau'$.

    We have proved so far that $\textnormal{Spec}(A)$ contains a dense open $D(\prod_{\theta}f_{\theta})$ with closed points satisfying $\textnormal{WD}(\rho_x)|_{I_{L}}\sim\tau'$ and that, for every $x\in \textnormal{m-Spec}(A)$ we have $\textnormal{WD}(\rho_x)|_{I_{L}}\preceq\tau'$. Set $\textnormal{Spec}(A^{\preceq \tau})\subset \textnormal{Spec}(A)$ to be an irreducible component of the Zariski closure of
    \begin{equation*}
        S_{\preceq \tau}=\{x\in \textnormal{m-Spec}(A)| \textnormal{WD}(\rho_x)\mid_{I_L}\preceq \tau\}
    \end{equation*}
    in $\textnormal{Spec}(A)$ with its induced reduced subscheme structure. To conclude, it suffices to see that every closed point of $\textnormal{Spec}(A^{\preceq\tau})$ is of type $\preceq \tau$.

    By considering the base change $W_{A_{\preceq\tau}}$ with the induced action of the inertia subgroup and nilpotent operator, and running the previous argument, we see that there is a Weil--Deligne inertial type $\widetilde{\tau}$ and an open dense subspace $\widetilde{U}\subset \textnormal{Spec}(A^{\preceq \tau})$ such that 
    \begin{enumerate}
        \item every closed point in $\widetilde{U}$ is of type $\sim \widetilde{\tau}$, and
        \item every closed point in $\textnormal{Spec}(A^{\preceq\tau})$ is of type $\preceq \widetilde{\tau}$.
    \end{enumerate}
    In particular, we must have $\widetilde{\tau} \preceq \tau$ and the proof is finished.
\end{proof}
\end{proof}

\subsection{Interpolation of local Langlands}
An essential ingredient to formulate torsion local-global compatibility is the existence of a (semisimple) local Langlands correspondence over the generic fiber of our potentially semistable deformation rings. Such a correspondence was already defined in \cite{CEGGPS16} for the potentially crystalline deformation rings of Kisin. However, as it is shown in \cite{Pyv20}, one can define it for the whole potentially semistable deformation ring following the same strategy. Let $\tau$ be a Weil--Deligne inertial type, $\Omega$ be the corresponding Bernstein block, and $\overline{\rho}:G_L\to \textnormal{GL}_n(k)$ be a continuous Galois representation as in the previous subsection.
\begin{Th}
    There is an $E$-algebra homomorphism
    \begin{equation*}
        \eta:\mathfrak{z}_{\Omega}\to R^{\lambda,\rho_{\tau}}_{\overline{\rho}}[1/p]
    \end{equation*}
    such that, for any $x\in \textnormal{m-Spec}(R_{\overline{\rho}}^{\lambda,\rho_{\tau}}[1/p])$ with residue field $E_x$ (necessarily a finite extension of $E$), the character
    \begin{equation*}
        x\circ \eta:\mathfrak{z}_{\Omega}\to E_x
    \end{equation*}
    coincides with the one induced by the natural action of $\mathfrak{z}_{\Omega}$ on $\textnormal{rec}^{-1}(\textnormal{WD}(\rho_x)^{F-ss})\otimes |\det|^{\frac{n-1}{2}}_L$.
    In particular, we also have a map $\eta:\mathfrak{z}_{\Omega}\to R^{\lambda,\preceq\tau}_{\overline{\rho}}[1/p]$ with the same property.
\end{Th}
\begin{proof}
    The result follows from (the proof of) \cite{Pyv20}, Theorem 3.3. We note that he states the existence of a map with source being a Hecke algebra $\mathcal{H}(\sigma_{\textnormal{min}})$ for a certain type $\sigma_{\textnormal{min}}$ of Schneider--Zink. This Hecke algebra in fact is isomorphic to $\mathfrak{z}_{\Omega}$ and, in the proof of Theorem 3.3 of \textit{loc. cit.}, the map is first constructed from $\mathfrak{z}_{\Omega}$ and it is only at the very end that it is precomposed with the identification $\mathcal{H}(\sigma_{\textnormal{min}})\cong \mathfrak{z}_{\Omega}$.

    The rough sketch is as follows. The two key ingredients in the proof are:
    \begin{enumerate}
        \item The existence of a pseudo-representation $T:W_L\to \mathfrak{z}_{\Omega}$ constructed by Chenevier interpolating local Langlands (cf. \cite{Che09}, Proposition 3.11);
        \item and the existence of a universal $(\varphi,N,L_{\tau}/L,E)$-module over $R_{\overline{\rho}}^{\lambda,\rho_{\tau}}[1/p]\otimes_{\mathbf{Q}_p}L_{\tau,0}$ (cf. Theorem~\ref{Th5.1}).
    \end{enumerate}

    Then, for $w\in W_L$ it is clear that $\eta(T(w))$ should be (some normalisation of) the trace of $w$ acting on the universal $(\varphi,N,L_{\tau}/L,E)$-module. In particular, we have defined $\eta$ for every element in the image of $T$. Moreover, \cite{CEGGPS16}, Lemma 4.5 shows that $\mathfrak{z}_{\Omega}$ is generated (as an $E$-vector space) by the image of $T$. Therefore, we can linearly extend the domain of $\eta$ from the image of $T$ to the whole Bernstein centre.
\end{proof}

\begin{Rem}\label{rem5.3}
    It is interesting to investigate on a possible integral avatar of $\eta$. For instance, one can ask under what circumstance does $\eta$ send $\mathfrak{z}_{\lambda,\tau}^{\circ}:= Z(\mathcal{H}(\sigma(\lambda,\tau)^{\circ}))\cap \mathfrak{z}_{\Omega}$ into $R_{\overline{\rho}}^{\lambda,\preceq \tau}$? Moreover, is there a subring in $R_{\overline{\rho}}^{\lambda,\preceq \tau}[1/p]$, say finite over $R_{\overline{\rho}}^{\lambda,\preceq \tau}$, such that $\mathfrak{z}_{\lambda,\tau}^{\circ}$ always lands in this ring? This turns out to be a rather subtle question, and the answer in most cases depends on some folklore conjectures. More precisely, in \cite{CEGGPS16} they remark that, at least in the potentially crystalline case, $\eta(\mathfrak{z}_{\lambda,\tau}^{\circ})$ should land in the normalisation of $R_{\overline{\rho}}^{\lambda,\preceq \tau}$ inside its generic fiber. This prediction is supported by what they see after patching (cf. \textit{loc. cit.} Lemma 4.18, 3), and Remark 4.21). However, their observation can only be made into a rigorous proof for the components of the local deformation rings for which we know automorphy. 
    
    In some very special cases, it seems plausible that one can prove that $\eta$ sends $\mathfrak{z}_{\lambda,\tau}^{\circ}$ into the normalisation by other means. Namely, when $n=2$, $\rho_{\tau}$ is a tame type, $\tau=(\rho_{\tau},0)$, and we look at the potentialy Barsotti--Tate deformation ring $R^{0,\preceq\tau}_{\Bar{\rho}}$, one can appeal to the work of Caraiani--Emerton--Gee--Savitt to realise the normalisation of the local deformation ring as the direct image of the Kisin variety\footnote{These spaces were first constructed in \cite{Kis09b}, in the Barsotti--Tate case, and have been constructed over any potentially Barsotti--Tate deformation ring associated with a tame type in \cite{CEGS21}, \S5 (where it is denoted by $X_{\overline{r}}$).} (parametrising Breuil--Kisin models of the deformations of $\overline{\rho}$) along its map towards $\textnormal{Spec}(R_{\Bar{\rho}}^{0,\preceq\tau})$. Moreover, using the construction of the map, the fact that the Kisin variety parametrises Breuil--Kisin models of the deformations appearing in $\textnormal{Spec}(R_{\Bar{\rho}}^{0,\preceq\tau})$, and comparisons between Breuil--Kisin modules and Weil--Deligne representations associated to the corresponding Galois representations, it seems possible to show that the image of $\mathfrak{z}_{\lambda,\tau}^{\circ}$ already lies in the direct image of the Kisin variety.
    
    There is some further evidence already present in the literature. Namely, \cite{CEGGPS16b} Lemma 2.15 (under some further assumptions on $\overline{\rho}$) shows that when $n=2$, $\tau=(\mathbf{1},0)$, $L=\mathbf{Q}_p$ and $\lambda$ is in the Fontaine--Laffaille range, $\eta$ sends $\mathfrak{z}_{\lambda,\tau}^{\circ}$ into $R^{\lambda,\preceq\tau}$. Moreover, up to a suitable completion of $\mathfrak{z}_{\lambda,\tau}^{\circ}$, the integral avatar of $\eta$ becomes an isomorphism. Note that in this case $R_{\overline{\rho}}^{\lambda,\preceq\tau}$ is a regular local ring hence normal. In particular, their result is compatible with the prediction of \cite{CEGGPS16}, Remark 4.21.
\end{Rem}
\subsection{The local-global compatibility conjecture}\label{sec5.3}
Fix a dominant weight $\lambda\in (\mathbf{Z}_+^n)^{\Hom(F,E)}$ for $G$, and a collection of $n$-dimensional Weil--Deligne inertial types $\tau=\{\tau_v\}_{v\in S_p(F)}$. Set $\lambda_v:=(\lambda_{\iota})_{\iota:F_v\hookrightarrow E}\in (\mathbf{Z}_+^n)^{\Hom(F_v,E)}$. Recall that we introduced the abstract spherical Hecke algebra $\mathbf{T}^T$ and the corresponding faithful Hecke algebra $\mathbf{T}^{T}(K,\lambda,\tau)$ acting on $R\Gamma(X_K,\mathcal{V}_{(\lambda,\tau)})$. We further introduce a refinement of the abstract Hecke algebra that also consists of Hecke operators at $p$. To do this, note that since $\textnormal{c-Ind}_{\textnormal{GL}_n(\mathcal{O}_{F_v})}^{\textnormal{GL}_n(F_v)}\sigma(\lambda_v,\tau_v)^{\circ}$ is finitely generated as an $\mathcal{O}[\textnormal{GL}_n(F_v)]$-module, we have $\mathcal{H}(\sigma(\lambda_v,\tau_v)^{\circ})[1/p]=\mathcal{H}(\sigma(\lambda_v,\tau_v))$. We have an identification 
\begin{equation*}
    \left(\textnormal{c-Ind}_{\textnormal{GL}_n(\mathcal{O}_{F_v})}^{\textnormal{GL}_n(F_v)}\sigma(\tau_v)\right)\otimes_E V_{\lambda_v^{\vee}} \cong \textnormal{c-Ind}_{\textnormal{GL}_n(\mathcal{O}_{F_v})}^{\textnormal{GL}_n(F_v)}\sigma(\lambda_v,\tau_v),
\end{equation*}
\begin{equation*}
    f\otimes w\mapsto [g\mapsto f(g)V_{\lambda_v^{\vee}}(g)w].
\end{equation*}
We can then define a natural map
\begin{equation*}
    \mathcal{H}(\sigma(\tau_v))\to \mathcal{H}(\sigma(\lambda_v,\tau_v)),
\end{equation*}
\begin{equation*}
    \phi\mapsto \phi\otimes \textnormal{id}_{V_{\lambda^{\vee}}}
\end{equation*}
that, thanks to \cite{ST06}, Lemma 1.4,
is an isomorphism of $E$-algebras. In particular, we have $\mathfrak{z}_{\tau_v}=Z(\mathcal{H}(\sigma(\lambda_v,\tau_v)^{\circ}))[1/p]$. Let $\Omega_v$ denote the Bernstein block corresponding to $\tau_v$ and set $\mathfrak{z}_{\Omega_v}$ to be the corresponding Bernstein centre. We then set
\begin{equation*}
    \mathfrak{z}_{\lambda_v,\tau_v}^{\circ}:=Z(\mathcal{H}(\sigma(\lambda_v,\tau_v)^{\circ}))\cap \mathfrak{z}_{\Omega_v}\subset \mathfrak{z}_{\tau_v},
\end{equation*}
a commutative $\mathcal{O}$-subalgebra and note that $\mathfrak{z}_{\lambda_v,\tau_v}^{\circ}[1/p]=\mathfrak{z}_{\Omega_v}$.
Finally, set
\begin{equation*}
    \mathbf{T}^{T,\lambda,\tau}:=\mathbf{T}^T\otimes_{\mathcal{O}}(\bigotimes_{v\in S_p(F)}\mathfrak{z}_{\lambda_v,\tau_v}^{\circ}),
\end{equation*}
a commutative $\mathcal{O}$-algebra. As a consequence of Lemma~\ref{Lem2.4}, Frobenius reciprocity, and \cite{NT16}, Lemma 3.11, we get a map of $\mathcal{O}$-algebras
\begin{equation*}
    \mathbf{T}^{T,\lambda,\tau}\to \textnormal{End}_{D^+(\mathcal{O})}(R\Gamma(X_K,\mathcal{V}_{(\lambda,\tau)}))
\end{equation*}
for every good subgroup $K\subset G(\mathbf{A}_{F^+}^{\infty})$ with $K_p=\prod_{\Bar{v}\in \overline{S}_p}G(\mathcal{O}_{F^+_{\Bar{v}}})$. Denote by $\mathbf{T}^{T,\lambda,\tau}(K^p)$ the corresponding faithful Hecke algebra $\mathbf{T}^{T,\lambda,\tau}(R\Gamma(X_K,\mathcal{V}_{(\lambda,\tau)}))$. Since $R\Gamma(X_K,\mathcal{V}_{(\lambda,\tau)})$ is a perfect complex in $D^+(\mathcal{O})$, $\mathbf{T}^{T,\lambda,\tau}(K^p)$ is a commutative finite $\mathcal{O}$-algebra. In particular, we obtain a decomposition
\begin{equation*}
    \mathbf{T}^{T,\lambda,\tau}(K^p)=\prod_{\mathfrak{m}}\mathbf{T}^{T,\lambda,\tau}(K^p)_{\mathfrak{m}}
\end{equation*}
where we run over all maximal ideals of $\mathbf{T}^{T,\lambda,\tau}(K^p)$. Therefore, for each $\mathfrak{m}$, we have a corresponding $\mathbf{T}^{T,\lambda,\tau}$-equivariant direct summand $R\Gamma(X_K,\mathcal{V}_{(\lambda,\tau)})_{\mathfrak{m}}$ of $R\Gamma(X_K,\mathcal{V}_{(\lambda,\tau)})$ and the natural map $\mathbf{T}^{T,\lambda,\tau}(K^p)_{\mathfrak{m}}\to \mathbf{T}^{T,\lambda,\tau}(R\Gamma(X_K,\mathcal{V}_{(\lambda,\tau)})_{\mathfrak{m}})$ becomes an isomorphism (cf. \cite{NT16}, \S3.2).

We recall the following conjecture (cf. \cite{CG18}, Conjecture B).
\begin{Conj}[Construction of torsion Galois representations]\label{Conj5.4}Let $K\subset G(\mathbf{A}_{F^+}^{\infty})$ be any good subgroup and $\mathfrak{m}\subset \mathbf{T}^T(K,\lambda,\tau)=\mathbf{T}^T(R\Gamma(X_K,\mathcal{V}_{(\lambda,\tau)})_{\mathfrak{m}})$ be a maximal ideal. Then there exists a continuous semisimple Galois representation
\begin{equation*}
    \overline{\rho}_{\mathfrak{m}}:G_{F,T}\to \textnormal{GL}_n(\mathbf{T}^T(K,\lambda,\tau)/\mathfrak{m})
\end{equation*}
such that, for each finite place $v\notin T$ of $F$, the characteristic polynomial of $\overline{\rho}_{\mathfrak{m}}(\textnormal{Frob}_v)$ is equal to the image of $P_v(X)$ in $(\mathbf{T}^T(K,\lambda,\tau)/\mathfrak{m})[X]$.

Moreover, if $\overline{\rho}_{\mathfrak{m}}$ is absolutely irreducible, then it admits a lift to a continuous homomorphism
\begin{equation*}
    \rho_{\mathfrak{m}}:G_{F,T}\to \textnormal{GL}_n(\mathbf{T}^T(K,\lambda,\tau)_{\mathfrak{m}})
\end{equation*}
such that, for each finite place $v\notin T$ of $F$, the characteristic polynomial of $\rho_{\mathfrak{m}}(\textnormal{Frob}_v)$ is equal to the image of $P_v(X)$ in $\mathbf{T}^T(K,\lambda,\tau)_{\mathfrak{m}}[X]$.
\end{Conj}
Since we have a natural map $\mathbf{T}^T(K,\lambda,\tau)\to \mathbf{T}^{T,\lambda,\tau}(K^p)$, assuming Conjecture~\ref{Conj5.4}, a maximal ideal $\mathfrak{m}\subset \mathbf{T}^{T,\lambda,\tau}(K^p)$ gives rise to a continuous semisimple Galois representation
\begin{equation*}
    \overline{\rho}_{\mathfrak{m}}:G_{F,T}\to \textnormal{GL}_n(\mathbf{T}^{T,\lambda,\tau}(K^p)/\mathfrak{m}).
\end{equation*}
Moreover, if we denote by $\mathfrak{m}^T$ the induced maximal ideal of $\mathbf{T}^T$, we have a natural inclusion
\begin{equation*}
\mathbf{T}^T(K,\lambda,\tau)_{\mathfrak{m}^T}\hookrightarrow \mathbf{T}^{T,\lambda,\tau}(K^p)_{\mathfrak{m}}.
\end{equation*}
In particular, assuming that $\overline{\rho}_{\mathfrak{m}}$ is absolutely irreducible, Conjecture~\ref{Conj5.4} provides a lift 
\begin{equation*}
    \rho_{\mathfrak{m}}:G_{F,T}\to \textnormal{GL}_n(\mathbf{T}^{T,\lambda,\tau}(K^p)_{\mathfrak{m}})
\end{equation*}
of $\overline{\rho}_{\mathfrak{m}}$ satisfying the assertion of Conjecture~\ref{Conj5.4}.

To state the local-global compatibility conjecture, we need to further introduce the subring \begin{equation*}
\mathfrak{z}_{\lambda_v,\tau_v}^{\circ,\textnormal{int}}:=\eta^{-1}(R_{\overline{\rho}_v}^{\lambda_v,\preceq \tau_v})\cap \mathfrak{z}_{\lambda_v,\tau_v}^{\circ}\subset \mathfrak{z}_{\lambda_v,\tau_v}^{\circ}.
\end{equation*}
Note that we still have the property
$\mathfrak{z}_{\lambda_v,\tau_v}^{\circ,\textnormal{int}}[1/p]=\mathfrak{z}_{\Omega_v}$.
\begin{Conj}[\textnormal{Torsion local-global compatibility at} $\ell=p$ ]\label{Conj5.5} Assume Conjecture~\ref{Conj5.4}. Let $\mathfrak{m}\subset \mathbf{T}^{T,\lambda,\tau}(K^p)$ be a non-Eisenstein maximal ideal. For $v\in S_p(F)$, set $\overline{\rho}_v:=\overline{\rho}_{\mathfrak{m}}|_{G_{F_v}}$ and $\rho_v:=\rho_{\mathfrak{m}}|_{G_{F_v}}$. Then, for any $v\in S_p(F)$, there is a (necessarily unique) dotted arrow making the following diagram commutative

$$\begin{tikzcd}
R^{\Box}_{\overline{\rho}_v} \arrow{r}{\rho_v} \arrow[two heads]{d}{}
&\mathbf{T}^{T,\lambda,\tau}(K^p)_{\mathfrak{m}} \\
R^{\lambda_v,\preceq \tau_v}_{\overline{\rho}_v} \arrow[hook]{d}{} \arrow[dotted]{ur}{} &\mathfrak{z}_{\lambda_v,\tau_v}^{\circ,\textnormal{int}}\arrow{u}[swap]{\textnormal{nat}}\arrow{l}{\eta\mid_{\mathfrak{z}_{\lambda_v,\tau_v}^{\circ,\textnormal{int}}}}\arrow[hook]{d}{}\\
R^{\lambda_v,\preceq\tau_v}_{\overline{\rho}_v}[1/p]&\arrow{l}{\eta} \mathfrak{z}_{\Omega_v}.
\end{tikzcd}$$
\end{Conj}
\begin{Rem}
In the diagram above, we denote by $\textnormal{nat}$ the canonical map coming from the fact that the target of the map is a faithful Hecke algebra.
    We also note that the complication of introducing the ring $\mathfrak{z}_{\lambda_v,\tau_v}^{\circ,\textnormal{int}}$ originates in the fact the $\eta$ does not necessarily send $\mathfrak{z}_{\lambda_v,\tau_v}^{\circ}$ into $R_{\overline{\rho}_v}^{\lambda_v,\preceq \tau_v}$ (see Remark~\ref{rem5.3}). We could also, instead of using the lower part of the diagram, simply just ask that whenever $z\in \mathfrak{z}_{\lambda_v,\tau_v}^{\circ}$ such that $\eta(z)$ lies in $R^{\lambda_v,\preceq\tau_v}_{\overline{\rho}_v}$, then $\textnormal{nat}(z)$ coincides with $\rho_v\circ\eta(z)$.

    Note that since we have $\mathfrak{z}_{\lambda_v,\tau_v}^{\circ,\textnormal{int}}[1/p]=\mathfrak{z}_{\Omega_v}$, whenever we specialise to an $E$-algebra $\mathbf{T}^{T,\lambda,\tau}(K^p)_{\mathfrak{m}}\to B$, we get that the induced natural $\mathfrak{z}_{\Omega_v}$-action coincides with the one induced by $\eta$.
\end{Rem}

We finally state our theorem that settles a large part of Conjecture~\ref{Conj5.5} in the case when $F$ is assumed to be an imaginary CM field.
\begin{Th}\label{Th5.7}
    Let $F$ be an imaginary CM field that contains an imaginary quadratic field $F_0$ in which $p$ splits. Assume that $T$ is stable under complex conjugation and satisfies:
    \begin{itemize}
        \item Let $v\notin T$ be a finite place of $F$, with residue characteristic $\ell$. Then either $T$ contains no $\ell$-adic places and $\ell$ is unramified in $F$, or there exists an imaginary quadratic subfield of $F$ in which $\ell$ splits.
    \end{itemize}
Let $K\subset G(\mathbf{A}_F^{\infty})$ be a good compact open subgroup with $K^T$ hyperspecial. Fix a place $\Bar{v}\in S_p(F^+)$. Let $\mathfrak{m}\subset \mathbf{T}^{T,\lambda,\tau}(K^p)$ be a non-Eisenstein maximal ideal. Assume that:
\begin{enumerate}
    \item There is a place $\Bar{v}'\in S_p(F^+)$ such that $\Bar{v}\neq \Bar{v}'$ and
    \begin{equation*}
        \sum_{\Bar{v}''\neq \Bar{v},\Bar{v}'}[F^+_{\Bar{v}''}:\mathbf{Q}_p]\geq \frac{1}{2}[F^+:\mathbf{Q}]
    \end{equation*}
    where the sum runs over $\Bar{v}''\in S_p(F^+)$;
    \item and $\overline{\rho}_{\mathfrak{m}}$ is decomposed generic.
\end{enumerate}
Then there exists an integer $N\geq 1$, depending only on $n$ and $[F^+:\mathbf{Q}]$, a nilpotent ideal $I\subset \mathbf{T}^{T,\lambda,\tau}(K^p)_{\mathfrak{m}}$ with $I^N=0$ and a continuous homomorphism
\begin{equation*}
    \rho_{\mathfrak{m}}:G_{F,T}\to \textnormal{GL}_n(\mathbf{T}^{T,\lambda,\tau}(K^p)_{\mathfrak{m}}/I)
\end{equation*}
lifting $\overline{\rho}_{\mathfrak{m}}$ such that
\begin{itemize}
    \item for each finite place $v\notin T$ of $F$, the characteristic polynomial of $\rho_{\mathfrak{m}}(\textnormal{Frob}_v)$ is equal to the image of $P_v(X)$ in $\mathbf{T}^{T,\lambda,\tau}(K^p)[X]$. 
\end{itemize}
Moreover, if, for $v|\Bar{v}$, we set $\overline{\rho}_v:=\overline{\rho}_{\mathfrak{m}}|_{G_{F_v}}$ and $\rho_v:=\rho_{\mathfrak{m}}|_{G_{F_v}}$, there is a (necessarily unique) dotted arrow making the following diagram commutative
    $$\begin{tikzcd}
R^{\Box}_{\overline{\rho}_v} \arrow{r}{\rho_v} \arrow[two heads]{d}{}
&\mathbf{T}^{T,\lambda,\tau}(K^p)_{\mathfrak{m}}/I \\
R^{\lambda_v,\preceq \tau_v}_{\overline{\rho}_v} \arrow[hook]{d}{} \arrow[dotted]{ur}{} &\mathfrak{z}_{\lambda_v,\tau_v}^{\circ,\textnormal{int}}\arrow{u}[swap]{\textnormal{nat}}\arrow{l}{\eta\mid_{\mathfrak{z}_{\lambda_v,\tau_v}^{\circ,\textnormal{int}}}}\arrow{u}[swap]{\textnormal{nat}}\arrow[hook]{d}{}\\
R^{\lambda_v,\preceq\tau_v}_{\overline{\rho}_v}[1/p]&\arrow{l}{\eta} \mathfrak{z}_{\Omega_v}.
\end{tikzcd}$$
\end{Th}
\begin{Rem}
    We say a few words about the assumptions appearing in Theorem~\ref{Th5.7}.

    The assumption on $T$ is already present in \cite{Sch15} and it makes sure that, up to nilpotent ideal, Conjecture~\ref{Conj5.4} is known to be true (see Theorem~\ref{Th2.20} and Theorem~\ref{Th2.22}). It is a rather mild condition that can always be fulfilled after enlarging $T$.

    Assumption (i) is a lot more serious and excludes the case of imaginary quadratic fields. Its appearance originates in the use of the Fontaine--Laffaille style degree shifting argument (cf. \cite{CN23}, Lemma 4.2.5).

    Finally, (ii) is again essential for our methods to work. It ensures that we can appeal to the vanishing results of Caraiani--Scholze. To our knowledge this is the only known way of producing the necessary congruences between cusp forms for $U(n,n)_{F^+}$ and Eisenstein series coming from cusp forms for $\textnormal{GL}_{n,F}$.
\end{Rem}

\begin{Rem}
    Although we don't discuss it here, one could possibly formulate a more general conjecture treating the case when our choice of parabolic subgroup $Q_p=\prod_v Q_v\subset \prod_{v}\textnormal{GL}_{n,F_v}$ is not the trivial one, and the complexes appearing are $Q_p$-ordinary. However, to do this, one would need to introduce the appropriate potentially semistable $Q_v$-ordinary deformation rings. This should be possible using ideas of \cite{Ger18}, \S3.3.

    Nevertheless, in the next section we will prove a general result (Proposition~\ref{Prop6.6}) that has the potential to settle the more general conjecture too, once it's formulated.
\end{Rem}
\section{Proof of Theorem~\ref{Th5.7}}
In this section, we provide a proof of Theorem~\ref{Th5.7} refining the argument of \cite{CN23}. In fact, we prove a local-global compatibility result for the more general $Q_p$-ordinary complexes considered in \S\ref{sec3.2}. For the rest of the section, we fix the following setup. Let $n\geq 2$ be an integer, $F$ be an imaginary $CM$ field, and $T\subset S(F)$ as in Theorem~\ref{Th5.7}. In particular, we assume Assumption~\ref{blanket}. Fix a choice of lift $v|\Bar{v}$ in $F$ for each $\Bar{v}\in \overline{S}_p$. We consider the corresponding groups $G$, $P$, and $\widetilde{G}$ as before. Moreover, we fix a tuple $(Q_p,\lambda,\underline{\tau})=(Q_v,\lambda_v,\underline{\tau_v})_{v\in S_p(F)}$ as in \S\ref{sec2.7}.
\subsection{Degree shifting}\label{sec6.1}
One essential step in executing the strategy laid out in \cite{ACC23} is what they call the "degree shifting argument". The most robust version of which appears in \cite{CN23}. In fact, besides some enrichment of the method, their argument is already sufficient for our proof.
We first state the precise statement we need and then indicate the changes to the argument of \cite{CN23} one needs to prove our generalisation of \textit{loc. cit.} Proposition 4.2.6.

Fix a place $\Bar{v}\in \overline{S}_p$, and assume that, under the identification of \S\ref{sec2.7} $(\lambda_v,\lambda_{v^c})$ corresponds to a dominant weight $\Tilde{\lambda}_{\Bar{v}}=(-w_0^{\textnormal{GL}_n}\lambda_{v^c},\lambda_v)\in (\mathbf{Z}_+^{2n})^{\Hom(F^+_{\Bar{v}},E)}$. Extend it to an arbitrary dominant weight $\widetilde{\lambda}$ for $\widetilde{G}$. Set $Q_{\Bar{v}}:=Q_v\times Q_{v^c}\subset G_{F^{+}_{\Bar{v}}}=\textnormal{GL}_{n,F_v}\times \textnormal{GL}_{n,F_{v^c}}$, and consider the standard parabolic subgroup $\widetilde{Q}^{w_0}_{\Bar{v}}=\widetilde{M}^{w_0}_{\Bar{v}}\ltimes \widetilde{N}^{w_0}_{\Bar{v}}\subset \widetilde{G}_{F^+_{\Bar{v}}}$ associated to it (cf. \S\ref{sec2.7}). We then obtain a tuple $(\widetilde{Q}^{w_0}_{\Bar{v}},\Tilde{\lambda}_{\Bar{v}},\underline{\tau_{\Bar{v}}})$. Write $Q_v=P_{(n_1,...,n_k)}$, and $Q_{v^c}=P_{(m_1,....,m_{k^c})}$. Accordingly, write 
\begin{equation*}
    \underline{\tau_v}=(\tau_{v,1},...,\tau_{v,k})\textnormal{, }\underline{\tau_{v^c}}=(\tau_{v^c,1},...,\tau_{v^c,k^c}),
\end{equation*}
\begin{equation*}
    w_0^{Q_{v}}\lambda_{v}=(\lambda_{v,1},...,\lambda_{v,k})\in (\mathbf{Z}_+^{n_k})^{\Hom(F_v,E)}\times ...\times (\mathbf{Z}_+^{n_1})^{\Hom(F_v,E)}\textnormal{, and}
\end{equation*}
\begin{equation*}
    w_0^{Q_{v^c}}\lambda_{v^c}=(\lambda_{v^c,1},...,\lambda_{v^c,k^c})\in (\mathbf{Z}_+^{m_{k^c}})^{\Hom(F_{v^c},E)}\times ...\times (\mathbf{Z}_+^{m_1})^{\Hom(F_{v^c},E)}.
\end{equation*} Set 
\begin{equation*}
\mathfrak{z}_{\lambda_{v},\underline{\tau_{v}}}^{\circ}:= \otimes_{i=1}^k\mathfrak{z}_{\lambda_{v,i},\tau_{v,i}}^{\circ}\textnormal{,   
   }\mathfrak{z}_{\lambda_{v^c},\underline{\tau_{v^c}}}^{\circ}:= \otimes_{i=1}^{k^c}\mathfrak{z}_{\lambda_{v^c,i},\tau_{v^c,i}}^{\circ},
\end{equation*} and introduce the abstract Hecke algebras
\begin{equation*}
    \widetilde{\mathbf{T}}^{T,\lambda_{\Bar{v}},\underline{\tau_{\Bar{v}}}}:=\widetilde{\mathbf{T}}^T\otimes_{\mathcal{O}}\left( \mathfrak{z}_{\lambda_{v},\underline{\tau_{v}}}^{\circ}\otimes \mathfrak{z}_{\lambda_{v^c},\underline{\tau_{v^c}}}^{\circ} \right),
\end{equation*}
\begin{equation*}
    \mathbf{T}^{T,\lambda_{\Bar{v}},\underline{\tau_{\Bar{v}}}}:=\mathbf{T}^T\otimes_{\mathcal{O}}\left( \mathfrak{z}_{\lambda_{v},\underline{\tau_{v}}}^{\circ}\otimes \mathfrak{z}_{\lambda_{v^c},\underline{\tau_{v^c}}}^{\circ} \right).
\end{equation*}
By Corollary~\ref{Cor3.15} and its analogue for $\widetilde{G}$, $\widetilde{\mathbf{T}}^{T,\lambda_{\Bar{v}},\underline{\tau_{\Bar{v}}}}$ respectively, $\mathbf{T}^{T,\lambda_{\Bar{v}},\underline{\tau_{\Bar{v}}}}$ naturally acts on $R\Gamma(\widetilde{X}_{\widetilde{K}},\mathcal{V}_{(\widetilde{\lambda},\underline{\tau_{\Bar{v}}})}^{\widetilde{Q}^{w_0}_{\Bar{v}}})^{\widetilde{Q}^{w_0}_{\Bar{v}}\textnormal{-ord}}$ respectively, $R\Gamma(X_K,\mathcal{V}_{(\lambda,\underline{\tau})}^{Q_p})^{Q_{\Bar{v}}\textnormal{-ord}}$ as long as $\widetilde{K}_{\Bar{v}}=\widetilde{\mathcal{Q}}_{\Bar{v}}^{w_0}(0,c)$, and $K_{\Bar{v}}=\mathcal{Q}_{\Bar{v}}(0,c)$ with $c\geq c_p$. More precisely, in the case of the former, the action of $\mathfrak{z}_{\lambda_{v^c},\underline{\tau_{v^c}}}^{\circ}$ is via the identification $\mathcal{H}(\sigma(\lambda_{v^c},\underline{\tau_{v^c}})^{\circ})\cong\mathcal{H}((\theta_n^{-1})^{\ast}\sigma(\lambda_{v^c},\underline{\tau_{v^c}})^{\circ}) $ induced by $\theta_n$ (see Remark~\ref{Rem3.22}). We introduce an extension of the unnormalised Satake transform
\begin{equation*}
    \mathcal{S}^{\Bar{v}}:=\mathcal{S}\otimes \textnormal{id}:\widetilde{\mathbf{T}}^{T,\lambda_{\Bar{v}},\underline{\tau_{\Bar{v}}}} \to \mathbf{T}^{T,\lambda_{\Bar{v}},\underline{\tau_{\Bar{v}}}}.
\end{equation*}
Let $\mathfrak{m}\subset \mathbf{T}^T$ be a non-Eisenstein maximal ideal, and set $\widetilde{\mathfrak{m}}:=\mathcal{S}^{\ast}(\mathfrak{m})\subset \widetilde{\mathbf{T}}^T$. By Theorem~\ref{Th2.24}, we have an associated Galois representation
\begin{equation*}
    \overline{\rho}_{\widetilde{\mathfrak{m}}}:G_{F,T}\to \textnormal{GL}_{2n}(\widetilde{\mathbf{T}}^T/\widetilde{\mathfrak{m}})
\end{equation*}
and, by Proposition~\ref{Prop2.16}, we have $\overline{\rho}_{\widetilde{\mathfrak{m}}}=\overline{\rho}_{\mathfrak{m}}\oplus\overline{\rho}_{\mathfrak{m}}^{\vee,c}(1-2n)$. Introduce the faithful Hecke algebras
\begin{equation*}
A(K,\lambda,\underline{\tau},q,\Bar{v}):=\mathbf{T}^{T,\lambda_{\Bar{v}},\underline{\tau_{\Bar{v}}}}(H^q(X_K,\mathcal{V}_{(\lambda,\underline{\tau})}^{Q_p})_{\mathfrak{m}}^{Q_{\Bar{v}}\textnormal{-ord}}),
\end{equation*}
\begin{equation*}
A(K,\lambda,\underline{\tau},q,\Bar{v},m):=\mathbf{T}^{T,\lambda_{\Bar{v}},\underline{\tau_{\Bar{v}}}}(H^q(X_K,\mathcal{V}_{(\lambda,\underline{\tau})}^{Q_p}/\varpi^m)_{\mathfrak{m}}^{Q_{\Bar{v}}\textnormal{-ord}})\textnormal{, and}
\end{equation*}
\begin{equation*}
    \widetilde{A}(\widetilde{K},\Tilde{\lambda},\underline{\tau_{\Bar{v}}},\Bar{v}):=\widetilde{\mathbf{T}}^{T,\lambda_{\Bar{v}},\tau_{\Bar{v}}}(H^d(\widetilde{X}_{\widetilde{K}},\mathcal{V}_{(\Tilde{\lambda},\underline{\tau_{\Bar{v}}})}^{\widetilde{Q}_{\Bar{v}}^{w_0}})_{\widetilde{\mathfrak{m}}}^{\widetilde{Q}^{w_0}_{\Bar{v}}\textnormal{-ord}})
\end{equation*}
for $d=\dim_{\mathbf{C}}\widetilde{X}_{\widetilde{K}}$, the middle degree for the (Betti) cohomology of $\widetilde{X}_{\widetilde{K}}$ and integers $0\leq q\leq d-1$.\footnote{Note that the real dimension of $X_K$ is $d-1$ so its Betti cohomology has top degree $d-1$.}

Given a good subgroup $K\subset G(\mathbf{A}_{F^+}^{\infty})$, a subset $\overline{S}\subset \overline{S}_p$, and an integer $m\in \mathbf{Z}_{\geq 1}$, define the subgroup $K(m,\overline{S})\subset K$ by setting
\begin{equation*}
    K(m,\overline{S})_{\Bar{v}}:=K_{\Bar{v}}\cap K_{\Bar{v}}^m\footnote{Recall that $K^m_{\Bar{v}}:=\ker(G(\mathcal{O}_{F^+_{\Bar{v}}})\to G(\mathcal{O}_{F^+_{\Bar{v}}}/\varpi_{\bar{v}}^m))$.}
\end{equation*}
if $\Bar{v}\in\overline{S}$ and $K(m,\overline{S})_{\Bar{v}}=K_{\Bar{v}}$ otherwise. Also, given a good subgroup $\widetilde{K}\subset \widetilde{G}(\mathbf{A}_{F^+}^{\infty})$, define the good subgroup $\widetilde{K}(m,\overline{S})\subset \widetilde{K}$ by setting
\begin{equation*}
    \widetilde{K}(m,\overline{S})_{\Bar{v}}:=\widetilde{K}_{\Bar{v}}\cap \mathcal{P}_{\Bar{v}}(m,m)
\end{equation*}
if $\Bar{v}\in \overline{S}$, and $\widetilde{K}(m,\overline{S})_{\Bar{v}}:=\widetilde{K}_{\Bar{v}}$ otherwise.

\begin{Prop}[\textnormal{Degree shifting}]\label{Prop6.1}
    Let $\Bar{v},\Bar{v}'$ be two distinct places of $\overline{S}_p$. Let $\overline{S}_1:=\{\Bar{v}'\}$, $\overline{S}_3:=\{\Bar{v}\}$, and $\overline{S}_2:=\overline{S}_p\setminus \{\Bar{v},\Bar{v}'\}$ their complement. Let $\widetilde{K}\subset \widetilde{G}(\mathbf{A}_{F^+}^{\infty})$ be a good subgroup and $m\in \mathbf{Z}_{\geq 1}$ be an integer. Assume that the following conditions are satisfied.
    \begin{enumerate}
        \item We have
        \begin{equation*}
            \sum_{\Bar{v}''\in \overline{S}_2}[F^+_{\Bar{v}''}:\mathbf{Q}_p]\geq \frac{1}{2}[F^+:\mathbf{Q}].
        \end{equation*}
        \item For $\Bar{v}''\in \overline{S}_1\cup \overline{S}_2$, we have $U(\mathcal{O}_{F^+_{\Bar{v}''}})\subset \widetilde{K}_{\Bar{v}''}$, and $\widetilde{K}_{\Bar{v}''}=\widetilde{K}(m,\overline{S}_1\cup \overline{S}_2)_{\Bar{v}''}$. Finally, we have $\widetilde{K}_{\Bar{v}}=\widetilde{\mathcal{Q}}^{w_0}_{\Bar{v}}(0,c_p)$.
        \item For each $\iota:F\hookrightarrow E$ inducing $\Bar{v}$ or $\Bar{v}'$, we have $-\lambda_{\iota c,1}-\lambda_{\iota,1}\geq 0$.
        \item $\overline{\rho}_{\widetilde{\mathfrak{m}}}$ is decomposed generic.
    \end{enumerate}
Define a weight $\Tilde{\lambda}\in (\mathbf{Z}^{2n}_+)^{\Hom(F^+,E)}$ as follows: if $\iota:F^+\hookrightarrow E$ does not induce either $\Bar{v}$ or $\Bar{v}'$, set $\Tilde{\lambda}_{\iota}=0$. Otherwise, set $\Tilde{\lambda}_{\iota}=(-w_{0,n}\lambda_{\iota c},\lambda_{\iota})$. Set $K:=(\widetilde{K}^{\Bar{v}}\cap G(\mathbf{A}_{F^+}^{\infty,\Bar{v}})\big)\cdot \big( \widetilde{\mathcal{Q}}_{\Bar{v}}(0,c_p)\cap G(F^+_{\Bar{v}})\big)=(\widetilde{K}^{\Bar{v}}\cap G(\mathbf{A}_{F^+}^{\infty,\Bar{v}})\big)\cdot\big( \mathcal{Q}_v(0,c_p)\times \mathcal{Q}_{v^c}(0,c_p)\big)$.

   Let $q\in [\lfloor \frac{d}{2}\rfloor, d-1]$. Then there exist an integer $m'\geq m$, an integer $N\geq 1$, a nilpotent ideal $I\subset A(K,\lambda,\underline{\tau},q,\Bar{v},m)$ satisfying $I^N=0$, and a commutative diagram
$$\begin{tikzcd}
\widetilde{\mathbf{T}}^{T,\lambda_{\Bar{v}},\underline{\tau_{\Bar{v}}}} \arrow{r}{} \arrow{d}{\mathcal{S}^{\Bar{v}}}
&\widetilde{A}(\widetilde{K}(m',\overline{S}_2),\Tilde{\lambda},\underline{\tau_{\Bar{v}}},\Bar{v}) \arrow{d}{}\\
\mathbf{T}^{T,\lambda_{\Bar{v},\tau_{\Bar{v}}}} \arrow{r}{} & A(K,\lambda,\underline{\tau},q,\Bar{v},m)/I.
\end{tikzcd}$$
Moreover, $N$ can be chosen to only depend on $n$ and $[F^+:\mathbf{Q}]$.
\end{Prop}

As a preliminary step, one proves the following lemma (compare with \cite{CN23}, Proposition 4.2.2). This already consists of (one of) the ideas coming from the ordinary degree shifting argument. Moreover, this lemma is one of the points where the main result of \cite{CS19}, in particular the decomposed generic assumption gets used. Finally, for this lemma to hold it is crucial that $\mathfrak{m}$ is non-Eisenstein as otherwise there is possible contribution to the boundary cohomology from strata corresponding to parabolic subgroups other than $P$ and consequently, \cite{CN23}, Corollary 4.1.9 could fail to hold.

Before stating the lemma, we need to introduce some notation. Given a subset $\overline{S}\subset \overline{S}_p$ and an integer $m\geq 1$,  set $\mathcal{V}_U(\overline{S},m):=R\Gamma(U(\mathcal{O}_{F^+,\overline{S}}),\mathcal{O}/\varpi^m)\in D^b_{\textnormal{sm}}(\mathcal{O}/\varpi^m[K_{\overline{S}}])$. We can view it as an object of $D^b_{\textnormal{sm}}(\mathcal{O}/\varpi^m[G^{\overline{S}}\times K_{\overline{S}}])$ via inflation. In particular, it gives rise to a bounded complex of $G^{\overline{S}}\times K_{\overline{S}}$-equivariant sheaves on $\overline{\mathfrak{X}}_G$ and descends to an object in $D^b(\textnormal{Sh}(X_K,\mathcal{O}/\varpi^m))$ for any good subgroup $K\subset G(\mathbf{A}_{F^+}^{\infty})$. It has locally constant cohomology sheaves $\mathcal{V}^j_U(\overline{S},m)$ that are non-zero if and only if $j\in[0,n^2\sum_{\Bar{v}\in \overline{S}}[F^+_{\Bar{v}}:\mathbf{Q}_p]]$ according to Lemma 2.3.17 of \cite{CN23} (combined with the K\"unneth formula). By \textit{loc. cit.} Lemma 2.1.9, we have
\begin{equation*}
R\Gamma(\overline{\mathfrak{X}}_G,\mathcal{O}/\varpi^m)\otimes^{\mathbf{L}}\mathcal{V}_U(\overline{S},m)\footnote{See Page 13 of \cite{CN23} for how to define this derived tensor product.}\cong R\Gamma(\overline{\mathfrak{X}}_G,\mathcal{V}_U(\overline{S},m))
\end{equation*}
in $D^b_{\textnormal{sm}}(\mathcal{O}/\varpi^m[G^{\overline{S}}\times K_{\overline{S}}])$. In particular, by the proof of Corollary~\ref{Cor3.15}, the complex $R\Gamma(X_K,\mathcal{V}_U(\overline{S},m)\otimes \mathcal{V}_{(\lambda_{\overline{S}_p\setminus \overline{S}},\underline{\tau_{\overline{S}_p\setminus \overline{S}}})}^{Q_{\overline{S}_p\setminus \overline{S}}}/\varpi^m)^{Q_{\Bar{v}}\textnormal{-ord}}$ carries a natural action of $\mathbf{T}^{T,\lambda_{\Bar{v}},\underline{\tau_{\Bar{v}}}}$ for any $\Bar{v}\in \overline{S}_p\setminus \overline{S}$ such that $K_{\Bar{v}}=\mathcal{Q}_{\Bar{v}}(0,c_p)$. We can pass to the homotopy limit over $m$ to get $\mathcal{V}_U(\overline{S})\in D^b(\textnormal{Sh}(X_K,\mathcal{O}))$ which once again has locally constant cohomology sheaves $\mathcal{V}^j_U(\overline{S})$. Assuming that for $\Bar{v}\in \overline{S}_p\setminus \overline{S}$ we have $K_{\Bar{v}}=\mathcal{Q}_{\Bar{v}}(0,c_p)$, we get a natural action of $\mathbf{T}^{T,\lambda_{\Bar{v}},\tau_{\Bar{v}}}$ on $R\Gamma(X_K,\mathcal{V}_U(\overline{S})\otimes\mathcal{V}_{(\lambda_{\overline{S}_p\setminus \overline{S}},\underline{\tau_{\overline{S}_p\setminus \overline{S}}})}^{Q_{\overline{S}_p\setminus \overline{S}}} )^{Q_{\Bar{v}}\textnormal{-ord}}$ by passing to the limit over $m$.

\begin{Lemma}\label{Lem6.2}
    Let $\widetilde{K}\subset \widetilde{G}(\mathbf{A}_{F^+}^{\infty})$ be a good subgroup that is decomposed with respect to $P$ such that, for each $\Bar{v}\in \overline{S}_p$, $\widetilde{K}_{\Bar{v}}\cap U(F^+_{\Bar{v}})=U(\mathcal{O}_{F^+_{\Bar{v}}}).$ Let $\mathfrak{m}\subset \mathbf{T}^{T}$ be a non-Eisenstein maximal ideal, let $\widetilde{\mathfrak{m}}:=\mathcal{S}^{\ast}(\mathfrak{m})\subset \widetilde{\mathbf{T}}^{T}$ and assume that $\overline{\rho}_{\widetilde{\mathfrak{m}}}$ is decomposed generic.

     Fix places $\Bar{v}, \Bar{v}'\in \overline{S}_p$ and introduce $\overline{S}_1,\overline{S}_2$ and $\overline{S}_3$ as before. Let $(Q_p,\lambda,\underline{\tau})$ be a tuple as in \S\ref{sec2.7}, and let $(\widetilde{Q}^{w_0}_{\Bar{v}}, \Tilde{\lambda},\underline{\tau_{\Bar{v}}})$  be the tuple associated to it as in Proposition~\ref{Prop6.1}. Assume that $\widetilde{K}_{\Bar{v}}=\widetilde{\mathcal{Q}}^{w_0}_{\Bar{v}}(0,c_p)$ and define $K$ as in Proposition~\ref{Prop6.1}. Then the Satake transform $\mathcal{S}^{\Bar{v}}:\widetilde{\mathbf{T}}^{T,\Tilde{\lambda}_{\Bar{v}},\underline{\tau_{\Bar{v}}}}\to\mathbf{T}^{T,\Tilde{\lambda}_{\Bar{v}},\underline{\tau_{\Bar{v}}}}$ descends to a homomorphism
    \begin{equation*}
\widetilde{\mathbf{T}}^{T,\lambda_{\Bar{v}},\underline{\tau_{\Bar{v}}}}(H^d(\widetilde{X}_{\widetilde{K}},\mathcal{V}_{(\widetilde{\lambda},\underline{\tau_{\Bar{v}}})}^{\widetilde{Q}^{w_0}_{\Bar{v}}})_{\widetilde{\mathfrak{m}}}^{\widetilde{Q}^{w_0}_{\Bar{v}}\textnormal{-ord}})\to
    \end{equation*}
    \begin{equation*}
        \mathbf{T}^{T,\lambda_{\Bar{v}},\underline{\tau_{\Bar{v}}}}\Big(\mathbf{H}^d\big(X_K,\mathcal{V}_{\lambda_{\Bar{v}'}}\otimes \mathcal{V}_U(\overline{S}_2)\otimes \mathcal{V}_{(\lambda_{\Bar{v}},\tau_{\Bar{v}})}^{Q_{\Bar{v}}}\big)_{\mathfrak{m}}^{Q_{\Bar{v}}\textnormal{-ord}}\Big)
    \end{equation*}
    where $\mathbf{H}^d$ denotes the degree $d$ hypercohomology.
\end{Lemma}
\begin{proof}
    By Theorem~\ref{Th2.27}, the genericity of $\widetilde{\mathfrak{m}}$ implies that we have a $\widetilde{\mathbf{T}}^{T,\lambda_{\Bar{v}},\underline{\tau_{\Bar{v}}}}$-equivariant surjection
    \begin{equation*}
H^d(\widetilde{X}_{\widetilde{K}},\mathcal{V}_{(\widetilde{\lambda},\underline{\tau_{\Bar{v}}})}^{\widetilde{Q}^{w_0}_{\Bar{v}}})_{\widetilde{\mathfrak{m}}}^{\widetilde{Q}^{w_0}_{\Bar{v}}\textnormal{-ord}}\twoheadrightarrow H^d(\partial\widetilde{X}_{\widetilde{K}},\mathcal{V}_{(\widetilde{\lambda},\underline{\tau_{\Bar{v}}})}^{\widetilde{Q}^{w_0}_{\Bar{v}}})_{\widetilde{\mathfrak{m}}}^{\widetilde{Q}^{w_0}_{\Bar{v}}\textnormal{-ord}}.
    \end{equation*}
    In particular, it suffices to prove that $\mathcal{S}^{\Bar{v}}$ descends to a homomorphism
    \begin{equation*}
\widetilde{\mathbf{T}}^{T,\lambda_{\Bar{v}},\underline{\tau_{\Bar{v}}}}(H^d(\partial\widetilde{X}_{\widetilde{K}},\mathcal{V}_{(\widetilde{\lambda},\underline{\tau_{\Bar{v}}})}^{\widetilde{Q}^{w_0}_{\Bar{v}}})_{\widetilde{\mathfrak{m}}}^{\widetilde{Q}^{w_0}_{\Bar{v}}\textnormal{-ord}})\to
        \mathbf{T}^{T,\lambda_{\Bar{v}},\underline{\tau_{\Bar{v}}}}\Big(\mathbf{H}^d\big(X_K,\mathcal{V}_{\lambda_{\Bar{v}'}}\otimes \mathcal{V}_U(\overline{S}_2)\otimes \mathcal{V}_{(\lambda_{\Bar{v}},\underline{\tau_{\Bar{v}})}}^{Q_{\Bar{v}}}\big)_{\mathfrak{m}}^{Q_{\Bar{v}}\textnormal{-ord}}\Big).
    \end{equation*}
Consider $\pi_{\partial}(\widetilde{K}^{\Bar{v}},\mathcal{V}_{\widetilde{\lambda}^{\Bar{v}}}/\varpi^m):=R\Gamma(\partial \widetilde{X}_{\widetilde{K}^{\Bar{v}}},\mathcal{V}_{\widetilde{\lambda}^{\Bar{v}}}/\varpi^m)\in D^+_{\textnormal{sm}}(\mathcal{O}/\varpi^m[\widetilde{G}(F^+_{\Bar{v}})])$, the $\Bar{v}$-completed boundary cohomology.
We then have a $\widetilde{\mathbf{T}}^{T,\lambda_{\Bar{v}},\underline{\tau_{\Bar{v}}}}$-equivariant isomorphism
\begin{equation*}
    H^d(\partial\widetilde{X}_{\widetilde{K}},\mathcal{V}_{(\widetilde{\lambda},\underline{\tau_{\Bar{v}}})}^{\widetilde{Q}^{w_0}_{\Bar{v}}}/\varpi^m)_{\widetilde{\mathfrak{m}}}^{\widetilde{Q}^{w_0}_{\Bar{v}}\textnormal{-ord}}\cong
\end{equation*}
\begin{equation*}
\mathbf{R}^d\Hom_{\mathcal{O}/\varpi^m[\widetilde{M}^{w_0,0}_{\Bar{v}}]}\Big(\widetilde{\sigma}(\Tilde{\lambda}_{\Bar{v}},\underline{\tau_{\Bar{v}}})^{\circ}/\varpi^m,R\Gamma\big(\widetilde{N}^{w_0}_{\Bar{v}}(\mathcal{O}_{F^+_{\Bar{v}}}),\pi_{\partial}(\widetilde{K}^{\Bar{v}},\mathcal{V}_{\widetilde{\lambda}^{\Bar{v}}}/\varpi^m)\big)^{\widetilde{Q}^{w_0}_{\Bar{v}}\textnormal{-ord}}\Big).
\end{equation*}
On the other hand, \cite{CN23}, Corollary 4.1.9 shows that $\pi_{\partial}(\widetilde{K}^{\Bar{v}},\mathcal{V}_{\widetilde{\lambda}^{\Bar{v}}}/\varpi^m)_{\widetilde{\mathfrak{m}}}$ admits
\begin{equation*}
\mathcal{S}^{\ast}\textnormal{Ind}_{P(F^+_{\Bar{v}})}^{\widetilde{G}(F^+_{\Bar{v}})}R\Gamma(\overline{X}_{K^{\Bar{v}}},\mathcal{V}_{\lambda_{\Bar{v}'}}/\varpi^m\otimes \mathcal{V}_U(\overline{S}_2,m))_{\mathfrak{m}}
\end{equation*}
as a $\widetilde{\mathbf{T}}^T$-equivariant direct summand in $D^+_{\textnormal{sm}}(\mathcal{O}/\varpi^m[\widetilde{G}(F^+_{\Bar{v}})])$. To simplify notation, set $\pi:=\pi(K^{\Bar{v}},\mathcal{V}_{\lambda_{\Bar{v}'}}/\varpi^m\otimes \mathcal{V}_U(\overline{S}_2,m))_{\mathfrak{m}}$. Then, by the previous discussion, $H^d(\partial\widetilde{X}_{\widetilde{K}},\mathcal{V}_{(\widetilde{\lambda},\underline{\tau_{\Bar{v}}})}^{\widetilde{Q}^{w_0}_{\Bar{v}}}/\varpi^m)_{\widetilde{\mathfrak{m}}}^{\widetilde{Q}^{w_0}_{\Bar{v}}\textnormal{-ord}}$ admits
\begin{equation*}
    \mathbf{R}^d\Hom_{\mathcal{O}/\varpi^m[\widetilde{M}^{w_0,0}_{\Bar{v}}]}\Big(\widetilde{\sigma}(\Tilde{\lambda}_{\Bar{v}},\underline{\tau_{\Bar{v}}})^{\circ}/\varpi^m,R\Gamma\big(\widetilde{N}^{w_0}_{\Bar{v}}(\mathcal{O}_{F^+_{\Bar{v}}}),\mathcal{S}^{\ast}\textnormal{Ind}_{P(F^+_{\Bar{v}})}^{\widetilde{G}(F^+_{\Bar{v}})}\pi\big)^{\widetilde{Q}^{w_0}_{\Bar{v}}\textnormal{-ord}}\Big)
\end{equation*}
as a $\widetilde{\mathbf{T}}^{T,\lambda_{\Bar{v}},\underline{\tau_{\Bar{v}}}}$-equivariant direct summand.

By Corollary~\ref{Cor3.29}, the latter admits
\begin{equation*}
    \mathbf{R}^d\Hom_{\mathcal{O}/\varpi^m[M^0_{\Bar{v}}]}\Big(\sigma(\lambda_v,\underline{\tau_v})^{\circ}\otimes \sigma(\lambda_{v^c},\underline{\tau_{v^c}})^{\circ}/\varpi^m,R\Gamma\big(N_{v}(\mathcal{O}_{F_v})\times N_{v^c}(\mathcal{O}_{F_{v^c}}),\pi\big)^{Q_v\times Q_{v^c}\textnormal{-ord}}\Big)
\end{equation*}
\begin{equation*}
    \cong \mathbf{H}^d(X_K,\mathcal{V}_{\lambda_{\Bar{v}'}}\otimes \mathcal{V}_U(\overline{S}_2)\otimes \mathcal{V}_{(\lambda_{\Bar{v}},\tau_{\Bar{v}})}^{Q_{\Bar{v}}}/\varpi^m)_{\mathfrak{m}}^{Q_{\Bar{v}}\textnormal{-ord}}
\end{equation*}
as a $\widetilde{\mathbf{T}}^{T,\lambda_{\Bar{v}},\underline{\tau_{\Bar{v}}}}$-equivariant subquotient, giving the desired map with mod $\varpi^m$ coefficients. Here we implicitly used the identification $\sigma(\lambda_v,\underline{\tau_v})^{\circ}\otimes \sigma(\lambda_{v^c},\underline{\tau_{v^c}})^{\circ}\cong \tau_{w_0^P}^{-1}\widetilde{\sigma}(\widetilde{\lambda}_{\Bar{v}},\underline{\tau_{\Bar{v}}})^{\circ}$ induced by $\iota_v:\widetilde{G}(F^+_{\Bar{v}})\cong \textnormal{GL}_{2n}(F_v)$.

We finally note that these identifications are compatible when we vary $m$, and, since all the cohomology groups appearing are finitely generated $\mathcal{O}$-modules, we can conclude by passing to the limit over $m\geq 1$.
\end{proof}
\begin{proof}[Proof of Proposition~\ref{Prop6.1}]
    This is a generalisation of \cite{CN23}, Proposition 4.2.6. More precisely, the role of the faithful Hecke algebras $A(K,\lambda,q)$, $A(K,\lambda,q,m)$, and $\widetilde{A}(\widetilde{K},\Tilde{\lambda},\Bar{v})$ of \textit{loc. cit.} are now played by $A(K,\lambda,\underline{\tau},q,\Bar{v})$, $A(K,\lambda,\underline{\tau},q,\Bar{v},m)$, and $\widetilde{A}(\widetilde{K},\Tilde{\lambda},\underline{\tau_{\Bar{v}}},\Bar{v})$. Consequently, in our case $\mathbf{T}$, respectively $\widetilde{\mathbf{T}}$ will denote $\mathbf{T}^{T,\lambda_{\Bar{v}},\underline{\tau_{\Bar{v}}}}$, respectively $\widetilde{\mathbf{T}}^{T,\lambda_{\Bar{v}},\underline{\tau_{\Bar{v}}}}$ and our goal is to show the existence of non-negative integers $m'\geq m$, and $N$ such that
    \begin{equation*}
        \mathcal{S}^{\Bar{v}}(\textnormal{Ann}_{\widetilde{\mathbf{T}}}H^d(\widetilde{X}_{\widetilde{K}(m',\overline{S}_2)},\mathcal{V}_{(\Tilde{\lambda},\underline{\tau_{\Bar{v}}})}^{\widetilde{Q}_{\Bar{v}}^{w_0}})_{\widetilde{\mathfrak{m}}}^{\widetilde{Q}_{\Bar{v}}^{w_0}\textnormal{-ord}})^N\subset \textnormal{Ann}_{\mathbf{T}}H^q(X_K,\mathcal{V}_{(\lambda,\underline{\tau})}^{Q_p}/\varpi^m)_{\mathfrak{m}}^{Q_{\Bar{v}}\textnormal{-ord}}.
    \end{equation*}
    Regardless the change of setup, the proof is identical to that of \textit{loc. cit.} In particular, we only indicate the necessary new inputs for the argument to work in our case and direct the reader to \cite{CN23} for the proof.
    \begin{itemize}
        \item The proof uses Poincar\'e duality at certain points. To be able to appeal to Poincar\'e duality in our case, we need that it is Hecke-equivariant also at $\Bar{v}$. This is the content of Proposition~\ref{Prop3.20}.
        \item To deepen the level at certain steps, their proof uses the Hochschild--Serre spectral sequence, and we need to verify that all of these spectral sequences are $\mathbf{T}^{T,\lambda_{\Bar{v}},\underline{\tau_{\Bar{v}}}}$-equivariant. However, this is clear as we only have to go deeper level at places in $\overline{S}_2$ and our Hecke operators are at places in $T\cup \{\Bar{v}\}$.
        \item We also need to argue that the hypercohomology spectral sequences with respect to $\mathcal{V}_U(\overline{S}_2)$ and $\mathcal{V}_U(\overline{S}_2,m)$ (that are denoted in \cite{CN23} by $E_n^{i,j}(\mathcal{O})$ resp. $E_n^{i,j}(\mathcal{O}/\varpi^m)$) are $\mathbf{T}^{T,\lambda_{\Bar{v}},\underline{\tau_{\Bar{v}}}}$-equivariant. To see this, we note that the $\mathbf{T}^{T,\lambda_{\Bar{v}},\underline{\tau_{\Bar{v}}}}$-action on the target of the map in Lemma~\ref{Lem6.2} is induced by the identification
        \begin{equation*}
            R\Gamma(X_K,\mathcal{V}_{\lambda_{\Bar{v}'}}/\varpi^m\otimes \mathcal{V}_U(\overline{S}_2,m)\otimes \mathcal{V}_{(\lambda_{\Bar{v}},\underline{\tau}_{\Bar{v}})}^{Q_{\Bar{v}}})^{Q_{\Bar{v}}\textnormal{-ord}}\cong 
        \end{equation*}
        \begin{equation*}
            \cong R\textnormal{Hom}_{\mathcal{O}/\varpi^m[M^0_{\overline{v}}]}(\sigma(\lambda_{\overline{v}},\underline{\tau_{\overline{v}}})^{\circ}/\varpi^m,\pi^{Q_{\overline{v}}\textnormal{-ord}}(K^{\overline{v}},\mathcal{V}_{\lambda_{\overline{v}'}}\otimes \mathcal{V}_U(\overline{S}_2,m)))
        \end{equation*}
        in $D^+(\mathcal{O}/\varpi^m)$. We can then construct $E_n^{i,j}(\mathcal{O}/\varpi^m)$ of \textit{loc. cit.} by taking the hypercohomology spectral sequence of
        \begin{equation*}
            R\textnormal{Hom}_{\mathcal{O}/\varpi^m[M^0_{\overline{v}}]}(\sigma(\lambda_{\overline{v}},\underline{\tau_{\overline{v}}})^{\circ}/\varpi^m,\pi^{Q_{\overline{v}}\textnormal{-ord}}(K^{\overline{v}},\mathcal{V}_{\lambda_{\overline{v}'}}\otimes-))\cong
        \end{equation*}
        \begin{equation*}
        R\Biggl(\textnormal{Hom}_{\mathcal{O}/\varpi^m[M^0_{\overline{v}}]}\biggl(\sigma(\lambda_{\overline{v}},\underline{\tau_{\overline{v}}})^{\circ}/\varpi^m,\Gamma\Bigl(N^0_{\Bar{v}},\Gamma\bigl(K^{\Bar{v}},\Gamma(\overline{\mathfrak{X}}_G,\mathcal{V}_{\lambda_{\Bar{v}'}}\otimes-)\bigr)\Bigr)^{Q_{\Bar{v}}\textnormal{-ord}}\biggr)\Biggr)
        \end{equation*}
        applied to $\mathcal{V}_U(\overline{S}_2,m)\in D^b\textnormal{Sh}_{G^{\overline{S}_2}\times K_{\overline{S}_2}}(\overline{\mathfrak{X}}_G,\mathcal{O}/\varpi^m)$. In particular, it will be $\mathbf{T}^{T,\lambda_{\Bar{v}},\underline{\tau_{\Bar{v}}}}$-equivariant by construction. Since all the members of $E_n^{i,j}(\mathcal{O}/\varpi^m)$ are finite $\mathcal{O}/\varpi^m$-modules, the Mittag--Leffler condition is satisfied and, in particular, the limit of the spectral sequences $E_n^{i,j}(\mathcal{O}/\varpi^m)$ over $m\geq 1$ produces $E_n^{i,j}(\mathcal{O})$.
        \item As input, we use Lemma~\ref{Lem6.2} instead of \textit{loc. cit.} Proposition 4.2.2.
    \end{itemize}
\end{proof}

Just as in \cite{CN23}, we will need a dual version of the degree shifting argument. We only explain the setup and the statements here as the proofs are identical to that of Proposition~\ref{Prop6.1}, and Lemma~\ref{Lem6.2}. Set $\widetilde{\mathbf{T}}^{T,\lambda_{\Bar{v}},\underline{\tau_{\Bar{v}}},\Tilde{\iota}}$, respectively $\mathbf{T}^{T,\lambda_{\Bar{v}},\underline{\tau_{\Bar{v}}},\iota}$ to be the image of $\widetilde{\mathbf{T}}^{T,\lambda_{\Bar{v}},\underline{\tau_{\Bar{v}}}}$, respectively $\mathbf{T}^{T,\lambda_{\Bar{v}},\underline{\tau_{\Bar{v}}}}$ under the anti-isomorphism
\begin{equation*}
\Tilde{\iota}:\widetilde{\mathbf{T}}^T\otimes_{\mathcal{O}}\mathcal{H}(\sigma(\lambda_{\Bar{v}},\underline{\tau}_{\Bar{v}})^{\circ})\cong \widetilde{\mathbf{T}}^T\otimes_{\mathcal{O}}\mathcal{H}(\sigma(\lambda_{\Bar{v}},\underline{\tau}_{\Bar{v}})^{\circ,\vee}), 
\end{equation*}
respectively
\begin{equation*}
    \iota:\mathbf{T}^T\otimes_{\mathcal{O}}\mathcal{H}(\sigma(\lambda_{\Bar{v}},\underline{\tau}_{\Bar{v}})^{\circ})\cong \mathbf{T}^T\otimes_{\mathcal{O}}\mathcal{H}(\sigma(\lambda_{\Bar{v}},\underline{\tau}_{\Bar{v}})^{\circ,\vee})
\end{equation*}
given by $[g,\psi]\mapsto [g^{-1},\psi^{t}]$ on double coset operators.\footnote{Recall that $\Tilde{\iota}$ and $\iota$ are the maps intertwining the Hecke actions between the two sides of Poincar\'e duality.} We denote by $\mathcal{S}^{\Bar{v},\vee}:\widetilde{\mathbf{T}}^{T,\lambda_{\Bar{v}},\underline{\tau_{\Bar{v}}},\Tilde{\iota}} \to \mathbf{T}^{T,\lambda_{\Bar{v}},\underline{\tau_{\Bar{v}}},\iota}$ the extension of the unnormalised Satake transform given by $\mathcal{S}\otimes \textnormal{id}$. We consider a tuple $(Q_p,\lambda,\underline{\tau})$ as before, but now, just as in the last paragraph of \S\ref{sec2.7}, we assume that $\widetilde{\lambda}:=(-w_0^{\textnormal{GL}_n}\lambda_{v^c},\lambda_v)$ is dominant for $\widetilde{G}$ (instead of $w_0^P\widetilde{\lambda}_{\Bar{v}}$). We also pick some dominant weight $\Tilde{\lambda}\in (\mathbf{Z}_+^{2n})^{\Hom(F^+,E)}$ for $\widetilde{G}$ extending $\Tilde{\lambda}_{\Bar{v}}$. We consider the dual local systems $\mathcal{V}_{(\lambda,\underline{\tau})}^{Q_p,\vee}$, and $\mathcal{V}_{(\Tilde{\lambda},\underline{\tau_{\Bar{v}}})}^{\widetilde{Q},w_0,\vee}$. Then $\widetilde{\mathbf{T}}^{T,\lambda_{\Bar{v}},\underline{\tau_{\Bar{v}}},\Tilde{\iota}}$, respectively $\mathbf{T}^{T,\lambda_{\Bar{v}},\underline{\tau_{\Bar{v}}},\iota}$ acts on $R\Gamma(\widetilde{X}_{\widetilde{K}},\mathcal{V}_{(\Tilde{\lambda},\underline{\tau_{\Bar{v}}})}^{\widetilde{Q}_{\Bar{v}},w_0,\vee}/\varpi^m)^{\overline{\widetilde{Q}_{\Bar{v}}}\textnormal{-ord}}$, respectively $R\Gamma(X_K,\mathcal{V}_{(\lambda,\underline{\tau})}^{Q_p,\vee}/\varpi^m)^{\overline{Q_{\Bar{v}}}\textnormal{-ord}}$ when $\widetilde{K}\subset \widetilde{G}(\mathbf{A}_{F^+}^{\infty})$ is a good subgroup with $\widetilde{K}_{\Bar{v}}=\widetilde{\mathcal{Q}}_{\Bar{v}}(0,c_p)$, and $K=\widetilde{K}\cap G(\mathbf{A}_{F^+}^{\infty})$. In the case of the latter, it follows from Corollary~\ref{Cor3.19}. The action in the case of the former comes from the identification
\begin{equation*}
R\Gamma(\widetilde{X}_{\widetilde{K}},\mathcal{V}_{(\Tilde{\lambda},\underline{\tau_{\Bar{v}}})}^{\widetilde{Q}_{\Bar{v}},w_0,\vee}/\varpi^m)^{\overline{\widetilde{Q}_{\Bar{v}}}\textnormal{-ord}}\cong 
\end{equation*}
\begin{equation*}
R\Hom_{\mathcal{O}/\varpi^m[\widetilde{M}^0_{\Bar{v}}]}(\theta_n^{-1}\sigma(\lambda_{v^c},\tau_{v^c})^{\circ,\vee}\otimes \sigma(\lambda_v,\tau_v)^{\circ,\vee},\pi^{\overline{\widetilde{Q}_{\Bar{v}}}\textnormal{-ord}}(\widetilde{K}^{\Bar{v}},\mathcal{V}_{\widetilde{\lambda}^{\Bar{v}}}^{\vee}/\varpi^m))
\end{equation*}
in $D^+(\mathcal{O}/\varpi^m)$ and the isomorphism $\mathcal{H}(\sigma(\lambda_v,\underline{\tau_{v}})^{\circ,\vee})\otimes \mathcal{H}(\sigma(\lambda_{v^c},\underline{\tau_{v^c}})^{\circ,\vee})\cong \mathcal{H}(\theta_n^{-1}\sigma(\lambda_{v^c},\underline{\tau_{v^c}})^{\circ,\vee})\otimes \mathcal{H}(\sigma(\lambda_v,\underline{\tau_v})^{\circ,\vee})$ induced by $\iota_v$.

Given a maximal ideal $\mathfrak{m}\subset \mathbf{T}^T$, set $\mathfrak{m}^{\vee}:=\iota|_{\mathbf{T}^T}(\mathfrak{m})\subset \mathbf{T}^T$. Define the faithful Hecke algebras
\begin{equation*}
A^{\vee}(K,\lambda,\underline{\tau},q,\Bar{v}):=\mathbf{T}^{T,\lambda_{\Bar{v}},\underline{\tau_{\Bar{v}}},\iota}(H^q(X_K,\mathcal{V}_{(\lambda,\underline{\tau})}^{Q_p,\vee})_{\mathfrak{m}^{\vee}}^{\overline{Q_{\Bar{v}}}\textnormal{-ord}}),
\end{equation*}
\begin{equation*}
A^{\vee}(K,\lambda,\underline{\tau},q,\Bar{v},m):=\mathbf{T}^{T,\lambda_{\Bar{v}},\underline{\tau_{\Bar{v}}},\iota}(H^q(X_K,\mathcal{V}_{(\lambda,\underline{\tau})}^{Q_p,\vee}/\varpi^m)_{\mathfrak{m}^{\vee}}^{\overline{Q_{\Bar{v}}}\textnormal{-ord}})\textnormal{, and}
\end{equation*}
\begin{equation*}
    \widetilde{A}^{\vee}(\widetilde{K},\Tilde{\lambda},\underline{\tau_{\Bar{v}}},\Bar{v}):=\widetilde{\mathbf{T}}^{T,\lambda_{\Bar{v}},\underline{\tau_{\Bar{v}}},\Tilde{\iota}}(H^d(\widetilde{X}_{\widetilde{K}},\mathcal{V}_{(\Tilde{\lambda},\underline{\tau_{\Bar{v}}})}^{\widetilde{Q}_{\Bar{v}},w_0^P,\vee})_{\mathcal{S}^{\ast}\mathfrak{m}^{\vee}}^{\overline{\widetilde{Q}_{\Bar{v}}}\textnormal{-ord}}).
\end{equation*}
Then the dual version of degree shifting reads as follows.
\begin{Prop}[\textnormal{Dual degree shifting}]\label{Prop6.3}
    Let $\Bar{v},\Bar{v}'$ be two distinct places of $\overline{S}_p$. Let $\overline{S}_1:=\{\Bar{v}'\}$, $\overline{S}_3:=\{\Bar{v}\}$, and $\overline{S}_2:=\overline{S}_p\setminus \{\Bar{v},\Bar{v}'\}$ their complement. Let $\widetilde{K}\subset \widetilde{G}(\mathbf{A}_{F^+}^{\infty})$ be a good subgroup and $m\in \mathbf{Z}_{\geq 1}$ be an integer. Assume that the following conditions are satisfied.
    \begin{enumerate}
        \item We have
        \begin{equation*}
            \sum_{\Bar{v}''\in \overline{S}_2}[F^+_{\Bar{v}''}:\mathbf{Q}_p]\geq \frac{1}{2}[F^+:\mathbf{Q}].
        \end{equation*}
        \item For $\Bar{v}''\in \overline{S}_1\cup \overline{S}_2$, we have $U(\mathcal{O}_{F^+_{\Bar{v}''}})\subset \widetilde{K}_{\Bar{v}''}$, and $\widetilde{K}_{\Bar{v}''}=\widetilde{K}(m,\overline{S}_1\cup \overline{S}_2)_{\Bar{v}''}$. Finally, we have $\widetilde{K}_{\Bar{v}}=\widetilde{\mathcal{Q}}_{\Bar{v}}(0,c_p)$.
        \item For each $\iota:F\hookrightarrow E$ inducing $\Bar{v}$ or $\Bar{v}'$, we have $\lambda_{\iota ,n}+\lambda_{\iota c,n}\geq 0$.
        \item $\overline{\rho}_{\mathcal{S}^{\ast}\mathfrak{m}^{\vee}}$ is decomposed generic.
    \end{enumerate}
Define a weight $\Tilde{\lambda}\in (\mathbf{Z}^{2n}_+)^{\Hom(F^+,E)}$ as follows: if $\iota:F^+\hookrightarrow E$ does not induce either $\Bar{v}$ or $\Bar{v}'$, set $\Tilde{\lambda}_{\iota}=0$. Otherwise, set $\Tilde{\lambda}_{\iota}=(\lambda_{\iota},-w_{0,n}\lambda_{\iota c})$. Set $K:=\widetilde{K}\cap G(\mathbf{A}_{F^+}^{\infty})$.

   Let $q\in [\lfloor \frac{d}{2}\rfloor, d-1]$. Then there exists an integer $m'\geq m$, an integer $N\geq 1$, a nilpotent ideal $I\subset A^{\vee}(K,\lambda,\underline{\tau},q,\Bar{v},m)$ satisfying $I^N=0$, and a commutative diagram
$$\begin{tikzcd}
\widetilde{\mathbf{T}}^{T,\lambda_{\Bar{v}},\underline{\tau_{\Bar{v}}},\widetilde{\iota}} \arrow{r}{} \arrow{d}{\mathcal{S}^{\Bar{v},\vee}}
&\widetilde{A}^{\vee}(\widetilde{K}(m',\overline{S}_2),\Tilde{\lambda},\underline{\tau_{\Bar{v}}},\Bar{v}) \arrow{d}{}\\
\mathbf{T}^{T,\lambda_{\Bar{v}},\tau_{\Bar{v}},\iota} \arrow{r}{} & A^{\vee}(K,\lambda,\underline{\tau},q,\Bar{v},m)/I.
\end{tikzcd}$$
Moreover, $N$ can be chosen to only depend on $n$ and $[F^+:\mathbf{Q}]$.
\end{Prop}
\begin{proof}
    One can argue the same way as in the proof of Proposition~\ref{Prop6.1}. In particular, one first proves the dual analogue of Lemma~\ref{Lem6.2} (see \cite{CN23}, Proposition 4.2.4 for how the dual statement might look like). This can be done analogously using \textit{loc. cit.}, Lemma 4.2.3 (taking into account the discussion above the lemma), and replacing Corollary~\ref{Cor3.29} by Corollary~\ref{Cor3.30}.
\end{proof}

\subsection{Middle degree cohomology}
As before, consider a collection of data
\begin{equation*}
    \Bar{v}\in \overline{S}_p \textnormal{, }(\lambda_{v},\lambda_{v^c})\in (\mathbf{Z}_+^n)^{\Hom(F_v,E)}\times(\mathbf{Z}_+^n)^{\Hom(F_{v^c},E)},
\end{equation*}
\begin{equation*}
    Q_v\times Q_{v^c}=P_{(n_1,...,n_k)}\times P_{(m_1,...,m_{k^c})}\subset G_{F^+_{\Bar{v}}}\textnormal{, }
\end{equation*}
\begin{equation*}
    (\underline{\tau_v},\underline{\tau_{v^c}})=((\tau_{v,1},...,\tau_{v,k}),(\tau_{v^c,1},...,\tau_{v^c,k^c})).
\end{equation*}
Let $((\Omega_{v,1},...,\Omega_{v,k}),(\Omega_{v^c,1},...,\Omega_{v^c,k^c}))$ to be the corresponding collection of Bernstein centres. Consider the tuple $(\widetilde{Q}_{\Bar{v}}^{w_0},\Tilde{\lambda}_{\Bar{v}}=(-w_{0,n}\lambda_{v^c},\lambda_v),\underline{\tau_{\Bar{v}}}=(\underline{\tau_{v}},\underline{\tau_{v^c}}))$. Set $\Tilde{\lambda}\in (\mathbf{Z}^{2n}_+)^{\Hom(F^+,E)}$ to be an extension of $\Tilde{\lambda}_{\Bar{v}}$. Let $\widetilde{K}\subset \widetilde{G}(\mathbf{A}_{F^+}^{\infty})$ be a good subgroup such that $\widetilde{K}_{\Bar{v}}=\widetilde{\mathcal{Q}}_{\Bar{v}}(0,c_p)$. Let $\mathfrak{m}\subset \mathbf{T}^T$ be a non-Eisenstein maximal ideal and set $\widetilde{\mathfrak{m}}:=\mathcal{S}^{\ast}(\mathfrak{m})$. Recall that $d=\dim_{\mathbf{C}}(\widetilde{X}_{\widetilde{K}})$. Then the goal of this short subsection is to decompose the $\widetilde{\mathbf{T}}^{T,\lambda_{\Bar{v}},\underline{\tau_{\Bar{v}}}}[1/p]$-module
\begin{equation}\label{eq6.1}
    H^d(\widetilde{X}_{\widetilde{K}},\mathcal{V}_{(\widetilde{\lambda},\underline{\tau_{\Bar{v}}})}^{\widetilde{Q}_{\Bar{v}}^{w_0}})_{\widetilde{\mathfrak{m}}}^{\widetilde{Q}_{\Bar{v}}^{w_0}\textnormal{-ord}}[1/p]
\end{equation}
in terms of cuspidal automorphic representations for $\widetilde{G}$. As without further assumptions the possibility of some Eisenstein series for $G$ contributing to \ref{eq6.1} cannot be ruled out, following \cite{ACC23}, we put an extra assumption on the weight $\Tilde{\lambda}$ to ensure that only cuspidal automorphic representations for $\widetilde{G}$ contribute. We recall the definition here.
\begin{Def}
    A weight $\widetilde{\lambda}\in (\mathbf{Z}_+^{2n})^{\Hom(F^+,E)}$ is CTG ("cohomologically trivial for $G$") if it satisfies the following condition
    \begin{itemize}
        \item Given $w\in W^P$, define $\lambda_w=w(\Tilde{\lambda}+\rho)-\rho$, viewed as an element of $(\mathbf{Z}^n_+)^{\Hom(F,E)}$ as usual where $\rho$ denotes the half-sum of positive roots. For each $w\in W^P$ and $i_0\in \mathbf{Z}$, there exists $\iota\in \Hom(F,E)$ such that $\lambda_{w,\iota}-\lambda_{w,\iota c}^{\vee}\neq (i_0,...,i_0)$.
    \end{itemize}
\end{Def}
We recall that $\widetilde{\mathbf{T}}^{T,\lambda_{\Bar{v}},\underline{\tau_{\Bar{v}}}}[1/p]$ is naturally identified with \begin{equation*}
\widetilde{\mathbf{T}}^{T}[1/p]\otimes_E\left((\bigotimes_{i=1}^k\mathfrak{z}_{\Omega_{v,i}})\otimes (\bigotimes_{i=1}^{k^c}\mathfrak{z}_{\Omega_{v^c,i}})\right).
\end{equation*}
\begin{Prop}\label{Prop6.5}
    Assume that $\widetilde{\mathfrak{m}}=\mathcal{S}^{\ast}(\mathfrak{m})$ is decomposed generic, and $\Tilde{\lambda}$ is CTG. Then, after possibly enlarging $E$,
    \begin{equation*}
H^d(\widetilde{X}_{\widetilde{K}},\mathcal{V}_{(\widetilde{\lambda},\underline{\tau_{\Bar{v}}})}^{\widetilde{Q}_{\Bar{v}}^{w_0}})_{\widetilde{\mathfrak{m}}}^{\widetilde{Q}_{\Bar{v}}^{w_0}\textnormal{-ord}}[1/p]
    \end{equation*}
    is a semisimple $\widetilde{\mathbf{T}}^{T,\lambda_{\Bar{v}},\underline{\tau_{\Bar{v}}}}[1/p]$-module. Moreover, for any homomorphism
    \begin{equation*}
x:\widetilde{\mathbf{T}}^{T,\lambda_{\Bar{v}},\underline{\tau_{\Bar{v}}}}(H^d(\widetilde{X}_{\widetilde{K}},\mathcal{V}_{(\Tilde{\lambda},\underline{\tau_{\Bar{v}}})}^{\widetilde{Q}_{\Bar{v}}^{w_0}})_{\widetilde{\mathfrak{m}}}^{\widetilde{Q}_{\Bar{v}}^{w_0}\textnormal{-ord}})\to \overline{\mathbf{Q}}_p
    \end{equation*}
    and isomorphism $t:\overline{\mathbf{Q}}_p\xrightarrow[]{\sim}\mathbf{C}$, there is a cuspidal automorphic representation $\widetilde{\pi}$ of $\widetilde{G}(\mathbf{A}_{F^+})$ such that $t^{-1}(\widetilde{\pi}_{\Bar{v}}\circ \iota_v^{-1})\otimes V_{\Tilde{\lambda}_{\Bar{v}}}^{\vee}$ is $\widetilde{Q}_{\Bar{v}}^{w_0}$-ordinary (in the sense of Definition~\ref{Def4.13}) and $x$ is induced by the natural Hecke action of
    \begin{equation*}
\widetilde{\mathbf{T}}^{T,\lambda_{\Bar{v}},\underline{\tau_{\Bar{v}}}}[1/p]=\widetilde{\mathbf{T}}^{T}[1/p]\otimes_E\left((\bigotimes_{i=1}^k\mathfrak{z}_{\Omega_{v,i}})\otimes (\bigotimes_{i=1}^{k^c}\mathfrak{z}_{\Omega_{v^c,i}})\right)
    \end{equation*}
    on
    \begin{equation}\label{eq6.2}
        (t^{-1}\widetilde{\pi}^{\{\infty\}\cup T})^{\widetilde{G}(\widehat{\mathcal{O}}_{F^+}^T)}\otimes_{\overline{\mathbf{Q}}_p}\Hom_{\widetilde{M}^{w_0}_{\Bar{v}}(F_v)}\left(\sigma(\underline{\tau_v})\otimes (\theta_n^{-1})^{\ast}\sigma(\underline{\tau_{v^c}}),(t^{-1}(\widetilde{\pi}_{\Bar{v}}\circ \iota_v^{-1}))^{\widetilde{Q}_{\Bar{v}}^{w_0}\textnormal{-ord}}\right).\footnote{For the definition of $(-)^{\widetilde{Q}^{w_0}_{\Bar{v}}\textnormal{-ord}}$, see Definition~\ref{Def4.26}.}
    \end{equation}
\end{Prop}
Before starting the proof, recall that $\mathfrak{z}_{\Omega_{v^c,1}}\otimes...\otimes \mathfrak{z}_{\Omega_{v^c,k^c}}$ acts on the second factor of the $\Hom$ in \ref{eq6.2} via $\theta_n$. Moreover, we emphasise that in the statement we are implicitly using that, by Corollary~\ref{Cor4.21}, $(t^{-1}(\widetilde{\pi}_{\Bar{v}}\circ \iota_v^{-1}))^{\widetilde{Q}_{\Bar{v}}^{w_0}\textnormal{-ord}}$ is irreducible, so $\widetilde{\mathbf{T}}^{T,\lambda_{\Bar{v}},\underline{\tau_{\Bar{v}}}}[1/p]$ indeed acts on \ref{eq6.2} through scalars.
\begin{proof}
    To prove the statement, we show that there is a $\widetilde{\mathbf{T}}^{T,\lambda_{\Bar{v}},\underline{\tau}_{\Bar{v}}}[1/p]$-equivariant direct sum decomposition
    \begin{equation*}
        H^d(\widetilde{X}_{\widetilde{K}},\mathcal{V}_{(\Tilde{\lambda},\underline{\tau_{\Bar{v}}})}^{\widetilde{Q}^{w_0}_{\Bar{v}}})_{\widetilde{\mathfrak{m}}}^{\widetilde{Q}_{\Bar{v}}^{w_0}\textnormal{-ord}}\otimes_{\mathcal{O}}\overline{\mathbf{Q}}_p\cong
    \end{equation*}
    \begin{equation*}
        \bigoplus_{\widetilde{\pi}}d(\widetilde{\pi})(t^{-1}\widetilde{\pi}^{\{\infty\}\cup T})^{\widetilde{G}(\widehat{\mathcal{O}}_{F^+}^T)}\otimes_{\overline{\mathbf{Q}}_p}\Hom_{\widetilde{M}^{w_0}_{\Bar{v}}(F_v)}\left(\sigma(\underline{\tau_v})\otimes (\theta_n^{-1})^{\ast}\sigma(\underline{\tau_{v^c}}),(t^{-1}(\widetilde{\pi}_{\Bar{v}}\circ \iota_v^{-1}))^{\widetilde{Q}_{\Bar{v}}^{w_0}\textnormal{-ord}}\right)
    \end{equation*}
    where the sum runs over cohomological cuspidal automorphic representations of $\widetilde{G}(\mathbf{A}_{F^+})$ of weight $\Tilde{\lambda}$ and $d(\widetilde{\pi})\geq 0$ denotes some integer. A $\widetilde{\mathbf{T}}^T$-equivariant decomposition of this kind is given in the proof of \cite{ACC23}, Theorem 2.4.11. To see that this decomposition is also Hecke equivariant at $\Bar{v}$, we unravel the definition of this action.

    To ease the notation, set $Q:=\widetilde{Q}_{\Bar{v}}^{w_0}\subset \textnormal{GL}_{2n,F_v}$, $M:=\widetilde{M}_{\Bar{v}}^{w_0}$ , and $\mathcal{V}:=\mathcal{V}_{(\widetilde{\lambda},\underline{\tau_{\Bar{v}}})}^{\widetilde{Q}_{\Bar{v}}^{w_0}}$. Recall that we have
    \begin{equation*}
        R\Gamma(\widetilde{X}_{\widetilde{K}},\mathcal{V})_{\widetilde{\mathfrak{m}}}^{Q\textnormal{-ord}}\cong \varprojlim_m R\Gamma(\widetilde{X}_{\widetilde{K}},\mathcal{V}/\varpi^m)_{\widetilde{\mathfrak{m}}}^{Q\textnormal{-ord}}.
    \end{equation*}
    Moreover, as explained in \S\ref{Sec3.4}, the $\widetilde{\mathbf{T}}^{T,\lambda_{\Bar{v}},\underline{\tau_{\Bar{v}}}}$-action is induced by the identification
    \begin{equation*}
        R\Gamma(\widetilde{X}_{\widetilde{K}},\mathcal{V}/\varpi^m)_{\widetilde{\mathfrak{m}}}^{Q\textnormal{-ord}}\cong R\Hom_{M^0}\left(\Tilde{\sigma}(\lambda_{\Bar{v}},\underline{\tau_{\Bar{v}}})^{\circ}/\varpi^m,\pi^{Q\textnormal{-ord}}(\widetilde{K}^{\Bar{v}},\mathcal{V}_{\Tilde{\lambda}^{\Bar{v}}}/\varpi^m)_{\widetilde{\mathfrak{m}}}\right).
    \end{equation*}
Since $\widetilde{\mathfrak{m}}$ is decomposed generic, the cohomology of $\pi^{Q\textnormal{-ord}}(\widetilde{K}^{\Bar{v}},\mathcal{V}_{\Tilde{\lambda}^{\Bar{v}}}/\varpi^m)_{\widetilde{\mathfrak{m}}}$ vanishes for degrees below $d$ thanks to \cite{CS19}, Theorem 1.1. Therefore, a standard argument with a hypercohomology spectral sequence (combined with the previous observations) gives an identification
\begin{equation*}
    H^d(\widetilde{X}_{\widetilde{K}},\mathcal{V}/\varpi^m)^{Q\textnormal{-ord}}_{\widetilde{\mathfrak{m}}}\cong \Hom_{M^0}(\Tilde{\sigma}(\lambda_{\Bar{v}},\underline{\tau_{\Bar{v}}})^{\circ}/\varpi^m,H^d(\pi^{Q\textnormal{-ord}}(\widetilde{K}^{\Bar{v}},\mathcal{V}_{\Tilde{\lambda}^{\Bar{v}}}/\varpi^m)_{\widetilde{\mathfrak{m}}})).
\end{equation*}
Another application of the vanishing result of \cite{CS19} combined with a standard argument with a Hocschild--Serre spectral sequence gives a $\widetilde{\mathbf{T}}^T$-equivariant isomorphism of smooth $\mathcal{O}/\varpi^m[M(F_v)]$-modules
\begin{equation*}
    H^d(\pi^{Q\textnormal{-ord}}(\widetilde{K}^{\Bar{v}},\mathcal{V}_{\Tilde{\lambda}^{\Bar{v}}}/\varpi^m)_{\widetilde{\mathfrak{m}}})\cong \textnormal{Ord}_{Q}(\widetilde{H}^d(\widetilde{K}^{\Bar{v}},\mathcal{V}_{\Tilde{\lambda}^{\Bar{v}}}/\varpi^m)_{\widetilde{\mathfrak{m}}}).
\end{equation*}
Here $\widetilde{H}^d(\widetilde{K}^{\Bar{v}},\mathcal{V}_{\Tilde{\lambda}^{\Bar{v}}}/\varpi^m)$ denotes the degree $d$ $\Bar{v}$-completed cohomology of level $\widetilde{K}^{\Bar{v}}$ and weight $\mathcal{V}_{\Tilde{\lambda}^{\Bar{v}}}/\varpi^m$ and $\textnormal{Ord}_Q(-)$ is Emerton's ordinary part functor from \S\ref{sec4.1}. Using that thanks to finiteness of finite level cohomology Mittag--Leffler applies in our situation, we deduce an identification
\begin{equation}\label{eq6.3}
H^d(\widetilde{X}_{\widetilde{K}},\mathcal{V})_{\widetilde{\mathfrak{m}}}^{Q\textnormal{-ord}}[1/p]\cong \Hom_{M^0}\left(\Tilde{\sigma}(\lambda_{\Bar{v}},\underline{\tau}_{\Bar{v}}),\textnormal{Ord}_Q(\widetilde{H}^d(\widetilde{K}^{\Bar{v}},V_{\Tilde{\lambda}^{\Bar{v}}})_{\widetilde{\mathfrak{m}}})\right).\footnote{For the definition of $\textnormal{Ord}_Q$ applied to $E$-Banach space representations, see the discussion below Remark~\ref{Rem4.7}.}
\end{equation}
The $\widetilde{\mathbf{T}}^{T,\lambda_{\Bar{v}},\underline{\tau_{\Bar{v}}}}$-action on the former then is induced from this identification.
As a consequence of Proposition~\ref{Prop4.8}, \ref{eq6.3} is further identified ($\widetilde{\mathbf{T}}^{T,\lambda_{\Bar{v}},\underline{\tau_{\Bar{v}}}}$-equivariantly) with
\begin{equation*}
    \Hom_{M^0}\left(\Tilde{\sigma}(\lambda_{\Bar{v}},\underline{\tau}_{\Bar{v}}),\textnormal{Ord}_Q(\widetilde{H}^d(\widetilde{K}^{\Bar{v}},V_{\Tilde{\lambda}^{\Bar{v}}})_{\widetilde{\mathfrak{m}}})^{V_{-w_0^Q\Tilde{\lambda}_{\Bar{v}}}\textnormal{-lalg}}\right)\cong 
\end{equation*}
\begin{equation*}
    \Hom_{M^0}\left(\Tilde{\sigma}(\lambda_{\Bar{v}},\underline{\tau}_{\Bar{v}}),\textnormal{Ord}_Q^{\textnormal{lalg}}(\widetilde{H}^d(\widetilde{K}^{\Bar{v}},V_{\Tilde{\lambda}^{\Bar{v}}})_{\widetilde{\mathfrak{m}}}^{V_{\Tilde{\lambda}_{\Bar{v}}^{\vee}}\textnormal{-lalg}})\right).
\end{equation*}
Using Emerton's spectral sequence (cf. \cite{Em06}, Corollary 2.2.18), and the fact that $\widetilde{H}^{\ast}(\widetilde{K}^{\Bar{v}},V_{\Tilde{\lambda}^{\Bar{v}}})_{\widetilde{\mathfrak{m}}}$ vanishes below the middle degree, we see that there is a $\widetilde{\mathbf{T}}^T$-equivariant isomorphism
\begin{equation*}
    \widetilde{H}^{d}(\widetilde{K}^{\Bar{v}},V_{\Tilde{\lambda}^{\Bar{v}}})_{\widetilde{\mathfrak{m}}}^{V_{\Tilde{\lambda}_{\Bar{v}}^{\vee}}\textnormal{-lalg}}\cong \left(\varinjlim_{\widetilde{K}'_{\Bar{v}}}H^d(\widetilde{X}_{\widetilde{K}^{\Bar{v}}\widetilde{K}_{\Bar{v}}'},V_{\Tilde{\lambda}^{\Bar{v}}})_{\widetilde{\mathfrak{m}}}  \right)\otimes_E V_{\Tilde{\lambda}^{\vee}_{\Bar{v}}}
\end{equation*}
of locally algebraic $E$-representations of $\textnormal{GL}_{2n}(F_v)$. Moreover, by the proof of \cite{ACC23}, Theorem 2.4.11, we have a $\widetilde{\mathbf{T}}^T$-equivariant direct sum decomposition
\begin{equation*}
    \left(\varinjlim_{\widetilde{K}'_{\Bar{v}}}H^d(\widetilde{X}_{\widetilde{K}^{\Bar{v}}\widetilde{K}_{\Bar{v}}'},V_{\Tilde{\lambda}^{\Bar{v}}})_{\widetilde{\mathfrak{m}}}  \right)\otimes_E \overline{\mathbf{Q}}_p\cong
\end{equation*}
\begin{equation*}
    \bigoplus_{\widetilde{\pi}}d(\widetilde{\pi})(t^{-1}\widetilde{\pi}^{\{\infty\}\cup T})^{\widetilde{G}(\widehat{\mathcal{O}}_{F^+}^T)}\otimes t^{-1}(\widetilde{\pi}_{\Bar{v}}\circ \iota_v^{-1})
\end{equation*}
of smooth admissible $\overline{\mathbf{Q}}_p[\textnormal{GL}_{2n}(F_v)]$-modules.

In particular, to conclude, it suffices to prove that for any $\widetilde{\pi}$ appearing in the direct sum decomposition above, there is an identification
\begin{equation}\label{eq6.4}
    \Hom_{M^0}\left(\Tilde{\sigma}(\underline{\tau_{\Bar{v}}}),t^{-1}(\widetilde{\pi}_{\Bar{v}}\circ \iota_v^{-1}\right)^{Q\textnormal{-ord}})  \cong
\end{equation}
\begin{equation*}
    \Hom_{M^0}\left(\Tilde{\sigma}(\lambda_{\Bar{v}},\underline{\tau_{\Bar{v}}}),\textnormal{Ord}_Q^{\textnormal{lalg}}\left(t^{-1}(\widetilde{\pi}_{\Bar{v}}\circ \iota_v^{-1})\otimes_E V_{\Tilde{\lambda}_{\Bar{v}}^{\vee}}\right)\right)
\end{equation*}
such that the isomorphism
\begin{equation}\label{eq6.5}
    \mathcal{H}(\Tilde{\sigma}(\underline{\tau_{\Bar{v}}}))\xrightarrow[]{\sim}\mathcal{H}(\Tilde{\sigma}(\lambda_{\Bar{v}},\underline{\tau_{\Bar{v}}})),
\end{equation}
\begin{equation*}
      \phi\mapsto \phi \otimes\textnormal{id}_{-w_0^Q\Tilde{\lambda}_{\Bar{v}}}
\end{equation*}
(provided by \cite{ST06}, Lemma 1.4) intertwines the Hecke actions on the two sides. After noting that $\textnormal{Ord}_Q^{\textnormal{lalg}}(t^{-1}(\Tilde{\pi}_{\Bar{v}}\circ \iota_v^{-1})\otimes_E V_{\Tilde{\lambda}_{\Bar{v}}^{\vee}})\cong (t^{-1}(\widetilde{\pi}_{\Bar{v}}\circ \iota_v^{-1}))^{Q\textnormal{-ord}}\otimes_E V_{-w_0^Q\Tilde{\lambda}_{\Bar{v}}}$, and that $\Tilde{\sigma}(\lambda_{\Bar{v}},\underline{\tau}_{\Bar{v}})=\Tilde{\sigma}(\underline{\tau}_{\Bar{v}})\otimes_EV_{-w_0^Q\Tilde{\lambda}_{\Bar{v}}}$, the isomorphism \ref{eq6.4} is given by the natural isomorphism
\begin{equation*}
    \Hom_{M^0}(\Tilde{\sigma}(\underline{\tau}_{\Bar{v}}),-)\cong \Hom_{M^0}(\Tilde{\sigma}(\underline{\tau}_{\Bar{v}})\otimes_EV_{-w_0^Q\Tilde{\lambda}_{\Bar{v}}},-\otimes_E V_{-w_0^Q\Tilde{\lambda}_{\Bar{v}}}).
\end{equation*}
The induced identification is clearly $\mathcal{H}(\Tilde{\sigma}(\underline{\tau_{\Bar{v}}}))$-equivariant when on the RHS of \ref{eq6.4} we act through \ref{eq6.5}.
\end{proof}

\subsection{The end of the proof}
We are now ready to prove our main local-global compatibility result in this level of generality. The reader is invited to compare it with \cite{CN23}, Proposition 4.2.13. The content of this result is, for every $p$-adic place $v\in S_p(F)$, constructing a characteristic zero lift of $\rho_{\mathfrak{m}}|_{G_{F_v}}$ with the right shape according to the tuple $(Q_v,\lambda_v,\underline{\tau}_v)$. We then show that in the case of $Q_v=\textnormal{GL}_n$ the existence of such a lift easily implies Theorem~\ref{Th5.7}.
\begin{Prop}\label{Prop6.6}
    Assume that $p$ splits in an imaginary quadratic subfield of $F$. Let $K\subset G(\mathbf{A}_{F^+}^{\infty})$ be a good subgroup and fix distinct places $\Bar{v},\Bar{v}'\in \overline{S}_p$. Fix a preferred lift $v\mid \Bar{v}$ in $S_p(F)$. Let $(Q_p,\lambda,\underline{\tau})=(Q_v,\lambda_v,\underline{\tau}_v)_{v\in S_p(F)}$ be a tuple as in \S\ref{sec2.7} and assume that $K_p\subset \mathcal{Q}_p(0,c_p)$ and $K_{\Bar{v}}=\mathcal{Q}_v(0,c_p)\times \mathcal{Q}_{v^c}(0,c_p)$.

    Let $m\in \mathbf{Z}_{\geq 1}$ be an integer, and $\mathfrak{m}\subset \mathbf{T}^{T,\lambda_{\Bar{v}},\underline{\tau_{\Bar{v}}}}$ be a maximal ideal in the support of $H^{\ast}(X_K,\mathcal{V}_{(\lambda,\underline{\tau})}^{Q_p}/\varpi^m)^{Q_v\times Q_{v^c}\textnormal{-ord}}$. Write
    \begin{equation*}
        Q_v=P_{(n_1,...,n_k),F_v}\subset \textnormal{GL}_{n,F_v}\textnormal{, and }Q_{v^c}=P_{(m_1,...,m_{k^c}),F_{v^c}}\subset \textnormal{GL}_{n,F_{v^c}}.
    \end{equation*}
    Assume that:
    \begin{enumerate}
        \item We have \begin{equation*}
        \sum_{\Bar{v}''\neq \Bar{v},\Bar{v}'}[F^+_{\Bar{v}''}:\mathbf{Q}_p]\geq \frac{1}{2}[F^+:\mathbf{Q}]
    \end{equation*}
    where the sum runs over $\Bar{v}''\in S_p(F^+)$.
    \item The maximal ideal $\mathfrak{m}$ is non-Eisenstein such that $\overline{\rho}_{\mathfrak{m}}$ is decomposed generic.
    \item Let $v\notin T$ be a finite place of $F$, with residue characteristic $\ell$. Then either $T$ contains no $\ell$-adic places and $\ell$ is unramified in $F$, or there exists an imaginary quadratic subfield of $F$ in which $\ell$ splits.
    \end{enumerate}
    Then, for each $q\in [0,d-1]$, there exists an integer $N\geq 1$, depending only on $n$ and $[F^+:\mathbf{Q}]$, a nilpotent ideal $I\subset \mathbf{T}^{T,\lambda_{\Bar{v}},\underline{\tau_{\Bar{v}}}}\left(H^q(X_K,\mathcal{V}_{(\lambda,\underline{\tau})}^{Q_p}/\varpi^m)_{\mathfrak{m}}^{Q_{\Bar{v}}\textnormal{-ord}}\right)$ with $I^N=0$, and a continuous $n$-dimensional representation
    \begin{equation*}
        \rho_{\mathfrak{m}}: G_{F,T}\to \textnormal{GL}_n\left(\mathbf{T}^{T,\lambda_{\Bar{v}},\underline{\tau_{\Bar{v}}}}\left(H^q(X_K,\mathcal{V}_{(\lambda,\underline{\tau})}^{Q_p}/\varpi^m)_{\mathfrak{m}}^{Q_{\Bar{v}}\textnormal{-ord}}\right)/I\right)
    \end{equation*}
    such that the following conditions hold:
\begin{enumerate}
    \item For each finite place $v\notin T$ of $F$, the characteristic polynomial of $\rho_{\mathfrak{m}}(\textnormal{Frob}_v)$ is equal to the image of $P_v(X)$.
    \item For $\Tilde{v}=v,v^c$, the representation $\rho_{\mathfrak{m}}|_{G_{F_v}}$  has a lift to $\widetilde{\rho}_{\Tilde{v}}:G_{F_{\Tilde{v}}}\to \textnormal{GL}_n(\widetilde{A})$, where $\widetilde{A}$ is a finite flat local $\mathcal{O}$-algebra equipped with a $\mathfrak{z}_{\lambda_{\Tilde{v}},\underline{\tau_{\Tilde{v}}}}^{\circ}$-algebra structure such that $\widetilde{A}[1/p]\cong \prod_{x}E$ is a semisimple $E$-algebra with a morphism
    \begin{equation*}
        \overline{\mathcal{S}}^{\Tilde{v}}:\widetilde{A}\to
        \mathbf{T}^{T,\lambda_{\Bar{v}},\underline{\tau_{\Bar{v}}}}\left(H^q(X_K,\mathcal{V}_{(\lambda,\underline{\tau})}^{Q_p}/\varpi^m)_{\mathfrak{m}}^{Q_{\Bar{v}}\textnormal{-ord}}\right)/I
    \end{equation*}
    of $\mathfrak{z}_{\lambda_{\Tilde{v}},\underline{\tau_{\Tilde{v}}}}^{\circ}$-algebras.
    \item For $\Tilde{v}=v,v^c$, $\Tilde{\rho}_{\Tilde{v}}[1/p]$ is potentially semistable with labelled Hodge--Tate weights $(\lambda_{\iota,1}+n-1>...>\lambda_{\iota,n})_{\iota:F_{\Tilde{v}}\hookrightarrow E}$.
    \item Moreover, these lifts admit isomorphisms
    \begin{equation*}
        \Tilde{\rho}_{v}[1/p]\sim \begin{pmatrix}
\Tilde{\rho}_{v,1} & \ast & ... & \ast\\
0 & \Tilde{\rho}_{v,2} & ... & \ast\\
. & . & . & .\\
. & . & . & .\\
0 & ... & 0 & \Tilde{\rho}_{v,k}
\end{pmatrix}\textnormal{, }\Tilde{\rho}_{v^c}[1/p]\sim \begin{pmatrix}
\Tilde{\rho}_{v^c,1} & \ast & ... & \ast\\
0 & \Tilde{\rho}_{v^c,2} & ... & \ast\\
. & . & . & .\\
. & . & . & .\\
0 & ... & 0 & \Tilde{\rho}_{v^c,k^c}
\end{pmatrix}
    \end{equation*}
    such that, for $1\leq i\leq k$, and $1\leq j\leq k^c$, $\Tilde{\rho}_{v,i}$, respectively $\Tilde{\rho}_{v^c,j}$ has Weil--Deligne inertial type bounded by $\tau_{v,i}$, respectively $\tau_{v^c,j}$. Moreover, the labelled Hodge--Tate weights of $\Tilde{\rho}_{v,i}$, respectively $\Tilde{\rho}_{v^c,j}$ are determined by the property that they are increasing as $i$, respectively $j$ grows.
    \item For $\Tilde{v}=v$ or $v^c$, an integer $1\leq i\leq \Tilde{k}$ for $\Tilde{k}=k$ or $\Tilde{k}=k^c$ (depending on $\Tilde{v}$), and a morphism $x: \widetilde{A}\to E$ the following property is satisfied. The morphism $\mathfrak{z}_{\Omega_{\Tilde{v},i}}\to E$ induced by postcomposing the structure map $\mathfrak{z}_{\Omega_{\Tilde{v},i}}\to \widetilde{A}$ with $x$ coincides with the one induced by the natural action of $\mathfrak{z}_{\Omega_{\Tilde{v},i}}$ on
    \begin{equation*}
        \textnormal{rec}^{T,-1}(\textnormal{WD}(x\circ \Tilde{\rho}_{\Tilde{v},i}))\otimes |\cdot|^{n_1+...+n_{i-1}}.\footnote{Here the convention is that $n_0:=0$.}
    \end{equation*}
\end{enumerate}
\end{Prop}
\begin{proof}
    As the argument follows very closely the proof of \cite{CN23}, Proposition 4.2.13, we only provide a sketch. Since the existence of $\rho_{\mathfrak{m}}$ satisfying the first condition is known, we are free to enlarge $T$.\footnote{This will be used in the various upcoming twisting arguments.} By various twisting arguments, and an application of the Hocschild--Serre spectral sequence, we can assume the following:
    \begin{itemize}
        \item $\overline{\rho}_{\widetilde{\mathfrak{m}}}$ is decomposed generic.
        \item $K=K(m,\overline{S}_p\setminus \{\Bar{v}\})$.
        \item For $\Bar{v}''\in \overline{S}_p\setminus\{\Bar{v},\Bar{v}'\}$, $\lambda_{\Bar{v}''}=0$, and $\tau_{\Bar{v}''}$ is trivial.
        \item For each $\iota: F_v\hookrightarrow E$, $-\lambda_{\iota c,1}-\lambda_{\iota,1}\geq 0$, and $\Tilde{\lambda}=(-w_0^{\textnormal{GL}_n}\lambda_{v^c},\lambda_v)_{\Bar{v}\in \overline{S}_p}$ is CTG.\footnote{See Remark~\ref{Rem6.7} for the role of this condition.}
    \end{itemize}
    In particular, by setting $\widetilde{K}\subset \widetilde{G}(\mathbf{A}_{F^+}^{\infty})$ to be a good subgroup satisfying
    \begin{itemize}
        \item $\widetilde{K}^{\Bar{v}}\cap G(\mathbf{A}_{F^+}^{\infty})=K^{\Bar{v}}$;
        \item $\widetilde{K}^T=\widetilde{G}(\widehat{\mathcal{O}}_{F^+}^T)$;
        \item For $\Bar{v}''\in \overline{S}_p\setminus \{\Bar{v}\}$, $U(\mathcal{O}_{F^+_{\Bar{v}}})\subset \widetilde{K}_{\Bar{v}''}$, $\widetilde{K}=\widetilde{K}(m,\overline{S}_p\setminus \{\Bar{v}\})$;
        \item $\widetilde{K}_{\Bar{v}}=\widetilde{\mathcal{Q}}_{\Bar{v}}^{w_0}(0,c_p)$ (associated with $Q_{\Bar{v}}$);
    \end{itemize}
    and assuming that $q\in [\lfloor \frac{d}{2}\rfloor, d-1]$, we are in the situation of Proposition~\ref{Prop6.1}.
    
    We briefly explain what we mean by "twisting argument" as it will be also used in the next step of the proof. Assume that we have a continuous character $\overline{\chi}:G_F\to k^{\times}$. Since for the upcoming step this is the only relevant case, we also assume that $\overline{\chi}$ is unramified at $S_p$. We set $\chi:G_F\to \mathcal{O}^{\times}$ to be its Teichm\"uller lift. Choose a finite set of finite places $T\subset T'$ of $F$ that is closed under complex conjugation, containing all places at which $\chi$ is ramified, and a good normal subgroup $K'\subset K$ satisfying:
\begin{itemize}
    \item $(K')^{T'\setminus T}=K^{T'\setminus T}$.
    \item $K/K'$ is abelian of order prime to $p$.
    \item For each finite place $v$ of $F$, $\chi|_{G_{F_v}}\circ \textnormal{Art}_{F_v}$ is trivial on $\det(K'_v)$.
    \item $T'$ again satisfies assumption iii) from the Proposition.
\end{itemize}
Then set $\mathfrak{m}(\chi)\subset \mathbf{T}^{T,\lambda_{\Bar{v}},\underline{\tau_{\Bar{v}}}}$ to be $f_{\chi}(\mathfrak{m})$, where $f_{\chi}:\mathbf{T}^{T,\lambda_{\Bar{v}},\underline{\tau_{\Bar{v}}}}\xrightarrow{\sim} \mathbf{T}^{T,\lambda_{\Bar{v}},\underline{\tau_{\Bar{v}}}}$ is the map defined in the discussion preceding Lemma~\ref{Lemma2.5}. Note that $\overline{\rho}_{\mathfrak{m}(\chi)}=\overline{\rho}_{\mathfrak{m}}\otimes \chi$. Set
\begin{equation*}
A(K',\lambda,\underline{\tau},q,\Bar{v},m,\chi):=\mathbf{T}^{T,\lambda_{\Bar{v}},\underline{\tau_{\Bar{v}}}}\left(H^q(X_{K'},\mathcal{V}_{(\lambda,\underline{\tau_{\Bar{v}}})}^{Q_p}/\varpi^m)_{\mathfrak{m}(\chi)}^{Q_{\Bar{v}}\textnormal{-ord}}\right).
\end{equation*}
\begin{claim}
    Verifying the Proposition for any of the $\mathfrak{m}(\chi)$'s with corresponding level $K'$ will imply it for $\mathfrak{m}$ with level $K$.
\end{claim}
\begin{proof}[Proof of Claim]
    Note that we have a surjection 
    \begin{equation*}
    A(K',\lambda,\underline{\tau},q,\Bar{v},m)\to A(K,\lambda,\underline{\tau},q,\Bar{v},m)
    \end{equation*} of $\mathbf{T}^{T,\lambda_{\Bar{v}},\underline{\tau_{\Bar{v}}}}$-algebras so we can and do assume that $K=K'$. Use the abbreviation $A(\chi):=A(K,\lambda,\underline{\tau},q,\Bar{v},m,\chi)$. Assume that we found an integer $N_{\chi}$, a nilpotent ideal $I_{\chi}\subset A(\chi)$, a continuous representation $\rho_{\mathfrak{m}(\chi)}:G_{F,T}\to \textnormal{GL}_n(A(\chi)/I_{\chi})$, a surjection $\widetilde{A}(\chi)\twoheadrightarrow A(\chi)/I_{\chi}$ of $\mathfrak{z}_{\lambda_v,\underline{\tau_{v}}}^{\circ}$-algebras, and a continuous representation $\widetilde{\rho}_{v,\chi}:G_{F_v}\to \textnormal{GL}_n(\widetilde{A})$ satisfying the conditions of the Proposition. Set $N:=N_{\chi}$, $I:=f_{\chi^{-1}}(I_{\chi})$, $\rho_{\mathfrak{m}}:=(f_{\chi^{-1}}\circ \rho_{\mathfrak{m}(\chi)})\otimes \chi^{-1}$, and $\widetilde{A}:=f_{\chi^{-1}}^{\ast}\widetilde{A}(\chi)$ as a $\mathfrak{z}_{\lambda_{v},\underline{\tau_{v}}}^{\circ}$-algebra. Define the surjection $\widetilde{A}\to A/I$ to be $\widetilde{A}(\chi)\to A(\chi)/I_{\chi}\cong^{f_{\chi^{-1}}}A/I$ and set $\widetilde{\rho}_v:=(f_{\chi^{-1}}\circ \widetilde{\rho}_{v,\chi})\otimes \chi^{-1}:G_{F_v}\to \textnormal{GL}_n(\widetilde{A})$. Then, as already observed at the end of the proof of \cite{ACC23}, Corollary 4.4.8, $\rho_{\mathfrak{m}}$ satisfies the first condition of the Proposition. An easy diagram chase shows that condition ii) is satisfied. Conditions iii) and iv) are clearly preserved under twisting by $\chi$, so they are also satisfied. To check the last condition, pick $x:\widetilde{A}\to E$, and consider one of the subquotients $\widetilde{\rho}_{v,\chi,i}$ of $\widetilde{\rho}_{v,\chi}[1/p]$ from condition iii) for $\mathfrak{m}(\chi)$. Then the corresponding subquotient of $\widetilde{\rho}_v[1/p]$ is $\widetilde{\rho}_{v,\chi,i}\otimes \chi^{-1}$, and the morphism $\mathfrak{z}_{\Omega_{v,i}}\to E$ is induced by $\mathfrak{z}_{\Omega_{v,i}}\to \widetilde{A}(=\mathfrak{z}_{\Omega_{v,i}}\xrightarrow{f_{\chi}}\mathfrak{z}_{\Omega_{v,i}}\to \widetilde{A}(\chi))$. Then condition v) follows from the computation
    \begin{equation*}
        \textnormal{rec}^{T,-1}(\textnormal{WD}(x\circ \widetilde{\rho}_{v,i}))=\textnormal{rec}^{T,-1}(\textnormal{WD}(x\circ \widetilde{\rho}_{v,\chi,i})\otimes \chi^{-1})=
    \end{equation*}
    \begin{equation*}
        \textnormal{rec}^{T,-1}(\textnormal{WD}(x\circ \widetilde{\rho}_{v,\chi,i}))\otimes (\chi^{-1}\circ \textnormal{Art}_{F_v})
    \end{equation*}
    and the definition of $f_{\chi}:\mathfrak{z}_{\Omega_{v,i}}\to \mathfrak{z}_{\Omega_{v,i}}$.
\end{proof}
For now we assume that $q\in [\lfloor \frac{d}{2}\rfloor, d-1]$. Then, according to Proposition~\ref{Prop6.1}, for every choice of $\overline{\chi}$ such that $\widetilde{\mathfrak{m}(\chi)}$ is decomposed generic, we obtain an integer $N_{\chi}$, a nilpotent ideal $I_{\chi}$, a flat $\mathcal{O}$-algebra $\widetilde{A}(\chi)$, and a map $\widetilde{A}(\chi)\to A(K',\lambda,\underline{\tau},q,\Bar{v},m,\chi)/I_{\chi}$. Write $\widetilde{A}(\chi)=\prod_x E$ using Proposition~\ref{Prop6.5}. Applying Theorem~\ref{Th2.19}, and Theorem~\ref{Th4.30} to each $x:\widetilde{A}\to E$ (keeping in mind Proposition~\ref{Prop6.5}), we obtain a continuous representation
\begin{equation*}
    \widetilde{\rho}_{\mathfrak{m}(\chi)}:=\prod_x\widetilde{\rho}_{\mathfrak{m}(\chi),x}:G_F\to \textnormal{GL}_{2n}(\widetilde{A}(\chi)[1/p]),
\end{equation*}
admitting a block upper-triangular shape with blocks of size $(n_1,...,n_k,m_{k^c},...,m_1)$. In particular, for each $x$, we have an isomorphism
\begin{equation*}
    \widetilde{\rho}_{\mathfrak{m}(\chi),x}|_{G_{F_v}}\sim \begin{pmatrix}
r_{1,\chi,x} & \ast \\
0 & r_{2,\chi,x}
\end{pmatrix}.
\end{equation*}
As shown in Sub-Lemma 1 in the proof of \cite{CN23}, Proposition 4.2.13, we can find $\overline{\chi}$ such that
\begin{itemize}
    \item the set of isomorphism classes of the Jordan-H\"older constituents of $\overline{\rho}_{\mathfrak{m}(\chi)}|_{G_{F_v}}$ and $\overline{\rho}_{\mathfrak{m}(\chi)}^{\vee,c}(1-2n)|_{G_{F_v}}$ are disjoint;
    \item and, for all $x:\widetilde{A}\to E$, the isomorphism classes of the Jordan-H\"older constituents of $\overline{r}_{1,\chi,x}$ coincide with those of $\overline{\rho}_{\mathfrak{m}(\chi)}|_{G_{F_v}}$.
\end{itemize}
For such a $\chi$, we can apply \cite{CN23}, Proposition 3.2.4 to obtain a $\widetilde{A}(\chi)$-valued lift $\widetilde{\rho}_{v,\chi}:G_{F_v}\to \textnormal{GL}_n(\widetilde{A}(\chi))$ of $\rho_{\mathfrak{m}(\chi)}|_{G_{F_v}}$ such that $\widetilde{\rho}_{v,\chi}[1/p]$ becomes isomorphic to $\prod_xr_{1,\chi,x}$. In particular, by Theorem~\ref{Th4.30} and Proposition~\ref{Prop6.5}, conditions iii), iv), and v) are all satisfied for $\widetilde{\rho}_{v,\chi}$. This finishes the proof in the case of $q\in [\lfloor \frac{d}{2}\rfloor, d-1]$.

To treat the case of $q< \lfloor \frac{d}{2}\rfloor$, one uses the Poincar\'e duality isomorphisms
\begin{equation*}
    \iota: A(K,\lambda,\underline{\tau},q,\Bar{v},m)\cong A^{\vee}(K,\lambda,\underline{\tau},d-1-q,\Bar{v},m),
\end{equation*}
\begin{equation*}
\widetilde{\iota}:\widetilde{\mathbf{T}}^{T,\lambda_{\Bar{v}},\underline{\tau_{\Bar{v}}}}\left(H^d(\widetilde{X}_{\widetilde{K}},\mathcal{V}_{(\Tilde{\lambda},\underline{\tau_{\Bar{v}}})}^{\widetilde{Q}_{\Bar{v}},w_0^P}[1/p])_{\Tilde{\iota}^{\ast}\mathcal{S}^{\ast}(\mathfrak{m}^{\vee})}^{\widetilde{Q}_{\Bar{v}}\textnormal{-ord}}\right)\cong \widetilde{A}^{\vee}(\widetilde{K},\Tilde{\lambda},\underline{\tau_{\Bar{v}}},\Bar{v})[1/p]
\end{equation*}
provided by Proposition~\ref{Prop3.20}. Then an analogous argument using the dual degree shifting (Proposition~\ref{Prop6.3}), and a version of Proposition~\ref{Prop6.5} for the cohomology group $H^d(\widetilde{X}_{\widetilde{K}},\mathcal{V}_{(\Tilde{\lambda},\underline{\tau_{\Bar{v}}})}^{\widetilde{Q}_{\Bar{v}},w_0^P}[1/p])_{\Tilde{\iota}^{\ast}\mathcal{S}^{\ast}(\mathfrak{m}^{\vee})}^{\widetilde{Q}_{\Bar{v}}\textnormal{-ord}}$ proves the Proposition also for $q$. For more details, see \cite{CN23}, Proposition 4.2.13.
\end{proof}
\begin{Rem}\label{Rem6.7}
    We emphasise the role of reserving the place $\Bar{v}'$ in the degree shifting argument (cf. Proposition~\ref{Prop6.1}). Namely, combined with the assumption that the level is deep enough at $\Bar{v}'$, it allows us to freely change the weight $\lambda_{\Bar{v}}$ without changing the faithful Hecke algebra $A(K,\lambda,\underline{\tau},q,\Bar{v},m)$. This lets us assume in the proof of Proposition~\ref{Prop6.6} that $\Tilde{\lambda}$ is CTG (cf. \cite{ACC23}, Lemma 4.3.6). This is crucial for us in order to have access to Proposition~\ref{Prop6.5}.

    Nevertheless, as was pointed out to the author by James Newton, once Theorem~\ref{Th5.7}, in particular, Theorem~\ref{Th1.1.1} is proved, one could try and run the proof of Theorem~\ref{Th5.7} with a version of Proposition~\ref{Prop6.5} where we drop the CTG assumption. This way the role of $\Bar{v}'$ would not be relevant anymore, leading to a strengthening of Theorem~\ref{Th5.7} where condition i) is weakened to asking for
    \begin{equation}
        \sum_{\Bar{v}''\in \overline{S}_p,\Bar{v}''\neq \Bar{v}}[F^+_{\Bar{v}''}:\mathbf{Q}_p]\geq \frac{1}{2}[F^+:\mathbf{Q}]
    \end{equation}
    instead. The author is planning to revisit this strategy in future work.
\end{Rem}

\begin{proof}[Proof of Theorem~\ref{Th5.7}]
Pick a collection of data $(F, K,\lambda,\tau,\mathfrak{m}, \Bar{v} )$ as in the statement of the Theorem. Pick a choice of lift $v\mid \Bar{v}$ in $S_p(F)$.
We will apply Proposition~\ref{Prop6.6} with the choice $(Q_v,\lambda_v,\underline{\tau_v})_{v\in S_p(F)}:=(\textnormal{GL}_n,\lambda_v,\tau_v)$.

Note that it suffices to prove the statement for Galois representations with coefficients in $\mathbf{T}^{T,\lambda,\tau}(R\Gamma(X_K,\mathcal{V}_{(\lambda,\tau)}/\varpi^m)_{\mathfrak{m}})$ for integers $m\in \mathbf{Z}_{\geq 1}$. This is because, by \cite{NT16}, Lemma 3.11, we have an isomorphism
\begin{equation*}
    \mathbf{T}^{T,\lambda,\tau}(R\Gamma(X_K,\mathcal{V}_{(\lambda,\tau)})_{\mathfrak{m}})\xrightarrow{\sim}\varprojlim_m \mathbf{T}^{T,\lambda,\tau}(R\Gamma(X_K,\mathcal{V}_{(\lambda,\tau)}/\varpi^m)_{\mathfrak{m}})
\end{equation*}
of $\mathbf{T}^{T,\lambda,\tau}$-algebras. Moreover, by an argument using \cite{KT17}, Lemma 2.5, and Carayol's lemma (cf. the proof of \cite{ACC23} Corollary 4.4.8, and Theorem 4.5.1), it suffices to verify the Theorem one cohomological degree at a time. Therefore, fix $q\in [0,d-1]$, and set $\mathfrak{m}_{\Bar{v}}\subset \mathbf{T}^{T,\lambda_{\Bar{v}},\tau_{\Bar{v}}}$ to be the maximal ideal obtained by pulling back $\mathfrak{m}$ along $\mathbf{T}^{T,\lambda_{\Bar{v}},\tau_{\Bar{v}}}\hookrightarrow \mathbf{T}^{T,\lambda,\tau}$. Set $A_{\Bar{v}}:=\mathbf{T}^{T,\lambda_{\Bar{v}},\tau_{\Bar{v}}}(H^q(X_K,\mathcal{V}_{(\lambda,\tau)}/\varpi^m)_{\mathfrak{m}_{\Bar{v}}})$. It is then clearly sufficient to verify the statement for a Galois representation with coefficients in $A_{\Bar{v}}$.

Applying Proposition~\ref{Prop6.6} gives a nilpotent ideal $I_{\Bar{v}}\subset A_{\Bar{v}}$, a continuous representation $\rho_{\mathfrak{m}}:G_F\to \textnormal{GL}_n(A_{\Bar{v}}/I_{\Bar{v}})$, a finite flat local $\mathcal{O}$-algebra $\widetilde{A}$, a surjective homomorphism $\overline{\mathcal{S}}^{v}:\widetilde{A}\to A_{\Bar{v}}/I_{\Bar{v}}$ of $\mathfrak{z}_{\lambda_{v},\tau_v}^{\circ}$-algebras, and a lift $\widetilde{\rho}_v:G_{F_v}\to \textnormal{GL}_n(\widetilde{A})$ of $\rho_v:=\rho_{\mathfrak{m}}|_{G_{F_v}}$ under $\overline{\mathcal{S}}^{v}$ satisfying the five listed conditions. In particular, by the first condition, it satisfies local-global compatibility outside $T$, and the only thing left to check is the existence of a dotted arrow making the diagram in the statement of the Theorem commutative.

We first note that an easy unravelling of the definitions (using condition iii), iv), and v) from Proposition~\ref{Prop6.6}) shows that we have a (necessarily unique) dotted arrow making the diagram
\begin{equation}\label{eq6.3.6}
    \begin{tikzcd}
	{R^{\Box}_{\overline{\rho}_v}[1/p]} & {\widetilde{A}[1/p]} \\
	{R^{\lambda_v,\preceq\tau_v}_{\overline{\rho}_v}[1/p]} & {\mathfrak{z}_{\Omega_v}}
	\arrow["{\widetilde{\rho}_v[1/p]}", from=1-1, to=1-2]
	\arrow[two heads, from=1-1, to=2-1]
	\arrow["\eta", from=2-2, to=2-1]
	\arrow["{\textnormal{nat}_{\widetilde{A}}}"', from=2-2, to=1-2]
	\arrow[dotted, from=2-1, to=1-2]
\end{tikzcd}
\end{equation}
commutative. Here $\textnormal{nat}_{\widetilde{A}}$ denotes the natural map towards the faithful Hecke algebra and $\eta$ denotes the interpolation of local Langlands.
We then obtain the following diagram
$$\begin{tikzcd}
	{\mathfrak{z}_{\Omega_v}} & {R_{\overline{\rho}_v}^{\lambda_v,\preceq\tau_v}[1/p]} & {\widetilde{A}[1/p]} \\
	& {R_{\overline{\rho}_v}^{\lambda_v,\preceq\tau_v}} & {\widetilde{A}} & A_{\Bar{v}}/I_{\Bar{v}} \\
	& {\mathfrak{z}_{\lambda_v,\tau_v}^{\circ}}
	\arrow["{\widetilde{\rho}_v[1/p]}", from=1-2, to=1-3]
	\arrow["\eta", from=1-1, to=1-2]
	\arrow[hook, from=2-2, to=1-2]
	\arrow["{\widetilde{\rho}_v}", from=2-2, to=2-3]
	\arrow[hook, from=2-3, to=1-3]
	\arrow["{\overline{\mathcal{S}}^{v}}", two heads, from=2-3, to=2-4]
	\arrow[ hook', from=3-2, to=1-1]
	\arrow["{\textnormal{nat}_{\widetilde{A}}}", from=3-2, to=2-3]
	\arrow["{\textnormal{nat}_{A_{\Bar{v}}}}"', from=3-2, to=2-4]
\end{tikzcd}$$
where all the inclusions are the natural ones induced by inverting $p$. Note that, by abuse of notation, we write Kisin's local deformation rings as the source of $\widetilde{\rho}_v[1/p]$ resp. $\widetilde{\rho}_v$. This is justified by \ref{eq6.3.6}. Moreover, by the very definition of $\widetilde{\rho}_v$, we have $\overline{\mathcal{S}}^{v}\circ \widetilde{\rho}_v=\rho_{\mathfrak{m}}|_{G_{F_v}}$. In particular, we see that in the local-global compatibility diagram of Theorem~\ref{Th5.7} we indeed have a dotted arrow making the upper triangle commutative. To see that the obtained arrow also commutes with the lower triangle, we need to prove the following.
\begin{claim}
    If $z\in \mathfrak{z}_{\lambda_{v},\tau_v}^{\circ}$ such that $\eta(z)$ lies in $R^{\lambda_v,\preceq \tau_v}_{\overline{\rho}_v}$, then $\textnormal{nat}_{A_{\Bar{v}}}(z)=\rho_v(z)$.
\end{claim}
\begin{proof}[Proof of Claim]
    We have
    \begin{equation*}
    \textnormal{nat}_{A_{\Bar{v}}}(z)=\overline{\mathcal{S}}^{v}(\textnormal{nat}_{\widetilde{A}}(z))=\overline{\mathcal{S}}^{v}(\widetilde{\rho}_v[1/p](\eta(z)))=
    \end{equation*}
    \begin{equation*}
        \overline{\mathcal{S}}^{v}(\widetilde{\rho}_v(\eta(z)))=\rho_v(\eta(z)),
    \end{equation*}
    where the first equality is by the definition of the degree shifting map, the second equality is the content of \ref{eq6.3.6}, and the last equality follows from the definition of $\widetilde{\rho}_v$.
\end{proof}

\end{proof}

\begin{proof}[Proof of Theorem~\ref{Th1.1.1}]
    Consider an imaginary CM field $F$, an identification $t:\overline{\mathbf{Q}}_p\cong \mathbf{C}$, and a regular algebraic cuspidal automorphic representation $\pi$ of $\textnormal{GL}_n(\mathbf{A}_F)$ as in the statement of the Theorem. Fix also a $p$-adic place $v$ in $F$ where we wish to prove local-global compatibility. We can then find a cyclic CM field extension $F'/F$ such that $v$ and $v^c$ split completely in $F'$, $F'$ is linearly disjoint from $\overline{F}^{\ker \overline{r_t(\pi)}}$, $\overline{r_t(\pi)}|_{G_{F'}}$ stays decomposed generic, and $F'$ and $v$ satisfy the conditions of Theorem~\ref{Th5.7}. In particular, it suffices to prove local-global compatibility for the cyclic base change $\pi':=\textnormal{BC}_{F'/F}(\pi)$ for any place $v'|v$ in $F'$.

    Let $\tau$ be the Weil--Deligne inertial type of $\pi'$, set $\lambda'$ to be its weight, and note that $\lambda_{v'}=\lambda_v$. Let $T$ be a suitable finite set of places containing $S_p(F')$ such that $(\pi')^{T\cup\{\infty\}}$ is unramified. Pick a good subgroup $K^p\subset \textnormal{GL}_n(\mathbf{A}_{F'}^{\infty,p})$ such that $((\pi')^{p})^{K^p}\neq 0$. Set $K=K^pK_p^0$ where $K^0_p:=\prod_{v\in S_p(F')}\textnormal{GL}_n(\mathcal{O}_{F'_{v}})$. Then, by Theorem~\ref{Th2.12},
    \begin{equation*}
        \Hom_{K^0_p}(\sigma(\tau),(\pi')^{K^p})\neq 0,
    \end{equation*}
    and $\mathbf{T}^{T,\lambda',\tau}[1/p]$ acts via scalars, inducing a map $x:\mathbf{T}^{T,\lambda',\tau}[1/p]\to \overline{\mathbf{Q}}_p$. For a large enough field extension $E/\mathbf{Q}_p$, \cite{Fr98}, and \cite{FS98} show (cf. \cite{ACC23}, Theorem 2.4.10) that $(\pi')^{K^pK_p'}$ can be found in
    \begin{equation*}
        H^{\ast}(X_{K^pK_p'},\mathcal{V}_{(\lambda',\tau)}[1/p])
    \end{equation*}
    as a $\mathbf{T}^T$-equivariant direct summand for any compact open normal subgroup $K'_p\subset K^0_p$ such that $\sigma(\tau)|_{K_p^0}$ is trivial. Since finite group cohomology is torsion, an argument with Hocschild--Serre spectral sequence shows that
    $\Hom_{K_p^0}(\sigma(\tau),(\pi')^{K^p})$ can be found in
    \begin{equation*}
        H^{\ast}(X_K,\mathcal{V}_{(\lambda',\tau)}[1/p])
    \end{equation*}
    as a $\mathbf{T}^T$-equivariant direct summand. To see that this direct summand is also $\mathbf{T}^{T,\lambda',\tau}[1/p]$-equivariant one has to compare the natural action of $\mathbf{T}^{T,\lambda',\tau}[1/p]$ (cf. \S\ref{sec2.6}) on
    \begin{equation*}
        \Hom_{K_p^0}(\sigma(\tau),(\pi')^{K^p})
    \end{equation*}
with the one on
\begin{equation*}
    H^{\ast}(X_K,\mathcal{V}_{(\lambda',\tau)}[1/p])
\end{equation*}
given by Lemma~\ref{Lem2.4}. The two induced actions on the direct summand can be seen to coincide by writing both actions in terms of correspondences as in \S\ref{sec2.4}. 

In the previous paragraph we proved that the map $x$ factors through $\mathbf{T}^{T,\lambda',\tau}(K^p)_{\mathfrak{m}}$ where $\mathfrak{m}\subset \mathbf{T}^{T,\lambda',\tau}$ is the (non-Eisenstein and decomposed generic) maximal ideal so that $\overline{r_t(\pi')}\cong \overline{\rho}_{\mathfrak{m}}$. The Theorem then follows from applying Theorem~\ref{Th5.7} to $\mathfrak{m}$ and specialising the diagram in the statement at $x$.
\end{proof}

\appendix

\section{Bernstein--Zelevinsky and Langlands classifications}
Since the theory of Bernstein--Zelevinsky and its relation to the Langlands classification is used at several points in \S\ref{sec4.2}, we recollect the necessary results of the theory. Throughout this section, we fix a finite field extension $L/\mathbf{Q}_p$ with a choice of uniformiser $\varpi_L$ and, for any integer $m\in \mathbf{Z}_{\geq 1}$, we set $G_m:=\textnormal{GL}_m(L)$. All of our representations will have $\overline{\mathbf{Q}}_p$-coefficients and recall that we fix an identification $t:\overline{\mathbf{Q}}_p\cong \mathbf{C}$. For a smooth representation $\pi$ of $G_m$, we set $\deg(\pi):=m$.

For an integer $n\in\mathbf{Z}_{\geq1}$ with a partition $n=n_1+...+n_s$, consider the corresponding standard parabolic subgroup $Q\subset G_n$ with its Levi decomposition $Q=M\ltimes N$. Recall that for any smooth representations $\pi_i$ of $G_{n_i}$ $(i=1,...,s)$ we denoted by $\textnormal{n-Ind}_Q^{G_n}\pi_1\otimes...\otimes \pi_s$ the normalised parabolic induction. For this appendix, we will abbreviate this by writing
\begin{equation*}
    \pi_1\times...\times \pi_s:=\textnormal{n-Ind}_Q^{G_n}\pi_1\otimes...\otimes \pi_s.
\end{equation*}
Moreover, for any smooth representation $\pi$, we set $J_Q(\pi)$ to be the (unnormalised) Jacquet module of $\pi$ with respect to $Q$. By \cite{Zel80}, Proposition 1.4, both operations carry finite length representations to finite length ones. (See \cite{Zel80}, Proposition 1.1 for some of the properties of these functors.)

For any smooth irreducible representation $\pi$ of $G_n$, there is a unique multi-set $\{\pi_1,...,\pi_s\}$ of supercuspidal representations of auxiliary general linear groups $G_{n_i}$ such that $\pi$ is a subquotient of $\pi_1\times...\times \pi_s$. We will refer to this multi-set as the \textit{supercuspidal} support of $\pi$.\footnote{We warn the reader that this is slightly unconventional as in the literature it is the pair $(\pi_1\otimes...\otimes \pi_s, G_{n_1}\times...\times G_{n_s})$ that is referred to as the supercuspidal support of $\pi$.} Moreover, the supercuspidal support can be ordered in a way so that $\pi$ is actually a subrepresentation of $\pi_1\times...\times \pi_s$ (cf. \cite{BZ77}, Theorem 2.5, Theorem 2.9).

Given a smooth representation $\pi$ of $G_n$, for some $n\geq 1$, and a real number $a\in \mathbf{R}$, set $\pi(a)$ to be $\pi\otimes |\textnormal{det}|_L^a$. Zelevinsky introduced the notion of a segment, a set of isomorphism classes of irreducible supercuspidal representations of the form
\begin{equation*}
    \Delta(\pi,r):=\{\pi,\pi(1),...,\pi(r-1)\}
\end{equation*}
where $r\geq 1$ is some integer. We also use the notation $[\pi,\pi(r-1)]$ for the set $\Delta(\pi,r)$. We further say that two segments $\Delta_1$ and $\Delta_2$ are linked if neither contains the other and $\Delta_1\cup \Delta_2$ is also a segment. Finally, for two linked segments $\Delta_1=[\pi(r_1),\pi(r_2)]$ and $\Delta_2=[\pi(u_1),\pi(u_2)]$, respectively, we say that $\Delta_1$ precedes $\Delta_2$ if $r_1\leq u_1$. Given a segment $\Delta=\Delta(\pi,r)$, we write
\begin{equation*}
    \pi(\Delta):=\pi\times...\times \pi(r-1).
\end{equation*}
Then the Bernstein--Zelevinsky and Langlands classifications read as follows (cf. \cite{Zel80}, Theorem 6.1, \cite{Rod82}, Theorem 3).
\begin{Th}\label{ThA.1}
\begin{enumerate}
    \item Given a segment $\Delta=\Delta(\pi,m)$, $\pi(\Delta)$ has length $2^{m-1}$ and admits a unique irreducible subrepresentation $Z(\Delta)$ and a unique irreducible quotient $L(\Delta)$ (called the Langlands quotient of the segment).
    \item Call an ordered multiset of segments $(\Delta_1,...,\Delta_l)$ \textit{well-ordered}, if it is ordered in a way that, for $i<j$, $\Delta_i$ \textit{doesn't} precede $\Delta_j$. Then, for such a well-ordered multiset of segments, the representation
    \begin{equation*}
        Z(\Delta_1)\times...\times Z(\Delta_l)
    \end{equation*}
    admits a unique irreducible subrepresentation $Z(\Delta_1,...,\Delta_l)$. Similarly, the representation 
    \begin{equation*}
        L(\Delta_1)\times...\times L(\Delta_l)
    \end{equation*}
    admits a unique irreducible quotient $L(\Delta_1,...,\Delta_l)$ (called the Langlands quotient of the multiset of segments $(\Delta_1,...,\Delta_l)$). Moreover, in both cases, the isomorphism class of the obtained representation does not depend on the chosen order.
    \item For any smooth irreducible representation $\pi$ of $G_n$, up to reordering, there is a unique well-ordered multiset of segments
    $(\Delta_1,...,\Delta_l)$ resp. $(\Delta_1',...,\Delta_h')$ such that $\pi\cong Z(\Delta_1,...,\Delta_l)$ resp. $\pi\cong L(\Delta_1',...,\Delta_h')$.
\end{enumerate}
\end{Th}

The relation between the Bernstein--Zelevinsky and Langlands classifications is slightly subtle. Nevertheless, Zelevinsky introduced an involution which allows one to pass between the two to some extent. Namely, for every $n\geq 1$, one can look at the Grothendieck group $\mathcal{R}_n$ of finite length smooth representations of $G_n$. Then the operation of normalised parabolic induction makes $\mathcal{R}:=\oplus_{n\geq 0}\mathcal{R}_n$ into a graded commutative ring (cf. \cite{Zel80}, 1.9). Moreover, according to \cite{Zel80}, Corollary 7.5, $\mathcal{R}$ is in fact a polynomial algebra over $\mathbf{Z}$ with indeterminates given by $Z(\Delta)$. Therefore, one can define a ring endomorphism
\begin{equation*}
    \mathcal{D}:\mathcal{R}\to\mathcal{R}
\end{equation*}
by sending, for any segment $\Delta$, the element $Z(\Delta)$ to $L(\Delta)$ and linearly extending it. One observes that $\mathcal{D}$ is in fact an involution which sends $Z(\Delta(\pi,r))$ to $Z(\pi,...,\pi(r-1))$ (cf. \cite{Zel80}, 9.15). Moreover, Zelevinsky conjectured that $\mathcal{D}$ sends irreducible representations to irreducible representations (\cite{Zel80}, 9.17). Assuming Zelevinsky's conjecture, Rodier deduced that, for a (well-ordered) multiset of segments $(\Delta_1,...,\Delta_l)$, $\mathcal{D}$ sends $Z(\Delta_1,...,\Delta_l)$ resp. $L(\Delta_1,...,\Delta_l)$ to $L(\Delta_1,...,\Delta_l)$ resp. $Z(\Delta_1,...,\Delta_l)$. Bernstein proposed a proof of the conjecture which stayed unpublished. The first written up proofs can be found in \cite{Aub95} and \cite{Pro98}, respectively. Another property of the involution which will be useful for us is the fact that it commutes with the operation of twisting by the determinant character. This can be easily seen by first verifying it on the generators of $\mathcal{R}$ as a polynomial algebra over $\mathbf{Z}$ corresponding to segments.

Given the Langlands classification, we can explain the reduction of the local Langlands correspondence for $\textnormal{GL}_n$ to a correspondence between the set of supercuspidal representations and irreducible Weil--Deligne representations. Namely, given a supercuspidal representation $\pi$ of $G_n$, \cite{HT01,Hen00} attaches to it an $n$-dimensional irreducible Weil--Deligne representation $\textnormal{rec}^T(\pi)$. Moreover, if $\pi$ is only essentially square integrable i.e., according to Bernstein, it is of the form $L(\Delta(\pi',r))$ (\cite{Zel80}, Theorem 9.3), we set $\textnormal{rec}^T(\pi)=\textnormal{rec}^T(\pi')\otimes \textnormal{Sp}(r)$ where $\textnormal{Sp}(r)$ is the Steinberg representation (see page 213 of \cite{Rod82}). Finally, if $\pi$ is a general smooth irreducible representation of the form $L(\Delta_1,...,\Delta_l)$, we set $\textnormal{rec}^T(\pi)=\oplus_{i=1}^l\textnormal{rec}^T(L(\Delta_i))$.

Since we will mainly be working with smooth irreducible representations of $G_n$ coming from cuspidal automorphic representations, we can always assume that our representations $\pi$ are so that $t \pi\otimes | \textnormal{det} |^{-s}$ is unitary for some $s\in \mathbf{R}$. We will refer to such representations as \textit{$t$-preunitary} and if $s=0$, then we further call it \textit{$t$-unitary}. In particular, it will be useful for us to recall Tadic's classification of unitary irreducible smooth representations of $G_n$ \cite{Tad86} which reveals that, in the case of unitary representations, the relation between the Bernstein--Zelevinsky and Langlands classifications becomes explicit. To state the classification, we first need to introduce some further notation. Given a $t$-unitary supercuspidal representation $\pi$ of $G_m$ for some $m\in\mathbf{Z}_{\geq 1}$ and an integer $d\in\mathbf{Z}_{\geq 1}$, we define the unitary segment
\begin{equation*}
    \Delta^u(d,\pi):=[\pi(\frac{1-d}{2}),...,\pi(\frac{d-1}{2})].
\end{equation*}
Then the building blocks of the classification will be the following two types of representations.
\begin{enumerate}
    \item Given $d,n\in\mathbf{Z}_{\geq 1}$ and $\pi$ a $t$-unitary supercuspidal representation of $G_m$ for some $m\in\mathbf{Z}_{\geq 1}$, we set
    \begin{equation*}
        a(d,n,\pi)=Z\left(\Delta^u(d,\pi)(\frac{n-1}{2}),...,\Delta^u(d,\pi)(\frac{1-n}{2})\right).
    \end{equation*}
    \item Given $d,n\in\mathbf{Z}_{\geq 1}$, $\alpha\in (0,\frac{1}{2})$ and $\pi$ a $t$-unitary supercuspidal representation of $G_m$ for some $m\in\mathbf{Z}_{\geq 1}$, we set
    \begin{equation*}
        a(d,n,\pi,\alpha)=a(d,n,\pi)(\alpha)\times a(d,n,\pi)(-\alpha).
    \end{equation*}
\end{enumerate} The classification is as follows (\cite{Tad86}, Theorem A, Theorem B).

\begin{Th}\label{ThA.2}
Given an integer $m\geq 1$, all smooth representations of $G_m$ obtained as normalised parabolic induction of type \textit{i)} and type \textit{ii)} representations are $t$-unitary and irreducible. 

Moreover, any smooth irreducible $t$-unitary representation of $G_m$ can be obtained this way and the associated multiset of type \textit{i)} and type \textit{ii)} representations is well-defined (i.e., the associated ordered multiset is unique up to permutation).

Finally, under the Zelevinsky involution a type \textit{i)} representation $a(d,n,\pi)$, respectively a type $\textit{ii)}$ representation $a(d,n,\pi,\alpha)$ is sent to $a(n,d,\pi)$, respectively $a(n,d,\pi,\alpha)$. In particular, we have
\begin{equation*}
a(d,n,\pi)=a^L(n,d,\pi):=L\left(\Delta^u(n,\pi)(\frac{d-1}{2}),...,\Delta^u(n,\pi)(\frac{1-d}{2})\right),
\end{equation*}
respectively
\begin{equation*}
a(d,n,\pi,\alpha)=a^L(d,n,\pi)(\alpha)\times a^L(d,n,\pi)(-\alpha).
\end{equation*}
\end{Th}

We conclude the appendix with two technical lemmas, the first of which relies on Tadic's classification. The role of the first computation is to understand the monodromy under local Langlands of the ordinary support in \S\ref{sec4.2} (see Corollary~\ref{Cor4.25}).
\begin{Lemma}\label{LemA.3}
Let $\{\Delta_1,...,\Delta_l\}$ be a multiset of segments, $r\in\mathbf{Z}_{\geq 2}$ and $0=k_0<k_1<...<k_r=l$ be integers such that
\begin{equation*}
    \frac{\textnormal{val}_p(\pi_1'(\langle\varpi\rangle))}{\deg\pi_1'}\leq \frac{\textnormal{val}_p(\pi_2'(\langle\varpi\rangle))}{\deg\pi_2'}+\textnormal{val}_p(|\varpi|_K)
\end{equation*}
whenever $\pi_1'$ resp. $\pi_2'$ lies in the underlying supercuspidal support of $\{\Delta_{k_{i-1}+1},...,\Delta_{k_i}\}$ resp. $\{\Delta_{k_{j-1}+1},...,\Delta_{k_j}\}$ for $i<j$. In particular, if $\pi_1'$ happens to be of the form $\pi_2'\otimes|\textnormal{det}|^s$, then  $s\geq 1$. Assume that for each $1\leq i\leq r$, $(\Delta_{k_{i-1}+1},...,\Delta_{k_i})$ is well-ordered.\footnote{Note that this ensures that already $(\Delta_1,...,\Delta_l)$ is well-ordered.} Finally, assume that
\begin{equation*}
    \widetilde{\pi}:=Z(\Delta_1,...,\Delta_l)=L(\Delta_1',...,\Delta_h')
\end{equation*}
is $t$-preunitary and, for $i=1,...,r$, set
\begin{equation*}
    \pi_i:=Z(\Delta_{k_{i-1}+1},...,\Delta_{k_i})=L(\Delta_{t_{i-1}+1}'',...,\Delta_{t_i}'')
\end{equation*}
where $0=t_0<t_1<...<t_r$ are the appropriate integers. Then every segment $\Delta_i'$ is of the form
\begin{equation*}
    \bigcup_{j=1}^r\Delta'''_j
\end{equation*}
with $\Delta'''_j\in \{\emptyset,\Delta_{k_{j-1}+1}'',...,\Delta_{k_j}''\}$.
\end{Lemma}
\begin{proof}
Note that we can assume that $\widetilde{\pi}$ is $t$-unitary. This follows from the fact that $\textnormal{rec}^T$ is compatible with character twists and that the involution $\mathcal{D}$ commutes with twisting by the determinant character. Arguing by induction on $r$, we can further restrict ourselves to the case when $r=2$. Therefore, we do assume the above restrictions and set $k:=k_1$. In particular, we have access to Tadic's classification. This means that, since we know the shape of $(\Delta_1',...,\Delta_h')$ by Theorem~\ref{ThA.2}, we need to compute the shape of the segments $(\Delta''_1,...,\Delta_{t_2}'')$. 

First, we pin down the possible unitary representations appearing in Tadic's classification as building blocks which have underlying multiset of segments strictly separated by 
\begin{equation}\label{parteq}
    \{(\Delta_1,...,\Delta_k),(\Delta_{k+1},...,\Delta_l)\}.
\end{equation} By the assumption on the central characters of the supercuspidal support of the partition
\ref{parteq},
the representations of the second type appearing in $\widetilde{\pi}=Z(\Delta_1,...,\Delta_l)$'s description via Tadic's classification have associated multiset of segments lying either in $(\Delta_1,...,\Delta_k)$ or $(\Delta_{k+1},...,\Delta_l)$, respectively. Indeed, this is because all the neighbours of the supercuspidal support of such a representation differ by a power of the determinant character with exponent strictly smaller than $1$. Now assume that a representation of the first type appears in $\widetilde{\pi}$ and its associated multiset of segments is strictly separated by the partition $\{(\Delta_1,...,\Delta_k),(\Delta_{k+1},...,\Delta_l)\}$. Say it is given by $a(d,n,\pi)$ for some integers $d,n\geq 1$ and $t$-unitary supercuspidal representation $\pi$ of $G_m$ for some integer $m\geq 1$. Note that in the case when $n=1$, we have only $1$ segment so $n$ must be at least $2$. On the other hand, notice that $d$ must be equal to $1$. Indeed, if we assume that $d$ is at least $2$, then all the neighbours of
\begin{equation*}
    (\Delta^u(d,\pi)(\frac{n-1}{2}),...,\Delta^u(d,\pi)(\frac{1-n}{2}))
\end{equation*}
have overlapping supercuspidal supports. In particular, by the assumption on the central characters, all the segments must be contained in $(\Delta_1,...,\Delta_k)$ or $(\Delta_{k+1},...,\Delta_{l})$, respectively. This leads to a contradiction. Therefore, $d$ necessarily equals $1$ and
\begin{equation*}
    a(d,n,\pi)=Z(\pi(\frac{n-1}{2}),...,\pi(\frac{1-n}{2})).
\end{equation*}

We now have the following claim.
\begin{claim}
    $Z(\Delta_1,...\Delta_k)$ is obtained by applying normalised parabolic induction to a product of representations of type \textit{i)}, type \textit{ii)} and representations of the form
    \begin{equation*}
        Z(\pi(\frac{n-1}{2}),...,\pi(j))=L([\pi(j),...,\pi(\frac{n-1}{2})])
    \end{equation*}
    for some $t$-unitary supercuspidal representation $\pi$ of $G_m$, $n\in \mathbf{Z}_{\geq 0}$ and $\frac{1-n}{2}\leq j\leq \frac{n-1}{2}$.
\end{claim}
\begin{proof}[Proof of Claim]
For a choice of $t$-unitary supercuspidal representation of $G_m$, set $S_{\pi}\subset (\Delta_1,...,\Delta_k)$ to be the ordered subset of segments with supercuspidal support lying in $\{\pi\otimes|\det|^s\}_{s\in \mathbf{R}}$. By \cite{Zel80}, Proposition 8.5, $Z(\Delta_1,...,\Delta_k)$ is the parabolic induction of the representations $Z(S_{\pi})$ for all $\pi$ such that $S_{\pi}\neq \emptyset$.\footnote{Note that the order does not matter.} Moreover, $S_{\pi}$ decomposes into $S_{\pi}^{\textnormal{int}}$ and $S_{\pi}^{\textnormal{nonint}}$ where the first ordered subset consists of the segments with supercuspidal support lying in $\{\pi\otimes |\det|^{s}\}_{s\in \frac{1}{2}\mathbf{Z}}$ and define the latter to be its complement. Again, $Z(S_{\pi})$ is the parabolic induction of $Z(S_{\pi}^{\textnormal{int}})$ and $Z(S_{\pi}^{\textnormal{nonint}})$.

By looking at Tadic's classification, we see that $S_{\pi}^{\textnormal{nonint}}$ consists of all the segments corresponding to type \textit{ii)} representations and nothing else. In particular, using Tadic's classification and \cite{Zel80}, Proposition 8.4, we see that it must be the parabolic induction of type \textit{ii)} representations. We are left with spelling out $Z(S^{\textnormal{int}}_{\pi})$.

We further divide $S_{\pi}^{\textnormal{int}}$ into the ordered multiset of segments $S^{\textnormal{nonres}}_{\pi}$ given by the collection of segments which correspond to a type \textnormal{i)} representations which are not separated by the partition \ref{parteq} and $S_{\pi}^{\textnormal{res}}$ consisting of segments that are part of the multiset of segments of a type \textit{i)} representation which is strictly separated by the partition \ref{parteq}. Then, by \cite{Zel80}, Proposition 8.4, $Z(S^{\textnormal{int}}_{\pi})$ occurs with multiplicity $1$ in the set of Jordan-H{\"o}lder factors of
\begin{equation}\label{eqB.2}
    Z(S^{\textnormal{nonres}}_{\pi})\times Z(S_{\pi}^{\textnormal{res}}).
\end{equation}
 Therefore, in order to prove the claim, it suffices to prove that \ref{eqB.2} is irreducible and that both $Z(S^{\textnormal{nonres}}_{\pi})$ and $Z(S^{\textnormal{res}}_{\pi})$ have the claimed shape.

Note that by our previous discussion $Z(S^{\textnormal{res}}_{\pi})$ is of the form
\begin{equation*}
    Z(\{\pi(\frac{n_1-1}{2}),...,\pi(j),\pi(\frac{n_2-1}{2}),...,\pi(j),...,\pi(\frac{n_f-1}{2}),...,\pi(j)\}^{\textnormal{ord}})
\end{equation*}
where $\{-\}^{\textnormal{ord}}$ means that we chose any allowable ordering of the underlying multiset of segments, $f\in \mathbf{Z}_{\geq 0}$, $n_1,...,n_f\in \mathbf{Z}_{\geq 1}$ and $j$ is some integer between $\frac{1-n_i}{2}$ and $\frac{n_i-1}{2}$ for any $1\leq i\leq f$. This is simply given by 
\begin{equation}\label{eqB.3}
    Z(\pi(\frac{n_1-1}{2}),...,\pi(j))\times ....\times Z(\pi(\frac{n_f-1}{2}),...,\pi(j)).
\end{equation}
To see this, note that \ref{eqB.3} can be rewritten as
\begin{equation*}
    L([\pi(j),...,\pi(\frac{n_1-1}{2})])\times ....\times L([\pi(j),...,\pi(\frac{n_f-1}{2})]).
\end{equation*}
This is irreducible by \cite{Zel80}, Proposition 9.7 so it must be $Z(S_{\pi}^{\textnormal{res}})$ by \cite{Zel80}, Proposition 8.4.

By combining Tadic's classification with \cite{Zel80}, Proposition 8.4, we also see that $Z(S_{\pi}^{\textnormal{nonres}})$ is of the form
\begin{equation*}
    a(\Tilde{d}_1,\Tilde{n}_1,\pi)\times...\times a(\Tilde{d}_h,\Tilde{n}_h,\pi)=a^L(\Tilde{n}_1,\Tilde{d}_1,\pi)\times...\times a^L(\Tilde{n}_h,\Tilde{d}_h,\pi)
\end{equation*}
for some $h\in \mathbf{Z}_{\geq 0}$, $\Tilde{n}_1,...,\Tilde{n}_h,\Tilde{d}_1,...,\Tilde{d}_h\in \mathbf{Z}_{\geq 1}$.

We are left with proving that \ref{eqB.2} is irreducible. Note that if we apply Zelevinsky duality to \ref{eqB.2} bearing in mind the previous discussion and Theorem~\ref{ThA.2}, we get
\begin{equation*}
    \Bigl(a(\Tilde{n}_1,\Tilde{d}_1,\pi)\times...\times a(\Tilde{n}_h,\Tilde{d}_h,\pi)\bigr)
    \times
\end{equation*}
\begin{equation*}
    \left(Z([\pi(j),...,\pi(\frac{n_1-1}{2})])\times ....\times Z([\pi(j),...,\pi(\frac{n_f-1}{2})])
    \right).
\end{equation*}
By Proposition 8.5 of \cite{Zel80}, it suffices to see that for any $1\leq a\leq h$, $\frac{1-\Tilde{d}_a}{2}\leq b\leq \frac{\Tilde{d}_a-1}{2}$ and $1\leq c\leq f$, the segments
\begin{equation*}
    [\pi(\frac{1-\Tilde{n}_a}{2}+b),...,\pi(\frac{\Tilde{n}_a-1}{2}+b)]\textnormal{ and }[\pi(j),...,\pi(\frac{n_c-1}{2})]
\end{equation*}
are not linked. Note that by assumption $j\in \frac{1}{2}\mathbf{Z}$ is the lowest such that $\pi(j)$ appears as a supercuspidal support of some segment in $S_{\pi}^{\textnormal{int}}$. Therefore, we have
\begin{equation*}
    j\leq \frac{1-\Tilde{n}_a}{2}+\frac{1-\Tilde{d}_a}{2}\leq \frac{1-\Tilde{n}_a}{2}+b.
\end{equation*}
On the other hand, if the two segments were linked, we would need to have
\begin{equation*}
    \frac{n_c-1}{2}<\frac{\Tilde{n}_a-1}{2}+b\leq \frac{\Tilde{n}_a-1}{2}+\frac{\Tilde{d}_a-1}{2}.
\end{equation*}
So, by multiplying by $-1$, we would get
\begin{equation*}
    \frac{1-\Tilde{n}_a}{2}+\frac{1-\Tilde{d}_a}{2}\leq \frac{1-n_c}{2}<j,
\end{equation*}
a contradiction. Therefore, \ref{eqB.2} is indeed irreducible and we proved the claim.
\end{proof}
The same observations apply to $Z(\Delta_{k+1},...,\Delta_{l})$ and the lemma now follows. To be more precise, from the claim we see what can be the shape of a segment $\Delta''_{i_1}$ appearing in $Z(\Delta_1,...,\Delta_k)=L(\Delta_1'',...,\Delta_{t_1}'')$. Moreover, the proof of the claim shows that if any such segment is of the form $[\pi(j),...,\pi(\frac{n-1}{2})]$ with $\frac{1-n}{2}< j$, then there must be a corresponding segment $\Delta_{i_2}$ with $t_1+1\leq i_2\leq t_2$ such that they are disjoint and linked and their union is $[\pi(\frac{1-n}{2}),...,\pi(\frac{n-1}{2})]$. Such pairs then build up to segments in $L(\Delta_1',...,\Delta_h')$ and any other segments of the latter were already a segment of $(\Delta''_1,..,\Delta''_{t_2})$.
\end{proof}

\begin{Lemma}\label{LemA.4}
Let $(\Delta_1,...,\Delta_l)$ be a well-ordered multiset of segments such that $Z(\Delta_1,...,\Delta_l)$ has degree $n$ and set $Q_{\textnormal{sc}}=M_{\textnormal{sc}}\ltimes N_{\textnormal{sc}}\subset G_n$ to be the standard parabolic subgroup corresponding to the underlying supercuspidal support with the induced ordering. Then the Jacquet module $J_{Q_{\textnormal{sc}}}(Z(\Delta_1,...,\Delta_l))$ admits $\delta_{Q_{\textnormal{sc}}}^{1/2}\Delta_1\otimes...\otimes\Delta_l$ as a quotient where, by abuse of notation, for a segment $\Delta=[\pi,...,\pi(r-1)]$ we also denote by $\Delta$ the representation $\pi\otimes...\otimes \pi(r-1)$.
\end{Lemma}
\begin{proof}
Let $Q=M\ltimes N\subset G_n$ be the parabolic subgroup corresponding to $Z(\Delta_1)\otimes...\otimes Z(\Delta_l)$. Then Frobenius reciprocity applied to Theorem~\ref{ThA.1} gives a surjection $J_Q(Z(\Delta_1,...,\Delta_l))\to \delta_{Q}^{1/2}Z(\Delta_1)\otimes...\otimes Z(\Delta_l)$. Then \cite{Zel80} Proposition 1.1 parts a) and c) combined with \textit{loc. cit.} 3.1 implies that taking Jacquet module with respect to $Q_{\textnormal{sc}}$ gives the desired result.
\end{proof}

\medskip
\printbibliography

\end{document}